\documentclass{amsart}
\usepackage{etex, stmaryrd, epsfig, graphics, psfrag, latexsym,
  mathtools, mathrsfs,enumitem, longtable,booktabs,ifthen,amsfonts, 
amsmath, amsthm,  amssymb,xcolor,bm,caption}

\setlist[itemize]{leftmargin=*}
\setlist[enumerate]{leftmargin=*}

\usepackage{ifthen,calc, sseq,  pdflscape, pdfpages}

\usepackage[all,cmtip]{xy}

\usepackage{tikz} \usetikzlibrary{matrix,arrows,shapes,calc,patterns,3d}

\definecolor{mygreen}{rgb}{0.09, 0.45, 0.27}
\definecolor{darkblue}{rgb}{0,0,0.4} 

\usepackage{xr-hyper}
\usepackage[colorlinks=true, citecolor=darkblue, filecolor=darkblue, linkcolor=darkblue,urlcolor=darkblue]{hyperref} 
\usepackage[all]{hypcap}

\allowdisplaybreaks

\usepackage[color=blue!20!white]{todonotes}
\makeatletter
\providecommand\@dotsep{5}
\def\listtodoname{List of Todos}
\def\listoftodos{\@starttoc{tdo}\listtodoname}
\makeatother

\newtheorem{lem}{Lemma}[section]               
\newtheorem{lemma}[lem]{Lemma}    

\newtheorem{theorem}[lem]{Theorem}

\newtheorem{proposition}[lem]{Proposition}

\theoremstyle{definition}

\newtheorem{definition}[lem]{Definition}

\theoremstyle{remark}

\newtheorem{remark}[lem]{Remark}

\newtheorem{example}[lem]{Example}

\newtheorem{convention}[lem]{Convention}
\numberwithin{equation}{section}

\newcommand{\RR}{\mathbb{R}}
\newcommand{\R}{\mathbb{R}}
\newcommand{\HalfSpace}{\mathbb{H}}
\newcommand{\Affine}{\mathbb{A}}

\newcommand{\ZZ}{\mathbb{Z}}
\newcommand{\Z}{\mathbb{Z}}
\newcommand{\FF}{\mathbb{F}}
\newcommand{\CC}{\mathbb{C}}
\newcommand{\C}{\mathbb{C}}

\newcommand{\NN}{\mathbb{N}}

\newcommand{\pt}{\mathrm{pt}}

\newcommand{\CP}{\mathbb{CP}}
\newcommand{\cp}{\mathit{CP}}

\newcommand{\smim}{\Sigma}

\newcommand{\mf}{\mathfrak}

\newcommand{\wh}{\widehat}
\newcommand{\wt}{\widetilde}
\newcommand{\ol}{\overline}

\newcommand{\del}{\partial}
\newcommand{\bdy}{\del}

\renewcommand{\emptyset}{\varnothing}

\newcommand{\maximal}{\mathfrak{m}}
\newcommand{\word}{\mathfrak{w}}

\newcommand{\smas}{\wedge}

\newcommand{\card}[1]{\left\vert{#1}\right\vert}
\newcommand{\set}[2]{\{#1\mid#2\}}

\newcommand{\domains}{\mathscr{D}}
\newcommand{\pdomains}{\mathscr{D}^+}
\newcommand{\periodic}{\mathscr{P}}
\newcommand{\pperiodic}{\mathscr{P}^+}
\newcommand{\rectangles}{\mathscr{R}}
\newcommand{\const}[1]{c_{#1}}
\newcommand{\coefficients}[1]{\mathbb{O}(#1)}
\newcommand{\Xcoefficients}[1]{\mathbb{X}(#1)}

\newcommand{\from}{\nobreak\mskip2mu\mathpunct{}\nonscript
  \mkern-\thinmuskip{:}\penalty300\mskip6muplus1mu\relax}

\newcommand{\into}{\hookrightarrow}

\newcommand{\defeq}{\mathrel{\vcenter{\baselineskip0.5ex
      \lineskiplimit0pt \hbox{\scriptsize.}\hbox{\scriptsize.}}}=}

\renewcommand{\th}{^{\text{th}}}

\newcommand{\fr}{{\operatorname{fr}}}

\newcommand{\Spinc}{\text{Spin}^{c}}

\newcommand{\Filt}{\mathcal{F}}

\newcommand{\diff}{\delta}

\newcommand{\Grid}{\mathbb{G}}
\newcommand{\Heeg}{\mathcal{H}}

\newcommand{\HF}{\mathit{HF}}
\newcommand{\CFK}{\mathit{CFK}}
\newcommand{\CFKhat}{\widehat{\CFK}}

\newcommand{\HFK}{\mathit{HFK}}
\newcommand{\HFKhat}{\widehat{\HFK}}
\newcommand{\HFhat}{\widehat{\mathit{HF}}}
\newcommand{\HFKnew}{{\HFK}^{\diamond}\!}
\newcommand{\CFL}{\mathit{CFL}}

\newcommand{\HFL}{\mathit{HFL}}

\newcommand{\CD}{\mathit{CD}}
\newcommand{\CDP}{\mathit{CDP}}
\newcommand{\DNlambda}{\mathscr{T}}
\newcommand{\DNlambdad}{\DNlambda^{\dagger}}
\newcommand{\Cp}{\mathit{GC}}
\newcommand{\Ct}{\widetilde{\mathit{GC}}}
\newcommand{\Chat}{\widehat{\mathit{GC}}}

\DeclareMathOperator{\Sym}{Sym}
\DeclareMathOperator{\sh}{sh}

\DeclareMathOperator{\Ob}{Ob}
\DeclareMathOperator{\Id}{Id}
\DeclareMathOperator{\Sq}{Sq}
\DeclareMathOperator{\colim}{colim}

\DeclareMathOperator{\Hom}{Hom}

\DeclareMathOperator{\image}{im}

\DeclareMathOperator{\Cone}{Cone}

\DeclareMathOperator{\hocolim}{hocolim}

\newcommand{\Cat}{\mathscr{C}}
\newcommand{\Bat}{\mathscr{B}}
\newcommand{\Dat}{\mathscr{D}}

\newcommand{\olsi}[1]{\,\overline{\!{#1}}}
\newcommand{\Moduli}{\mathcal{M}}
\newcommand{\bModuli}{\olsi{\mathcal{M}}}
\newcommand{\bM}{\bModuli}
\newcommand{\Noduli}{\mathcal{N}}
\newcommand{\bNoduli}{\olsi{\mathcal{N}}}
\newcommand{\bN}{\bNoduli}

\newcommand{\bX}{\olsi{X}}

\newcommand{\gr}{\mathrm{gr}}

\newcommand{\codim}[1]{\langle #1\rangle}

\newcommand{\basis}[1]{\langle #1\rangle}

\newcommand{\SphereS}{\mathbb{S}} 

\renewcommand{\Re}{\operatorname{Re}}
\renewcommand{\Im}{\operatorname{Im}}

\newcommand{\vp}{\vec p}
\newcommand{\Part}{\operatorname{Part}}

\def\Set{\mathscr{S}}
\def\Tube{\mathscr{T}}

\def\vM{\vec{M}}
\def\vN{\vec{N}}
\def\vP{\vec{P}}
\def\vlambda{\vec{\lambda}}
\def\veta{\vec{\eta}}
\def\Ta{\mathbb{T}_{\alpha}}
\def\Tb{\mathbb{T}_{\beta}}
\def\Os{\mathbb{O}}

\def\S{\mathbb{S}}
\def\tOmega{\tilde{\Omega}}
\def\EC{\operatorname{EC}}
\def\UE{\operatorname{UE}}
\def\IR{\operatorname{IR}}
\def\FR{\operatorname{FR}}
\def\ve{\vec{e}}

\def\vv{\vec{v}}
\def\vw{\vec{w}}
\def\vmu{\vec{\mu}}
\def\vnu{\vec{\nu}}

\def\dI{\delta^{\textrm{I}}}
\def\dII{\delta^{\textrm{II}}}
\def\dIII{\delta^{\textrm{III}}}
\def\dIV{\delta^{\textrm{IV}}}
\def\xid{x^{\Id}}
\def\bZ{\olsi{Z}}

\def\deltasupp{\delta\kern-2pt\operatorname{-supp}}
\def\thick{\operatorname{th}} 
\def\CDPd{\CDP_*^{\dagger}}

\def\nueps{\nu_{\epsilon}}
\def\E{\mathbb{E}}
\def\tD{\widetilde{D}}
\def\tE{\widetilde{E}}
\def\tF{\widetilde{F}}
\def\tdim{\operatorname{tdim}}
\def\Nbhd{\mathcal{N}}

\def\inte{\operatorname{int}}
\def\tZ{\widetilde{Z}}
\def\smo{\operatorname{sm}}
\def\L{\mathcal{L}}
\def\dold{\bdy^{\operatorname{old}}}

\def\gCtilde{g\widetilde{\mathit{GC}}}

\def\gChat{g\widehat{\mathit{GC}}}
\def\GS{\mathcal{X}}
\def\gGS{g\mathcal{X}}
\def\tGS{\widetilde{\GS}}
\def\gtGS{g\widetilde{\GS}}
\def\hGS{\widehat{\GS}}
\def\ghGS{g\widehat{\GS}}
\def\GSnew{\mathcal{X}^\diamond}
\def\tH{\widetilde{H}}
\def\Cell{\mathcal{C}}
\def\Dell{\mathcal{D}}
\def\Sp{S}

\def\hZ{\tfrac{1}{2}\Z}
\def\bK{\mathbb{K}}

\begin{document}

\title{A knot Floer stable homotopy type}

\author[Ciprian Manolescu]{Ciprian Manolescu}
\thanks{Manolescu was supported by NSF Grant DMS-2003488 and a Simons Investigator Award.}
\address{Department of Mathematics, Stanford University, Stanford, CA 94305}
\email{\href{mailto:cm5@stanford.edu}{cm5@stanford.edu}}

\author[Sucharit Sarkar]{Sucharit Sarkar}
\thanks{Sarkar was supported by NSF Grant DMS-2403558.}
\address{Department of Mathematics, University of California, Los Angeles, CA 90095}
\email{\href{mailto:sucharit@math.ucla.edu}{sucharit@math.ucla.edu}}



\date{}

\begin{abstract}
Given a grid diagram for a knot or link $K$ in $S^3$, we construct a filtered spectrum whose homology is the knot Floer homology of $K$. We conjecture that the filtered homotopy type of the spectrum is an invariant of $K$. Our construction does not use holomorphic geometry, but rather builds on the combinatorial definition of grid homology. We inductively define models for the moduli spaces of pseudo-holomorphic strips and disk bubbles, and patch them together into a framed flow category. The inductive step relies on the vanishing of an obstruction class that takes values in a complex of positive domains with partitions.
\end{abstract}
\maketitle
\tableofcontents

\section{Introduction}
In \cite{CJS}, Cohen, Jones, and Segal proposed the problem of lifting Floer homology to a Floer spectrum or pro-spectrum, in the sense of stable homotopy theory. Since then, stable homotopy refinements of Floer homology have been constructed in Seiberg-Witten theory \cite{MSpectrum, KLS, SS} and symplectic geometry \cite{Cohen, Kragh, AB, PS}. In a similar vein, there is a lift of Khovanov homology to a stable homotopy type \cite{LipshitzSarkar, LLS}.

The purpose of this paper is to construct a stable homotopy refinement of knot Floer homology. Knot Floer homology was developed by Ozsv\'ath-Szab\'o \cite{OS-knots} and Rasmussen \cite{RasmussenThesis}, and has many applications; see \cite{Msurvey, HomSurvey, OSSurvey} for some surveys. There is also a generalization to links, called link Floer homology \cite{OS-links}. Knot and link Floer homology were given a combinatorial description in \cite{MOS-combinatorial, MOSzT-hf-combinatorial}. This description is based on representing the link in terms of a grid diagram, as in Figure~\ref{fig:grid}, and counting empty rectangles on the grid. When defined combinatorially from a grid, link Floer homology is sometimes called {\em grid homology}. The corresponding chain complex is called a {\em grid complex}. There are many versions of grid complexes, depending on how we keep track of the $O$ and $X$ markings inside the rectangles. We refer to the book \cite{OSS-book} for an extensive treatment. 

\begin{figure}
  \centering
  \begin{tikzpicture}[scale=0.7]

    \foreach\grid in {0,1}{
      \begin{scope}[xshift=\grid*8cm]
    
    \foreach\i in {0,...,6}{
      \draw (0,\i)--++(6,0);
      \draw (\i,0)--++(0,6);
    };

    \foreach\i in {1,...,6}{
      \node[anchor=east] at (0,\i-1) {\small $\alpha_\i$};
      \node[anchor=north] at (\i-1,0) {\small $\beta_\i$};
    };
    \node[anchor=east] at (0,6) {\small $\alpha_1$};
    \node[anchor=north] at (6,0) {\small $\beta_1$};

    \begin{scope}[xshift=-0.5cm,yshift=-0.5cm]

      \ifnum\grid=1
      \draw[ultra thick] (1,1)--(5,1) (4,2)--(6,2) (5,3)--(2,3) (3,4)--(1,4) (2,5)--(4,5) (6,6)--(3,6);
      \foreach\i/\j in {2/4,3/5,4/3,5/2}{\fill[white] (\i,\j) circle (0.2cm);}
      \draw[ultra thick] (1,1)--(1,4) (2,5)--(2,3) (3,4)--(3,6) (4,2)--(4,5) (5,3)--(5,1) (6,6)--(6,2);

      \else
      \foreach \i/\j [count=\c from 1] in {6/6,3/4,1/1,5/3,2/5,4/2}{
        \node at (\i,\j) {$O_\c$};
      }
      \foreach \i/\j [count=\c from 1] in {3/6,1/4,5/1,2/3,4/5,6/2}{
        \node at (\i,\j) {$X_\c$};
      }

      \fi
    \end{scope}

    \end{scope}}
  \end{tikzpicture}
\caption {Left: A grid diagram. Right: The figure-eight knot that it represents, drawn on the same grid.}
\label{fig:grid}
\end{figure}

In this paper we let $\Grid$ denote a grid diagram representing a link $L$, and let $O_1$ be a fixed $O$-marking on $\Grid$. We will  work with a variant of the grid complex which we denote by $\Cp$, where we do not count rectangles going over $O_1$. The complex $\Cp$ is a module over $\Z[U_2, \dots, U_n]$, with one variable for each $O$-marking different from $O_1$. Further, $\Cp$ has several Alexander filtrations, one from each component of the link. The multi-filtered chain homotopy type of $\Cp$ is an invariant of the link with a distinguished component (the one containing $O_1$). It is closely related to the plus version of the link Floer complex;  see Section~\ref{sec:grid-complex-background} for the definitions. From $\Cp$ one can recover several other versions of grid complexes, such as $\Chat$ or $\Ct$.

From the grid diagram $\Grid$ we will construct a multi-filtered CW-spectrum $\GS(\Grid)$, whose associated chain complex is $\Cp$. The spectrum comes equipped with maps
$$ U_i : \GS(\Grid) \to \Sigma^2 \GS(\Grid), \ \ i=2, \dots, n,$$
where $\Sigma^2$ denotes the double suspension. 

\begin{theorem}
\label{thm:main}
Up to multi-filtered stable homotopy equivalence (commuting with the maps $U_i$), the spectrum $\GS(\Grid)$ is an invariant of the grid $\Grid$ and the distinguished marking $O_1$.
\end{theorem}

A variation of this construction produces a multi-filtered spectrum $\hGS(\Grid)$ with associated graded $\ghGS(\Grid)$. The reduced homology of $\ghGS(\Grid)$ is the hat flavor of grid homology (i.e., link Floer homology), $\widehat{\mathit{GH}}(\Grid)=\widehat{\HFL}(L)$. Just as link Floer homology splits according to Alexander gradings $h \in (\hZ)^{\ell}$, the spectrum $\ghGS(\Grid)$ decomposes into a wedge sum
\[
  \ghGS(\Grid) = \bigvee_{h \in (\hZ)^{\ell}} \ghGS(\Grid, h)
\]
such that
\[
  \tH_i(\ghGS(\Grid, h); \Z) = \widehat{\HFL}_{i}(L, h).
\]
There are also tilde versions of the spectrum, $\tGS(\Grid)$ and $\gtGS(\Grid)$. The homology of the latter,  $\widetilde{\mathit{GH}}(\Grid)=\widetilde{\HFL}(\Grid)$, is isomorphic to the direct sum of several copies of $\widehat{\HFL}(L)$.

The construction of $\GS(\Grid)$ (and its variants) is based on first building a framed flow category; once this is done, the machinery of Cohen, Jones and Segal \cite{CJS} automatically produces a spectrum. In \cite{CJS}, the framed flow category is obtained from moduli spaces of pseudo-holomorphic curves.
In our work, we do not use any holomorphic geometry, but rather build models $\Moduli([D])$ for these moduli spaces, inductively on their dimension, in a manner similar to the construction of the Khovanov stable homotopy type in \cite{LipshitzSarkar}.

The rectangles counted in the definition of grid homology are the positive domains associated to $0$-dimensional moduli spaces of holomorphic strips in the symmetric product of the grid. Whereas rectangles are domains of index $1$, in order to construct the spectrum $\GS(\Grid)$ we have to consider positive domains of arbitrary index. Indeed, each moduli space $\Moduli([D])$ in the framed flow category is associated to an equivalence class of positive domains $D$ on the grid, where two domains are equivalent if they differ by a periodic domain (a linear combination of vertical and horizontal annuli) which has coefficient zero on all $O$ markings. We only consider domains $D$ that do not cross the specified marking $O_1$.

\subsection{Bubbling}
The spaces $\Moduli([D])$ admit compactifications $\bModuli([D])$ which correspond to moduli spaces of broken holomorphic strips. Furthermore, each $\bModuli([D])$ will be the union of spaces $\bModuli_0(D)$ associated to positive domains $D$ in the equivalence class $[D]$, where the different $\bModuli_0(D)$ are glued along their common boundaries. These common boundaries correspond to moduli spaces of disk bubbles in symplectic geometry.
 
We are thus forced to also build models for the moduli spaces of
bubbles. This is one of the novel aspects of our
construction. Previously, stable homotopy refinements of Floer
homologies have mostly been done in the absence of bubbles. (One
notable exception is the work of Abouzaid and Blumberg \cite{AB},
which produces a lift of Hamiltonian Floer homology to Morava K-theory
allowing for bubbles.) In general situations where bubbles appear,
even Floer homology is not always well-defined since the differential
on the Floer complex may not square to zero.

In the link Floer complex (and, more generally, in Heegaard Floer complexes), bubbles appear but they cancel in pairs, so that the differential does square to zero. In the setting of grid diagrams, bubbles correspond to vertical and horizontal annuli, and the two annuli going through the same $O$-marking cancel each other out. In our construction of $\GS(\Grid)$, we implement a higher dimensional analogue of this cancellation: the spaces $\bModuli_0(D)$ by themselves are stratified spaces with a complicated structure, but after we glue them together the resulting $\bModuli([D])$ is a manifold-with-corners of the kind that is used to define a framed flow category.

To understand the strata in the compactifications $\bModuli_0(D)$, we will construct more general spaces $\bModuli_{\vN, \vlambda}(D)$, which are models for the moduli spaces of pseudo-holomorphic strips with disk bubbles attached. The bubble configuration is described by vectors
\[
  \vN = (N_2, \dots, N_n), \ \ \ \vlambda=(\lambda_2, \dots, \lambda_n)
\]
 where $N_j$ are non-negative integers, and $\lambda_j$ is an ordered partition of $N_j$. The number $N_j$ counts the bubbles going through the $j\th$ $O$-marking. These bubbles are grouped according to the partition $\lambda_j$, with those in the same part appearing at the same height on the boundary of the pseudo-holomorphic strip.

Each $\bModuli_{\vN, \vlambda}(D)$ is a stratified space. The local models for the strata are quite interesting, being based on a stratification of the symmetric product $\Sym^N(\CC)$ modulo translation  by $\R$. Specifically, we consider the stratification of $\Sym^N(\CC)/\R$ given by the signs of the imaginary parts of the $N$ complex numbers. For example, when $N=2$, we will encounter the 
Whitney umbrella
\[
  W = \{(a,b, c) \in \R^3 \mid b \leq 0, a^2b+c^2=0\}.
\]

We hope that these models for the moduli spaces of trajectories with bubbles are of independent interest, as they may appear in other settings. However, we warn the reader that the bubble configurations we use in this paper are different from the ones usually considered in the Gromov compactification in symplectic geometry. See Remark~\ref{rem:gromov} for more details.

\subsection{The inductive construction} We now sketch the construction of the spaces $\bModuli_{\vN, \vlambda}(D)$. These will come equipped with suitable embeddings (called {\em neat}) in Euclidean spaces, and also with normal framings. Since the spaces $\bModuli_{\vN, \vlambda}(D)$ are not manifolds, it is not immediate what we mean by framings. We will in fact distinguish two different collections of vector fields, the {\em internal} and {\em external} framings. More details on these can be found in Section~\ref{sec:neat}.

The construction of the spaces $\bModuli_{\vN, \vlambda}(D)$ goes as follows:
\begin{itemize}[leftmargin=*]

\item We first construct them when $D$ is trivial, and all the entries of $\vN$ are $0$'s and $1$'s. In this case we define $\bModuli_{\vN, \vlambda}(D)$ to be a permutohedron, and explain how to give it a normal framing;

\item We define the rest of the spaces $\bModuli_{\vN, \vlambda}(D)$ inductively on their dimension $k$. For the base case $k=0$, we define them to be points, and give them suitable framings;

\item For the inductive step, we suppose all spaces up to dimension
  $(k-1)$ have been constructed. To construct a $k$-dimensional space
  $\bModuli_{\vN, \vlambda}(D)$, we start with its (already
  constructed) boundary $\bdy \bModuli_{\vN, \vlambda}(D)$ and smooth
  it to get a $(k-1)$-dimensional framed manifold
  $\bdy' \bModuli_{\vN, \vlambda}(D)$;

\item From here we get an element
  $[\bdy' \bModuli_{\vN, \vlambda}(D)] \in \tOmega^{k-1}_\fr$, where
  $\tOmega^{k-1}_\fr$ is a slight variant of the usual framed cobordism
  group $\Omega^{k-1}_{\fr}$ (and, in fact, is isomorphic to
  $\Omega^{k-1}_\fr$);

\item We define a chain complex $\CDP_*$ whose generators are
  ``positive domains with partitions,'' i.e., triples
  $(D, \vN, \vlambda)$. We let $\CDP'_*$ be the quotient of $\CDP_*$
  by the subcomplex generated by $(D, \vN, \lambda)$ where $D$ is a
  chosen trivial domain, and $\vN$ is made of $0$'s and
  $1$'s. Altogether, the classes $[\bdy' \bModuli_{\vN, \vlambda}(D)]$
  produce an obstruction class
  \[
    \mf{o}_k\in\Hom(\CDP'_{k+1},\tOmega^{k-1}_\fr);
  \]

\item We show that $\mf{o}_k$ is a cocycle, and that $\CDP'_*$ is
  acyclic. It follows that $\mf{o}_k$ is the coboundary of some
  element $\mf{b}\in\Hom(\CDP'_{k},\tOmega^{k-1}_\fr)$;

\item We use $\mf{b}$ to adjust the definition of the $(k-1)$-dimensional
  moduli spaces that we had previously constructed, so that all cocycles
  $\mf{o}_k$ vanish. (We do not change the definition of any moduli
  spaces of dimension $(k-2)$ or lower.)

\item Then $\bdy' \bModuli_{\vN, \vlambda}(D)$ is framed
  null-cobordant. We fill it in arbitrarily to obtain the desired
  framed moduli space $\bModuli_{\vN, \vlambda}(D)$, and continue with
  the induction.
\end{itemize}

A key role in this construction is played by the complex $\CDP'_*$. To define $\CDP'_*$, we first introduce a chain complex $\CD_*$ generated by positive domains on the grid; this is a close cousin of the complex of positive pairs $\mathit{CP}^*$ used in \cite[Section 4]{MOT}. We then enhance $\CD_*$ by adding vectors of partitions to its generators; the result is the complex $\CDP_*$. We show that the homology of $\CDP_*$ is supported by triples $(D, \vN, \vlambda)$ where $D$ is a fixed trivial domain  and $\vN$ is made of $0$'s and $1$'s; hence, the quotient $\CDP'_*$ of $\CDP_*$ by these triples is acyclic. Thus, it is important that we first defined some moduli spaces by hand (to be permutohedra); otherwise we would have had to work with $\CDP_*$, which is not acyclic. 

Once the framed moduli spaces $\bModuli_{\vN, \vlambda}(D)$ are defined, the spectrum $\GS(\Grid)$ is obtained by a standard procedure from \cite{CJS,LipshitzSarkar}. 

We remark that for the simpler versions $\ghGS(\Grid)$, $\tGS(\Grid)$ and $\gtGS(\Grid)$, we only use moduli spaces $\bModuli_{0}(D)$ for domains $D$ that do not cross certain markings ($X$'s, $O$'s, or both). These spaces do not involve configurations of bubbles, because $D$ cannot contain a full row or column. Nevertheless, if we had tried to construct only these spaces $\bModuli_{0}(D)$, we would have run into the problem that the analogue of $\CDP'_*$  (using domains that do not cross the $X$-markings) is not acyclic. Thus, even if we were only interested in the simpler versions, we still had to build all the spaces $\bModuli_{\vN, \vlambda}(D)$ and discuss bubbling. 

\subsection{Further directions}
Theorem~\ref{thm:main} proves a weak form of invariance for $\GS(\Grid)$: that it depends only on the grid (and its special marking), not on the other choices made in its construction. We further conjecture that the stable homotopy type of $\GS(\Grid)$ is a link invariant, that is, it is independent of the choice of grid diagram $\Grid$ representing a given link. The proof of this is beyond the scope of the present paper. Invariance of grid homology is proved in \cite{MOSzT-hf-combinatorial} by checking the Cromwell-Dynnikov moves: cyclic permutation, commutation, and stabilization. We expect that a combination of those arguments with the techniques from this paper will yield invariance for $\GS(\Grid)$. The main challenge is to prove that suitable complexes of positive domains and partitions associated to the commutation and stabilization moves are acyclic.

Another limitation of our paper is that we only consider domains that
do not cross a given marking $O_1$. The reason for this is to ensure
the acyclicity of $\CDP'_*$. One can check that the analogue of
$\CDP'_*$ using all domains on the grid is not
acyclic~\cite{Tao}. Nevertheless, one can compute its homology and
attempt to get a handle on the analogues of the obstruction classes
$[\mf{o}_k]$. We expect that all versions of grid homology (including
those involving domains that go over $O_1$) admit stable homotopy
refinements, in the form of spectra or pro-spectra.

\subsection{Organization of the paper.} 
In Section~\ref{sec:background} we fix notation and review some facts about grid diagrams and grid homology. 

In Section~\ref{sec:CD} we define the complex $\CD_*$ whose generators are positive domains on the grid. 

In Section~\ref{sec:CDP} we define the complex $\CDP_*$ of positive domains with partitions, we compute its homology, and introduce the acyclic quotient $\CDP'_*$. 

In Section~\ref{sec:nmflds} we review $\langle n \rangle$-manifolds, the type of manifolds with corners that are used in framed flow categories. 

In Section~\ref{sec:stratified} we discuss different notions of stratified spaces, such as Whitney and Thom-Mather stratifications.

In Section~\ref{sec:local} we describe the local models for the stratified spaces that appear in this paper; these are generalizations of the Whitney umbrella. 

In Section~\ref{sec:moduli} we give examples of stratified spaces that can be associated to some simple domains on the grid. 

In 
Section~\ref{sec:strata} we list the strata that should be included in the compactification of each space $\Moduli_{\vN, \vlambda}(D)$. 

In Section~\ref{sec:neat} we introduce the notion of neat embedding for a space $\bModuli_{\vN, \vlambda}(D)$, and explain what we mean by internal and external framings. 

In Section~\ref{sec:group} we define the embedded framed cobordism group $\tOmega^k_\fr$, and show that it is isomorphic to the usual $\Omega^k_\fr$.

Section~\ref{sec:construction} is the heart of the paper, in which we construct the spaces $\bModuli_{\vN, \vlambda}(D)$ inductively. 

The case where $D$ is trivial and $\vN$ is made of $0$'s and $1$'s is relegated to Section~\ref{sec:permutohedron}, where we describe a neat embedding of the permutohedron, and give it a normal framing. 

In Section~\ref{sec:cjs} we review the Cohen-Jones-Segal construction of a spectrum from a framed Floer category. We then define $\GS(\Grid)$ and its variants.

In Section~\ref{sec:invariance} we prove our weak invariance result, Theorem~\ref{thm:main}. This section also introduces the notion of maps of normally framed flow categories.

Finally, in Section~\ref{sec:examples} we give some examples. We show that the filtered equivalence class of $\GS(\Grid)$ is determined by homological data (i.e., by the link Floer complex) for several families of knots, such as thin or L-space knots.

\subsection{Conventions.} Throughout the paper $\NN$ denotes the
natural numbers including $0$. We also let $\RR_+ = [0, \infty)$.

\subsection{Acknowledgments.} We would like to thank Mohammed Abouzaid, Mike Hill, Tyler Lawson, Mona Merling, Danny Ruberman, and Sander Kupers for helpful conversations, and Ciprian Bonciocat, Vinicius Ramos, Yan Tao, and the referee for comments on a previous version of the paper. We are particularly indebted to Robert Lipshitz who suggested a key idea for framing the permutohedron in Section~\ref{sec:permutohedron}.

\section{Background}
\label{sec:background}



\subsection{Grid diagrams}\label{sec:grid-background}

Definitions and notions related to grid diagrams have been listed in the following enumerated list. For details, see~\cite{MOS-combinatorial, MOSzT-hf-combinatorial, OSS-book}.
\begin{enumerate}[leftmargin=*,label=(G-\arabic*),ref=G-\arabic*]
\item An index-$n$ \emph{grid diagram} $\Grid$ consists of the torus, obtained
  from $[0,n]\times[0,n]$ by identifying opposite edges, $n$
  `horizontal' \emph{$\alpha$-circles}, $\alpha_1,\dots,\alpha_n$, with
  $\alpha_i$ being the image of $[0,n]\times\{i-1\}$, and $n$ `vertical'
  \emph{$\beta$-circles}, $\beta_1,\dots,\beta_n$, with $\beta_i$ being the
  image of $\{i-1\}\times[0,n]$.
\item The $n$ components of the complement of $\alpha$ circles are
  called \emph{horizontal annuli} or \emph{rows}, the $n$ components of the
  complement of $\beta$ circles are called \emph{vertical annuli} or \emph{columns}, and
  the $n^2$ components of the complement of $\alpha$ and $\beta$
  circles are called \emph{square regions}. 
\item Grid diagrams are decorated with $n$ \emph{$O$-markings},
  $O_1,\dots,O_n$, placed in $n$ distinct square regions so that each
  horizontal annulus has one $O$ marking and each vertical annulus
  has one $O$ marking. Let $H_i$, respectively $V_i$, be the horizontal,
  respectively vertical, annulus that contains $O_i$; since we are working on a torus, without loss
  of generality, we will assume that $O_1$ lies in the `top-right'
  square region $(n-1,n)\times(n-1,n)$.
\item We can also order and label the annuli more naturally, without
  regard for the position of the $O$'s. We define the horizontal
  annulus $H_{(i)}$ to be the image of $[0, n] \times (i-1, i)$, and
  the vertical annulus $V_{(i)}$ to be the image of
  $(i-1, i) \times [0,n].$
\item Grid diagrams will also be decorated with $n$ \emph{$X$-markings},
  $X_1,\dots,X_n$, placed in $n$ distinct square regions so that each
  horizontal annulus has one $X$ marking and each vertical annulus
  has one $X$ marking.
\item By joining the $O$ and $X$ markings by segments in each row and
  column, and letting the vertical segments be overpasses, we obtain a
  planar diagram for a link $L \subset S^3$. (For square regions
  containing both an $O$ and an $X$ marking, we put a small unknot in
  that region.) We say that $\Grid$ is a grid diagram presentation for
  the link $L$.
\item\label{item:permutation-generator} A \emph{generator} or a \emph{state} $x$ is a unordered
  $n$-tuple $(x_1,\dots,x_n)$ of points on the torus, so that each
  $\alpha$-circle contains some $x_i$ and each $\beta$-circle contains
  some $x_j$. The $x_i$'s are called the \emph{coordinates} of $x$. We
  sometimes view $x$ as a formal sum of its coordinates,
  $x_1+x_2+\dots+x_n$. Generators are in one-to-one correspondence
  with permutations of $\{1,2,\dots,n\}$, with permutation $\sigma$
  corresponding to the generator
  \[
  x^{\sigma}=(\alpha_{\sigma(1)}\cap\beta_1,\alpha_{\sigma(2)}\cap\beta_2,\dots,\alpha_{\sigma(n)}\cap\beta_n).
  \]
  The set of all generators on a grid diagram $\Grid$ is denoted $\S=\S(\Grid)$.
\item A \emph{domain} $D$ from a generator $x$ to a generator $y$ is a
  $2$-chain given by a $\ZZ$-linear combination of (the closures of)
  the square regions, with the property that $\bdy(\bdy
  D\cap\alpha)=y-x$. In this paper, we are only interested in domains
  that have coefficient zero at $O_1$, and we will let $\domains(x,y)$
  denote the set of domains from $x$ to $y$ that avoid $O_1$.
\item For any domain $D$, let $\coefficients{D}=(O_2(D),\dots,O_n(D))\in\ZZ^{n-1}$ be the
  vector that records the coefficients of $D$ at the $O$-markings;
  that is, component $O_i(D)$ is the
  coefficient of $D$ at $O_i$, for $2\leq i\leq n$. Similarly, we let $\Xcoefficients{D}=(X_1(D),\dots,X_n(D))\in \ZZ^n$ be the vector that records the coefficients of $D$ at the $X$-markings.
\item\label{item:domain-concatenation-grid} Given $D\in\domains(x,y),E\in\domains(y,z)$, by adding the
  underlying $2$-chains, we get a domain $D*E\in\domains(x,z)$.
\item A domain is said to be \emph{positive} if it has no negative
  coefficients. Let $\pdomains(x,y)\subset\domains(x,y)$ be the subset
  of positive domains. (Note that this includes the zero domain.)
\item For any generators $x,y$, the set $\domains(x,x)$ can be identified with
  $\domains(y,y)$ by identifying the underlying $2$-chains. We
  call either of these sets $\periodic$, the set of \emph{periodic
    domains}. Further, we denote by $\pperiodic$ the subset consisting
  of positive periodic domains (including zero). We have
  \[
  \periodic=\ZZ\basis{H_{(1)},\dots,H_{(n-1)},V_{(1)},\dots,V_{(n-1)}}\qquad\pperiodic=\NN\basis{H_{(1)},\dots,H_{(n-1)},V_{(1)},\dots,V_{(n-1)}}.
  \]
  Indeed, for any periodic domain, its multiplicity at the region
  $H_{(i)}\cap V_{(n)}$, respectively $H_{(n)}\cap V_{(i)}$, gives the coefficient of
  $H_{(i)}$, respectively $V_{(i)}$, in the above formula.
\item For every domain $D$, there is an associated integer $\mu(D)$
  called its Maslov index, defined as follows. For any point $p$ in
  the intersection of the $\alpha$ and $\beta$-circles, define the
  \emph{coefficient} of $D$ at $p$ to be average of the coefficients
  of $D$ at the four square regions adjacent to $p$. Then the Maslov
  index $\mu(D)$ is the sum of the coefficients of $D$ at the $n$
  coordinates of $x$ and the $n$ coordinates of $y$. The Maslov index
  satisfies the following properties:
  \begin{enumerate}
  \item For any $D\in\domains(x,y),E\in\domains(y,z)$, $\mu(D*E)=\mu(D)+\mu(E)$.
  \item For any $D\in\pdomains(x,y)$, $\mu(D)\geq 0$.
  \item If $D\in\pdomains(x,y)$, then $\mu(D)= 0$ if and only if $x=y$
    and $D$ is the trivial domain; let $\const{x}\in\pdomains(x,x)$
    denote the trivial domain.
  \item If $D\in\pdomains(x,y)$, then $\mu(D)= 1$ if and only if $D$ is a
    \emph{rectangle} in the torus: its `bottom-left' and `top-right' corners
    are coordinates of $x$ and its `bottom-right' and `top-left'
    corners are coordinates of $y$; the other $(n-2)$-coordinates of
    $x$ and $y$ agree and none of them lie in $D$. Let 
    \[
    \rectangles(x,y)=\set{D\in\pdomains(x,y)}{\mu(D)=1}.
    \] 
  \item\label{item:maslov-index-2} If $D\in\pdomains(x,y)$, then
    $\mu(D)= k$ if and only if $D$ has a (possibly non-unique)
    \emph{decomposition} into rectangles
    \[
      D=R_1*R_2*\cdots*R_k\qquad
      R_1\in\rectangles(x=w_0,w_1),R_2\in\rectangles(w_1,w_2),\dots,
      R_k\in\rectangles(w_{k-1},w_k=y).
    \]
    (This is \cite[Lemma 3.5]{SarkarShellable}, but the proof given in
    that paper is wrong, so we give a correct proof as
    Lemma~\ref{lem:decompose-into-rectangles} below.)  In particular,
    $D$ is a positive index-$2$ domain if and only if it can be
    decomposed into two rectangles; that is, it can be two rectangles,
    either overlapping like a cross or disjoint,
  
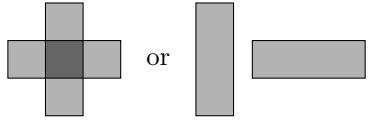
\[
    \begin{tikzpicture}[scale=0.5]
      \begin{scope}
      \fill[black!30] (-1,0) rectangle (2,1);
      \fill[black!30] (0,-1) rectangle (1,2);
      \fill[black!60] (0,0) rectangle (1,1);
      \draw (0,-1) rectangle (1,2);
      \draw (-1,0) rectangle (2,1);
      \end{scope}
      \node at (3,0.5) {or};
      \begin{scope}[xshift=4cm]
      \draw[fill=black!30] (0,-1) rectangle (1,2);
      \draw[fill=black!30] (-1+2.5,0) rectangle (2+2.5,1);
      \node[anchor=west] at (2+2.5,0.5) {,};
      \end{scope}
    \end{tikzpicture}
  \]
  or a hexagon in one of four possible shapes,
  
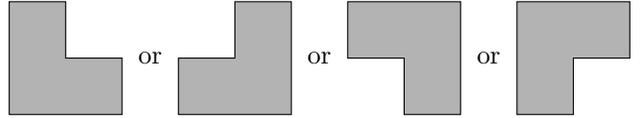
\[
    \begin{tikzpicture}[scale=0.75]
      \foreach \i in {0,1,2,3}{
        \begin{scope}[xshift=3*\i cm,rotate=90*\i]
      \draw[fill=black!30] (-1,-1)-- (1,-1)-- (1,0)-- (0,0)-- (0,1)-- (-1,1)--cycle;
      \end{scope}}
    \foreach \i in {0,1,2}{
        \begin{scope}[xshift=3*\i cm]
          \node at (1.5,0) {or};
      \end{scope}}
      \node[anchor=west] at (9+1,0) {,};
    \end{tikzpicture}
  \]
    or a horizontal annulus, or a vertical annulus. Note, in the first
    six cases, $D$ has exactly two decompositions into rectangles,
    while in the last two cases, $D$ has exactly one.
  \end{enumerate}
\item Generators carry a well-defined integer-valued grading---called
  the Maslov grading and denoted $\gr(x)$---so that for any domain
  $D\in\domains(x,y)$,
  \[
    \gr(x)-\gr(y)=\mu(D)-2\card{\coefficients{D}},
  \]
  where $\card{\coefficients{D}}=\sum_i O_i(D)$. 
 \item Generators also admit an Alexander grading $A(x) \in \ZZ$ with the property that for any $ D\in\domains(x,y)$,
 \[
 A(x) - A(y) = |\Xcoefficients{D}|-|\coefficients{D}|.
 \]
In fact, if $L$ is a link of $\ell$ components, we have an Alexander multi-grading $(A_1(x), \dots, A_\ell(x)) \in (\tfrac{1}{2}\ZZ)^\ell$ such that $A(x) = A_1(x) + \dots + A_{\ell}(x)$.
\item \label{item:signa} \emph{A sign assignment} $s$ is a function
  $s\from\bigcup_{x,y}\rectangles(x,y)\to\{\pm 1\}$ satisfying the
  following. For any $D\in\pdomains(x,y)$ with $\mu(D)=2$ that is not
  a horizontal or a vertical annulus (that is, one of the types
  pictured above), if $R_1*S_1$ and $R_2*S_2$ are the two
  decompositions of $D$ into rectangles, then
  \begin{equation}
  \label{eq:sRS}
  s(R_1)s(S_1)=-s(R_2)s(S_2).
  \end{equation}
  Furthermore, if $R*S$ is a decomposition of a horizontal annulus into rectangles, then
\begin{equation}
  \label{eq:RSh}
 s(R) s(S) = 1,
   \end{equation}
  and if $R*S$ is a decomposition of a vertical annulus into rectangles, then
  \begin{equation}
  \label{eq:RSv}
   s(R) s(S) = -1.
     \end{equation}
  \end{enumerate}

  As promised, we give a correct proof of the following lemma.
  \begin{lemma}\label{lem:decompose-into-rectangles}
    Let $D\in\pdomains(x,y)$ with $\mu(D)= k$.  Then $D$ has a
    (possibly non-unique) \emph{decomposition} into rectangles
    \[
      D=R_1*R_2*\cdots*R_k\qquad
      R_1\in\rectangles(x=w_0,w_1),R_2\in\rectangles(w_1,w_2),\dots,
      R_k\in\rectangles(w_{k-1},w_k=y).
    \]
  \end{lemma}

  \begin{proof}
    Let $p$ be a point of intersection of an $\alpha$ and a
    $\beta$-circle, and assume the coefficients of $D$ at the four square
    regions around $p$ are
    \[
      \begin{tikzpicture}[scale=0.8]
        \draw (-1,0) -- (1,0);
        \draw (0,-1) -- (0,1);
        \node at (0,0) {$\bullet$};
        \node at (-0.5,0.5) {$a$};
        \node at (0.5,0.5) {$b$};
        \node at (-0.5,-0.5) {$c$};
        \node at (0.5,-0.5) {$d$};
      \end{tikzpicture}
    \]
    From the condition $\bdy D\cap\alpha=y-x$, the following implications are immediate:
    \begin{align*}
      b+c-a-d&=1&\text{if $p$ is a coordinate of $x$ and not of $y$;}\\
      b+c-a-d&=-1&\text{if $p$ is a coordinate of $y$ and not of $x$;}\\
      b+c-a-d&=0&\text{otherwise.}
    \end{align*}
    
    We can assume $\mu(D)=k>0$, so  $D$ is not the
    constant domain. Thus, there is some coordinate of $x$, say $x_1$,
    such that the coefficient of $D$ at $x_1$ is positive. Therefore,
    the coefficient of $D$ at either the top-right or the bottom-left
    square region of $x_1$ must be non-zero. Assume, after rotating
    everything by $\pi$ if necessary, that the coefficient of the
    top-right square region is non-zero.  By translating in the torus
    if necessary, assume $x_1=\{0\}\times\{0\}$. Consider all rectangles $R$ that
    are contained in $D$ as 2-chains (that is, the 2-chain $D-R$ has
    non-negative coefficients) and that have $x_1$ as their
    bottom-left corner. Partially order such rectangles by inclusion,
    and let $[0,k-1]\times[0,l-1]$ be a maximal element under this
    partial order.

    Let $D_0$ be the restriction of $D$ to $[0,k]\times[0,l]$. It
    might not be a domain, but it is a non-negative $2$-chain
    supported on $[0,k]\times[0,l]$ such that its $(k-1)(l-1)$
    coefficients in $[0,k-1]\times[0,l-1]$ are strictly positive, and
    at least one of the $(l-1)$ coefficients in
    $[k-1,k]\times [0,l-1]$ is zero, and at least one of the $(k-1)$
    coefficients in $[0,k-1]\times[l-1,l]$ is zero. From non-negative
    2-chains satisfying such hypothesis, we will produce some
    coordinate of $x$ in $(0,k)\times (0,l)$.

    This is accomplished by induction on the sum of coefficients of
    $D_0$. Consider the rightmost 0-coefficient in
    $[0,k-1]\times[l-1,l]$, say it is at $[i-1,i]\times[l-1,l]$ with
    $i\leq k-1$. Consider at the topmost 0-coefficient in
    $[k-1,k]\times[0,l-1]$, say it is at $[k-1,k]\times[j-1,j]$ with
    $j\leq l-1$. If $i=k-1$ and $j=l-1$, then the coefficients near
    the vertex $\{k-1\}\times\{l-1\}$ looks like
    \[
      \begin{tikzpicture}[scale=0.8]
        \draw (-1,0) -- (1,0);
        \draw (0,-1) -- (0,1);
        \node at (0,0) {$\bullet$};
        \node at (-0.5,0.5) {$0$};
        \node at (0.5,0.5) {$b$};
        \node at (-0.52,-0.5) {$c>0$};
        \node at (0.5,-0.5) {$0$};
      \end{tikzpicture}
    \]
    and so this point is a coordinate of $x$.

  If not, we have either $i<k-1$ or $j<l-1$. The two cases are
  similar, so assume $j<l-1$. Look at the coefficients of $D$ on two
  sides of the horizontal line $[0,k]\times \{j\}$, numbered as follows:
  \[
    \begin{tikzpicture}[scale=0.8]
      \draw (0,0) -- (10,0);
      \foreach \i in {0,1,2,3}{
        \draw (\i,-1) -- (\i,1);
        \ifnum\i>0
        \node at (\i-0.5,0.5) {$a_{\i}$};
        \node at (\i-0.5,-0.5) {$b_{\i}$};
        \node at (\i,0) {$\bullet$};
        \fi
      }
      \foreach \i/\ii in {0/,1/k,2/k-1,3/k-2}{
        \draw (10-\i,-1) -- (10-\i,1);
        \ifnum\i>0
        \node at (10-\i+0.5,0.5) {$a_{\ii}$};
        \node at (10-\i+0.5,-0.5) {$b_{\ii}$};
        \node at (10-\i,0) {$\bullet$};
        \fi
      }
    \end{tikzpicture}
  \]
  Assume none of the points on this line are coordinates of $x$. Then we have
  $a_n-b_n\geq a_{n+1}-b_{n+1}$ for all $n=1,\dots,k-1$. By assumption
  $a_k>0$ and $b_k=0$, so using the previous inequality, we get
  $a_n>b_n$ for all $n$. Moreover, $b_n\geq 1$ for all $n<k$, so
  $a_n\geq 2$ for all $n<k$.

  Now consider the 2-chain $D_0'$ whose coefficients are same as that
  of $D_0$ except in the row $[0,k]\times[j,j+1]$; and on that row,
  the coefficients are $a_1-1,a_2-1,\dots,a_k-1$. (That is, $D'_0$ is
  obtained from $D_0$ by decreasing the coefficients in the row
  $[0,k]\times[j,j+1]$ by one.) By the inequalities obtained above,
  $D'_0$ also satisfies the hypothesis, and by induction, has a
  coordinate of $x$.

  Therefore, we conclude that the rectangle $[0,k-1]\times[0,l-1]$
  contains another coordinate of $x$, say $x_2$. Now consider all
  rectangles that are contained in $D$, have $x_1$ as their
  bottom-left corner, and has some coordinate of $x$ as their
  top-right corner. (The above argument shows that the set of such
  rectangles is non-empty.) Again partially order them by inclusion,
  and let $R$ be the minimal element under this partial order. Then
  $R\in\rectangles(x,z)$ for some $z$ and $D-R\in\pdomains(z,y)$, so
  we have a decomposition $D=R*(D-E)$. By induction, this produces the
  required decomposition of $D$ into rectangles.
\end{proof}

\subsection{Grid complexes}\label{sec:grid-complex-background} Let
$\Grid$ be a grid diagram decorated with $O$- and $X$-markings and
equipped with a sign assignment $s$, representing a link $L$; let
$L_1,\dots,L_\ell$ denote the components of $L$.  Assume the marking
$O_1$ (which we are avoiding) lies on the link component $L_1$.

To $\Grid$ one can associate chain complexes in various flavors, which
are typically called \emph{grid complexes}. The chain homotopy types
of these chain complexes (in appropriate senses) are invariants of
$(L,L_1)$, namely, the underlying link with a preferred component
$L_1$. We will concentrate on the following flavor.  As an Abelian
group, the chain group $\Cp=\Cp(\Grid)$ is freely generated by
elements of the form
\[
  [x, j_2, \dots, j_n], \ \ x \in \S, \ j_2, \dots, j_n \in \NN.
\]
The homological grading
of a generator is
\[
  \gr([ x, j_2, \dots, j_n]) = \gr(x) + 2j_2 + \dots + 2j_n.
\]
The $\ell$ Alexander gradings are defined as follows:
\[
  A_k([x,j_2,\dots,j_n])=A_k(x)+\sum_{\{i\mid O_i\in L_k\}} 2j_i.
\]
We equip $\Cp$ with the structure of a module over
$\ZZ[U_2, \dots, U_n]$, by letting $U_i$ act on
$[ x, j_2, \dots, j_n]$ by decreasing $j_i$ by $1$, if $j_i \geq 1$;
if $j_i=0$, then $U_i$ acts by zero. Notice that $U_i$ decreases 
homological grading by two, decreases the Alexander grading $A_k$ by one if $O_i\in L_k$, and preserves the other Alexander gradings. We can alternatively describe the generators of $\Cp$ as
\[
  U_2^{-j_2} \cdots U_n^{-j_n} x = [x, j_2, \dots, j_n].
\]
 
The differential on $\Cp$ is given by
\[
  \bdy([ x, j_2, \dots, j_n])=\sum_y \sum_{R\in\rectangles(x,y)}s(R)U^{\coefficients{R}}[y, j_2, \dots, j_n],
\]
where we used the notation
\[
  U^{\coefficients{R}} \defeq U_2^{O_2(R)} \cdots U_n^{O_n(R)}.
\]
Each of the $\ell$ Alexander gradings are either preserved or
decreased by the differential, so the complex $\Cp$ admits an
Alexander multi-filtration induced from these $\ell$ Alexander
gradings on the generators.

We can endow $\Cp$ with the structure of a
$\ZZ[U'_2,\dots,U'_\ell]$ module by picking a preferred $O$-marking
$O_{i_k}\in L_k$ and letting $U'_k$ act as $U_{i_k}$, for
$2\leq k\leq \ell$. (The module structure of the chain complex does
depend on this additional choice of a preferred $O$-marking on each
component.) This gives $\Cp$ the structure of a graded
$\ell$-filtered chain complex over $\ZZ[U'_2,\dots,U'_\ell]$. The
multi-filtered chain homotopy type of $\Cp$ over
$\ZZ[U'_2,\dots,U'_\ell]$ is an invariant of $(L,L_1)$. We will
outline a proof of this assertion below, and will assume the reader
is familiar with Heegaard Floer homology basics for this part.

Grid diagrams are particular examples of Heegaard diagrams for link
complements, and grid complexes correspond to link Floer
complexes. From any Heegaard diagram
$\Heeg=(\Sigma_g,\alpha_1,\dots,\alpha_{g+n-1},\allowbreak\beta_1,\dots,\beta_{g+n-1},\allowbreak
O_1,\dots,O_n,\allowbreak X_1,\dots,X_n)$ representing an 
$\ell$-component link $L \subset S^3$, one can define a link Floer
complex $\CFL^+(\Heeg)$, in a similar way as we did for $\Cp$, but
using pseudo-holomorphic disks instead of rectangles. We will study
$\CFL^+$ in full generality, allowing differentials to go over all
basepoints, including $O_1$. We will assume that we have picked the
preferred $O$-marking on $L_1$ to be $O_1$, so this becomes a module
over $\ZZ[U'_1,\dots,U'_\ell]$ with $U'_1$ acting as $U_1$. Our grid
complex $\Cp$ defined above is the kernel of the $U'_1$ action, which
is a subcomplex of $\CFL^+$ of the grid diagram.

These complexes were studied in \cite{OS-knots} when $L$ is a knot
and the Heegaard diagram has only two basepoints, in \cite{OS-links}
for general links, but when each link component has only two
basepoints, and in \cite{MOS-combinatorial} in general;
\cite{MOSzT-hf-combinatorial} and \cite{OSS-book} study the
specializations of these complexes to grid diagrams in more
detail. Of these, \cite{OS-knots} defines a version of knot Floer
complex $\CFK^+$ which is most closely related to our versions, but
the later papers study a version---denoted $\CFL^-$ in
\cite{OS-links,MOS-combinatorial}, $C^-$ in
\cite{MOSzT-hf-combinatorial}, and $\bm{\mathcal{GC}}^-$ in
\cite{OSS-book}---with generators
\[
  U_1^{j_1} \cdots U_n^{j_n} x, \ \ x \in \S, \ j_1, \dots, j_n \in \NN,
\]
and differentials going over both types of basepoints. It is
proved in \cite[Section 2]{MOS-combinatorial} that the filtered
chain homotopy type of $\CFL^-$ over $\ZZ[U'_1,\dots,U'_\ell]$ is an
invariant of $L$.

For completeness, we include here the invariance result for the plus
version.

\begin{proposition}
  \label{prop:HFLp}
  The multifiltered chain homotopy type of $\CFL^+(\Heeg)$ over $\Z[U_1', \dots ,U_{\ell}']$ is an invariant of the link $L$.
\end{proposition}

\begin{proof}
  If we restrict to Heegaard diagrams with only two basepoints on each
  link component, the argument is entirely similar to that in
  \cite[Theorem 4.7]{OS-links}; it involves checking invariance under
  isotopies, handleslides, and index one/two stabilizations. Note that
  in this case there is a single $U_k$ variable for each component
  $L_k$, and we can call it $U'_k$.

  Once we allow more basepoints, we also need to check invariance
  under index zero/three stabilizations. This was done for the minus
  version in \cite[Section 2]{MOS-combinatorial}. For the plus
  version, assuming without loss of generality that the stabilization
  occurs on the link component $L_1$ near the marking $O_1$, the
  arguments there show that the stabilized complex $C'$ is isomorphic
  to a mapping cone
  \begin{equation}
    \label{eq:cone}
    C[U_{n+1}^{-1}]  \xrightarrow{U_{n+1} - U_1} C[U_{n+1}^{-1}], 
  \end{equation}
  where $C$ is the complex for the diagram before stabilization,
  $U_{n+1}$ is the new variable corresponding to the new $O$-marking,
  and $U_1$ is the old variable for the marking $O_1$. We will prove
  that $C'$ is filtered chain homotopy to $C$, as a module over the
  old variables $\ZZ[U_1,\dots,U_n]$. Once this is done, the desired
  conclusion follows inductively: By \cite[Lemma
  2.4]{MOS-combinatorial}, we can choose any of the markings on a
  given component $L_k$ to be the new one. Hence, when we do induction
  on the number of markings on $L_k$, we can ensure that the preferred
  $O$-marking $O_{i_k}\in L_k$ is the oldest one, and so the above
  filtered chain homotopy equivalences will respect the
  $U_{i_k}=U'_k$-action on $\HFL^+$. 

  We first perform a change of basis on the stabilized complex $C'$,
  viewed as the mapping cone from Equation~\eqref{eq:cone}. Keep the
  right $C[U_{n+1}^{-1}]$ unchanged, but for the left
  $C[U_{n+1}^{-1}]$, replace the generator $U_1^{-j_1} \cdots U_{n+1}^{-j_{n+1}} x$ by
  \[
    (U_1^{-j_1} \cdots U_{n+1}^{-j_{n+1}} x)'\defeq (U_1^{-j_1}+U_1^{-j_1+1}U_{n+1}^{-1}+\dots+U_1^{-1}U_{n+1}^{-j_1+1}+U_{n+1}^{-j_1})U_2^{-j_2}\cdots U_{n+1}^{-j_{n+1}}x.
  \]
  Note that the change of basis isomorphism $z\mapsto z'$ respects the
  homological grading and the Alexander multigrading, and also the
  $U_1,\dots,U_n$-actions, but not the $U_{n+1}$-action.

  In terms of this new basis, the map in the mapping cone
  $C[U_{n+1}^{-1}]\to C[U_{n+1}^{-1}]$ sends a generator $z'$ on the
  left to the generator $U_{n+1}z$ on the right. Therefore, this
  complex decomposes into two summands. The first summand is generated
  by elements of the form $(U_1^{-j_1} \cdots U_{n+1}^{-j_{n+1}} x)'$
  with $j_{n+1}=0$ in the left $C[U^{-1}_{n+1}]$, and the second
  summand is generated by the entire right $C[U_{n+1}^{-1}]$ and
  elements of the form $(U_1^{-j_1} \cdots U_{n+1}^{-j_{n+1}} x)'$
  with $j_{n+1}>0$ in the left $C[U^{-1}_{n+1}]$. The first summand is
  isomorphic to $C$ (over $\ZZ[U_1,\dots,U_n]$, respecting all
  gradings), while the second summand is multifiltered chain homotopic
  to the zero chain complex.
\end{proof}

Specializing to grid diagrams, and taking the kernel of the $U'_1$
action, we get our $\ell$-filtered chain complex $\Cp$ over
$\ZZ[U'_2,\dots,U'_\ell]$, whose filtered chain homotopy type is an
invariant of $(L,L_1)$.

Here are a few other flavors of grid complexes. We may consider the
intersection of the kernels of each of the $U'_k$-actions. This is a
subcomplex $\Chat\subset\Cp$ generated by $[x,j_2,\dots,j_n]$ where
$j_{i_k}=0$ for all the preferred $O_{i_k}$ markings that we
picked. (If $L$ is a knot, this complex $\Chat$ is simply our original
complex $\Cp$.)  By adapting the proof of Proposition~\ref{prop:HFLp}
to this setting, we see that $\Chat$ is $\ell$-filtered chain homotopy
equivalent to $\widehat{\CFL}$, the hat flavor of link Floer chain
complex. In particular, if we take the homology of its associated
graded complex $\gChat$, we get $\widehat{\HFL}(L)$, the hat flavor of link
Floer homology.

If instead we ask for all $j_i$ to be zero (that is, the complex is
generated over $\ZZ$ by $x \in \S$), we have a complex denoted $\Ct$,
which is multifiltered chain homotopy equivalent to
$\Chat\otimes V_1\otimes\dots\otimes V_{n-\ell}$, where each $V_i$ is
a direct sum of two copies of $\ZZ$ supported in certain homological
and Alexander gradings. Therefore, the homology of its associated
graded complex $\gCtilde$ is
$\widehat{\HFL}(L) \otimes V_1\otimes\dots\otimes V_{n-\ell}$.

\section{The complex of positive domains}
\label{sec:CD}
In this section we will study a different chain complex, $\CD_*$, associated to
grid diagrams. Unlike the grid complex, the complex $\CD_*$ does not carry
any interesting topological information. Rather, it is the first step towards constructing a slightly more complicated complex, $\CDP_*$, which will be defined in Section~\ref{sec:CDP}. The obstruction classes that we will encounter while constructing our CW complex will
live in $\CDP_*$. 

\begin{definition}\label{def:complex-positive-domains}
  Given a grid diagram $\Grid$ and a sign assignment $s$, the
  \emph{complex of positive domains}, $\CD_*=\CD_*(\Grid)$, is freely generated
  over $\ZZ$ by the positive domains (avoiding $O_1$), with the
  homological grading being the Maslov index:
  \[
    \CD_k=\ZZ\basis{\set{(x,y,D)}{D\in\pdomains(x,y),\mu(D)=k}}.
  \]
  We will usually drop $x$ and $y$ from the notation for a generator of $\CD_*$, and just write it as $D$. 

The
  differential $\diff\from\CD_k\to\CD_{k-1}$, on a basis element
  $D\in\pdomains(x,y)$, is given as follows:
  \[
    \diff(D)=\sum_{\substack{(R,E)\in\rectangles(x,w)\times\pdomains(w,y)\\R*E=D}}\!\!\!\!\!\!s(R)E+(-1)^k\!\!\!\!\!\!\sum_{\substack{(E,R)\in\pdomains(x,w)\times\rectangles(w,y)\\E*R=D}}\!\!\!\!\!\!s(R)E.
  \]
  Note that $\CD_*$ is independent of the locations of the markings $O_2,\dots,O_n$ and $X_1, \dots,X_n$.
\end{definition}

\begin{lemma}
\label{lemma:d2}
  The complex from Definition~\ref{def:complex-positive-domains} is
  indeed a chain complex, that is, $\diff^2=0$.
\end{lemma}

\begin{proof}
  The proof is essentially the same as the proof that the grid complex
  is a chain complex.
  \begin{align*}
    \diff^2(D)&=\sum_{\substack{(R,E)\in\rectangles(x,w)\times\pdomains(w,y)\\R*E=D}}\!\!\!\!\!\!s(R)\diff(E)+(-1)^k\!\!\!\!\!\!\sum_{\substack{(E,R)\in\pdomains(x,w)\times\rectangles(w,y)\\E*R=D}}\!\!\!\!\!\!s(R)\diff(E)\\
              &=\sum_{\substack{(R,S,F)\in\rectangles(x,w)\times\rectangles(w,z)\times\pdomains(z,y)\\R*S*F=D}}\!\!\!\!\!\!s(R)s(S)F+(-1)^{k-1}\!\!\!\!\!\!\sum_{\substack{(R,F,S)\in\rectangles(x,w)\times\pdomains(w,z)\times\rectangles(z,y)\\R*F*S=D}}\!\!\!\!\!\!s(R)s(S)F\\
              &\qquad{}+(-1)^k\!\!\!\!\!\!\sum_{\substack{(S,F,R)\in\rectangles(x,z)\times\pdomains(z,w)\times\rectangles(w,y)\\S*F*R=D}}\!\!\!\!\!\!s(R)s(S)F-\!\!\!\!\!\!\sum_{\substack{(F,S,R)\in\pdomains(x,z)\times\rectangles(z,w)\times\rectangles(w,y)\\F*S*R=D}}\!\!\!\!\!\!s(R)s(S)F.
  \end{align*}
  The second and the third terms cancel. For the first term, if the
  index-$2$ domain $R*S\in\pdomains(x,z)$ is not a horizontal
  annulus or a vertical annulus, then it has a unique other
  decomposition which contributes with the opposite sign. Therefore,
  it only contributes when $x=z$ and $R*S$ is a horizontal or a
  vertical annulus. Similarly, the fourth term only contributes when
  $z=y$ and $S*R$ is a horizontal annulus or a vertical
  annulus. These two terms contribute with opposite signs, and hence cancel.
\end{proof}

\begin{remark}
  A similar complex of positive pairs, denoted $\cp^*$, is defined in
  \cite[Section 4.1]{MOT}, and a certain obstruction class lives in
  its cohomology. Roughly, the complex $\cp^*$ is generated by pairs
  of generators such that there exists a positive domain between them;
  in other words, it is generated by positive domains modulo an
  equivalence relation given by adding or subtracting periodic
  domains. By contrast, the complex $\CD_*$ is generated by positive
  domains, without dividing by an equivalence relation.
\end{remark}

We will spend the rest of this section in showing that the complex $\CD_*$ has no
interesting homology.

\begin{proposition}\label{prop:pos-domain-cx-homology}
  The complex of positive domains, $\CD_*$, has homology $\ZZ$
  supported in grading $0$, generated by the trivial domain
  $\const{x}$ for some generator $x$.
\end{proposition}

In order to prove this, we need to define a few objects and establish
some of their properties, which we do in the following subsections.

\subsection{Decompositions into rectangles}  For any domain $D$, let
$A(D)\in\NN^{n-1}$ be the vector recording the coefficients of $D$ in
the rightmost vertical annulus; that is, the $i\th$ component of
$A(D)$ is the coefficient of $D$ at the region $H_{(i)}\cap
V_{(n)}$. Similarly, let $B(D)\in\NN^{n-1}$ be the vector recording the
coefficients of $D$ in the topmost horizontal annulus.

\begin{lemma}\label{lem:starting-rectangle}
  If $D\in\pdomains(x,y)$ contains no horizontal (respectively,
  vertical) annulus---that is, if $D*(-H_i)$ (respectively,
  $D*(-V_i)$) is not a positive domain for any $i$---and has
  $A(D)\neq 0$ (respectively, $B(D)\neq 0$), then there is a
  decomposition $D=E*R$, with $E\in\pdomains(x,z)$ and
  $R\in\rectangles(z,y)$ and $A(R)\neq 0$ (respectively,
  $B(R)\neq 0$).
\end{lemma}

\begin{proof}
  We prove the case when $D$ contains no horizontal annulus and
  $A(D)\neq 0$. The other case is similar.

  By Lemma~\ref{lem:decompose-into-rectangles}, there exists at least one
  decomposition of $D$ into rectangles
  \[
    D=R_1*R_2*\cdots*R_n\qquad
    R_1\in\rectangles(x=w_0,w_1),R_2\in\rectangles(w_1,w_2),\dots,
    R_n\in\rectangles(w_{n-1},w_n=y).
  \]
  Since we assumed that $A(D)\neq 0$, there is some $i$
  such that $A(R_i)\neq 0$. Given such a decomposition
  $\maximal$ of $D$, let $\iota(\maximal)$ be the largest such $i$.

  We claim that if $\iota(\maximal)\neq n$, then there is some other
  decomposition $\maximal'$ with
  $\iota(\maximal')=\iota(\maximal)+1$. If $R_1,R_2,\dots,R_n$ are the
  rectangles appearing in $\maximal$, look at the domain
  \[
    H=R_{\iota(\maximal)}*R_{\iota(\maximal)+1}\in\pdomains(w_{\iota(\maximal)-1},w_{\iota(\maximal)+1}).
  \]
  By assumption $A(R_{\iota(\maximal)})\neq 0$ and
  $A(R_{\iota(\maximal)+1})= 0$. Therefore, $H$ is not a vertical
  annulus.  We have already assumed that $D$ does not contain any
  horizontal annulus, so $H$ is not a horizontal annulus either.
  Therefore, $H$ is either a (possibly non-disjoint) union of two
  rectangles or a hexagon, as pictured in
  Item~(\ref{item:maslov-index-2}).

  In each case, we claim that if $H=S*T$ is the (unique) other
  decomposition of $H$ into rectangles with
  $S\in\rectangles(w_{\iota(\maximal)-1},w')$ and
  $T\in\rectangles(w',w_{\iota(\maximal)+1})$, then $A(T)\neq 0$. This
  is clear in the first case when $H$ is a union of two rectangles. In
  the second case, depending on the shape of $H$ and how it intersects
  $V_{(n)}$, the rightmost vertical annulus (shown as a vertical line in the
  following figure), there are the following eight possibilities; in
  each case, we have shown a decomposition of $H=S*T$ with $A(T)\neq
  0$. (Two of the following configurations---the second and the
  eighth---cannot actually appear since they do not admit any
  decomposition $R_{\iota(\maximal)}*R_{\iota(\maximal)+1}$ with
  $A(R_{\iota(\maximal)+1})=0$.)
  \[
    \begin{tikzpicture}[scale=0.75]
      \foreach \i/\r in {0/0,1/0,2/1,3/1}{
        \foreach \j in {0,1}{
        \begin{scope}[yshift=3*\j cm, xshift=3*\i cm,rotate=90*\i]
      \draw[fill=black!30] (-1,-1)-- (1,-1)-- (1,0)-- (0,0)-- (0,1)-- (-1,1)--cycle;
      \end{scope}
      \begin{scope}[yshift=3*\j cm, xshift=3*\i cm]
      \draw (-0.5+\j,-1.5)--++(0,3);
      \begin{scope}[rotate=\r*90+\j*90]
        \draw (-1,0) --++(2,0);
      \end{scope}
      \end{scope}
    }}
    \end{tikzpicture}
  \]

  Therefore, if we look at the decomposition $\maximal'$
  \[
    D=R_1*\dots*R_{\iota(\maximal)-1}*S*T*R_{\iota(\maximal)+2}*\dots*R_n,
  \]
  then $\iota(\maximal')=\iota(\maximal)+1$. Consequently, there is
  some decompsition $\maximal$ with $\iota(\maximal)=n$. That is, $D$
  has a decomposition $E*R$, with $E\in\pdomains(x,z)$ and
  $R\in\rectangles(z,y)$ with $A(R)\neq 0$.
\end{proof}

\subsection{The partial order on generators}\label{sec:generator-poset} 
Let us first introduce some notation. In the symmetric group, we will denote by $\tau_p$ the adjacent transposition $(p, p+1)$. Further, in any partially ordered set, when $y \leq x$, we denote by $[y,x]$ the interval consisting of all $z$ with $y \leq z \leq x$.

Next, recall that the standard (strong) \emph{Bruhat order} on the symmetric group is defined as
follows. For any permutation $\sigma$, a \emph{reduced word} for
$\sigma$ is a minimal decomposition of $\sigma$ as a product of
adjacent transpositions. All reduced words for $\sigma$ have the same
length, which we denote $\card{\sigma}$. Define $\sigma\leq\tau$ if
some (not necessarily consecutive) substring of some (equivalently,
every) reduced word for $\tau$ is a reduced word for $\sigma$.

Now define the following partial order on the set $\S$ of generators:
\[
y\leq x\text{ if }\set{D\in\pdomains(x,y)}{A(D)=B(D)=0}\neq\emptyset.
\]

The relation between this partial order and the Bruhat order is
explained below.

\begin{enumerate}[leftmargin=*,label=(P-\arabic*),ref=P-\arabic*]
\item If $x^\sigma$ denotes the generator corresponding to the
  permutation $\sigma$ of $\{1,2,\dots,n\}$ from
  Item~(\ref{item:permutation-generator}), then $x^\sigma\leq x^\tau$
  if and only if $\sigma\geq \tau$, that is, the above order is the
  \emph{opposite} of the usual Bruhat order on the symmetric group.
 \item The poset has a unique maximum $x^{\Id}$, the generator
  corresponding to the identity permutation. For any permutation
  $\sigma$, there is a unique (positive) domain
  $D_\sigma\in\pdomains(x^{\Id}, x^\sigma)$ with
  $A(D_\sigma)=B(D_\sigma)=0$---that is, $D_\sigma$ avoids the
  rightmost vertical annulus and the topmost horizontal annulus.
\item For any reduced word $\sigma_1\sigma_2\cdots\sigma_k$ for
  $\sigma$, there is a decomposition of $D_\sigma=R_1*R_2*\dots*R_k$
  into rectangles so that, for all $1\leq i\leq k$, $R_i$ is a
  width-one rectangle supported in the vertical annulus $V_{(j)}$,
  where $\sigma_i$ is the adjacent permutation
  $\tau_j=(j,j+1)$. Therefore, minimal words for $\sigma$ correspond
  to decompositions of $D_\sigma$ into width-one rectangles. In
  particular, $\mu(D_\sigma)=\card{\sigma}$.
\item\label{item:remove-from-end-is-reduced} If
  $\sigma_1\sigma_2\cdots\sigma_k$ is a reduced word for $\sigma$, then
  $\sigma_1\sigma_2\cdots\sigma_{k-1}$ is a reduced word for
  $\sigma\sigma_k$. To wit, if $R_1*R_2*\dots*R_k$ is the decomposition
  of $D_\sigma$ into width-one rectangles corresponding to
  $\sigma_1\sigma_2\cdots\sigma_k$, then $R_1*R_2*\dots*R_{k-1}$ is a
  decomposition of $D_{\sigma\sigma_k}$ into width-one rectangles.
\item\label{item:add-to-end-is-reduced-easy} If the coordinate of
  $x^\sigma$ on $\beta_p$ lies to the bottom-left of the coordinate of
  $x^\sigma$ on $\beta_{p+1}$ for some $1\leq p<n$, then for any
  reduced word $\sigma_1\sigma_2\cdots\sigma_k$ of $\sigma$,
  $\sigma_1\sigma_2\cdots\sigma_k\tau_p$ is a reduced word for
  $\sigma\tau_p$.  The proof is similar to that in
  Item~(\ref{item:remove-from-end-is-reduced}). There is a rectangle
  $R\in\rectangles(x^\sigma,x^{\sigma\tau_p})$ with width one,
  supported in $V_{(p)}$ and avoiding $H_{(n)}$. If
  $D_\sigma=R_1*R_2*\dots*R_k$ is the decomposition into width-one
  rectangles corresponding to $\sigma_1\sigma_2\cdots\sigma_k$, then
  $D_{\sigma\tau_p}=R_1*R_2*\dots*R_k*R$ is a decomposition into
  width-one rectangles corresponding to
  $\sigma_1\sigma_2\cdots\sigma_k\tau_p$.
\item\label{item:last-transposition-geometric-meaning} For any $1\leq
  p<n$, the permutation $\sigma$ has a reduced word ending in the transposition
  $\tau_p$ if and only if the coordinate of $x^\sigma$ on $\beta_p$
  lies to the top-left of the coordinate of $x^\sigma$ on
  $\beta_{p+1}$.

  One direction is clear. If $\sigma$ has a reduced word ending in
  $\tau_p$, then $D_\sigma$ has a decomposition $E*R$, with
  $E\in\pdomains(x^{\Id},y)$ and $R\in\rectangles(y,x^\sigma)$ being a
  width-one rectangle supported in the vertical annulus $V_{(p)}$ (and
  avoiding the top horizontal annulus $H_{(n)}$), and hence the
  coordinate of $x^\sigma$ on $\beta_p$ lies to the top-left of the
  coordinate of $x^\sigma$ on $\beta_{p+1}$.  The proof for the other
  direction is similar to Lemma~\ref{lem:starting-rectangle}. Let
  $\word=\sigma_1\sigma_2\cdots\sigma_k$ be a reduced word for
  $\sigma$. By Item~(\ref{item:remove-from-end-is-reduced}),
  $\eta_i=\sigma_1\sigma_2\cdots\sigma_i$ is a reduced word. Call a
  permutation to be \emph{inverted} if its coordinate on $\beta_p$
  lies to the top-left of its coordinate on $\beta_{p+1}$. By
  assumption $\eta_k=\sigma$ is inverted, while $\eta_0=x^{\Id}$ is
  not. Let $\iota(\word)$ be the smallest $i$, so that
  $\eta_i,\eta_{i+1},\dots,\eta_k$ are all inverted. Since
  $\eta_{\iota(\word)-1}$ is not inverted, but
  $\eta_{\iota(\word)}=\eta_{\iota(\word)-1}\sigma_{\iota(\word)}$ is,
  we must have $\sigma_{\iota(\word)}=\tau_p$.

  If $\iota(\word)\neq k$, we will find a new reduced word $\word'$
  for $\sigma$ with $\iota(\word')=\iota(\word)+1$. Continuing, we
  will eventually find a word with $\iota=k$, and we will be done. 

  If $\sigma_{\iota(\word)+1}$ is a transposition that is far from
  $\tau_p$, then switching $\sigma_{\iota(\word)}=\tau_p$ and
  $\sigma_{\iota(\word)+1}$ works; that is,
  $\word'=\sigma_1\sigma_2\cdots\sigma_{\iota(\word)-1}\sigma_{\iota(\word)+1}\tau_p\sigma_{\iota(\word)+2}\cdots\sigma_k$
  has $\iota(\word')=\iota(\word)+1$. Now let us do the case
  $\sigma_{\iota(\word)+1}=\tau_{p-1}$ (the case $\tau_{p+1}$ is
  similar).  Note that $\tau_p\tau_{p-1} = (p-1, p+1)\cdot \tau_p$,
  where $(p-1,p+1)$ denotes the non-adjacent transposition. Therefore,
  $\eta_{\iota(\word)+1}=\sigma'\tau_p$, where
  $\sigma'=\sigma_1\sigma_2\cdots\sigma_{\iota(\word)-1}\cdot
  (p-1,p+1)$. Since we have assumed that
  $\eta_{\iota(\word)+1}$ is also inverted, the index-two domain
  corresponding to $\tau_p\tau_{p-1}=(p-1,p+1)\cdot \tau_p$ looks like
  the third hexagon from Item~(\ref{item:maslov-index-2}). Therefore,
  $\sigma'$ is not inverted, and therefore, for any reduced word
  $\sigma'_1\sigma'_2\cdots\sigma'_{\iota(\word)}$ for $\sigma'$, we
  have that $\sigma'_1\sigma'_2\cdots\sigma'_{\iota(\word)}\tau_p$ is
  a reduced word for $\eta_{\iota(\word)+1}$ (by
  Item~\ref{item:add-to-end-is-reduced-easy}), and hence
  $\word'=\sigma'_1\sigma'_2\cdots\sigma'_{\iota(\word)}\tau_p\sigma_{\iota(\word)+2}\cdots\sigma_k$
  has $\iota(\word')=\iota(\word)+1$.
\item\label{item:add-to-end-is-reduced} If $\sigma$ does not have a
  reduced word ending in the transposition $\tau_p$, then for any
  reduced word $\sigma_1\sigma_2\cdots\sigma_k$ of $\sigma$, we have that 
  $\sigma_1\sigma_2\cdots\sigma_k\tau_p$ is a reduced word for
  $\sigma\tau_p$. This follows immediately from
  Items~(\ref{item:add-to-end-is-reduced-easy})
  and~(\ref{item:last-transposition-geometric-meaning}).
\end{enumerate} 
  
 \subsection{Plausible triples}
\label{sec:plausible}
Given two partially ordered sets $S_1, \dots, S_m$, the \emph{product partial order} on $S_1 \times \dots \times S_m$ is given by 
\[
  (s_1, \dots, s_m) \leq (s_1', \dots, s_m') \ \iff \ (s_i \leq s_i' \text{ for all } i).
\]
We will give $\NN^{n-1}$ the product partial order coming from its factors.

In the proof of Proposition~\ref{prop:pos-domain-cx-homology} that
will be given in Section~\ref{sec:acyclic}, we will filter positive
domains according to the vectors $A(D), B(D)$ that capture their
multiplicties on the rightnots column and topmost row. In the process,
given a triple $(a,b,y)$, with $y \in \S$ and
$(a,b)\in\NN^{n-1}\times\NN^{n-1}$, we will be interested in the set
of generators
\[ 
G^{a,b,y}=\set{x \in \S}{\exists D\in\pdomains(x,y),A(D)=a,B(D)=b}.
\]
This is an upward closed subset: that is, if $x\in G^{a,b,y}$ and
$x\leq x'$, then $x'\in G^{a,b,y}$. Therefore, $G^{a,b,y}$ always
contains $x^{\Id}$. Moreover, if $D\in\pdomains(x,y)$ with
$A(D)\leq a$ and $B(D)\leq b$, then $x\in G^{a,b,y}$, since there
exists a (unique) periodic domain $E\in\pperiodic$ with
$(A(E),B(E))=(a-A(D),b-B(D))$, and therefore, $D*E\in\pdomains(x,y)$
satisfies the required condition. That is, $G^{a,b,y}$ has an
alternate description
\[ 
G^{a,b,y}=\set{x}{\exists D\in\pdomains(x,y),A(D)\leq a,B(D)\leq b}.
\]
In particular, since $\const{y}\in\pdomains(y,y)$, the set $G^{a,b,y}$ contains $y$, and hence all $z$ with $z \geq y$. 

We would like to understand in what cases $G^{a,b,y}$ contains more elements than just those in the interval $[y, x^{\Id}]$. An example is shown in Figure~\ref{fig:plausible}, where $z \leq y$ but the rectangle $R$ with $A(R)=a > 0$ and $B(R)=0$ makes it so that $z \in G^{a,0,y}.$

\begin{figure}
  \centering
  \begin{tikzpicture}[scale=4]
    \draw[fill=black!30] (0.7,0.3) rectangle (1,0.6) node[midway] {\small $R$};
    \draw[fill=black!30] (0.2,0.3) rectangle (0,0.6) node[midway] {\small $R$};

    \foreach \i/\j/\a in {0.7/0.3/north east,0.2/0.6/south west}{
      \node[fill=black,minimum width=3pt, minimum height=3pt,inner sep=0,outer sep=0] at (\i,\j) {};
      \node[anchor=\a] at (\i,\j) {\small $z$};
    }
    \foreach \i/\j/\a in {0.7/0.6/south east,0.2/0.3/north west}{
      \node[fill=black,circle,minimum width=3pt, minimum height=3pt,inner sep=0,outer sep=0] at (\i,\j) {};
      \node[anchor=\a] at (\i,\j) {\small $y$};
    }

    \draw[->] (0,0.3-0.05)--++(0.2,0) node[midway,anchor=north] {\small $\tau(R)$};
    \draw[->] (1,0.3-0.05)--++(-0.3,0);

    \draw[<->] (1,0.6+0.05)--++(-0.3,0) node[midway,anchor=south] {\small $\omega(R)$};
    
    \draw (0,0) rectangle (1,1);
  \end{tikzpicture}
\caption {The shaded rectangle $R$ is an $A$-witness.}
\label{fig:plausible}
\end{figure}

It turns out that the specific condition we need is {\em
  plausibility}, as defined below.  Consider triples $(a,b,y)$ with
$y$ a generator and $(a,b)\in\NN^{n-1}\times\NN^{n-1}$. Call such a
triple \emph{A-plausible} (respectively, \emph{B-plausible}) if there
exist $z$ and $R\in\rectangles(z,y)$ with $0< A(R)\leq a$
(respectively, $0< B(R)\leq b$); call such rectangles
\emph{A-witnesses} (respectively, \emph{B-witnesses}). Assign to any
such witness $R$ a pair $(\omega(R),\tau(R))\in\NN^2$, where
$\omega(R)$ is number of vertical annuli to the left of $V_{(n)}$, including $V_{(n)}$,
(respectively, horizontal annuli below $H_{(n)}$, including $H_{(n)}$) that $R$ intersects,
and $\tau(R)$ is the horizontal width (respectively, vertical height)
of $R$. See again Figure~\ref{fig:plausible}.

\begin{lemma}\label{lem:easy-technical-poset-lemma}
  If $(a,b,y)$ is neither A-plausible nor B-plausible, and
  $D\in\pdomains(x,y)$ with $(A(D),B(D))=(a,b)$, then $x\geq y$.
\end{lemma}

\begin{proof}
  Look at decompositions $D=E*F$ where $E\in\pperiodic$ and
  $F\in\pdomains(x,y)$, and consider the one that maximizes the Maslov
  index of $E$. Then $F$ does not contain any vertical annulus or any
  horizontal annulus.

  By Lemma~\ref{lem:starting-rectangle}, if $(A(F),B(F))\neq(0,0)$,
  then $F$ has a decomposition $G*R$, with $G\in\pdomains(x,z)$ and
  $R\in\rectangles(z,y)$ with $(A(R),B(R))\neq (0,0)$. Then $R$ is a
  witness, which contradicts the hypothesis. Therefore, we must have
  $(A(F),B(F))=(0,0)$. Then $x\geq y$ due to the domain $F$, and we
  are done.
\end{proof}

\begin{lemma}\label{lem:main-technical-poset-lemma}
  If $(a,b,y)$ is A-plausible, and $D\in\pdomains(x,y)$ with
  $(A(D),B(D))=(a,b)$, then there exists an A-witness
  $R\in\rectangles(z,y)$ and $E\in\pdomains(x,z)$ with $E*R=D$. We may
  choose $R$ to be one that minimizes $\omega$ among all
  A-witnesses. In fact, we may choose $R$ to be the (unique) A-witness
  $R_0$ that minimizes the pair $(\omega,\tau)$, ordered
  lexicographically, among all A-witnesses.

  Analogous statements hold if $(a,b,y)$ is B-plausible.
\end{lemma}

\begin{proof}
  Let us only consider the case for A-plausible. The other case is
  similar. We prove this by induction on the Maslov index of
  $D$. There are three statements in the problem, and for clarity, we
  write them out. Each statement is weaker than the next.

  \begin{itemize}
  \item[$(P_1^n)$] If $(a,b,y)$ is A-plausible, and
    $D\in\pdomains(x,y)$ with $\mu(D)=n$ and $(A(D),B(D))=(a,b)$, then
    there exists an A-witness $R\in\rectangles(z,y)$ and
    $E\in\pdomains(x,z)$ with $E*R=D$.
  \item[$(P_2^n)$] If $(a,b,y)$ is A-plausible, and
    $D\in\pdomains(x,y)$ with $\mu(D)=n$ and $(A(D),B(D))=(a,b)$, then
    there exists an A-witness $R\in\rectangles(z,y)$ and
    $E\in\pdomains(x,z)$ with $E*R=D$, and $R$ minimizes $\omega$
    among all A-witnesses.
  \item[$(P_3^n)$] If $(a,b,y)$ is A-plausible, and
    $D\in\pdomains(x,y)$ with $\mu(D)=n$ and $(A(D),B(D))=(a,b)$, then
    there exists an A-witness $R\in\rectangles(z,y)$ and
    $E\in\pdomains(x,z)$ with $E*R=D$, and $R$ minimizes
    $(\omega,\tau)$ among all A-witnesses.
  \end{itemize}

  The base case for the induction is either vacuous or trivial,
  depending on whether one starts at $n=0$ or $n=1$. We will do induction on $n$, and at each  step, we will first get $(P_1^n)$, then $(P_2^n)$, and then $(P_3^n)$. For this, we will make use of the following implications.
  \begin{enumerate}[leftmargin=*,label=(Ind-\arabic*)]
  \item $(P_1^n) \wedge (P_2^{n-1}) \Rightarrow (P_2^n)$. Consider the
    decomposition $D=E*R$ as provided by $(P_1^n)$, with
    $R\in\rectangles(z,y)$. Consider an A-witness $S$ that minimizes
    $\omega$. If $\omega(R)=\omega(S)$, we are done. Otherwise
    $\omega(S)<\omega(R)$; therefore, the top-left corner of $S$ lies
    outside $R$, and the configuration of $S$ (shaded), $R$ (striped),
    and $y$-coordinates (dots) looks like one of the follows (the
    vertical column $V_{(n)}$ is once again shown as a vertical line).
    
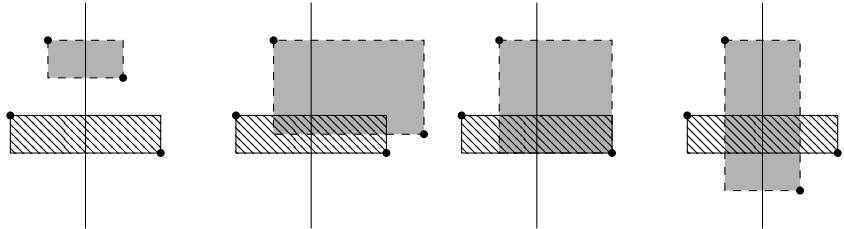
\[
    \begin{tikzpicture}[scale=0.5]
      \begin{scope}[xshift=6cm]
      \draw[dashed, fill=black!30] (-1,4) rectangle (1,3);
      \draw[pattern=north west lines] (-2,1) rectangle (2,2);
      \foreach \i/\j in {-1/4,1/3,-2/2,2/1}{
        \fill[black] (\i,\j) circle (3pt);
        }
      \draw (0,-1) --++(0,6);
      \end{scope}
      \begin{scope}[xshift=2*6cm]
      \draw[dashed, fill=black!30] (-1,4) rectangle (3,1.5);
      \draw[pattern=north west lines] (-2,1) rectangle (2,2);
      \foreach \i/\j in {-1/4,3/1.5,-2/2,2/1}{
        \fill[black] (\i,\j) circle (3pt);
        }
      \draw (0,-1) --++(0,6);
      \end{scope}
      \begin{scope}[xshift=3*6cm]
      \draw[dashed, fill=black!30] (-1,4) rectangle (2,1);
      \draw[pattern=north west lines] (-2,1) rectangle (2,2);
      \foreach \i/\j in {-1/4,2/1,-2/2}{
        \fill[black] (\i,\j) circle (3pt);
        }
      \draw (0,-1) --++(0,6);
      \end{scope}
      \begin{scope}[xshift=4*6cm]
      \draw[dashed, fill=black!30] (-1,4) rectangle (1,0);
      \draw[pattern=north west lines] (-2,1) rectangle (2,2);
      \foreach \i/\j in {-1/4,1/0,-2/2,2/1}{
        \fill[black] (\i,\j) circle (3pt);
        }
      \draw (0,-1) --++(0,6);
      \end{scope}
    \end{tikzpicture}
    \]
    In each case, the domain $E\in\pdomains(x,z)$ produces an
    A-plausible triple $(z, A(E), B(E))$ and the following A-witness
    (shaded) has minimum $\omega$, which equals $\omega(S)$. The
    coordinates of $z$ are shown as black squares. The fourth case has
    been subdivided into two cases: in the first subcase, there are no
    $z$ coordinates in the interior of the striped rectangle, while in
    the second subcase, there are a few; in that subcase, the
    A-witness uses the leftmost of those extra $z$-coordinates.
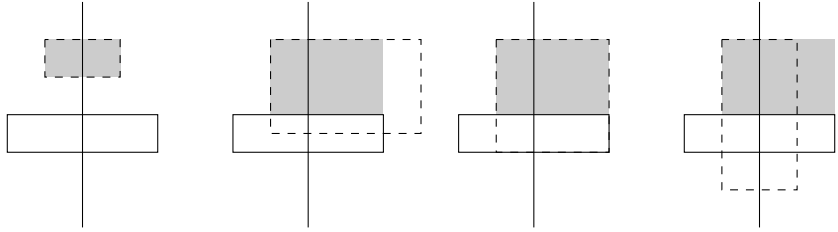
    \[
    \begin{tikzpicture}[scale=0.5]
      \begin{scope}[xshift=6cm]
        \fill[black!20] (-1,4) rectangle (1,3);
      \draw[dashed] (-1,4) rectangle (1,3);
      \draw (-2,1) rectangle (2,2);
      \draw (0,-1) --++(0,6);
      \foreach \i/\j in {-1/4,1/3,-2/1,2/2}{
        \node[fill=black,minimum width=3pt, minimum height=3pt,inner sep=0,outer sep=0] at (\i,\j) {};
        }
      \end{scope}
      \begin{scope}[xshift=2*6cm]
        \fill[black!20] (-1,4) rectangle (2,2);
      \draw[dashed] (-1,4) rectangle (3,1.5);
      \draw (-2,1) rectangle (2,2);
      \draw (0,-1) --++(0,6);
      \foreach \i/\j in {-1/4,3/1.5,-2/1,2/2}{
        \node[fill=black,minimum width=3pt, minimum height=3pt,inner sep=0,outer sep=0] at (\i,\j) {};
        }
      \end{scope}
      \begin{scope}[xshift=3*6cm]
        \fill[black!20] (-1,4) rectangle (2,2);
      \draw[dashed] (-1,4) rectangle (2,1);
      \draw (-2,1) rectangle (2,2);
      \draw (0,-1) --++(0,6);
      \foreach \i/\j in {-1/4,-2/1,2/2}{
        \node[fill=black,minimum width=3pt, minimum height=3pt,inner sep=0,outer sep=0] at (\i,\j) {};
        }      
      \end{scope}
      \begin{scope}[xshift=4*6cm]
        \fill[black!20] (-1,4) rectangle (2,2);
      \fill[pattern=north east lines, pattern color=black!60] (1,2) rectangle (2,4);        
      \draw[dashed] (-1,4) rectangle (1,0);
      \draw (-2,1) rectangle (2,2);
      \draw (0,-1) --++(0,6);
      \foreach \i/\j in {-1/4,1/0,-2/1,2/2}{
        \node[fill=black,minimum width=3pt, minimum height=3pt,inner sep=0,outer sep=0] at (\i,\j) {};
        }
      \end{scope}
      \begin{scope}[xshift=5*6cm]
        \fill[black!20] (-1,4) rectangle (1.3,3.4);
      \fill[pattern=north east lines, pattern color=black!60] (1,2) rectangle (2,4);        
      \draw[dashed] (-1,4) rectangle (1,0);
      \draw (-2,1) rectangle (2,2);
      \draw (0,-1) --++(0,6);
      \foreach \i/\j in {-1/4,1/0,-2/1,2/2,1.3/3.4,1.6/2.7}{
        \node[fill=black,minimum width=3pt, minimum height=3pt,inner sep=0,outer sep=0] at (\i,\j) {};
        }
      \end{scope}
    \end{tikzpicture}
    \]
    Therefore, by $(P_2^{n-1})$, we have a decomposition $E=F*T$, with
    $F\in\pdomains(x,w)$ and $T\in\rectangles(w,z)$ with
    $\omega(T)=\omega(S)$. Therefore, the Maslov index 2 domain
    $H=T*R\in\pdomains(w,y)$ looks like one of the following (the
    decomposition $T*R$ and the $y$-coordinates are also shown).
    
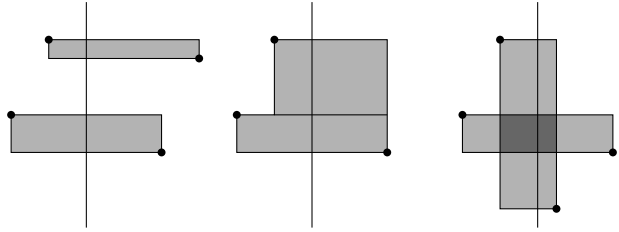
\[
    \begin{tikzpicture}[scale=0.5]
      \begin{scope}[xshift=6cm]
      \draw[fill=black!30] (-1,4) rectangle (3,3.5);
      \draw[fill=black!30] (-2,1) rectangle (2,2);
      \foreach \i/\j in {-1/4,3/3.5,-2/2,2/1}{
        \fill[black] (\i,\j) circle (3pt);
        }
      \draw (0,-1) --++(0,6);
      \end{scope}
      \begin{scope}[xshift=2*6cm]
      \draw[fill=black!30] (-1,4) -- (2,4) -- (2,1) -- (-2,1) -- (-2,2) -- (-1,2) -- cycle;
      \foreach \i/\j in {-1/4,-2/2,2/1}{
        \fill[black] (\i,\j) circle (3pt);
        }
      \draw (-1,2) -- (2,2);
      \draw (0,-1) --++(0,6);
      \end{scope}
      \begin{scope}[xshift=3*6cm]
      \fill[black!30] (-1,4) rectangle (0.5,-0.5);
      \fill[black!30] (-2,1) rectangle (2,2);
      \fill[black!60] (-1,1) rectangle (0.5,2);
      \draw (-1,4) rectangle (0.5,-0.5);
      \draw (-2,1) rectangle (2,2);
      \foreach \i/\j in {-1/4,0.5/-0.5,-2/2,2/1}{
        \fill[black] (\i,\j) circle (3pt);
        }
      \draw (0,-1) --++(0,6);
      \end{scope}
    \end{tikzpicture}
    \]
    In each case, the other decomposition $H=T'*R'$ satisfies
    $\omega(R')=\omega(S)$. Therefore, we have a decomposition
    $D=(F*T')*R'$, with $R'$ an A-witness minimizing $\omega$.
  \item $(P_2^n) \wedge (P_3^{n-1}) \Rightarrow (P_3^n)$. The proof is
    similar to (but easier than) the previous proof. Consider the
    decomposition $D=E*R$ as provided by $(P_2^n)$, with
    $R\in\rectangles(z,y)$ minimizing $\omega$. Consider an A-witness
    $S$ that minimizes $(\omega,\tau)$, ordered lexicographically. We
    must have $\omega(R)=\omega(S)$. If in addition,
    $\tau(R)=\tau(S)$, we are done. Otherwise $\tau(S)<\tau(R)$;
    therefore, the configuration of $S$ (shaded), $R$ (striped), and
    $y$-coordinates (dots) looks as follows.
    
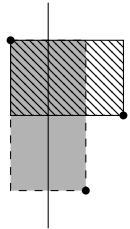
\[
    \begin{tikzpicture}[scale=0.5]
      \draw[dashed, fill=black!30] (-1,4) rectangle (1,0);
      \draw[pattern=north west lines] (-1,4) rectangle (2,2);
      \foreach \i/\j in {-1/4,1/0,2/2}{
        \fill[black] (\i,\j) circle (3pt);
        }
      \draw (0,-1) --++(0,6);
    \end{tikzpicture}
    \]
    Therefore, the domain $E\in\pdomains(x,z)$ is A-plausible and the
    following A-witness $T$ (shaded) has minimum $(\omega,\tau)$, which equals
    $(\omega(S),\tau(S))$.
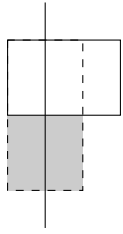
    \[
    \begin{tikzpicture}[scale=0.5]
      \fill[black!20] (-1,2) rectangle (1,0);
      \draw[dashed] (-1,4) rectangle (1,0);
      \draw (-1,4) rectangle (2,2);
      \draw (0,-1) --++(0,6);
    \end{tikzpicture}
    \]
    By $(P_3^{n-1})$, we have a decomposition $E=F*T$ with
    $F\in\pdomains(x,w)$ and $T\in\rectangles(w,z)$. The Maslov index
    2-domain $H=T*R\in\pdomains(w,y)$ is a hexagon, and the other
    decomposition $H=T'*R'$ satisfies
    $(\omega(R'),\tau(R'))=(\omega(S),\tau(S))$. Therefore, we have a
    decomposition $D=(F*T')*R'$, with $R'$ the unique A-witness
    minimizing $(\omega,\tau)$.
  \item $(P_1^{n-1})\wedge (P_3^n) \Rightarrow
    (P_1^{n+1})$. Let $D$ be the given domain with $\mu(D)=n+1$, and
    consider the A-witness $R$ for $D$ which minimizes
    $(\omega,\tau)$. Since $R$ minimizes $(\omega,\tau)$, none of the 
    $y$-coordinates can lie in the interior of $R$.

    If $D$ does not contain any horizontal annuli, we are done by
    Lemma~\ref{lem:starting-rectangle}. Therefore, assume $D$ contains
    some horizontal annulus $H$. If $H$ is disjoint from the interior
    of $R$, then $D*(-H)$ is a domain with $\mu=n-1$, which
    is still A-plausible since it still contains the A-witness
    $R$. Therefore, by $(P_1^{n-1})$, it admits a decomposition $E*T$
    with $E\in\pdomains(x,w),T\in\rectangles(w,y)$ with $A(T)\neq 0$;
    consequently, $D$ has a decomposition $(E*H)*T$ and we are done.

    Therefore, we may assume that $D$ contains some horizontal annulus
    $H$ that intersects $R$.  Let $H=S*T$ with
    $S\in\rectangles(y,w),T\in\rectangles(w,y)$ be the unique
    decomposition of $H$ into rectangles. Exactly one of $A(S)$ and
    $A(T)$ is non-zero. If $A(T)\neq 0$, we are done, since $D$ then
    has a decomposition $(D*(-T))*T$. So we may assume $A(S)\neq 0$ and
    $A(T)=0$.

    Therefore, the configuration of the A-witness $R$ (shaded), the
    horizontal annulus $H$ (striped), and the $y$-coordinates (dots)
    looks as follows.
    
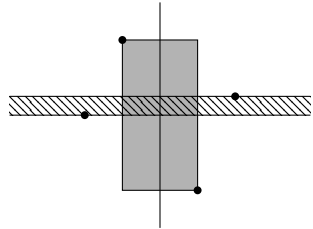
\[
    \begin{tikzpicture}[scale=0.5]
      \draw[fill=black!30] (-1,4) rectangle (1,0);
      \fill[pattern=north west lines] (-4,2.5) rectangle (4,2);
      \draw (-4,2.5)--(4,2.5);
      \draw (-4,2)--(4,2);
      \foreach \i/\j in {-1/4,1/0,-2/2,2/2.5}{
        \fill[black] (\i,\j) circle (3pt);
        }
      \draw (0,-1) --++(0,6);
    \end{tikzpicture}
    \]
    Therefore, the domain $D*(-T)\in\pdomains(x,w)$ is A-plausible, with
    $R$ still being the A-witness that minimizes $(\omega,\tau)$. By
    $(P_3^n)$, $D*(-T)$ contains the rectangle $R$, and
    therefore, $D$ contains $R$ as well.\qedhere
  \end{enumerate}
\end{proof}

\subsection{Proof of acyclicity}
\label{sec:acyclic}
Given a triple $(a,b,y)$, with $y$ a generator and 
$(a,b)\in\NN^{n-1}\times\NN^{n-1}$, in Section~\ref{sec:plausible} we defined the following set of generators
\begin{align*}
G^{a,b,y}&=\set{x \in \S}{\exists D\in\pdomains(x,y),A(D)=a,B(D)=b}\\
&=\set{x \in \S}{\exists D\in\pdomains(x,y),A(D)\leq a, B(D)\leq b}.
\end{align*}
This is an upward closed subset that contains $y$.

\begin{lemma}\label{lem:g-poset-has-minimum}
  The set $G^{a,b,y}$ has a unique minimum $m^{a,b,y}$, so that $G^{a,b,y}$ is the interval $ [m^{a,b,y},x^{\Id}].$ Furthermore, $m^{a,b,y}$ equals
  $x^{\Id}$ if and only if $a=b=0$ and $y=x^{\Id}$.
\end{lemma}

\begin{proof}
We prove this by induction on $(a,b)$, viewed as an element of the
poset $\NN^{2n-2}$ under the product partial order. For the base case,
we have $G^{0,0,y}=[y,x^{\Id}]$, and so it has a unique minimum
$m^{0,0,y}\defeq y$.

Now consider the case $(a,b)\neq (0,0)$. For the first part, if
$(a,b,y)$ is neither A-plausible nor B-plausible, then it follows from
Lemma~\ref{lem:easy-technical-poset-lemma} that $m^{a,b,y}$ exists and
equals $y$. On the other hand, if $(a,b,y)$ is A-plausible
(respectively, B-plausible), and $R_0\in\rectangles(z,y)$ is the
unique A-witness (respectively B-witness) that minimizes
$(\omega,\tau)$, then it follows from
Lemma~\ref{lem:main-technical-poset-lemma} (using induction on
$(a,b)$) that $m^{a,b,y}$ exists and equals $m^{a-A(R_0),b,z}$
(respectively, $m^{a,b-B(R_0),z}$).

For $(a,b)\neq (0,0)$, the second part follows from the first
part. Let $D\in\pdomains(y,y)$ with $(A(D),B(D))= (a,b)$, and consider
some decomposition of $D$ into rectangles:
\[
D=R_1*R_2*\dots*R_n\qquad R_1\in\rectangles(y=w_0,w_1),R_2\in\rectangles(w_1,w_2),\dots, R_n\in\rectangles(w_{n-1},w_n=y).
\]
Clearly, $w_i\in G^{a,b,y}$ for all $i$ because of the domain
$R_{i+1}*\dots*R_n\in\pdomains(w_i,y)$. Since $(a,b)\neq (0,0)$, $D$
is non-trivial, and therefore, there is at least one rectangle, and
consequently, the set $\{w_0,\dots,w_n\}$ contains at least two
elements. Therefore, $G^{a,b,y}$ is not the one-element set
$\{x^{\Id}\}$.
\end{proof}

We are now ready to prove Proposition~\ref{prop:pos-domain-cx-homology}.
\begin{proof}[Proof of Proposition~\ref{prop:pos-domain-cx-homology}]
  The idea of the proof is to construct a sequence of filtrations on the
  chain complex, and to prove that various associated graded complexes are
  acyclic.

  For domain $D\in \CD_*$, let $(A(D),B(D))$ be its filtration grading
  in the product partial order on $\NN^{2n-2}$. It is clear that the
  differential either preserves $(A,B)$ or lowers it. For
  $(a,b)\in\NN^{n-1}\times\NN^{n-1}$, let $\CD^{a,b}_*$ be the
  associated graded complex in filtration grading $(a,b)$. We will
  prove that $\CD^{a,b}_*$ is acyclic if $(a,b)\neq (0,0)$, and
  $\CD_*^{0,0}$ has homology $\ZZ$ generated by $\const{x^{\Id}}$.

  Now put a new filtration grading on $\CD^{a,b}_*$ as follows. For
  any domain $D\in\pdomains(x,y)$ with $(A(D),B(D))=(a,b)$, define its
  filtration grading to be $y$, viewed as an element of the poset from
  Section \ref{sec:generator-poset}. The differential $\delta$ on
  \emph{the associated graded complex} $\CD^{a,b}_*$ either preserves
  $y$ or increases it. Now let $\CD^{a,b,y}_*$ be associated graded
  complex consisting of only those domains that end at $y$. Now, it
  is enough to show that $\CD^{a,b,y}_*$ is acyclic, unless $a=b=0$
  and $y=x^{\Id}$. When $a=b=0$
  and $y=x^{\Id}$, the homology is clearly $\ZZ$, 
  generated by the trivial domain $\const{x^{\Id}}$.

  The complex $\CD^{a,b,y}_*$ is generated by domains
  $D\in\pdomains(x,y)$ with $(A(D),B(D))=(a,b)$. Note, if there is
  such a domain, then $x\in G^{a,b,y}$ by definition, and conversely,
  for any $x\in G^{a,b,y}$, there is a unique such positive domain
  $D$. Note that $G^{a,b,y}=[m^{a,b,y},x^{\Id}]$ by
  Lemma~\ref{lem:g-poset-has-minimum}.  Therefore, the complex
  $\CD^{a,b,y}_*$ is isomorphic to the following complex (which
  resembles the grid complex from
  Section~\ref{sec:grid-complex-background}). It is generated by the
  elements of $[m^{a,b,y},x^{\Id}]$, and the differential on a
  generator is given by
  \[
  \diff(x)=\sum_{\substack{m^{a,b,y}\leq z< x\\R\in\rectangles(x,z)\\A(R)=B(R)=0}}s(R)z.
  \]
  
  If in the above formula we have $m^{a,b,y}= x^\sigma$, $x=x^\theta$ and $z=x^\eta$, for some permutations $\sigma,\theta,\eta$, then the
  condition
  \[
  m^{a,b,y}\leq z< x\text{ and }\exists R\in\rectangles(x,z)\text{ with }A(R)=B(R)=0
  \]
  is equivalent to
  \[
  \sigma\geq \eta>\theta\text{ and }|\eta|=|\theta|+1\text{ and }R=(-D_\theta)*D_\eta.
  \]
  Therefore, the complex $\CD^{a,b,y}_*$ is isomorphic to the following. It is
  generated by permutations in $[\Id,\sigma]$, and the differential on
  a generator is given by
  \[
  \diff(\theta)=\sum_{\substack{\sigma\geq \eta>\theta\\|\eta|=|\theta|+1}}s((-D_\theta)*D_\eta)\eta.
  \]

Now fix some reduced word $\sigma_1\sigma_2\cdots\sigma_k$ for
  $\sigma$. If $(a,b)\neq (0,0)$, then $m^{a,b,y}\neq x^{\Id}$ (once
  again, using Lemma~\ref{lem:g-poset-has-minimum}), and hence
  $k>0$. Let $\sigma_k$ be the transposition $\tau_p=(p,p+1)$. Define a grading
  on permutations by declaring its value on $\theta$ to be
  $\card{\theta}-1$ if $\theta$ has a reduced word ending in
  $\tau_p$, and $\card{\theta}$ otherwise. Since the differential
  increases the length $\card{\cdot}$ by one, this defines a
  filtration grading on above complex.

  We claim that the associated graded complex is a direct sum of
  two-generator acyclic complexes, and hence is acyclic. If $\eta$ and
  $\theta$ are in the same filtration grading and $\eta$ appears in
  $\diff(\theta)$, then the filtration grading must be
  $\card{\eta}-1=\card{\theta}$. Therefore, $\eta$ has a reduced word,
  say $\word$ of length $\ell$, ending in $\tau_p$; since
  $\theta<\eta$ with $\card{\theta}=\card{\eta}-1$, $\theta$ has a
  reduced word $\word'$ of length $\ell-1$ which is a sub-word of
  $\word$. But since $\theta$ does not have any reduced word ending
  in $\tau_p$, $\word'$ must be obtained from $\word$ by deleting
  $\tau_p$ from the end. That is, $\eta=\theta\tau_p$, and there is
  a (width-one) rectangle from $x^\theta$ to $x^\eta$.

  On the other hand, if $\theta$ does not have a reduced word ending
  in $\tau_p$, and $\theta\leq \sigma$, consider some reduced word
  $\word$ for $\theta$ that is a sub-word of
  $\sigma_1\sigma_2\cdots\sigma_k$, and hence a sub-word of
  $\sigma_1\sigma_2\cdots\sigma_{k-1}$. By
  Item~(\ref{item:add-to-end-is-reduced}), $\word\tau_p$ is a
  reduced word for $\theta\tau_p$; since it is a sub-word of
  $\sigma_1\sigma_2\cdots\sigma_k$, $\theta\tau_p\leq\sigma$. Similarly, if
  $\theta$ has a reduced word ending in $\tau_p$, by
  Item~(\ref{item:remove-from-end-is-reduced}), removing $\tau_p$
  from the end produces a reduced word for $\theta\tau_p$, and hence
  $\theta\tau_p<\theta$; so if $\theta\leq\sigma$, $\theta\tau_p\leq\sigma$
  as well. In either case, if $\theta\leq\sigma$,
  $\theta\tau_p\leq\sigma$. Therefore, $\theta$ and $\theta\tau_p$ span an
  acyclic summand of the associated graded complex. Therefore, the
  associated graded complex is acyclic, and this concludes the proof.
\end{proof}

\section{The complex of positive domains with partitions}
\label{sec:CDP}

\subsection{Ordered partitions}
\label{sec:part}
For $N \geq 0$, denote by $\Part(N)$ the set of ordered partitions of $N$ as sums of positive integers. Thus, an element $\lambda \in \Part(N)$ is of the form
\[
  \lambda = (\lambda_1, \dots, \lambda_m), \ m \geq 0,  \ \sum \lambda_j = N.
\]
We denote by $\ell(\lambda)=m$ the {\em length} of the partition. The quantity $N-\ell(\lambda)$ is called the {\em co-length}.

The number of ordered partitions of $N$ is $2^{N-1}$ for $N \geq 1$, and $1$ for $N=0$. Indeed, 
to each $\lambda \in \Part(N)$ we can uniquely associate an $(N-1)$-tuple
\begin{equation}
\label{eq:epsilonlambda}
 \epsilon(\lambda) = (\epsilon_1(\lambda), \dots, \epsilon_{N-1}(\lambda)) \in \{0,1\}^{N-1}
 \end{equation}
as follows: Consider $N$ objects (represented by bullets) in a row, with the first $\lambda_1$ in the first partition class, the next $\lambda_2$ in the second class, etc. We place a $0$ between objects in the same class, and a $1$ between objects in a different class. For example, the partition $2+3+1$ corresponds to the string $01001$:
\[
  ( \bullet \ 0 \  \bullet ) \ 1 \ ( \bullet \ 0 \  \bullet \ 0 \  \bullet)  \ 1 \ ( \bullet)
\]

For $\lambda= (\lambda_1, \dots, \lambda_m) \in \Part(N)$ and $\lambda'= (\lambda'_1, \dots, \lambda'_{m'}) \in \Part(N')$, we define their {\em concatenation}
\begin{equation}
\label{eq:concatenate}
 \lambda * \lambda' = (\lambda_1, \dots, \lambda_m, \lambda'_1, \dots, \lambda'_{m'}) \in \Part (N+N').
 \end{equation}

For $\lambda, \lambda' \in \Part(N)$, we write $\lambda \geq \lambda'$ if $\lambda$ is a {\em refinement} of $\lambda'=(\lambda'_1, \dots, \lambda'_m)$, that is, if there are partitions of each $\lambda'_j$ such that their concatenation gives $\lambda$. We have
\[
  \lambda \geq \lambda' \iff \epsilon(\lambda) \geq \epsilon(\lambda'),
\]
where on the right hand side we used the product partial order on $\{0,1\}^{N-1}$.

\begin{definition}
\label{def:EC}
If $\lambda, \lambda' \in \Part(N)$ are such that $\lambda \geq \lambda'$, we say that $\lambda$ is {\em finer} than $\lambda'$, and $\lambda'$ is {\em coarser} than $\lambda$. If $\lambda \geq \lambda'$ and $\ell(\lambda') = \ell(\lambda)- 1$, we say that $\lambda'$ is an {\em elementary coarsening} of $\lambda$. We denote by $\EC(\lambda)$ the set of elementary coarsenings of $\lambda$.
\end{definition}

If $\lambda'=(\lambda'_1, \dots, \lambda'_m) \in \EC(\lambda)$, then there is an index $k$ and $\lambda_k^1, \lambda_k^2 \geq 1$ such that
\[
  \lambda = (\lambda'_1, \dots, \lambda'_{k-1}, \lambda_k^1, \lambda_k^2, \lambda'_{k+1}, \dots, \lambda'_m), \ \ \lambda_k^1+ \lambda_k^2 = \lambda'_k.
\]
We define the {\em sign} of the elementary coarsening to be
\[
  s(\lambda, \lambda') = (-1)^{m+1-k}=(-1)^{\ell(\lambda)-k}.
\]
Alternatively, note that there is a unique $i \in \{1, \dots, N-1\}$ such that $\epsilon_i(\lambda) = 1$ and $\epsilon_i(\lambda')=0$; and for all $j \neq i$, we have $\epsilon_j(\lambda)=\epsilon_j(\lambda').$ We have
\[
  s(\lambda, \lambda') = (-1)^{1+\epsilon_{j+1}(\lambda) + \dots + \epsilon_{N-1}(\lambda)}.
\]

 \begin{definition}
 \label{def:UE}
 If $\lambda = (\lambda_1, \dots, \lambda_m) \in \Part(N)$, a {\em unit enlargement} of $\lambda$ is a partition $\lambda' \in \Part(N+1)$ of the form
 \[
   \lambda' = (\lambda_1, \dots, \lambda_{k-1}, 1, \lambda_{k}, \dots, \lambda_m)
 \]
 for some $k\in \{1, \dots, m\}$. The sign of the unit enlargement is defined to be
 \[
   s(\lambda, \lambda') = (-1)^{m+1-k}=(-1)^{\ell(\lambda')-k}.
 \]
 The set of unit enlargements of $\lambda$ is denoted $\UE(\lambda)$. 
 \end{definition}
 
 \begin{definition}
 \label{def:IRFR}
  If $\lambda = (\lambda_1, \dots, \lambda_m) \in \Part(N)$, the {\em initial reduction} of $\lambda$ is the partition 
  \[
    \lambda^- \defeq  (\lambda_2, \dots, \lambda_m) \in \Part(N - \lambda_1).
  \]
 The {\em final reduction} of $\lambda$ is the partition
 \[
   \lambda^+ \defeq  (\lambda_1, \dots, \lambda_{m-1}) \in \Part(N - \lambda_m).
 \]
The reductions are not well-defined when $N = 0$ (and $\lambda$ is the empty partition). We define the sets
\[
  \IR(\lambda) \defeq \begin{cases}
    \{ \lambda^-\} & \text{if } N > 0,\\
    \emptyset  & \text{if } N=0;
  \end{cases}
  \ \ \ \ \text{and} \ \ \ \
  \FR(\lambda) \defeq \begin{cases}
    \{ \lambda^+\} & \text{if } N > 0,\\
    \emptyset  & \text{if } N=0.
  \end{cases}
\]
\end{definition}
 
Let us now extend the order on partitions to the case where their sum may be different.
 \begin{definition}
 \label{def:orderpart}
 Suppose $\lambda\in \Part(N)$ and $\lambda' \in \Part(N')$ for $N' \geq N$. We write $\lambda \geq \lambda'$ if there is a partition $\eta \in \Part(N')$ such that $\eta \geq \lambda'$ and $\eta$ is obtained from $\lambda$ by $N'-N$ unit enlargements.
 \end{definition}
 
 For example, we have $\lambda=(2, 1) > \eta=(1, 2, 1) >  \lambda'=(1, 3).$

\begin{remark}
If $\lambda\in \Part(N)$ and $\lambda' \in \Part(N')$ satisfy $\lambda \geq \lambda'$, we must have the inequality
\begin{equation}
\label{eq:colength}
N -\ell(\lambda) \leq N'-\ell(\lambda').
\end{equation}
In other words, the co-length of partitions decreases with respect increasing partitons.
\end{remark}

\subsection{The new complex} \label{sec:new}
We now define a slightly more complicated complex,
$\CDP_*=\CDP_*(\Grid)$, associated to a grid diagram $\Grid$ and a
sign assignment $s$. We will call it the {\em complex of positive
  domains with partitions.} As an Abelian group, $\CDP_*$ is feely
generated by triples $(D, \vN, \vlambda)$ with
\[
D\in\pdomains(x,y), \vN=(N_2,\dots,N_{n})\in\NN^{n-1},
\vlambda=(\lambda_2, \dots, \lambda_{n}),  \lambda_j= (\lambda_{j,1}, \dots, \lambda_{j, m_j}) \in \Part(N_j).
\]
The intuition is that a triple of this type will be associated to a
configuration consisting of the following:
\begin{itemize}
\item A pseudo-holomorphic strip in
  $\Sym^n(T^2)$ with domain $D$. This consists of a map
  \[
    u\from(\RR\times[0,1],\RR\times0,\RR\times 1)\to(\Sym^n(T^2),\alpha_1\times\dots\times\alpha_n,\beta_1\times\dots\times\beta_n),
  \]
  up to translation in the first $\RR$ factor, which is
  $J$-holomorphic with respect to certain (1-parameter family of)
  almost complex structures on $\Sym^n(T^2)$ with
  $\lim_{s\to-\infty}u(s,t)=x,\lim_{s\to+\infty}u(s,t)=y$ and whose
  underlying $2$-chain on $T^2$ is $D$.
\item Several disk bubbles attached to the strip. There are $N_j$
  bubbles going through the marking $O_j$, and each can have as domain
  either the row $H_j$ or the column $V_j$. If the domain of a bubble
  is $H_j$ (respectively, $V_j$), it is attached to the strip at some
  point $(s,0)$ (respectively, $(s,1)$), and $s$ is called the
  \emph{height} of the bubble (which is only well-defined up to an
  overall translation). We do not distinguish between bubbles with
  domain $H_j$ and bubbles with domain $V_j$, and just record the
  total number $N_j$ of such bubbles.
\item Partitioning of the bubbles according to their heights. These
  bubbles are allowed to occur at the same height, and
  $\lambda_j=(\lambda_{j,1},\dots,\lambda_{j,m_j})$ is the
  partitioning of these $N_j$ bubbles so that the bubbles in the same
  partition class are attached to the strip at the same height.
  Further, the ordering of the partition classes corresponds to the
  ordering of the heights---the first $\lambda_{j,1}$ bubbles have the
  smallest height (closest to $x$), and the last $\lambda_{j,m_j}$
  bubbles have the largest height (closest to $y$).  For $j\neq j'$,
  we do not record the relative heights of the $N_j$ bubbles (with
  domain $H_j$ or $V_j$) and the $N_{j'}$ bubbles (with domain
  $H_{j'}$ or $V_{j'}$).
\end{itemize}
See Figure~\ref{fig:bubble-intuition}.

\begin{figure}
  \centering
  \begin{tikzpicture}

    \fill[black!30] (0,0) rectangle ++(6,1);
    \draw[thick] (0,0)--++(6,0) (0,1)--++(6,0);
    \node at (3,0.5) {\small $D$};
    \node[anchor=east] at (0,0.5) {\small $x$};
    \node[anchor=west] at (6,0.5) {\small $y$};

    \foreach\i/\t in {1/2,1.7/3,5.5/2}{
      \draw[thick,fill=black!30] (\i,0) to[out=-135,in=180] ++(0,-1) to[out=0,in=-45] ++(0,1);
      \node at ($(\i,0)+(0,-0.5)$) {\small $H_\t$};
    }
    \foreach\i/\t in {1/2,3.5/2}{
      \draw[thick,fill=black!30] (\i,1) to[out=135,in=180] ++(0,1) to[out=0,in=45] ++(0,-1);
      \node at ($(\i,1)+(0,0.5)$) {\small $V_\t$};
    }
  
    \foreach \i in {1,1.7,3.5,5.5}{\draw[dashed] (\i,0)--++(0,1);}
    
    \begin{scope}[xshift=8cm]

      \fill[black!30] (0,0) rectangle ++(6,1);
      \draw[thick] (0,0)--++(6,0) (0,1)--++(6,0);
      \node at (3,0.5) {\small $D$};
      \node[anchor=east] at (0,0.5) {\small $x$};
      \node[anchor=west] at (6,0.5) {\small $y$};

      \foreach\rot in {45,-45}{
        \begin{scope}[xshift=2cm,rotate=\rot]
          \draw[thick,fill=black!30] (0,0) to[out=-135,in=180] ++(0,-1) to[out=0,in=-45] ++(0,1);
          \node at (0,-0.5) {\small $H_2$};
        \end{scope}
      }
      
      \foreach\i/\t in {1/3,4/2,5/2}{
        \draw[thick,fill=black!30] (\i,1) to[out=135,in=180] ++(0,1) to[out=0,in=45] ++(0,-1);
        \node at ($(\i,1)+(0,0.5)$) {\small $V_\t$};
      }

      \foreach \i in {1,2,4,5}{\draw[dashed] (\i,0)--++(0,1);}
      
    \end{scope}
    
  \end{tikzpicture}
  \caption{Two pseudo-holomorphic strips with bubbles. We do not distinguish between bubbles with domain $H_j$ or $V_j$, nor do we record the relative heights of bubbles that pass through different $O_j$ markings; so both pictures correspond to the triple $(D,(N_2=4,N_3=1),(\lambda_2=(2,1,1),\lambda_3=(1)))$.}\label{fig:bubble-intuition}
\end{figure}

For future reference, set
\begin{align*}
  |\vN| &= N_2 + \dots + N_{n}\\
  |\vlambda| &= \ell(\lambda_2)+ \dots+ \ell(\lambda_{n}).
\end{align*}
and let $e_2,\dots,e_n$ denote the standard unit vectors in $\NN^{n-1}$.

The grading on $\CDP_*$ is given by
\begin{equation}\label{eq:CDP-grading}
\operatorname{gr}(D, \vN, \vlambda)= \mu(D) + |\vlambda|
\end{equation}
The differential $\delta\from \CDP_k \to \CDP_{k-1}$ has four kinds of terms: 
\begin{itemize}
\item {\bf Type I} terms, given by taking out a rectangle from the domain, just as in the complex $\CD_*$;
\item {\bf Type II} terms, given by boundary degenerations, i.e., taking out a row $H_j$ or a column $V_j$ from the domain $D$, and at the same time increasing $N_j$ by one, and changing $\lambda_j$ by a unit enlargement; 
\item {\bf Type III} terms, given by an elementary coarsening of one of the partitions $\lambda_j$. This corresponds to two groups of bubbles at two different heights reaching the same height.
\item {\bf Type IV} terms, given by taking the initial or final reduction of one of the partitions $\lambda_j$. This corresponds to removing a boundary degeneration, in the limit as its height goes to $-\infty$ (for initial reductions) or $+\infty$ (for final reductions). 
\end{itemize}

Precisely, we can write 
\begin{equation}
\label{eq:delta123}
 \delta = \dI + \dII + \dIII + \dIV
\end{equation}
such that, for $D \in \pdomains(x, y)$, we have  
\begin{align}
\label{eq:delta1}
\dI(D, \vN , \vlambda)& =\sum_{\substack{(R,E)\in\rectangles(x,w)\times\pdomains(w,y)\\R*E=D}}\!\!\!\!\!\!s(R)(E, \vN, \vlambda)+(-1)^{\mu(D)}\!\!\!\!\!\!\sum_{\substack{(E,R)\in\pdomains(x,w)\times\rectangles(w,y)\\E*R=D}}\!\!\!\!\!\!s(R)(E, \vN, \vlambda).\\
  \dII(D, \vN , \vlambda)&= (-1)^{\mu(D)+1} \sum_{j=1}^{n} (-1)^{|(\lambda_{j+1},\dots,\lambda_{n})|} \sum_{\substack{E\in\pdomains(x,y)\\E*H_j =D}} \! \sum_{\lambda'_j \in \UE(\lambda_j)} \!\!\!\!\! s(\lambda_j, \lambda'_j)(E, \vN+\ve_j, \vlambda') \label{eq:delta2}\\
  &\quad+ (-1)^{\mu(D)} \sum_{j=1}^{n} (-1)^{|(\lambda_{j+1},\dots,\lambda_{n})|} \sum_{\substack{E\in\pdomains(x,y)\\E*V_j =D}} \! \sum_{\lambda'_j \in \UE(\lambda_j)} \!\!\!\!\! s(\lambda_j, \lambda'_j)(E, \vN+\ve_j, \vlambda').\notag\\
\dIII(D, \vN , \vlambda) &=  (-1)^{\mu(D)+1}
     \sum_{j=1}^{n} (-1)^{|(\lambda_{j+1},\dots,\lambda_{n})|} \sum_{ \lambda'_j \in \EC(\lambda_j)}  s(\lambda_j, \lambda'_j) (D, \vN, \vlambda').\label{eq:delta3}\\
     \dIV(D, \vN , \vlambda) &=  (-1)^{\mu(D)+1}
     \sum_{j=1}^{n} (-1)^{|(\lambda_j,\dots,\lambda_n)|} \sum_{\lambda'_j \in \IR(\lambda_j)} (D, \vN - \lambda_{j, 1}\ve_j, \vlambda') \label{eq:delta4} \\
     &\quad + (-1)^{\mu(D)+1}
     \sum_{j=1}^{n} (-1)^{|(\lambda_{j+1},\dots,\lambda_{n})|}  \sum_{\lambda'_j \in \FR(\lambda_j)} (D, \vN - \lambda_{j, m_j}\ve_j, \vlambda').\notag
  \end{align}
In the expressions \eqref{eq:delta2}, \eqref{eq:delta3} and \eqref{eq:delta4} we used the notation
\[
  \vlambda' = (\lambda_1, \dots, \lambda_{j-1}, \lambda_j', \lambda_{j+1}, \dots, \lambda_{n}).
\]

\begin{lemma}
The complex $\CDP_*$ defined above is indeed a chain complex, i.e., $\delta^2=0$.
\end{lemma}

\begin{proof}
We claim that each of $\dI$, $\dII$ and $\dIII$ squares to zero, and that any two of these differentials anti-commute with each other. We also claim that $\dIV$ anti-commutes with $\dI$ and $\dII$, and that we have
\begin{equation}
\label{eq:d4}
(\dIV)^2 + \dIII \dIV + \dIV \dIII=0. 
\end{equation}
Together, these claims will show that $\delta^2=0$.

Let us start with the differential $\dI$. This gave the complex $\CD_*$, and fact that $(\dI)^2 = 0$ was established in Lemma~\ref{lemma:d2}.

To see that $(\dII)^2=0$, note that in the expression $(\dII)^2(D, \vN, \vlambda)$ we encounter  terms of two kinds. Some are of the form 
\begin{equation}
\label{eq:termsone}
\pm (E, \vN + \ve_i + \ve_j, \vlambda''), \ \ i > j
\end{equation} such that $E$ is obtained from $D$ by deleting a (vertical or horizontal) annulus going through $O_i$ and another annulus through $O_j$. Also, $\vlambda''=(\lambda''_1, \dots, \lambda''_{n})$ is obtained from $\vlambda=(\lambda_1, \dots, \lambda_{n})$ by doing unit enlargements to $\lambda_i$ and $\lambda_j$. The terms of the form \eqref{eq:termsone} come in pairs, corresponding to the order in which we delete the two annuli (and do the respective unit enlargements). The presence of the sign $(-1)^{|(\lambda_{j+1}, \dots,  \lambda_{n})|} $ guarantees that these terms cancel in pairs.

Second, we also have terms of the form
\[
  \pm (E, \vN + 2\ve_j, \vlambda'')
\]
where $E$ is obtained from $D$ by deleting two annuli through the same $O_j$, and $\vlambda''$ is obtained from $\vlambda$ by doing two unit enlargements to the same partition $\lambda_j$. Again, these terms cancel in pairs, due to presence of the signs $s(\lambda_j, \lambda'_j)$ and $s(\lambda'_j, \lambda''_j)$, where $\lambda'_j$ is the intermediate partition. This completes the proof that $(\dII)^2=0$.

The proof that $(\dIII)^2=0$ is similar, with elementary coarsenings
instead of unit enlargements. Once again, the signs
$(-1)^{|(\lambda_{j+1}, \dots,  \lambda_{n})|} $ and
$s(\lambda_j, \lambda'_j)$ ensure that the resulting terms cancel out
in pairs.

The same kind of argument can be used to show that
\[
  (\dII \dIII + \dIII \dII)(D, \vN , \vlambda)  = 0.
\]

Next, let us check that
\[
  (\dI \dII + \dII \dI)(D, \vN , \vlambda) = 0.
\]
Here we obtain terms of the form $\pm (E, \vN + \ve_j, \vlambda')$, where $E$ is obtained from $D$ by deleting a rectangle $R$ and an annulus $H_j$ or $V_j$, and $\vlambda'$ is obtained from $\vlambda$ by doing a unit enlargement to $\vlambda_j$. The terms come in pairs, corresponding to which of the operations $\dI$ and $\dII$ we do first. To see that they cancel out, observe that they get the same sign contributions from the factor $(-1)^{\mu(D)}$ in \eqref{eq:delta1}; the same goes for the factors $(-1)^{|(\lambda_{j+1}, \dots,  \lambda_{n})|} $ and $s(\lambda_j, \lambda'_j)$ in \eqref{eq:delta2}. However, the contributions due to the factor $(-1)^{\mu(D)}$ in \eqref{eq:delta2} differ: for one term we get $(-1)^{\mu(D)}$, and for the other $(-1)^{\mu(E)}$, where $\mu(D) = \mu(E) +1$.

The proofs that
\begin{align*}
  (\dI \dIII + \dIII \dI)(D, \vN , \vlambda) &= 0,\\
  (\dI \dIV + \dIV \dI)(D, \vN , \vlambda) &= 0
\end{align*}
are similar. The cancellations are due to the signs $(-1)^{\mu(D)}$ in \eqref{eq:delta3} and \eqref{eq:delta4}.

Next, we check that 
\[
  (\dII \dIV + \dIV \dII)(D, \vN , \vlambda)  = 0.
\]
On the left hand side we obtain terms of the form $\pm (E, \vN + \ve_j -\lambda_{i, 1} \ve_i, \vlambda'')$ (from a unit enlargement combined with an initial reduction, in either order) and $\pm (E, \vN + \ve_j -\lambda_{i, m_i} \ve_i, \vlambda'')$ (from a unit enlargement combined with a final reduction, in either order). These terms cancel in pairs as follows:
\begin{itemize}
\item When $i \neq j$, the term from a unit enlargement followed by a reduction cancels with the one where the operations are done in the opposite order. The signs of the terms differ due to the presence of the $(-1)^{|(\lambda_{j+1}, \dots,  \lambda_{n})|}$ in \eqref{eq:delta2} and \eqref{eq:delta4};

\item When $i=j$, the term from a unit enlargement in position $k$,
  which is not the first ($k \geq 2$), followed by an initial
  reduction, cancels with the one from the initial reduction followed
  by a unit enlargement in position $k-1$. This is because of the
  extra sign $(-1)^{\ell(\lambda_j)}$ in the initial reduction term
   in \eqref{eq:delta4};

 \item When $i=j$, the term from a unit enlargement in position $k$,
   which is not the last ($k < m_j$), followed by a final reduction,
   cancels with the one from the final reduction followed by the same
   unit enlargement in position $k$. This is because of the sign
   $s(\lambda_j, \lambda_j')$ in \eqref{eq:delta2};

 \item When $i=j$, the term from a unit enlargement in the first
   position, followed by an initial reduction, cancels with the one
   from a unit enlargement in the last position, followed by a final
   reduction. Indeed, the latter term is $ (D, \vN, \vlambda)$ and the
   former term is
   $(-1)^{m_j} (-1)^{m_j + 1} (D, \vN, \vlambda) = -(D, \vN,
   \vlambda).$
\end{itemize}

Finally, we prove Equation~\eqref{eq:d4}. In the expression
\[
  ((\dIV)^2 + \dIII \dIV + \dIV \dIII)(D, \vN, \vlambda)
\]
we encounter terms of the form $(D, \vN, \vlambda'')$, where $\vlambda''$ is obtained from $\vlambda$ either by a combination of an elementary coarsening and a reduction, or by two reductions. Most of the time, these terms cancel each other in pairs corresponding to reversing the order of the two operations. There are, however, two special cases: 
\begin{itemize}
\item The term obtained by doing an elementary coarsening by combining the first two pieces of the partition $\lambda_j$, followed by an initial reduction of that partition, cancels with the term obtained by doing two initial reductions of $\lambda_j$;

\item Similarly, the term obtained by doing an elementary coarsening by combining the last two pieces of the partition $\lambda_j$, followed by a final reduction of that partition, cancels with the term obtained by doing two final reductions of $\lambda_j$.
\end{itemize}
Checking that the signs of the paired terms differ is a straightforward exercise.
\end{proof}

Recall from Proposition~\ref{prop:pos-domain-cx-homology} that the simpler complex $\CD_*$ has  homology generated by the constant domain $\const{\xid}$, for the generator $\xid$. In fact, the span $\langle \const{\xid} \rangle \cong \ZZ$ is a subcomplex of $\CD_*$, and its quotient complex is acyclic. We will now establish a similar result for $\CDP_*$.

\begin{definition}
We denote by $\CDPd \subset \CDP_*$ the subcomplex generated by triples $(c_{\xid}, \vN, \vlambda)$ with $\vN$ made only of $0$'s and $1$'s. We let $\CDP'_*$ be the quotient complex $\CDP_*/\CDPd$.  
\end{definition}

\begin{proposition}
\label{prop:CDPhomology}
$(a)$ The complex $\CDP'_*$ is acyclic, and therefore the inclusion of $\CDPd$ in $\CDP_*$ is a quasi-isomorphism.

$(b)$ For a grid diagram $\Grid$ of size $n$, the homology of $\CDPd$ (and hence also of $\CDP_*$) is isomorphic to $\ZZ^{2^{n-1}}$. Its rank in degree $k$ is ${n  \choose k}$.
\end{proposition}

\begin{proof} $(a)$ As in the proof of Proposition~\ref{prop:pos-domain-cx-homology}, we filter the complex $\CDP'_*$ by the quantity
  \[
    (A(D), B(D)) = (a, b) \in \NN^{2n-2}
  \]
  capturing the multiplicities of the domain $D$ in the rightmost
  column and the topmost row. This is a bounded below increasing
  filtration, and in the associated graded, the differential has no
  more Type II terms. Then, note that the quantity $|\vN|$ is kept
  constant by Type I and III terms, and decreased by Type IV
  terms. Thus, we can filter the associated graded complex by $|\vN|$
  (also a bounded below increasing filtration), and in the new
  associated graded, Type IV terms also disappear. Next, again
  following the proof of
  Proposition~\ref{prop:pos-domain-cx-homology}, we filter with
  respect to the endpoint $y$ of the domain $D$. The resulting
  associated graded complex breaks as a direct sum of finite
  dimensional complexes
  \[
    \CDP_*^{a, b, y, \vN}
  \]
  generated by triples $(D, \vN, \vlambda)$ with
  $D \in \pdomains(x, y)$ such that $(A(D), B(D))=(a, b)$. It suffices
  to show that all of these complexes are acyclic.

When $y \neq x^{\Id}$ or $(a,b)\neq (0,0)$, we filter
$\CDP_*^{a, b, y, \vN}$ with respect to the quantity
$|\ell(\vlambda)|$, and get rid of the Type III terms in the
differential. We are left with only Type I terms. The resulting
associated graded is a direct sum of complexes of the form
$\CD_*^{a, b, y}$, which were shown to be acyclic in the proof of
Proposition~\ref{prop:pos-domain-cx-homology}. We deduce that
$\CDP_*^{a, b, y, \vN}$ is acyclic.

So we are only left with case when $y = x^{\Id}$ and
$(a,b)=(0,0)$. Then the domain $D$ has to be the constant domain
$\const{\xid}$, and our complex $\CDP_*^{0,0, x^{\Id}, \vN}$ has only
Type III terms in the differential. Here,
$\vN=(N_2, \dots, N_{n}) \in \NN^{n}$ and, because of how we defined
$\CDP_*'$, we only consider the case when at least one $N_i$ is
$\geq 2$. We find that $\CDP_*^{0,0, x^{\Id}, \vN}$ is the tensor
product of complexes $\CDP_*(\const{\xid}, N_j)$, for $j=2, \dots, n$,
where $\CDP_*(\const{\xid}, N_j)$ is generated by all the partitions
of $N_j$. By the K\"unneth formula, it suffices to show that
$\CDP_*(\const{\xid}, N_j)$ is acyclic when $N_j \geq 2$.

Let us represent the partitions of $N_j$ by sequences
$(\epsilon_1, \dots, \epsilon_{N_j-1})$ as in
\eqref{eq:epsilonlambda}. We see that $\CDP_*(\const{\xid}, N_j)$ is a
hypercube complex, with the differential decreasing one of the
$\epsilon_k$ by $1$. In fact, we can describe
$\CDP_*(\const{\xid}, N_j)$ as the tensor product of $N_j-1$ complexes
of the form $\ZZ \xrightarrow{\cong} \ZZ$, which are acyclic. Thus,
$\CDP_*(\const{\xid}, N_j)$ is acyclic for $N_j \geq 2$, and the
conclusion follows.

$(b)$ Note that the differential on $\CDPd$ only has terms of Type IV,
corresponding to initial or final reductions. We can identify the
generators of $\CDPd$ with sequences
$\vN = (N_2, \dots, N_n) \in \{0, 1\}^{n-1}$.  The terms in the
differential come in pairs, corresponding to an initial and final
reduction that do the same thing: change a value of $N_j$ from $1$ to
$0$. The paired terms come with opposite signs, because
$\ell(\lambda_j)=1$. We get that $\CDPd$ is the tensor product of $(n-1)$
copies of the complex $ \ZZ \xrightarrow{0} \ZZ.$ The calculation of
the homology of $\CDPd$ now follows from the K\"unneth formula.
\end{proof}

\section{$\langle n \rangle$-manifolds}
\label{sec:nmflds}


\subsection{Definitions and examples}
We recall the definition of an $\langle n\rangle$-manifold, following J\"anich \cite{Janich}; see also \cite{Laures} and  \cite[Section 3.1]{LipshitzSarkar}. We will borrow the terminology from \cite[Definiton 3.2]{LLS}.

 We say that a map from a subset $S \subset \RR^k$ to $\RR^n$ is {\em smooth} if it is the restriction of a smooth map defined on an open set containing $S$. In particular, this allows us to define diffeomorphisms between open subsets of $\RR_+^k$. Then, following Cerf \cite{Cerf} and Douady \cite{Douady}, we define a $k$-dimensional {\em manifold with corners} to be a topological space $X$ along with a maximal atlas, where an atlas is a collection of charts $(U,\phi)$, where $U \subseteq X$ is open and $\phi$ is a homeomorphism from $U$ to an open subset of $\RR_+^k$, such that the sets $U$ cover $X$, and, for any two charts $(U, \phi)$ and $(V, \psi)$, the map 
 \[
   \phi \circ \psi^{-1}\from \psi(U \cap V) \to \phi(U \cap V)
 \]
 is a diffeomorphism. 
For $x \in X$, let $c(x)$ denote the number of coordinates in $\phi(x)$ which are zero, for some (and hence any) chart $(U, \phi)$ with $x \in U$. The codimension-$i$ boundary of $X$ is the subspace 
$\{x \in X \mid c(x) = i\}$; the usual boundary $\del X$ is the closure of the codimension-$1$ boundary. A {\em facet} is the closure of a connected component of the codimension-$1$ boundary of $X$. A {\em multifacet} of $X$ is a (possibly empty) union of disjoint facets of $X$.

A $k$-dimensional {\em multifaceted manifold} is a $k$-dimensional manifold with corners $X$ such that every $x \in X$ belongs to exactly $c(x)$ facets of $X$.  For example, a simple polytope in $\RR^m$ is a multifaceted manifold. By contrast, the following ``teardrop'' manifold with corners
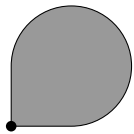
\[
\begin{tikzpicture}[scale=0.7]
  \draw[fill=black!40] (0,0) -- (1,0) arc(-90:180:1cm) -- (0,1) -- cycle;
  \node[inner sep=0,outer sep=0,circle,fill=black,minimum width=4pt] at (0,0) {};
\end{tikzpicture}
\]
is not a multifaceted manifold, because the codimension-$2$ corner belongs to a single facet.

A $k$-dimensional {\em $\langle n\rangle$-manifold $X$} is a $k$-dimensional multifaceted manifold,  together with an ordered $n$-tuple $(\del_1 X, \dots, \del_n X)$ of multifacets of $X$ such that
\begin{itemize}
\item $ \bigcup_i \del_iX=\del X$ and
\item $\del_i X \cap \del_j X$ is a multifacet of both $\del_i X$ and $\del_j X$ for all $i \neq j$.
\end{itemize}

 For a subset $I \subset \{1, \dots, n\}$, we write
\begin{equation}
\label{eq:deliX}
\del_IX  \defeq  \bigcap_{i \in I} \del_iX, \ \ \ \ \mathring{\bdy}_I X \defeq  \del_IX \setminus \bigcup_{I \subsetneq J} \mathring{\bdy}_J X.
\end{equation}
The subsets $\mathring{\bdy}_I X$ are called the {\em strata} of $X$, and $\del_I X$ are the {\em closed strata}.

\begin{example}
The $n$-dimensional hypercube $X=[0,1]^n$ is an $\langle n\rangle$-manifold, with $\del_i X$ being the union of the two facets given by setting the $i\th$ coordinate to either $0$ or $1$.
\end{example}

\begin{example}
\label{ex:permuto}
Consider the $(n-1)$-dimensional permutohedron $\Pi_n$, defined as the convex hull of all points in $\R^n$ whose coordinates are a permutation of $(1,2,3,\dots, n)$. This is an $\langle n-1\rangle$-manifold, with the facets being the convex hulls of points for permutations that preserve a given partition of $\{1,2,3, \dots, n\}$ into two subsets $A$ and $B$. The boundary $\del_i \Pi_n$ consists of those facets for which $|A|=i$ and $|B|=n-i$. We refer to \cite[Example 0.10]{Ziegler}, \cite[Section 2]{Bloom} or \cite[Section 3.3]{LLS} for more details. 
\end{example}


\subsection{Neat embeddings and smoothings}
 \label{sec:neatk}
 Let
 \[
   \E(n,N) = \RR_+^n\times \R^{N}
 \]
 for some $N, n \geq 0$. 
We will describe a class of embeddings of $\langle n \rangle$-manifolds into $\E(n,N)$, called {\em neat}. Neat embeddings of $\langle n \rangle$-manifolds were defined by Laures in \cite{Laures} and used by Lipshitz and Sarkar in \cite{LipshitzSarkar} to construct a Khovanov stable homotopy type. Our definition here will be slightly different, in that we require more than the intersections of strata with the boundaries of $\E(n,N)$ being perpendicular; we ask that that the strata contain small product neighborhoods of a special form near the boundaries.

Let $X$ be a $\langle n \rangle$-manifold. Let $t_1, \dots, t_n$ be the coordinates on $\E(n,N)$ corresponding to the $\R_+$ factors. We view $\E(n,N)$ as a $\langle n \rangle$-manifold, with $\del_I \E(n,N)$ being given by $t_i =0$ for $i \in I$. We also let $\nueps(\del_I \E(n,N))$ be an $\epsilon$-neighborhood of $\del_I \E(n,N)$, given by $t_i \in [0, \epsilon)$ for $i \in I$. Finally, we let
\[
  \pi_I \from \E(n,N) \to \del_I \E(n,N)
\]
be the orthogonal projection.

\begin{definition}
\label{def:neatangle}
A smooth embedding of the $\langle n \rangle$-manifold $X$ into $\E(n,N)$ is called {\em neat} if
\begin{enumerate}
\item It respects the strata, i.e., for every $i$, we have $\del_i X = X \cap \del_i \E(n,N)$.
\item For every $I \subset \{1, \dots, n\}$, there exists $\epsilon > 0$ such that
  \[
    \nueps(\del_I \E(n,N)) \cap X = \nueps(\del_I \E(n,N)) \cap \pi_I^{-1}(\del_IX).
  \]
\end{enumerate}
\end{definition}

\begin{remark}
The condition that the embedding be smooth makes sense in terms of maps of smooth manifolds with corners. However, once we assume conditions (1) and (2), we can rephrase smoothness by simply asking for a topological embedding such that its restriction to every stratum $\mathring{\bdy}_I X$ is a smooth embedding into the corresponding stratum $\mathring{\bdy}_I \E(n,N)$.
\end{remark}

\begin{example}
\label{ex:neathex}
The permutohedron $\Pi_2$ is a hexagon, and Figure~\ref{fig:neathex} shows a neat embedding of that hexagon. The edges are perpendicular to $\RR^{N}$ at vertices, and in fact contain small perpendicular intervals. We then fill in the hexagon so that, near an edge contained in one of the two hyperplanes $0 \times \RR_+\times\RR^N$ or $\RR_+ \times 0\times\RR^N$, it contains the product of that edge and an interval $[0, \epsilon)$ in the direction perpendicular to that hyperplane. 
\end{example}
\begin{figure}
  \centering
  \begin{tikzpicture}[x={(3cm,0)},y={(0,2cm)},z={(-0.4cm,-0.2cm)}]

    \foreach\i in {1,...,6}{
      \coordinate (v\i) at (0,0,\i);
      \coordinate (vx\i) at (0.2,0,\i);
      \coordinate (vy\i) at (0,0.2,\i);
    }


    \path (v1)--(vx1) ..controls (0.5,0,1) and (0.5,0,2).. coordinate[pos=0.4] (mid12) (vx2)--(v2);
    \path (v3)--(vx3) ..controls (0.5,0,3) and (0.5,0,4).. coordinate[pos=0.4] (mid34) (vx4)--(v4);
    \path (v5)--(vx5) ..controls (0.5,0,5) and (0.5,0,6).. coordinate[pos=0.4] (mid56) (vx6)--(v6);

    \path (v1)--(vy1) ..controls (0,1.5,1) and (0,1.5,6)..  coordinate[pos=0.4] (start16) coordinate[pos=0.49] (startmid16) coordinate[pos=0.58] (mid16) coordinate[pos=0.67] (endmid16) coordinate[pos=0.76] (end16)  (vy6)--(v6);
    \path (v3)--(vy3) ..controls (0,0.5,3) and (0,0.5,2).. coordinate[pos=0.5] (mid23) (vy2)--(v2);
    \path (v5)--(vy5) ..controls (0,0.5,5) and (0,0.5,4).. coordinate[pos=0.5] (mid45) (vy4)--(v4);

    \fill[black!20] (v1)--(vy1) ..controls (0,1.5,1) and (0,1.5,6).. (vy6)--(v6)--cycle;
    \fill[black!20] (0,0,1.5)--(mid12) ..controls ($(mid12)+(0,0.3,0)$) and ($(start16)+(0.3,0,0)$).. (start16)--cycle;
    \fill[black!20] (0,0,3.5)--(mid34) ..controls ($(mid34)+(0,0.3,0)$) and ($(mid16)+(0.3,0,0)$).. (mid16)--cycle;
    \fill[black!20] (0,0,5.5)--(mid56) ..controls ($(mid56)+(0,0.3,0)$) and ($(end16)+(0.3,0,0)$).. (end16)--cycle;
    \fill[black!20] (v1)--(vx1) ..controls (0.5,0,1) and (0.5,0,2).. (vx2)--(v2);
    \fill[black!20] (v3)--(vx3) ..controls (0.5,0,3) and (0.5,0,4).. (vx4)--(v4);
    \fill[black!20] (v5)--(vx5) ..controls (0.5,0,5) and (0.5,0,6).. (vx6)--(v6);

    \draw[black!50,thin] (vx2) .. controls ($(vx2)+(0,0.3,0)$) and ($(vx3)+(0,0.3,0)$).. coordinate[pos=0.5] (saddle23) (vx3);

    \draw[black!50,thin] (mid23) ..controls ($(mid23)+(0.1,0,0)$) and ($(saddle23)+(-0.1,0,0)$).. (saddle23) ..controls ($(saddle23)+(0.1,0,0)$) and ($(startmid16)+(0.3,0,0)$).. (startmid16);

    \draw[black!50,thin] (vx4) .. controls ($(vx4)+(0,0.3,0)$) and ($(vx5)+(0,0.3,0)$).. coordinate[pos=0.5] (saddle45) (vx5);

    \draw[black!50,thin] (mid45) ..controls ($(mid45)+(0.1,0,0)$) and ($(saddle45)+(-0.1,0,0)$).. (saddle45) ..controls ($(saddle45)+(0.1,0,0)$) and ($(endmid16)+(0.3,0,0)$).. (endmid16);
    
    \draw[black!50,thin] (mid12) ..controls ($(mid12)+(0,0.3,0)$) and ($(start16)+(0.3,0,0)$).. (start16);
    \draw[black!50,thin] (mid34) ..controls ($(mid34)+(0,0.3,0)$) and ($(mid16)+(0.3,0,0)$).. (mid16);
    \draw[black!50,thin] (mid56) ..controls ($(mid56)+(0,0.3,0)$) and ($(end16)+(0.3,0,0)$).. (end16);

    \draw[thick] (v1)--(vx1) ..controls (0.5,0,1) and (0.5,0,2).. (vx2)--(v2);
    \draw[thick] (v3)--(vx3) ..controls (0.5,0,3) and (0.5,0,4).. (vx4)--(v4);
    \draw[thick] (v5)--(vx5) ..controls (0.5,0,5) and (0.5,0,6).. (vx6)--(v6);

    \draw[thick] (v1)--(vy1) ..controls (0,1.5,1) and (0,1.5,6).. (vy6)--(v6);
    \draw[thick] (v3)--(vy3) ..controls (0,0.5,3) and (0,0.5,2).. (vy2)--(v2);
    \draw[thick] (v5)--(vy5) ..controls (0,0.5,5) and (0,0.5,4).. (vy4)--(v4);

    \draw[->](0,0)--(2/3,0) node[pos=1,anchor=west] {\small $\RR_+$};
    \draw[->](0,0)--(0,1) node[pos=1,anchor=south] {\small $\RR_+$};
    \draw (0,0,-3)--(0,0,7) node[pos=0,anchor=south west] {\small $\RR^N$};

    \foreach\i in {1,...,6}{\fill[black] (v\i) circle (1.5pt);}
    
  \end{tikzpicture}
\caption {A neat embedding of the permutohedron $\Pi_2$. The hexagon is puffed out in the interior of $\R_+^2\times\RR^N$. The boundary of the hexagon is drawn thick, and some curves in the interior of the hexagon are shown by thin lines (to aid visualization).}
\label{fig:neathex}
\end{figure}

\begin{proposition}
\label{prop:Whembedding}
Let $X$ be an $\langle n \rangle$-manifold such that $\del X$ is compact. Then $X$ admits a neat embedding into $\E(n,N)$ for some $N$.
\end{proposition}
\begin{proof}
This was proved by Laures in \cite[Proposition 2.1.7]{Laures}, for $X$ compact. He used his definition of neat embedding, which only required the intersections of strata with the boundaries of $\E(n,N)$ to be perpendicular. However, an inspection of his proof shows that the resulting embedding is neat in our sense. Further, the compactness condition can be weakened to $\del X$ being compact. Indeed, the proof proceeds by constructing collar neighborhoods of the strata (by integrating vector fields), then neatly embedding a neighborhood of $\del X$, and then extending the embedding to the interior. The last step can also be done when $X$ is not compact, in a similar way to the proof that ordinary smooth manifolds can be embedded in Euclidean space. 
\end{proof}

\begin{remark}
With a little more work, one can also drop the compactness assumption on $\del X$ in Proposition~\ref{prop:Whembedding}. However, we will not need this more general statement.
\end{remark}

Note that the boundary of $\E(n,N)$ is given by the equation
$ t_1t_2 \cdots t_n = 0$, which can be smoothed into $ t_1t_2 \cdots t_n = \delta.$ Neat embeddings allow us to smooth the boundaries on $\langle n \rangle$-manifolds in a similar fashion.

\begin{definition}
\label{def:smoothn}
Let $X$ be a compact $\langle n \rangle$-manifold. Pick a neat embedding of $X$ into some $\E(n,N)$, and a value $\delta > 0$ smaller than all $\epsilon$ appearing (for different $I$) in Definition~\ref{def:neatangle}. Then, the subset 
\[
  \smo[X] = X \cap \{(t_1, t_2, \dots, t_n,x) \in \E(n,N) \mid t_1t_2 \cdots t_n \geq \delta\}
\]
is a smooth manifold with boundary, called a {\em smoothing} of $X$. The boundary $\del \smo[X]$ is called the {\em smoothed boundary} of $X$.
\end{definition}

\section{Stratified spaces} \label{sec:stratified}
There are many different definitions of stratified spaces in the literature. We will start with the following simple minded one.
\begin{definition}
A {\em stratified space} is a topological space $X$ together with a locally finite decomposition of $X$ into disjoint subsets, called {\em strata}, such that each stratum is equipped with the structure of a smooth manifold. The decomposition is called a {\em stratification} of $X$.
\end{definition}

It will be helpful to know that the stratified spaces we will encounter in this paper satisfy certain properties; in particular, we will need to be able to smooth the boundary of each stratum, to obtain manifolds with boundary. This can be done for {\em Thom-Mather stratified spaces}. In turn, to show that a space is Thom-Mather stratified, it suffices to show that it is Whitney stratified, so we will start by defining {\em Whitney stratifications.}

\subsection{Whitney stratified spaces}
Whitney stratified spaces were defined in \cite{Whitney}. See \cite{GWPL} for another exposition.

\begin{definition}
Let $X, Y \subseteq \R^n$ be smooth submanifolds, and let $x\in X$. We say that $Y$ is {\em Whitney regular over $X$ at $x$} if, whenever two sequences $(x_i)$ of points in $X$ and $(y_i)$ of points in $Y$, with $x_i \neq y_i$, are such that:
\begin{itemize}
\item both sequences $(x_i)$ and $(y_i)$ converge to $x$,
\item the sequences of tangent spaces $T_{y_i}Y$ converge to a subspace $T \subseteq \R^n$, and
\item the secant lines $\overrightarrow{x_i y_i}$ converge to a line $L \subseteq \R^n$,
\end{itemize}
then $L$ is contained in $T$.
\end{definition}

\begin{definition}
Let $M$ be a smooth $m$-dimensional manifold, $X, Y \subset M$ be smooth submanifolds, and $x \in X$.  We say that $Y$ is {\em Whitney regular over $X$ at $x$} if their images in $\R^n$ are so, under one (and therefore under any) coordinate chart for $M$ at $x$. 

We define the {\em bad set} $B(X, Y)$ to be the set of $x \in X$ such that $Y$ is not Whitney regular over $X$ at $x$. We say that $Y$ is {\em Whitney regular over $X$} if $B(X, Y) =\emptyset$.
\end{definition}

\begin{definition}
Let $M$ be a smooth manifold, and $V \subseteq M$ a subset. A {\em Whitney stratification} of $V$ is a stratification such that all the strata are smooth submanifolds of $M$, and they are regular over each other.
\end{definition}

The commonly given example of a non-Whitney stratification is that of the Whitney umbrella $\{(x, y, z) \in \R^3 \mid x^2z=y^2\}$, where one stratum $X$ is the $z$-axis and the other stratum $Y$ is its complement. Then $Y$ is not Whitney regular over $X$ at the origin. However, if we make the origin into a separate stratum of its own, we get a Whitney stratification. We will come back to the Whitney umbrella in Section~\ref{sec:ZN}; see Figure~\ref{fig:Whitney}. 

More generally, Thom \cite{Thom} showed that all semialgebraic sets admits Whitney stratifications. Recall that 
the class of {\em semialgebraic sets} of $\R^n$ is the smallest Boolean algebra of subsets of $\R^n$ which contains all sets of the form
\[
  \{ x\in \R^n \mid f(x) > 0\}
\]
with $f\from \R^n \to \R$ a polynomial function. For our purposes, we will need the following three results:

\begin{proposition}[Proposition (2.1) in \cite{GWPL}]
\label{prop:imagesa}
Let $f\from \R^n \to \R^m$ be a polynomial map, and $V \subseteq \R^n$ a semialgebraic set. Then the image $f(V)$ is semialgebraic.
\end{proposition}

\begin{proof}
The Tarski-Seidenberg theorem says that the conclusion is true when $f$ is a linear projection. For the general statement, consider the graph of $f$ in $\R^n \times \R^m$. This is semialgebraic, and its linear projection to $\R^m$ is $f(V)$.
\end{proof}

\begin{proposition}[Whitney's theorem \cite{Whitney, Thom2}; Proposition (2.6) in \cite{GWPL}]
\label{prop:badset}
Let $X$, $Y$ be semialgebraic smooth submanifolds of $\R^n$. Then the bad set $B(X, Y)$ is semialgebraic, of dimension strictly smaller than the dimension of $X$.
\end{proposition}

\begin{proposition}[Proposition (1.2) in \cite{GWPL}]
\label{prop:products}
Let $V_1, \dots, V_m$ be Whitney stratified spaces. Then the product stratification on $V_1 \times \dots \times V_m$ (consisting of Cartesian products of the strata in each $V_i$) is a Whitney stratification.
\end{proposition}

\subsection{Thom-Mather stratified spaces} \label{sec:TM}
The following definition is based on \cite[Section 8]{Mather}. 

\begin{definition}
\label{def:stratified}
A {\em Thom-Mather stratified space} is a triple $(V, \Set, \Tube)$ satisfying the following axioms:

\begin{enumerate}[leftmargin=*,label=(A-\arabic*),ref=(A-\arabic*)]
 \item\label{item:stratified-1} $V$ is a Hausdorff, locally compact, second countable topological space;

 \item  $\Set$ is a family of locally closed subsets of $V$, such that $V$ is the disjoint union of the members of $\Set$. The members of $\Set$ are called the {\em strata} of $V$, and their closures are called the {\em closed strata}.

 \item  Each stratum of $V$ (with the induced topology from $V$) is a topological manifold, and additionally equipped with a $C^{\infty}$ structure; 

 \item  The family $\Set$ is locally finite.

 \item  If $X,Y \in \Set$ and $X \cap \overline{Y} \neq \emptyset,$ then $X \subseteq \overline{Y}$. If this is the case, we write $X \leq Y$. If $X \leq Y$ and $X \neq Y$, we write $X < Y$.

 \item  $\Tube$ is a collection of triples $(T_X, \pi_X, \rho_X)$, one for each $X \in \Set$, where $T_X$ is an open neighborhood of $X$ in $V$ (called a {\em tubular neighborhood}), $\pi_X\from T_X \to X$ is a continuous retraction, and $\rho_X \from T_X \to [0, \infty)$ a continuous function.

 \item  $X=\{v\in T_X \mid \rho_X(v)=0\}.$
  
 \item \label{item:A8} For $X, Y \in \Set$, denote 
   \[
     T_{X,Y} = T_X \cap Y, \ \ \pi_{X,Y} = \pi_X|_{T_{X,Y}}, \ \ \rho_{X,Y} = \rho_X |_{T_{X,Y}}.
   \]
   Then, we require that for any distinct strata $X$ and $Y$ the mapping 
   \[
     (\pi_{X,Y}, \rho_{X,Y})\from T_{X,Y} \to X\times (0,\infty)
   \]
is a smooth submersion.

 \item\label{item:stratified-9}  For any strata $X, Y,$ and $Z$, we have
   \[
     \pi_{X} \pi_{Y} (v) = \pi_{X} (v), \ \ \rho_{X} \pi_{Y} (v) = \rho_{X} (v)
   \]
whenever both sides of the respective equation are defined.

 \item\label{item:stratified-10}  If $X,Y \in \Set$ satisfy $T_{X,Y} \neq \emptyset$, then $X \leq Y$.

 \item\label{item:stratified-11}  If $X, Y \in\Set$ are such that $T_X \cap T_Y \neq \emptyset$, then $X$ and $Y$ are comparable, i.e., we have $X\leq Y$ or $Y \leq X$.
\end{enumerate}
\end{definition}

\begin{remark}
  The terminology used in \cite{Mather} is {\em abstract stratified
    set}. This is required to only satisfy the conditions
  \ref{item:stratified-1}--\ref{item:stratified-9}, but it is noted
  there that every such set is equivalent to one that also satisfies
  \ref{item:stratified-10} and \ref{item:stratified-11}.
\end{remark}

\begin{remark}
The function $\rho_X$ is called the {\em tubular function} of $X$. Roughly, it is meant to play the role of the distance to $X$.
\end{remark}

\begin{remark}
  As noted in \cite{Mather}, the assumptions
  \ref{item:stratified-1}--\ref{item:stratified-10} above also have
  the following implications:
\begin{itemize}
\item The relation $\leq$ is a partial order on $\Set$;
\item $X \leq Y$ if and only if $T_{X,Y} \neq \emptyset$;
\item $X$ and $Y$ are comparable if and only if $T_X \cap T_Y \neq \emptyset$.
\end{itemize}
\end{remark}

\begin{remark} 
\label{rem:cone}
Another implication is that if $x$ is a point in a $k$-dimensional stratum $X$, then there is a neighborhood $V_x$ of $x$ in $V$ homeomorphic to $\R^k \times C(L)$, where $L$ is a stratified space and $C(L)=\bigl( L \times [0,1) \bigr) / \bigl( L \times \{0\} \bigr)$ is the open cone on $L$. Specifically, we can take $V_x$ to be the preimage of a chart in $X$ under the map $\pi_X$. If we identify
\[
  L \cong \{0\} \times L \times \{1/2\} \subset \R^k \times C(L) \cong V_x,
\]
then the stratification of $L$ is given by intersections with the strata $Y$ such that $X < Y$.

The space $L$ is called the {\em link} of $X$ at $x$. We will refer to $\R^k \times C(L)$ as the {\em local model} of $V$ around $x$, and to $C(L)$ as the {\em local model in the normal directions}.
\end{remark}

\begin{remark}
Given a stratum $X$ in a Thom-Mather stratified space, its closure $\bX$ and the boundary $\del X = \bX  \setminus X$ have induced Thom-Mather stratifications. The boundary $\del X$ is the union of all strata $Y$ such that $Y < X$.
\end{remark}

We now turn to examples of Thom-Mather stratified spaces. First, an $\langle n \rangle$-manifold $X$ can be made into a Thom-Mather stratified space as follows. Let $\del_iX$, for $i=1, \dots, n,$ be the distinguished collection of multifacets, and $\del_I X$ and $\mathring{\bdy}_I X$ be as in Equation~\eqref{eq:deliX}. We let $\mathring{\bdy}_I X $ be the strata in $X$ and $\del_I X$ are their closures. Observe that
\[
  {\bdy}_J X \subseteq {\bdy}_I X  \iff I \subseteq J.
\]
The tubular neighborhood of a stratum $\mathring{\bdy}_I X $ in an $\langle n \rangle$-manifold $X$ can be constructed by integrating a smooth vector field that is transverse to $\mathring{\bdy}_I X $, and vanishes at the boundary of $\mathring{\bdy}_I X $. 

\begin{example}
Consider the quadrant $X = \RR_+^2$, and view it as a $\langle 2 \rangle$-manifold with $\del_1X = 0\times \RR_+$ and $\del_2X = \RR_+\times 0.$ There are four strata:
\[
  \mathring{\bdy}_\emptyset X = (0, \infty)\times (0,\infty), \ \ \mathring{\bdy}_{\{1\}}X = 0\times (0,\infty), \ \ \mathring{\bdy}_{\{2\}}X= (0,\infty)\times 0, \ \ \mathring{\bdy}_{\{1,2\}} X = 0\times 0.
\]
Their tubular neighborhoods are shown in Figure~\ref{fig:quadrant}: that of $\mathring{\bdy}_{\{1,2\}} X$ is the lightly shaded quarter-disk, those of $\mathring{\bdy}_{\{1\}}X$ and $\mathring{\bdy}_{\{2\}}X$ are darkly shaded, and the tubular neighborhood of $\mathring{\bdy}_\emptyset X$ is $\mathring{\bdy}_\emptyset X$ itself.
\end{example}

\begin{figure}
  \centering
  \begin{tikzpicture}[scale=1.3]

    \fill[black!30] (0,0)--(1,0)arc(0:90:1)--cycle;
    \fill[black!50] (0,0)--(3,0)--(3,0.5)--cycle;
    \fill[black!50] (0,0)--(0,3)--(0.5,3)--cycle;

    \draw[ultra thick] (3,0)--(0,0)--(0,3);
    \draw (3,0.5)--(0,0)--(0.5,3) (1,0)arc(0:90:1);
    \fill[black] (0,0) circle (2pt);

    \node[anchor=south east,inner sep=0,outer sep=1pt] at (3,0.5) {\small $T_{\mathring{\bdy}_{\{2\}}X}$};
    \node[anchor=north west,inner sep=0,outer sep=1pt] at (0.5,3) {\small $T_{\mathring{\bdy}_{\{1\}}X}$};
    \node[anchor=south west,inner sep=0,outer sep=0] at (45:1) {\small $T_{\mathring{\bdy}_{\{1,2\}}X}$};
    
  \end{tikzpicture}
\caption {The Thom-Mather stratified space $X = [0, \infty)^2$, with tubular neighborhoods.}
\label{fig:quadrant}
\end{figure}

\begin{remark}
In \cite[Lemma 2.1.6]{Laures} it is proved that the {\em closed} strata $\del_IX$ in an $\langle n \rangle$-manifold admit a system of collar neighborhoods, of the form $\R^{|I|} \times \del_I X$. These are different from the tubular neighborhoods that we consider in their paper, but serve similar purposes. The collar neighborhoods from \cite[Lemma 2.1.6]{Laures} were used in the construction of the Khovanov stable homotopy type in \cite{LipshitzSarkar}. We do not use them here because they do not admit a straightforward generalization to other stratified spaces.
\end{remark}

A larger class of Thom-Mather stratified spaces is provided by the following result.
\begin{theorem}[Thom \cite{Thom}, Mather \cite{Mather}] \label{thm:TM}
The strata in Whitney stratified spaces admit tubular neighborhoods as in Definition~\ref{def:stratified}, and therefore Whitney stratified spaces can be turned into Thom-Mather stratified spaces.
\end{theorem}

\subsection{Smoothings}
We now explain how one can smooth the boundary of a stratum in a Thom-Mather stratified space.
The lemma below is key: it allows us to find neighborhoods of the boundary that are submanifolds with corners. 

\begin{lemma}
\label{lemma:Nbhd}
Let $X$ be an $n$-dimensional stratum in a Thom-Mather stratified space $(V, \Set, \Tube)$. For any stratum $Y \subseteq \del X$, choose $\epsilon_Y > 0$ sufficiently small, inductively on the dimension of $Y$, so that $\epsilon_{Y} \ll \epsilon_Z$ when $Z < Y$. Consider the following closed neighborhood of $\del X$ in $\bX$:
\[
  \Nbhd = \bigcup_{Y < X} \bigl( \rho_{Y}^{-1}([0, \epsilon_Y]) \cap \bX).
\]
Then, the complement
\[
  M\defeq \bX \setminus \inte(\Nbhd)
\]
is an $n$-dimensional $\langle n \rangle$-manifold. 
\end{lemma}

\begin{proof} 
It is convenient to set $\rho_Y(x) = \infty$ where $\rho_Y$ is undefined, i.e., for $x \not\in T_Y$. Then we can write
\[
  M = \bX \cap \bigcap_{Y < X} \rho_Y^{-1}([ \epsilon_Y, \infty]).
\]
For $i=1, \dots, n$, set
\[
  \del_i M = \{x \in M \mid \rho_Y(x) =\epsilon_Y \ \text{for some } Y < X \text{ with }  \dim Y = i-1 \}.
\]
To see that this turns $M$ into an $\langle n \rangle$-manifold, pick any point $x \in \del M = \bigcup_i \del_iM$, and consider the strata $Y < X$ that satisfy $\rho_Y(x) =\epsilon_Y$. Using \ref{item:stratified-11}, we see that $\leq$ is a total order on these strata, so the strata have different dimensions and we can label them by $Y_1 < Y_2 < \cdots < Y_k.$ Thus, $x$ lies at the intersection of the boundaries $\del_i M$ where $i=\dim Y_j$ for some $j$. Furthermore, near $x$, the subset $M \subseteq X$ is given by the inequalities
\[
  \rho_{Y_i} \geq \epsilon_{Y_i}, \ i=1, \dots, k.
\]
We claim that $x$ is a codimension $k$ corner, that is, the local model for $M$ near $x$ is $0 \in \R^{n-k} \times \R_+^k.$ For this, it suffices to show that the map 
\[
  (\rho_{Y_1, X}, \rho_{Y_2, X}, \dots, \rho_{Y_k, X})\from U(x) \to \R^k
\]
is a submersion at $x$. (Here, $U(x)$ is a neighborhood of $x$ in $X$.)

We prove by induction on $j \leq k$ that 
\[
  (\rho_{Y_1, X}, \rho_{Y_2, X}, \dots, \rho_{Y_j, X})\from U(x) \to \R^j
\]
is a submersion at $x$. The base case $j=1$ follows from ~\ref{item:A8}. For the inductive step, using \ref{item:stratified-9} , we write
\[
  (\rho_{Y_1, X}, \rho_{Y_2, X}, \dots, \rho_{Y_j, X})= (\rho_{Y_1, Y_j}, \rho_{Y_2, Y_j}, \dots, \rho_{Y_{j-1}, Y_j}, \operatorname{id}) \circ (\pi_{Y_j, X}, \rho_{Y_j, X}).
\]
Note that $(\pi_{Y_j, X}, \rho_{Y_j, X})$ is a submersion by~\ref{item:A8}. Using the inductive hypothesis for $j-1$ and the fact that the composition of submersions is a submersion, the claim follows.
\end{proof}

\begin{definition}
\label{def:smoothstrata}
Let $X$ be an $n$-dimensional stratum of a Thom-Mather stratified space, such that $\del X$ is compact. We define the {\em smoothing of $X$} to be the $n$-dimensional smooth manifold with boundary
\[
  \smo[X] \defeq \smo[M]
\]
where $M \subseteq X$ is the $\langle n \rangle$-manifold from Lemma~\ref{lemma:Nbhd}, and $\smo[M] \subseteq M$ was defined in Definition~\ref{def:smoothn}. 
\end{definition}

See Figure~\ref{fig:smooth} for an example.

\begin{figure}
  \centering
  \begin{tikzpicture}[scale=1.3]

    \fill[black!50] (0,0)--(1,0)arc(0:90:1)--cycle;
    \fill[black!50] (0,0)--(3,0)--(3,0.5)--cycle;
    \fill[black!50] (0,0)--(0,3)--(0.5,3)--cycle;

    \draw[ultra thick] (3,0)--(0,0)--(0,3);
    \fill[black] (0,0) circle (2pt);

    \node[anchor=east,inner sep=0,outer sep=1pt] at (0,1) {\small $\bdy X$};
    \node[anchor=north] at (0.25,3) {\small $\mathcal{N}$};

    \fill[black!30] (3,3)--(3,0.6)--(1.5,0.35) to[out=-170,in=-60] (30:1.1) arc(30:60:1.1) to[out=150,in=-110] (0.35,1.5)--(0.6,3)--(3,3);

    \node[anchor=center,inner sep=0,outer sep=1pt] at (2,2) {\small $\smo[X]$};
    
  \end{tikzpicture}
\caption {A smoothing of the stratified space from Figure~\ref{fig:quadrant}.}
\label{fig:smooth}
\end{figure}

\section{Local models}
\label{sec:local}
In this paper we will work with a certain kind of stratified spaces, where we have boundaries and corners as in $\langle n\rangle$-manifolds, but we also allow a different type of boundary, modeled on ``generalized Whitney umbrellas'' and called the {\em special boundary}. When we glue several of these spaces along their special boundaries, we will obtain an $\langle n\rangle$-manifold.

 The spaces will be constructed in Section~\ref{sec:construction}. For now, we will limit ourselves to describing the local models that appear in their stratifications. 

\subsection{The spaces $I_N$} \label{sec:IN}
Let us first consider the space
\[
  I_N = \Sym^N(\RR)/\RR,
\]
where $N \geq 0$ and $\Sym^N$ denotes the $N\th$ symmetric product, and $\RR$ acts by simultaneous translation on all factors. Recall from Section~\ref{sec:part} that $\Part(N)$ denotes the set of ordered partitions of $N$. For 
\[
  \lambda = (\lambda_1, \dots, \lambda_m) \in \Part(N),
\]
let $I(\lambda) \subseteq I_N$ be the subset consisting of $N$-tuples of real numbers (modulo $\RR$) such that the first $\lambda_1$ of these numbers coincide (take the same value $x_1$), the next $\lambda_2$ coincide (taking a value $x_2$), and so on, with $x_i < x_j $ for $i < j.$ The decomposition
\begin{equation}
\label{eq:IN}
 I_N = \bigsqcup_{\lambda \in \Part(N)}  I(\lambda)
 \end{equation}
gives $I_N$ the structure of a stratified space, with strata $I(\lambda)$. 

If $N>0$, the dimension of $I(\lambda)$ is $\ell(\lambda)-1$, where $\ell(\lambda)=m$ is the length of the partition. The local model of $I_N$ around a point in $I(\lambda)$ is 
\[
  \R^{m-1} \times I_{\lambda_1} \times \dots \times I_{\lambda_m}.
\]

The refinement order on partitions introduced in Section~\ref{sec:part} is relevant to the decomposition of $I_N$, because it tells us which strata are in the closures of other strata:
\begin{equation}
\label{eq:INcompare}
 I(\lambda) \leq I(\mu) \iff \lambda \leq \mu.
 \end{equation}

\begin{example}
When $N=0$ or $1$, there is a unique partition of $N$, and in both cases $I_N$ is a point. When $N=2$, we have $I_2 \cong [0, \infty)$, with the strata in the decomposition being $I_2(2)=\{0\}$ and $I_2(1,1) = (0, \infty)$. When $N=3$, one can check that $I_3\cong [0, \infty)^2$, with $I_3(3)$ being the origin, $I_3(1,2)$ and $I_3(2,1)$ being the two half-lines on the boundary, and $I_3(1,1,1) \cong (0, \infty)^2$. When $N \geq 4$, the topology of $I_N$ is more complicated; see \cite{BorsukUlam}.
\end{example}

\subsection{The spaces $Z_N$} \label{sec:ZN}
Next, consider the space
\[
  Z_N = \Sym^N(\CC)/\RR,
\]
where $\RR$ acts on the $\CC$ factors by translating the real parts. If we let the coordinates on each copy of $\CC$ be
\[
  z_j = x_j + i y_j, \ j=1, \dots, N,
\]
note that $\Sym^N(\CC)$ can be identified with $\CC^N$ using the elementary symmetric polynomials in $z_1, \dots, z_N$:
 \begin{equation}
\label{eq:coordtZN}
 s_1 = \sum_j z_j, \ \  s_2 = \sum_{j < l} z_{j} z_{l}, \ \dots, \ s_N = \prod_j z_j.
\end{equation}
 Further, dividing by $\RR$ is equivalent to setting $\Re(s_1)= x_1 + \dots + x_N=0.$ Therefore, if $N>0$, we can identify $Z_N$ with $\RR^{2N-1}$, with real coordinates being 
 \begin{equation}
\label{eq:coordZN}
  \Im(s_1), \Re(s_2), \Im(s_2), \dots, \Re(s_N), \Im(s_N).
  \end{equation}
We put a stratification on $Z_N \cong \RR^{2N-1}$, with the strata being given by  the signs of the imaginary parts $y_j = \Im(z_j)$, as well as by which real coordinates coincide for the indices $j$ with $y_j=0$. Precisely, consider the decomposition
\[
  Z_N = \bigsqcup_{\substack{p^-+p^0+p^+=N \\ \lambda \in \Part(p^0) }} Z(p^-, p^0, p^+; \lambda),
\]
with $Z(p^-, p^0, p^+; \lambda)$ consisting of the multisets $\{z_1, \dots, z_N\}$ where $p^-$ of the $y_j$'s are less than zero, $p^0$ are zero, $p^+$ are greater than zero, and the $p^0$ coordinates $x_j$ (those for which $y_j=0$) are split according to the partition $\lambda$, as in the decomposition~\eqref{eq:IN} of the space $I_{p^0}$. Observe that
\begin{equation}
\label{eq:ZN}
 Z(p^-, p^0, p^+; \lambda) \leq Z(q^-, q^0, q^+; \mu)  \iff (p^- \leq q^-, p^+ \leq q^+ \text{ and } \lambda \leq \mu),
 \end{equation}
where we used the order on partitions introduced in Definition~\ref{def:orderpart}. We will denote the closure of $Z(p^-, p^0, p^+; \lambda)$ by $\bZ(p^-, p^0, p^+; \lambda)$.
 
 We have
 \begin{equation}\label{eq:dimZN}
   \dim Z(p^-, p^0, p^+; \lambda) = 2p^-+2p^+ +\ell(\lambda) -1.
 \end{equation}
 For example, there is a unique zero dimensional stratum, namely $Z(0, N, 0; N).$
 
 Observe that the codimension of a stratum $Z(p^-, p^0, p^+; \lambda) \subset Z_N$ is $2p^0-\ell(\lambda)$, which is at least $p^0$. In particular, there are $N+1$ codimension zero strata, corresponding to $p^0=0$. 

 We let $Z(p^-, p^0, p^+)$ be the union of $Z(p^-, p^0, p^+; \lambda)$
 over all $\lambda \in \Part(p^0)$. In particular, $Z(0, N,0)$ is our
 old space $I_N$. Also, note that when $p^0=0$ or $1$, there is a
 unique partition $(p^0)$, so
 $Z(p^-, p^0, p^+)= Z(p^-, p^0, p^+; (p^0))$.

\begin{example}
\label{ex:Z1}
When $N=0$, $Z_0$ is a point. When $N=1$, we have
\[
  Z_1 = \Sym^1(\CC)/ \RR \cong \RR,
\]
with the three strata $(-\infty, 0), \{0\}$, and $(0,\infty).$ \end{example}
  
 \begin{example}
 \label{ex:Z2}
 When $N=2$, we look at  $Z_2=\Sym^2(\CC)/\RR$, with the coordinates on the two copies of $\CC$ being
 \[
   z_1 = x_1  + i y_1, \ \ \ z_2 = x_2 + i y_2.
 \]
We identify $\Sym^2(\CC)$ with $\CC^2$ using the symmetric polynomials $z_1 + z_2$ and $z_1 z_2$. After dividing by $\RR$-translation (that is, setting $x_1+x_2=0$), we are left with three real coordinates on $Z_2$:
\begin{align*}
 a &= \Im(z_1+z_2) = y_1 + y_2, \\
 b &= \Re(z_1z_2) = -(x_1^2 + y_1y_2),\\
 c &= \Im(z_1z_2) = x_1(y_2-y_1).
 \end{align*}
 
 Let $W \subset Z_2$ be the hypersurface given by the condition that at least one of $z_1$ and $z_2$ be real, i.e., $y_1y_2=0$. From here we get $b=-x_1^2 \leq 0$ and $c=\pm x_1a$, so $a^2b+c^2=0$:
 \[
   W = \{(a,b, c) \in \R^3 \mid b \leq 0, a^2b+c^2=0\}.
 \]
 This is the Whitney umbrella shown in Figure~\ref{fig:Whitney}. The complement of $W$ in $Z_2$ splits into three connected components:
 \begin{align*}
  Z(2,0,0) &= \{(a,b, c) \in \R^3 \mid a^2b+c^2<0\text{ and }a < 0\},\\
   Z(1,0,1) &= \{(a,b, c) \in \R^3 \mid a^2b+c^2>0\text{ or }b>0\}, \\
    Z(0,0,2) &= \{(a,b, c) \in \R^3 \mid a^2b+c^2<0\text{ and }a > 0\}, 
    \end{align*}
 corresponding to none, one, or two of the $y_j$ coordinates being positive, and the rest negative.

\begin{figure}
  \centering
  \begin{tikzpicture}[scale=1.2]

    \begin{scope}[z={(0.8cm,-0.2cm)},y={(0,-1cm)},x={(0.8cm,0.4cm)}]

    \draw[thick,fill=black!20] (2,0) ..controls(2,0,-0.25) and (2,1,-0.8).. (2,2,-1) -- (0,2) -- (-2,2,-1) ..controls(-2,1,-0.8) and (-2,0,-0.25).. coordinate[midway] (z110) (-2,0) -- cycle;
    \draw[thick,fill=black!50] (2,0) ..controls(2,0,0.25) and (2,1,0.8).. coordinate[midway] (z011) (2,2,1) -- (0,2) -- (-2,2,1) ..controls(-2,1,0.8) and (-2,0,0.25).. (-2,0) -- cycle;
    \draw[dashed] (2,0) ..controls(2,0,-0.25) and (2,1,-0.8).. (2,2,-1) -- (0,2) -- (-2,2,-1);

    \node at (2,2) {$Z(0,0,2)$};
    \node at (-2,2) {$Z(2,0,0)$};
    \node at (0.5,-1.5,0.5) {$Z(1,0,1)$};
    
    \draw[->] (-3.5,0)--(3.5,0) node[pos=1,anchor=south west] {$a$};
    \draw[->] (0,0)--(0,-2) node[pos=1,anchor=south] {$b$};
    \draw[->] (0,0,-3.5)--(0,0,3.5) node[pos=1,anchor=west] {$c$};

    \draw[ultra thick] (0,0)--(0,2.5) coordinate[midway] (z02011);

    \draw[->] (-2,0,-2) ..controls (-2,0,-1) and ($(z110)+(0,0,-1)$).. node[pos=0,anchor=east] {$Z(1,1,0)$} ($(z110)+(0,0,-0.05)$);
    \draw[->] (2,0,2) ..controls (2,0,1) and ($(z011)+(0.5,0,0)$) .. node[pos=0,anchor=west] {$Z(0,1,1)$} ($(z011)+(0.05,0,0)$);
    
    \draw[->] (0,2,2) ..controls(0,2,1) and ($(z02011)+(0,0,1)$).. node[pos=0,anchor=west]{$Z(0,2,0;(1,1))$} ($(z02011)+(0,0,0.05)$);
    \draw[->] (0,-1,-1) ..controls(0,-1,-0.5) .. node[pos=0,anchor=east] {$Z(0,2,0;(2))$} (0,-0.07,-0.07);
    \end{scope}

    \fill[black] (0,0) circle (2pt);
    
  \end{tikzpicture}
\caption{The Whitney umbrella inside $Z_2\cong\RR^3$.}
\label{fig:Whitney}
\end{figure}

The codimension-$1$ strata are the two halves of $W$: 
\begin{align*}
  Z(1,1,0) &= \{(a,b, c) \in \R^3 \mid a^2b+c^2=0, a < 0\},\\
   Z(0,1,1) &= \{(a,b, c) \in \R^3 \mid a^2b+c^2=0, a > 0\}.
   \end{align*}
This leaves the strata corresponding to $y_1=y_2=0$, which give the half-line
\[
    Z(0,2,0) = \{(0,b,0) \in \R^3 \mid b \leq 0\}.
\]
This is further decomposed into the codimension-$2$ stratum
$Z(0,2,0; (1,1))$, which is the open half-line, and the stratum
$Z(0,2,0; (2)),$ which is just the point $\{(0,0,0)\}.$
\end{example} 

\subsection{Models for internal framings} \label{sec:modelinternal} We
will now construct explicit framings for the normal bundles to the
strata in $Z_N$; assume $N>0$. These will be models for the internal framings in
Section~\ref{sec:neat} later.

\begin{convention}
\label{conv:framings}
A framing of a vector bundle is defined to be a smoothly varying,
ordered basis of the fibers. (We do not ask it to be orthonormal.) The
normal bundle to a submanifold $X \subset V$ is defined to be the
quotient $TV/TX$, which we can identify with any complement of $TX$ in
$TV$. Throughout this paper, when talking about a framing of the
normal bundle, we will always mean that such a complement has been
chosen, and we consider a framing of it; thus, the frame consists of
vectors in $TV$. We do not ask for the complement to be the orthogonal
complement. Of course, from a framing as above one can get one of the
orthogonal complement, and/or an orthonormal framing, by using the
Gram-Schmidt process. However, it is convenient to have the extra
flexibility.
\end{convention}

Let $\tZ_N = \Sym^N(\C)$, which we can identify with $\R \times Z_N$
by letting the first coordinate be $\Re(s_1)$ in the notation of
Equation~\eqref{eq:coordtZN}. The space $\tZ_N$ has a stratification with
strata
\begin{equation}\label{eq:tZ-strata}
  \tZ(p^-, p^0, p^+; \lambda)\cong\R \times Z(p^-, p^0, p^+; \lambda).
\end{equation}
The normal bundle of $Z(p^-, p^0, p^+; \lambda)$ in $Z_N$ is
identified with the normal bundle of $\tZ(p^-, p^0, p^+; \lambda)$ in
$\tZ_N$, so it suffices to frame the latter.

Fix a small $\epsilon>0$.  Consider any stratum
$\tZ(p^-,p^0,p^+;\lambda)\subset \tZ_N$ and a point $z$ in the
stratum. Assume $\lambda=(\lambda_1,\dots,\lambda_m)$. The point $z$
consists of an unordered tuple
\[
  \{z_1=x_1+iy_1,\dots,z_N=x_N+iy_N\}
\]
such that $p^\pm$ of them have imaginary part positive/negative, and
the remaining $p^0$ numbers---say $z_1,\dots,z_{p^0}$---have imaginary
part $0$ with their real parts satisfying
\[
  x_1=\dots= x_{\lambda_1}<x_{\lambda_1+1}=\dots=x_{\lambda_1+\lambda_2}<\dots<x_{p^0-\lambda_m+1}=\dots=x_{p^0}.
\]
For convenience, let us relabel $z_1,\dots,z_{p^0}$ as $z_{1,1},\dots,z_{1,\lambda_1},z_{2,1},\dots,z_{2,\lambda_2},\dots,z_{m,1},\dots,z_{m,\lambda_m}$, so the above inequality becomes
\[
  x_{1,1}=\dots= x_{1,\lambda_1}<x_{2,1}=\dots=x_{2,\lambda_2}<\dots<x_{m,1}=\dots=x_{m,\lambda_m}.
\]

We will pick a smoothly varying $z'\in \tZ(p^-,p^0,p^+;(1,1,\dots,1))$
that is $\epsilon$-close to $z$. For specificity, pick a smoothly
varying $\epsilon_z$ which is less than
$\min\{\epsilon,x_{2,1}-x_{1,1},\dots,x_{m,1}-x_{m-1,1}\}$, and define
$z'$ to be the unordered tuple
$\{z'_{1,1}=x'_{1,1},\dots,z'_{m,\lambda_m}=x'_{m,\lambda_m},z_{p^0+1},\dots,z_{N}\}$ with
\[
  x'_{j,l}=x_{j,1}+\frac{(2l-\lambda_j-1)\epsilon_z}{2\lambda_j}.
\]
(In words, $z'_{1,1},\dots,z'_{m,\lambda_m}$ still have imaginary part
$0$, but their real parts are distinct, and for each $j$, the real
parts $x'_{j,1},\dots,x'_{j,\lambda_j}$ are equally spaced in an
interval of size $<\epsilon_z$ centered at $x_{j,1}$.)

Recall that $\tZ_N\cong\RR^{2N}$ has real coordinates
$\Re(s_1),\Im(s_1),\dots,\Re(s_N),\Im(s_N)$ and so the tangent bundle
has a basis given by the infinitesimal variations
$\delta\Re(s_1),\delta\Im(s_1),\dots,\allowbreak\delta\Re(s_N),\allowbreak\delta\Im(s_N)$,
which are the unit vectors along those coordinate
directions. However, near the point $z'$ we will choose a different
coordinate system, \emph{tailored to the point $z'$}, and consequently
a different basis of the tangent bundle.

Consider disjoint neighborhoods
$U_{1,1}\ni z'_{1,1},\dots, U_{m,\lambda_m}\ni z'_{m,\lambda_m}$ in
$\CC$ and a neighborhood $U\ni\{z_{p^0+1},\dots,z_N\}$ in
$\Sym^{N-p^0}(\CC)$ which is disjoint from each
$U_{j,l}\times\Sym^{N-p^0-1}(\CC)$. Then
\[
W\defeq U_{1,1}\times\dots\times U_{m,\lambda_m}\times U=\prod_{l=1}^{\lambda_1} U_{1,l}\times\dots\times\prod_{l=1}^{\lambda_m} U_{m,l}\times U
\]
is an open neighborhood of $z'$ in $\tZ_N=\Sym^N(\CC)$.

Any point $w\in W$ has coordinates
$w_{j,l}=u_{j,l}+iv_{j,l}\in U_{j,l}$ and an unordered tuple
$\{w_{p^0+1}=u_{p^0+1}+iv_{p^0+1},\dots,w_N=u_N+iv_N\}\in U$. Let
$t_1,\dots, t_{N-p^0}$ be the elementary symmetric polynomials in
$w_{p^0+1},\dots,w_N$. Then our tailored coordinates on $U$ are
$\Re(t_1),\Im(t_1),\dots,\Re(t_{N-p^0}),\Im(t_{N-p^0})$, similar to
Equation~\eqref{eq:coordZN}. However, for each $j$, we pick different
tailored coordinates on $\prod_{l=1}^{\lambda_j} U_{j,l}$:
\[
  v_{j,1},u_{j,2}-u_{j,1},v_{j,2},u_{j,3}-u_{j,2},\dots,u_{j,\lambda_i}-u_{j,\lambda_{i-1}},v_{j,\lambda_i},\sum_{l=1}^{\lambda_j}u_{j,l}.
\]
For notational convenience, let $\Delta_{j,l}=u_{j,l+1}-u_{j,l}$.
Together, these give tailored local coordinates on $W$, which we reorder and write as:
\begin{gather*}
  v_{1,1},\Delta_{1,1},v_{1,2},\Delta_{1,2},\dots,\Delta_{1,\lambda_1-1},v_{1,\lambda_1},\\
  \cdots\\
  v_{m,1},\Delta_{m,1},v_{m,2},\Delta_{m,2},\dots,\Delta_{m,\lambda_m-1},v_{m,\lambda_m},\\
  \sum_{l=1}^{\lambda_1}u_{1,l},\dots,\sum_{l=1}^{\lambda_m}u_{m,l},\Re(t_1),\Im(t_1),\dots,\Re(t_{N-p^0}),\Im(t_{N-p^0}).
\end{gather*}

Their infinitesimal variations give a basis for the tangent bundle of
$W$. Of these, infinitesimals of the form $\delta v_{j,l}$ form a
basis of the normal bundle of $\tZ(p^-,p^+,p^0;(1,\dots,1))$, while
the rest form a basis for its tangent bundle.

Note that $\tZ(p^-,p^+,p^0;\lambda)<\tZ(p^-,p^+,p^0;(1,\dots,1))$; as
$\epsilon\to 0$, the point $z'\in \tZ(p^-,p^+,p^0;(1,\dots,1))$
approaches $z\in\tZ(p^-,p^+,p^0;\lambda)$. Each of the $2N$ vectors
listed above has a limit, and in the limit, the vectors from the last
row
\[
  \lim_{\epsilon\to 0}\delta\sum_{l=1}^{\lambda_1}u_{1,l}|_{z'},\dots,\lim_{\epsilon\to 0}\delta\sum_{l=1}^{\lambda_m}u_{m,l}|_{z'},\lim_{\epsilon\to 0}\delta\Re(t_1)|_{z'},\dots,\lim_{\epsilon\to 0}\delta\Im(t_{N-p^0})|_{z'}
\]
give a basis for the tangent space of $\tZ(p^-,p^+,p^0;\lambda)$ at
$z$. Similarly, the limits of the other vectors give normal vectors of
$\tZ(p^-,p^+,p^0;\lambda)$ in $\tZ_N$ at $z$, but they do not give a
basis. Specifically, we get
\begin{align*}
  \lim_{\epsilon\to 0}\delta v_{j,l_1}|_{z'}&=\lim_{\epsilon\to 0}\delta v_{j,l_2}|_{z'}&\text{for all $j,l_1,l_2$, and}\\
  \lim_{\epsilon\to 0}\delta \Delta_{j,l}|_{z'}&=0&\text{for all $j,l$.}
\end{align*}

Instead we pick a path $\tau\from[0,1]\to \ol{\tZ}(p^-,p^0,p^+;(1,1,\dots,1))$ with $\tau(0)=z$, $\tau(1)=z'$, and $\tau((0,1])$ lying inside the contractible subspace of $\tZ(p^-,p^0,(1,1,\dots,1))$ with coordinates satisfying
\begin{align*}
  v_{j,l}&=0&\text{for all $1\leq j\leq m,1\leq l\leq \lambda_j$,}\\
  \sum_{l=1}^{\lambda_j-1}\Delta_{j,l}&<\epsilon_z&\text{for all $1\leq j\leq m$,}\\
  w_j&=z_j&\text{for all $j>p^0$.}
\end{align*}
Pick an isotopy of $\tZ_N$ supported in a small neighborhood of this
path that sends $\tau(1)=z'$ to $\tau(0)=z$. 

\begin{definition}
\label{def:standardframe1}
The {\em standard frame} for the normal bundle to
$\tZ(p^-, p^0, p^+; \lambda)$ in $\tZ_N$ (equivalently,
$Z(p^-,p^0,p^+;\lambda)$ in $Z_N$) at any point $z$ is given by the
images under the above isotopy of the following vectors in
$T_{z'}\tZ_N$ (in this order):
\begin{gather*}
  \delta v_{1,1},-\delta\Delta_{1,1},\delta v_{1,2},-\delta\Delta_{1,2},\dots,-\delta\Delta_{1,\lambda_1-1},\delta v_{1,\lambda_1},\\
  \cdots\\
  \delta v_{m,1},-\delta \Delta_{m,1},\delta v_{m,2},-\delta \Delta_{m,2},\dots,-\delta \Delta_{m,\lambda_m-1},\delta v_{m,\lambda_m}.
\end{gather*}
While this depends on the choice of $\epsilon$, the smoothly varying
function $\epsilon_z$, the path $\tau$ staying within an $\epsilon_z$
neighborhood, and the isotopy supported near $\tau$, each of these
choices come from a contractible space, and so the different standard
frames at $z$ coming from different choices are canonically isotopic.
\end{definition}

\begin{remark}\label{rem:pause-to-justify}
We pause here to justify our choice of these complicated tailored
coordinate systems and the consequent standard frames. Consider any
two strata $Z(p^-,p^0,p^+;\lambda)<Z(q^-,q^0,q^+;\mu)$ of consecutive
dimensions. From Equations~\eqref{eq:ZN} and~\eqref{eq:dimZN}, we
conclude that there are only three possibilities
\begin{align*}
  (p^-,p^0,p^+)=(q^-,q^0,q^+)&\text{ and }\lambda\in\EC(\mu),\\
  (p^-,p^0,p^+)=(q^--1,q^0+1,q^+)&\text{ and }\lambda\in\UE(\mu),\\
  (p^-,p^0,p^+)=(q^-,q^0+1,q^+-1)&\text{ and }\lambda\in\UE(\mu).
\end{align*}
In each case, for any point $z\in Z(p^-,p^0,p^+;\lambda)$, one of the
normal vectors in the standard frame at $z$ is an inward or outward
normal vector to $Z(q^-,q^0,q^-;\mu)$; and for any nearby
$z'\in Z(q^-,q^0,q^+;\mu)$, the standard frame at $z'$ is composed of
the remaining vectors of the standard frame at $z$ (up to canonical
isotopy). For instance, in the first case, if
\[
\mu=(\mu_1,\dots,\mu_m),\qquad \lambda=(\mu_1,\dots,\mu_{k-1},\mu_k+\mu_{k+1},\mu_{k+2},\dots,\mu_m),
\]
then the normal vector $-\delta\Delta_{k,\mu_k}$ in the standard
frame is the outward normal to $Z(q^-,q^0,q^+;\mu)$. In the second
(respectively third) case, if
\[
\mu=(\mu_1,\dots,\mu_m),\qquad \lambda=(\mu_1,\dots,\mu_{k-1},1,\mu_{k},\dots,\mu_m),
\]
then the normal vector $\delta v_{k,1}$ in the standard frame is the
outward (respectively inward) normal to $Z(q^-,q^0,q^+;\mu)$.
\end{remark}

\begin{example}
  Consider the stratified space $Z_2=\Sym^2(\CC)/\RR$ from
  Example~\ref{ex:Z2}, viewed as the hyperplane $\Re(s_1)=0$ of the
  space $\tZ_2=\Sym^2(\CC)$ which has real coordinates
  \begin{align*}
    \Re(s_1)&=x_1+x_2\\
    a=\Im(s_1)&=y_1+y_2\\
    b=\Re(s_2)&=x_1x_2-y_1y_2\\
    c=\Im(s_2)&=x_1y_2+x_2y_1;
  \end{align*}
  its various strata $Z(p^-,p^0,p^+;\lambda)$ are the intersections
  of the corresponding strata $\tZ(p^-,p^0,p^+;\lambda)$ with this
  hyperplane.

  The stratum $\tZ(1,1,0)$ is a codimension one stratum. Near points
  in $Z(1,1,0)$ we use tailored coordinates $(y_1,x_1,x_2,y_2)$. For
  points actually on $Z(1,1,0)$, we have $y_1=x_1+x_2=0$ and $y_2<0$,
  so the coordinates are $(y_1=0,x_1=c/a,x_2=-c/a,y_2=a)$. The
  standard frame at any point $(a,b,c)\in Z(1,1,0)$ consists of the
  single vector $\delta y_1$; since
  \begin{align*}
    \delta a&=\delta y_1\\
    \delta b&=-y_2\delta y_1=-a(\delta y_1)\\
    \delta c&=x_2\delta y_1=-c/a(\delta y_1),
  \end{align*}
  the standard frame is the vector $(1,-a,-c/a)$. Observe that this is a
  normal vector pointing towards $Z(1,0,1)$ and away from $Z(2,0,0)$.

  A similar calculation shows that for any point $(a,b,c)$ in the
  other codimension one stratum $Z(0,1,1)$, the standard frame is
  again $(1,-a,-c/a)$, which is now a normal vector pointing towards
  $Z(0,0,2)$ and away from $Z(1,0,1)$.

  Near the one-dimensional stratum $Z(0,2,0;(1,1))$, the tailored local
  coordinates are $(y_1,y_2,x_1,x_2)$. For points actually on
  $Z(0,2,0;(1,1))$, $y_1=y_2=x_1+x_2=0$ and $x_1<x_2$, so the
  coordinates are $(y_1=0,y_2=0,x_1=-\sqrt{-b},x_2=\sqrt{-b})$. The
  standard frame at any point $(0,b,0)\in Z(0,2,0;(1,1))$ is
  $[\delta y_1,\delta y_2]$; we have
  \begin{align*}
    \delta a&=\delta y_1+\delta y_2\\
    \delta b&=-y_2\delta y_1-y_1\delta y_2=0\\
    \delta c&=x_1\delta y_2+x_2\delta y_1=-\sqrt{-b}(\delta y_2) +\sqrt{-b}(\delta y_1),
  \end{align*}
  and hence, the standard frame is
  $[(1,0,\sqrt{-b}),(1,0,-\sqrt{-b})]$. The two vectors
  point towards the two sheets of $Z(0,1,1)$ entering
  $Z(0,2,0;(1,1))$, and away from the two sheets of $Z(1,1,0)$
  entering $Z(0,2,0;(1,1))$. Moroever, for any nearby point in
  $Z(1,1,0)$ or $Z(0,1,1)$, the standard frame is $(1,-a,-c/a)$, which
  approaches one of the two vectors in the standard frame as we
  approach $Z(0,2,0;(1,1))$ (that is, as $a,c\to 0$ with $c/a = \pm \sqrt{-b}$).

  Finally for the zero-dimensional stratum $Z(0,2,0;(2))$, we pick a
  nearby point $z'=\{x_1+iy_1,x_2+iy_2\}\in Z(0,2,0;(2))$ with
  $y_1=y_2=x_1+x_2=0$ and $x_1<x_2$, but use tailored coordinates
  $(y_1,\Delta_1=x_2-x_1,y_2,\Re(s_1)=x_1+x_2)$ near $z'$; in terms of
  $b$, the tailored coordinates of $z'$ are
  $(y_1=0,\Delta_1=2\sqrt{-b},y_2=0,\Re(s_1)=0)$. We consider the frame $[\delta y_1,-\delta \Delta_1,\delta y_2]$; we calculate
  \begin{align*}
    a&=y_1+y_2\\
    \delta a&=\delta y_1+\delta y_2\\
    b&=x_1x_2-y_1y_2=(\Re(s_1)^2-\Delta_1^2)/4-y_1y_2\\
    \delta b&=-\Delta_1(\delta\Delta_1)/2-y_2(\delta y_1)-y_1(\delta y_2)=-\sqrt{-b}(\delta\Delta_1)\\
    c&=x_1y_2+x_2y_1=y_2(\Re(s_1)-\Delta_1)/2+y_1(\Re(s_1)+\Delta_1)/2\\
    \delta c&=(y_1-y_2)(\delta\Delta_1)/2+(\Re(s_1)+\Delta_1)(\delta y_1)/2+(\Re(s_1)-\Delta_1)(\delta y_2)/2=\sqrt{-b}(\delta y_1-\delta y_2),
  \end{align*}
  so the frame is given by
  $[(1,0,\sqrt{-b}),(0,\sqrt{-b},0),(1,0,-\sqrt{-b})]$. Instead of
  taking the limit as $b\to 0$ (which would not produce a
  $3$-dimensional frame), we take the path $(0,bt,0)$, $t\in[0,1]$,
  choose an isotopy supported near this path that sends $(0,b,0)$ to
  $(0,0,0)$, and consider the image of this frame under this
  isotopy. Assuming the isotopy near the path is an affine translation
  in the $a,b,c$ coordinates, the standard frame at $(0,0,0)$ is also
  given by
  $[(1,0,\sqrt{-b}),(0,\sqrt{-b},0),(1,0,-\sqrt{-b})]$. This depends
  on $b$, but $b$ varies on the contractible space $(-\epsilon,0)$ for
  some $\epsilon>0$, and the different frames for different choices
  of $b$ are canonically isotopic. Observe that the second vector
  points away from the stratum $Z(0,2,0;(1,1))$; and by construction,
  the other two vectors are canonically isotopic to the frame
  $[(1,0,\sqrt{-b}),(1,0,-\sqrt{-b})]$ at $z'$,
  which is the standard frame at $z'$.
  \end{example}



 \subsection{Local models from $Z_N$} \label{sec:localZN}
 Consider a stratum 
 \[
   Y =  Z(p^-, p^0, p^+; \lambda) \subseteq Z_N
 \]
 with $\lambda=(\lambda_1,\dots, \lambda_m)$. Around any point
 $y \in Y$, consider the affine subspace generated by the standard
 frame for the normal bundle to $Y$, as in
 Definition~\ref{def:standardframe1}. This subspace is the {\em
   local model for $Y$ in the normal directions}, and we have an isomorphism
\begin{equation}
\label{eq:locmodZN}
\L(Y) \defeq  Z_{\lambda_1} \times \dots \times Z_{\lambda_m}
\end{equation}
where the subspace of $\L(Y)$ generated by the vectors
$\delta v_{j,1},-\delta\Delta_{j,1},\dots,\delta v_{j,\lambda_j}$ is
isomorphic to $Z_{\lambda_j}$ by the linear map that sends these
vectors to the standard frame of the unique $0$-dimensional stratum
inside $Z_{\lambda_j}$.

The above decomposition allows us to understand the local models
around $y$ inside all the other strata. Using the local coordinate
system around $y$, we can identify $\L(Y)$ with a small disk inside
$Z_N$, transverse to $Y$ at $y$ and of complementary dimension. The
stratification of each $Z_{\lambda_i}$ induces a product
stratification of $\L(Y)$, and the stratification of $Z_N$,
intersected with $\L(Y)$, also produces a stratification of
$\L(Y)$. These two stratifications are canonically isomorphic. Indeed,
consider another stratum $X = Z(q^-, q^0, q^+; \mu) $ with $Y \leq
X$. Then the intersection
\[
  \L(Y; X) \defeq  \L(Y) \cap X
\]
is the union of the following strata in the product stratification of
$\L(Y)$:
\[
  Z_{\lambda_1}(q_1^-, q_1^0, q_1^+; \mu^1) \times \cdots \times Z_{\lambda_m}(q_m^-, q_m^0, q_m^+;\mu^m),\qquad q_i^-+q_i^0+q_i^+=\lambda_i,\mu^i\in\Part(\lambda_i),
\]
satisfying
\[
q^- = p^- + \sum q_i^-,  \qquad
q^0 = \sum q_i^0, \qquad
q^+ = p^+  + \sum q_i^+\qquad
  \mu = \mu^1 * \dots * \mu^m,
\]
where $*$ is the concatenation of partitions from
\eqref{eq:concatenate}. We call $\L(Y, X)$ the {\em local model for
  $Y$ in the normal directions inside $X$}.
 
 \begin{example}
   In Example~\ref{ex:Z1}, the stratum $Z(0,1,0)$ lives inside the closed
   stratum $\bZ(0, 0, 1)$ as $ 0 \in \R_+$. (Note that for a multifaceted
   manifold $X$, the inclusion of any multifacet $\del_i X\into X$
   also has the same local model, but we will treat those differently;
   in the stratified spaces $X$ that we will construct, the points with
   the local model $Z(0,1,0) \in \bZ(0, 0, 1)$ will be part of the special
   boundary, while points in the multifacets will be part of the
   ordinary boundary.)
 \end{example}

 \begin{example}\label{ex:Z020-inside-Z2}
   In Example~\ref{ex:Z2}, the local model of $Z(0,2,0;(1,1))$ inside
   $Z_2$ is same as $\RR\times \{0\}$ inside
   $\RR \times Z_1\times Z_1$, which is a stratified
   space with four $3$-dimensional strata, four $2$-dimensional
   strata, and a unique $1$-dimensional stratum:
   \[
     \begin{tikzpicture}[scale=0.6]
       \fill[black!20] (-1,-1) rectangle (1,1);
       \draw[thick] (-1,0) to (1,0);
       \draw[thick] (0,-1) to (0,1);
       \fill[black] (0,0) circle (4pt);
       \node[anchor=east] at (-1,0) {$\RR\times$};
     \end{tikzpicture}
   \]
   The local model of $Z(0,2,0;(1,1))$ inside $\bZ(0,0,2)$ is same as
   $\RR\times\{0\}$ inside
   $\RR\times\bZ(0,0,1)\times \bZ(0,1,0)\cong\RR\times\RR_+^2$, which
   is one of the four closed $3$-dimensional strata of
   $\RR\times Z_1\times Z_1$. (Note once again, codimension-$2$
   boundary inside a multifaceted manifold has the same local model,
   but will be treated differently.)  Finally, the local model of
   $Z(0,2,0;(1,1))$ inside $\bZ(0,1,1)$ is same as
   $\RR\times\{0\}$ inside
   $\RR\times \big(Z(0,1,0)\times \bZ(0,0,1)\cup \bZ(0,0,1)\times
   Z(0,1,0)\big)\cong\RR\times\{(x,y)\in\R_+^2\mid xy=0\}$, which is
   the union of two of the four closed $2$-dimensional strata of
   $\RR\times Z_1\times Z_1$.
 \end{example}
 
We can now prove the following.
 \begin{proposition}
 \label{prop:WhitneyZN}
 The stratification of $Z_N \cong \R^{2N-1}$ described in Section~\ref{sec:ZN} is a Whitney stratification.
 \end{proposition}
 
 \begin{proof}
   Given the local models from Equation~\eqref{eq:locmodZN} and the
   fact that products of Whitney stratifications are Whitney
   (cf.~Proposition~\ref{prop:products}), it suffices to consider the
   origin $0 \in Z_N$, and show that all the bigger strata are Whitney
   regular over it.
 
 Consider the projection
 \[
   \pi\from \C^N \to Z_N, \ \ \  \pi(z_1, \dots, z_N) = (s_1, \dots, s_N)
 \]
 where $s_i$ are the elementary symmetric polynomials from
 \eqref{eq:coordtZN}. Any stratum $Y \subset Z_N$ is the image under
 $\pi$ of a subset of $\C^N$ given by linear equalities and
 inequalities. Since $\pi$ is a polynomial mapping,
 Proposition~\ref{prop:imagesa} implies that $Y$ is
 semialgebraic. Proposition~\ref{prop:badset} shows that the bad set
 $B(\{0\}, Y) \subset \{0\}$ is empty, and therefore $Y$ is Whitney
 regular over the origin.
 \end{proof}

 \subsection{More general local models}
\label{sec:generalLM}
We now complete the description of the local models that will appear in the stratified spaces in this paper. More generally than $Z_N$, let $\vN = (N_2, \dots, N_{n}) \in \NN^{n-1}$, and consider the space
\begin{equation}
\label{eq:ZvN}
Z_{\vN} \defeq  \left(\Sym^{N_2}(\CC)\times\dots\times\Sym^{N_{n}}(\CC)\right)/\RR,
\end{equation}
where we divided by the diagonal action of $\RR$. If $\vN=0$, this is a point; otherwise, this is homeomorphic to $\RR^{2|\vN|-1}$, and admits a decomposition into strata
\[
  Z(\vp^{\, \, -}, \vp^{\, \, 0}, \vp^{\, \, +}; \vlambda)
\]
where $\vp^{\, \ast}=(p^\ast_2, \dots, p^\ast_n)$ for $\ast \in \{-,0, +\}$ and $\vlambda=(\lambda_2, \dots, \lambda_n)$
satisfy
\[
  p_i^- + p_i^0 + p_i^+=N_i, \ \lambda_i \in \Part(p_i^0).
\]
Specifically, $Z(\vp^{\, \, -}, \vp^{\, \, 0}, \vp^{\, \, +}; \vlambda)$ is given by asking the imaginary parts of the coordinates in each $\Sym^{N_i}(\CC)$ to consist of $p_i^-$ negative numbers, $p_i^0$ zeros (with the corresponding real parts decomposed according to the partition $\lambda_i$), and $p_i^+$ positive numbers.

\begin{definition}
\label{def:standardframe2}
The {\em standard frame} for the normal bundle to a stratum $Z(\vp^{\, \, -}, \vp^{\, \, 0}, \vp^{\, \, +}; \vlambda) \subset Z_{\vN}$ is obtained by concatenating the standard frames for each $Z(p_i^-, p_i^0, p_i^+; \lambda_i)\subset Z_{N_i}$ described in Definition~\ref{def:standardframe1}.
\end{definition}


Even more generally, for the local models in Section~\ref{sec:moduli} we will consider products of the  ones considered above, as well as half-intervals $[0, \infty)$ that account for the usual (non-special) boundaries of $\langle n \rangle$-manifolds, and $\R$ factors that just correspond to some tangent directions inside the stratum. The most general model is of the form
\begin{equation}
\label{eq:mostgeneral}
 \R^a \times \R_+^{r-1}  \times Z_{\vN^1}  \times Z_{\vN^2} \times \dots \times Z_{\vN^r},
 \end{equation}
with the induced product stratification from its factors---where $\R$ has a single stratum, and $\R_+$ has two strata: $\{0\}$ and $(0, \infty)$. The local models will be based on the strata of this space inside the closures of bigger strata.

Note that we can define tailored local coordinates around any point in \eqref{eq:mostgeneral}, in a manner similar to what we did for $Z_N$ in Section~\ref{sec:modelinternal}.

\begin{definition}
\label{def:standardframe3}
The {\em standard frame} for the normal bundle to a stratum inside the space \eqref{eq:mostgeneral} is obtained by concatenating the standard frames for each of the factors, where for $\{0\} \subset \R^+$ we use the standard unit vector, and for the strata in each $Z_{\vN^i}$ we use the frames from Definition~\ref{def:standardframe2}.
\end{definition}

\begin{definition}
\label{def:lmn} 
Given strata
\[
  Y, X \subset  \R^a \times \R_+^{r-1}  \times Z_{\vN^1}  \times Z_{\vN^2} \times \dots \times Z_{\vN^r}
\]
 with $Y \leq X$, we let the {\em  local model for $Y$ in the normal directions inside $X$}, denoted $\L(Y; X)$, be the intersection of $X$ with a small ball in the linear subspace spanned by the standard frame for the normal bundle to $Y$ (in tailored local coordinates around any point of $Y$). The union of $\L(Y; X)$ over all $X \geq Y$ is denoted $\L(Y)$. 
\end{definition}

\begin{proposition}
\label{prop:productWhitneyZN}
 The stratification from Equation~\eqref{eq:mostgeneral}  is a Whitney stratification.
 \end{proposition}
 
 \begin{proof}
Since the Whitney condition is local, observe that a stratification of a space of the form $V/\R$ (where  $\R$ acts freely on $V$) is Whitney if and only if its pullback to $V$ is Whitney. In Proposition~\ref{prop:WhitneyZN} we established that the stratification of $Z_N = \Sym^n(\C)/\R$ is Whitney. Furthermore, Proposition~\ref{prop:products} says that the product of Whitney stratifications is Whitney. Combining these facts, we get the conclusion.
 \end{proof}

\section{Moduli spaces}
\label{sec:moduli}
Let us recall some notation from Section~\ref{sec:grid-background}.
Let $O_1,\dots,O_n$ and $X_1, \dots, X_n$ be the markings on the
grid. We only consider (positive) domains $D\in\pdomains(x,y)$ that do not go over
the marking $O_1$; let $\coefficients{D}\in\ZZ^{n-1}$ be the vector
that records the coefficients of $D$ at the remaining $O$-markings.
For each $j$, let $H_j$, respectively $V_j$, be the horizontal row,
respectively vertical column, that contains $O_j$. The set of periodic
domains $\periodic$ can be identified with $\domains(x,x)$ for any $x$,
and we have
$\periodic=\ZZ\basis{H_2,\dots,H_n,V_2,\dots,V_n}.$

Let $\DNlambda=\DNlambda(\Grid)$ be the set
of triples $(D,\vN,\vlambda)$ where $D\in\pdomains(x,y)$, and
\begin{gather*}
  \vN=(N_2,\dots,N_{n})\in\NN^{n-1},\\
  \vlambda=(\lambda_2, \dots, \lambda_{n}), \ \lambda_j \in \Part(N_j).
\end{gather*}
(These are the generators of the chain complex $\CDP_*(\Grid)$ from
Section~\ref{sec:new}.) For each $(D,\vN,\vlambda)\in\DNlambda$, we
will construct a stratified space
\[
  \bModuli_{\vN, \vlambda}(D)
\]
in Section~\ref{sec:construction}. This will come equipped with an
embedding in a Euclidean space, be Whitney stratified and hence (by
Theorem~\ref{thm:TM}) Thom-Mather stratified, and the local models for
the stratification will be those considered in
Section~\ref{sec:local}. In this section, we will present some model
spaces that could {\em potentially} play the role of
$\bModuli_{\vN, \vlambda}(D)$, in a few simple cases, as illustrated
in Figures~\ref{fig:Dind2}--\ref{fig:2row}.

\begin{figure}
  \centering
  \begin{tikzpicture}
    \foreach\l [count=\c from 0] in {a,b}{
      \begin{scope}[yshift=-2.5*\c cm]
        \node at (-2,0) {$(\l)$};

        \foreach \suba/\subb/\pos/\size/\decor in {0/0/center/0.7/0, 3.5/0.1/south/0.4/1,  5.5/0.1/south/0.4/2}{
        
          \node[anchor=\pos] at (\suba,\subb){
            \begin{tikzpicture}[scale=\size]

              \ifnum\c=0

              \draw (0,0)--++(2,0)--++(0,1)--++(-2,0)--++(0,-1) (1,1.5)--++(2,0)--++(0,1)--++(-2,0)--++(0,-1);

              \ifnum\decor=0
              \foreach \x/\y in {0/0,2/1,1/1.5,3/2.5}{\fill[black] (\x,\y) circle (0.12);}
              \foreach \x/\y in {2/0,0/1,1/2.5,3/1.5}{\draw[fill=white] (\x,\y) circle (0.12);}
              \fi

              \ifnum\decor=1
              \node[anchor=center] at (1,0.5) {\tiny 1};
              \node[anchor=center] at (2,2) {\tiny 2};
              \fi

              \ifnum\decor=2
              \node[anchor=center] at (1,0.5) {\tiny 2};
              \node[anchor=center] at (2,2) {\tiny 1};
              \fi

              \else

              \draw (0,0)--++(2,0)--++(0,1)--++(-1,0)--++(0,1)--++(-1,0)--++(0,-2);

              \ifnum\decor=0
              \foreach \x/\y in {0/0,2/1,1/2}{\fill[black] (\x,\y) circle (0.12);}
              \foreach \x/\y in {2/0,0/2,1/1}{\draw[fill=white] (\x,\y) circle (0.12);}
              \fi

              \ifnum\decor=1
              \draw(1,1)--(1,0);
              \node[anchor=center] at (0.5,1) {\tiny 1};
              \node[anchor=center] at (1.5,0.5) {\tiny 2};
              \fi

              \ifnum\decor=2
              \draw(1,1)--(0,1);
              \node[anchor=center] at (1,0.5) {\tiny 1};
              \node[anchor=center] at (0.5,1.5) {\tiny 2};
              \fi

              \fi
            \end{tikzpicture}};
          
        }
        
        \node at (2,0) {$\longrightarrow$};

        \draw[thick] (3.5,0)--++(2,0) node[pos=0]{$\bullet$} node[pos=1]{$\bullet$};

        \ifnum\c=1
        \node at (0,-1.3) {$D$};
        \node at (4.5,-1.3) {$\bModuli_0(D)$};
        \fi
      \end{scope}
    }
  \end{tikzpicture}
\caption {Domains of index two on the grid, and the associated moduli spaces $\bModuli_0(D)$. On the left hand side, the black dots are part of the initial point $x$ in each domain, and the white dots part of the final point $y$. The ends of the moduli space correspond to different decompositions $D = D^1 * D^2$. In each picture we indicate $D^i$ with the respective digit $i \in \{1,2\}$. }
\label{fig:Dind2}
\end{figure}

\begin{figure}
  \centering
  \begin{tikzpicture}
    
    \node at (0,0){
      \begin{tikzpicture}[scale=0.7]
        \draw (0,0)--++(2,0)--++(0,1)--++(-2,0)--++(0,-1) (0.67,1.5)--++(2,0)--++(0,1)--++(-2,0)--++(0,-1) (1.33,3)--++(2,0)--++(0,1)--++(-2,0)--++(0,-1);
        \foreach \x/\y in {0/0,2/1,0.67/1.5,2.67/2.5,1.33/3,3.33/4}{\fill[black] (\x,\y) circle (0.12);}
        \foreach \x/\y in {2/0,0/1,0.67/2.5,2.67/1.5,1.33/4,3.33/3}{\draw[fill=white] (\x,\y) circle (0.12);}
        \node at (1,0.5) {\small $C$};
        \node at (1.67,2) {\small $B$};
        \node at (2.33,3.5) {\small $A$};
      \end{tikzpicture}};

    \node at (2,0) {$\longrightarrow$};
    
    \begin{scope}[xshift=5.5cm,scale=2]
      
      \draw[fill=black!20] (0:1)--(60:1)--(120:1)--(180:1)--(240:1)--(300:1)--(0:1);
      \foreach \a in {0,120,240}{\begin{scope}[rotate=\a]
          \draw[ultra thick] (0:1)--(60:1) node[pos=0] {$\bullet$} node[pos=1] {$\bullet$};
        \end{scope}}
      
      \foreach \x/\y/\z/\r [count=\c from 0] in {A/C/B/1.2,C/A/B/1.3,C/B/A/1.3,B/C/A/1.2,B/A/C/1.3,A/B/C/1.3}{
        \node at (60*\c:\r) {\begin{tikzpicture}[scale=0.5]
            
            \draw (0,0) to[out=20,in=-20] ++(0,1) to[out=20,in=-20] ++(0,1) to[out=20,in=-20] ++(0,1);
            \draw (0,0) to[out=160,in=-160] ++(0,1) to[out=160,in=-160] ++(0,1) to[out=160,in=-160] ++(0,1);
            \node at (0,2.5) {\tiny $\x$};
            \node at (0,1.5) {\tiny $\y$};
            \node at (0,0.5) {\tiny $\z$};
            
          \end{tikzpicture}};
      }

      \foreach \x/\y/\r [count=\c from 0] in {AC/B/1.15,BC/A/1.15,AB/C/1.3}{
        \node at (30+120*\c:\r) {\begin{tikzpicture}[scale=0.5]
            
            \draw (0,0) to[out=20,in=-20] ++(0,1) to[out=20,in=-20] ++(0,2);
            \draw (0,0) to[out=160,in=-160] ++(0,1) to[out=160,in=-160] ++(0,2);
            \node at (0,2) {\tiny $\x$};
            \node at (0,0.5) {\tiny $\y$};
            
          \end{tikzpicture}};
      }     

      \foreach \x/\y/\r [count=\c from 0] in {A/BC/1.15,C/AB/1.3,B/AC/1.15}{
        \node at (-30+120*\c:\r) {\begin{tikzpicture}[scale=0.5]
            
            \draw (0,0) to[out=20,in=-20] ++(0,2) to[out=20,in=-20] ++(0,1);
            \draw (0,0) to[out=160,in=-160] ++(0,2) to[out=160,in=-160] ++(0,1);
            \node at (0,2.5) {\tiny $\x$};
            \node at (0,1) {\tiny $\y$};
            
          \end{tikzpicture}};
      }     

    \end{scope}
    
  \end{tikzpicture}
  \caption {A domain of index three on the grid, and the associated moduli spaces $\bModuli_0(D)$. For each edge and vertex on the boundary we show the corresponding decomposition $D=D^1 * D^2$ or $D=D^1 * D^2 * D^3$ by a picture of the trajectory breaking. (Holomorphic strips from $x$ to $y$ are drawn vertically, with $x$ at top and $y$ at bottom (and $\alpha$ on the left and $\beta$ on the right)---for example, the bottom edge corresponds to $(A + B) * C$.) The multifacet $\del_1\bModuli_0(D)$ is made of the thin edges, and the multifacet $\del_2\bModuli_0(D)$ is made of the thick edges.}
\label{fig:Dind3}
\end{figure}

\begin{figure}
  \centering
  \begin{tikzpicture}

    \node at (-0.5,1.25) {{\small $A+B$} $\longrightarrow$};
    
    \draw[thick] (1.5,0.5)--++(0,2) node[pos=1] {$\bullet$};
    \fill[blue!60] (1.5,0.5) circle (0.1);
    \node[anchor=east] at (1.5,2.5) {\begin{tikzpicture}[scale=0.5]
        \draw (0,0) to[out=20,in=-20] ++(0,1) to[out=20,in=-20] ++(0,1);
        \draw (0,0) to[out=160,in=-160] ++(0,1) to[out=160,in=-160] ++(0,1);
        \node[anchor=center] at (0,1.5) {\tiny $A$};
        \node[anchor=center] at (0,0.5) {\tiny $B$};
      \end{tikzpicture}};
    \node[anchor=east] at (1.5,0.5) {\begin{tikzpicture}[scale=0.5]
        \draw[fill=black!50] (0,0) to[out=70,in=-70] coordinate[midway] (r) ++(0,2) to[out=-110,in=110] coordinate[midway] (l) ++(0,-2);
        \draw (l) arc (0:360:0.5);        
        \node[anchor=center] at ($(l)+(-0.5,0)$) {\tiny $\mathit{AB}$};
      \end{tikzpicture}};
    
    \node at (-0.5,-1.25) {{\small $C+D$} $\longrightarrow$};
    \draw[thick] (1.5,-0.5)--++(0,-2) node[pos=1] {$\bullet$};
    \fill[blue!60] (1.5,-0.5) circle (0.1);
    \node[anchor=east] at (1.5,-2.5) {\begin{tikzpicture}[scale=0.5]
        \draw (0,0) to[out=20,in=-20] ++(0,1) to[out=20,in=-20] ++(0,1);
        \draw (0,0) to[out=160,in=-160] ++(0,1) to[out=160,in=-160] ++(0,1);
        \node[anchor=center] at (0,1.5) {\tiny $C$};
        \node[anchor=center] at (0,0.5) {\tiny $D$};
      \end{tikzpicture}};
    \node[anchor=east] at (1.5,-0.5) {\begin{tikzpicture}[scale=0.5]
        \draw[fill=black!50] (0,0) to[out=70,in=-70] coordinate[midway] (r) ++(0,2) to[out=-110,in=110] coordinate[midway] (l) ++(0,-2);
        \draw (r) arc (-180:180:0.5);        
        \node[anchor=center] at ($(r)+(0.5,0)$) {\tiny $\mathit{CD}$};
      \end{tikzpicture}};

    \draw (2,1.25) to[out=0,in=180] (3,0);
    \draw[->] (2,-1.25) to[out=0,in=180] (3,0);
    
    \draw[thick] (3.5,-2)--++(0,4) node[pos=0] {$\bullet$} node[pos=1] {$\bullet$};
    \fill[blue!60] (3.5,0) circle (0.1);
    \node[anchor=west] at (3.5,-2) {\begin{tikzpicture}[scale=0.5]
        \draw (0,0) to[out=20,in=-20] ++(0,1) to[out=20,in=-20] ++(0,1);
        \draw (0,0) to[out=160,in=-160] ++(0,1) to[out=160,in=-160] ++(0,1);
        \node[anchor=center] at (0,1.5) {\tiny $C$};
        \node[anchor=center] at (0,0.5) {\tiny $D$};
      \end{tikzpicture}};
    \node[anchor=west] at (3.5,2) {\begin{tikzpicture}[scale=0.5]
        \draw (0,0) to[out=20,in=-20] ++(0,1) to[out=20,in=-20] ++(0,1);
        \draw (0,0) to[out=160,in=-160] ++(0,1) to[out=160,in=-160] ++(0,1);
        \node[anchor=center] at (0,1.5) {\tiny $A$};
        \node[anchor=center] at (0,0.5) {\tiny $B$};
      \end{tikzpicture}};

    \begin{scope}[xshift=-5cm]

      \draw (-2,-1)--++(4.5,0) (-2,0)--++(4.5,0) (-1,-2)--++(0,4.5) (0,-2)--++(0,4.5);
      \draw[dashed] (0.5,-1)--++(0,1) (1.5,-1)--++(0,1) (-1,0.5)--++(1,0) (-1,1.5)--++(1,0);
      \foreach \i/\j in {0.5/-1,1.5/0,-1/0.5,0/1.5}{\fill[black] (\i,\j) circle (0.12*0.7);}
      \foreach \i/\j in {0.5/0,1.5/-1,0/0.5,-1/1.5}{\draw[fill=white] (\i,\j) circle (0.12*0.7);}

      \node at (-0.5,-0.5) {\small $O_i$};
      \node at (-0.5,-2) {\small $V_i$};
      \node at (-2,-0.5) {\small $H_i$};
      \node at (1,-0.5) {\small $A$};
      \node at (2,-0.5) {\small $B$};
      \node at (-0.5,1) {\small $C$};
      \node at (-0.5,2) {\small $D$};
    \end{scope}
    
  \end{tikzpicture}
\caption {Gluing the moduli spaces for the row and the column that contain the same marking $O_i$. The special boundary points are shown in blue. These blue points are associated to a configuration made of a trivial strip (shown in gray) and a boundary degeneration.}
\label{fig:row}
\end{figure}

\begin{figure}
  \centering
  \begin{tikzpicture}

    \begin{scope}[scale=2]
      \draw[fill=black!20] (60:1)--(120:1)--(180:1)--(240:1)--(300:1);
      \draw[ultra thick] (120:1)-- node[pos=0] {$\bullet$} node[pos=1] {$\bullet$}(180:1) (240:1)-- node[pos=0] {$\bullet$}(300:1);
      \draw[ultra thick,blue!60] (300:1)--(60:1);
      \fill[blue!60] (300:1) circle (0.05);
      \fill[blue!60] (60:1) circle (0.05);

      \foreach \x/\y/\z/\pos [count=\c from 2] in {C/A/B/south,A/C/B/east,A/B/C/north}{
        \node[anchor=\pos] (n\c) at (60*\c:1) {\begin{tikzpicture}[scale=0.5]
            \draw (0,0) to[out=20,in=-20] ++(0,1) to[out=20,in=-20] ++(0,1) to[out=20,in=-20] ++(0,1);
            \draw (0,0) to[out=160,in=-160] ++(0,1) to[out=160,in=-160] ++(0,1) to[out=160,in=-160] ++(0,1);
            \node[anchor=center] at (0,2.5) {\tiny $\x$};
            \node[anchor=center] at (0,1.5) {\tiny $\y$};
            \node[anchor=center] at (0,0.5) {\tiny $\z$};
          \end{tikzpicture}};
      }

      \node[anchor=east] at (n3) {{\small $A+B+C$} $\longrightarrow\quad$};
      
      \foreach \x/\y/\r [count=\c from 1] in {AC/B/1.15,AB/C/1.3}{
        \node at (30+120*\c:\r) {\begin{tikzpicture}[scale=0.5]
            \draw (0,0) to[out=20,in=-20] ++(0,1) to[out=20,in=-20] ++(0,2);
            \draw (0,0) to[out=160,in=-160] ++(0,1) to[out=160,in=-160] ++(0,2);
            \node at (0,2) {\tiny $\x$};
            \node at (0,0.5) {\tiny $\y$};
          \end{tikzpicture}};
      }     

      \foreach \x/\y/\r [count=\c from 1] in {C/AB/1.3,A/BC/1.15}{
        \node at (-30+120*\c:\r) {\begin{tikzpicture}[scale=0.5]
            \draw (0,0) to[out=20,in=-20] ++(0,2) to[out=20,in=-20] ++(0,1);
            \draw (0,0) to[out=160,in=-160] ++(0,2) to[out=160,in=-160] ++(0,1);
            \node at (0,2.5) {\tiny $\x$};
            \node at (0,1) {\tiny $\y$};
          \end{tikzpicture}};
      }     

      \node[anchor=south] at (60:1) {\begin{tikzpicture}[scale=0.5]

          \draw[fill=black!50] (0,0) to[out=70,in=-70] coordinate[midway] (r) ++(0,2) to[out=-110,in=110] coordinate[midway] (l) ++(0,-2);
          \draw (l) arc (0:360:0.5);        
          \node[anchor=center] at ($(l)+(-0.5,0)$) {\tiny $\mathit{AB}$};
          \draw (0,2) to[out=20,in=-20] ++(0,1) to[out=-160,in=160] ++(0,-1);
          \node[anchor=center] at (0,2.5) {\tiny $C$};
        \end{tikzpicture}};

      \node[anchor=north] at (-60:1) {\begin{tikzpicture}[scale=0.5]

          \draw[fill=black!50] (0,1) to[out=70,in=-70] coordinate[midway] (r) ++(0,2) to[out=-110,in=110] coordinate[midway] (l) ++(0,-2);
          \draw (l) arc (0:360:0.5);        
          \node[anchor=center] at ($(l)+(-0.5,0)$) {\tiny $\mathit{AB}$};
          \draw (0,0) to[out=20,in=-20] ++(0,1) to[out=-160,in=160] ++(0,-1);
          \node[anchor=center] at (0,0.5) {\tiny $C$};
        \end{tikzpicture}};
      
      \node[anchor=west] at (0:0.5) {\begin{tikzpicture}[scale=0.5]
          \draw (0,0) to[out=20,in=-20] coordinate[midway] (r) ++(0,1) to[out=-160,in=160] coordinate[midway] (l) ++(0,-1);
          \node[anchor=center] at (0,0.5) {\tiny $C$};
          \draw (l) arc (0:360:0.5);        
          \node[anchor=center] at ($(l)+(-0.5,0)$) {\tiny $\mathit{AB}$};
        \end{tikzpicture}};
      
    \end{scope}

    \begin{scope}[xshift=3cm,scale=2]
      \draw[fill=black!20] (60:1) to[out=0,in=120] (0:2) to[out=-120,in=0] (300:1);
      \draw[ultra thick] (0:2) to[out=-120,in=0] node[pos=0] {$\bullet$} (300:1);
      \draw[ultra thick,blue!60] (300:1)--(60:1);
      \fill[blue!60] (300:1) circle (0.05);
      \fill[blue!60] (60:1) circle (0.05);
      
      \node[anchor=west] (nn) at (0:2) {\begin{tikzpicture}[scale=0.5]
          \draw (0,0) to[out=20,in=-20] ++(0,1) to[out=20,in=-20] ++(0,1) to[out=20,in=-20] ++(0,1);
          \draw (0,0) to[out=160,in=-160] ++(0,1) to[out=160,in=-160] ++(0,1) to[out=160,in=-160] ++(0,1);
          \node[anchor=center] at (0,2.5) {\tiny $C$};
          \node[anchor=center] at (0,1.5) {\tiny $D$};
          \node[anchor=center] at (0,0.5) {\tiny $C$};
        \end{tikzpicture}};

      \node[anchor=west] at (nn) {$\quad\longleftarrow$ {\small $2C+D$}};

      \node at (-35:1.8) {\begin{tikzpicture}[scale=0.5]
          \draw (0,0) to[out=20,in=-20] ++(0,1) to[out=20,in=-20] ++(0,2);
          \draw (0,0) to[out=160,in=-160] ++(0,1) to[out=160,in=-160] ++(0,2);
          \node at (0,2) {\tiny $CD$};
          \node at (0,0.5) {\tiny $C$};
        \end{tikzpicture}};

      \node at (35:1.8) {\begin{tikzpicture}[scale=0.5]
          \draw (0,0) to[out=20,in=-20] ++(0,2) to[out=20,in=-20] ++(0,1);
          \draw (0,0) to[out=160,in=-160] ++(0,2) to[out=160,in=-160] ++(0,1);
          \node at (0,2.5) {\tiny $C$};
          \node at (0,1) {\tiny $CD$};
        \end{tikzpicture}};

      \node[anchor=south] at (60:1) {\begin{tikzpicture}[scale=0.5]

          \draw[fill=black!50] (0,0) to[out=70,in=-70] coordinate[midway] (r) ++(0,2) to[out=-110,in=110] coordinate[midway] (l) ++(0,-2);
          \draw (r) arc (-180:180:0.5);        
          \node[anchor=center] at ($(r)+(0.5,0)$) {\tiny $\mathit{CD}$};
          \draw (0,2) to[out=20,in=-20] ++(0,1) to[out=-160,in=160] ++(0,-1);
          \node[anchor=center] at (0,2.5) {\tiny $C$};
        \end{tikzpicture}};

      \node[anchor=north] at (-60:1) {\begin{tikzpicture}[scale=0.5]

          \draw[fill=black!50] (0,1) to[out=70,in=-70] coordinate[midway] (r) ++(0,2) to[out=-110,in=110] coordinate[midway] (l) ++(0,-2);
          \draw (r) arc (-180:180:0.5);        
          \node[anchor=center] at ($(r)+(0.5,0)$) {\tiny $\mathit{CD}$};
          \draw (0,0) to[out=20,in=-20] ++(0,1) to[out=-160,in=160] ++(0,-1);
          \node[anchor=center] at (0,0.5) {\tiny $C$};
        \end{tikzpicture}};

      \node[anchor=east] at (0:0.5) {\begin{tikzpicture}[scale=0.5]
          \draw (0,0) to[out=20,in=-20] coordinate[midway] (r) ++(0,1) to[out=-160,in=160] coordinate[midway] (l) ++(0,-1);
          \node[anchor=center] at (0,0.5) {\tiny $C$};
          \draw (r) arc (-180:180:0.5);        
          \node[anchor=center] at ($(r)+(0.5,0)$) {\tiny $\mathit{CD}$};
        \end{tikzpicture}};

    \end{scope}

    \draw (0,-3.5) to[out=-90,in=90] ++(3,-1);
    \draw[->] (6,-3.5) to[out=-90,in=90] ++(-3,-1);
    
    \begin{scope}[yshift=-6.5cm,xshift=2cm,scale=2]
      \draw[fill=black!20] (60:1)--(120:1)--(180:1)--(240:1)--(300:1) to[out=0,in=-120] (0:2) to[out=120,in=0] (60:1);
      \draw[ultra thick] (120:1)-- node[pos=0] {$\bullet$} node[pos=1] {$\bullet$}(180:1) (240:1)-- node[pos=0] {$\bullet$}(300:1) to[out=0,in=-120] node[pos=1] {$\bullet$} (0:2);
      \draw[ultra thick,blue!60] (300:1)--(60:1);
      \fill[blue!60] (300:1) circle (0.05);
      \fill[blue!60] (60:1) circle (0.05);
    \end{scope}

  \end{tikzpicture}
\caption {Reusing the rectangles $A,B,C,D$ from Figure~\ref{fig:row},
  we glue the moduli spaces $\bM_0(A+B+C)$ and $\bM_0(2C+D)$ along
  their special boundaries (the blue edges) to get a
  $\langle 2\rangle$-manifold, a $4$-gon. As in
  Figure~\ref{fig:Dind3}, the thin edges represent $\del_1\bModuli_0$,
  and the thick edges represent $\del_2\bModuli_0$.}
\label{fig:rowplus}
\end{figure}

\begin{figure}
  \centering
  \begin{tikzpicture}[scale=0.8]

    \foreach \i in {1,2,6,7}{\draw (\i,0) --++(0,8);\draw (0,\i) --++(8,0);}

    \foreach\i/\j/\k in {2.5/1/1,4.5/2/-1,3.5/6/1,5.5/7/-1}{
      \draw[dashed] (\i,\j)--++(0,\k);
      \draw[dashed] (\j,\i)--++(\k,0);
      \fill[black] (\i,\j) circle (0.12*0.7/0.8);
      \fill[black] (\j,\i) circle (0.12*0.7/0.8);
      \draw[fill=white] ($(\i,\j)+(0,\k)$) circle (0.12*0.7/0.8);
      \draw[fill=white] ($(\j,\i)+(\k,0)$) circle (0.12*0.7/0.8);
    }
    \node at (0,1.5) {\small $H_i$};
    \node at (1.5,0) {\small $V_i$};
    \node at (1.5,1.5) {\small $O_i$};

    \node at (0,6.5) {\small $H_j$};
    \node at (6.5,0) {\small $V_j$};
    \node at (6.5,6.5) {\small $O_j$};

    \foreach \i/\j/\l in {3.5/1.5/A,5.5/1.5/B,4.5/6.5/C,2.5/6.5/D,1.5/3.5/E,1.5/5.5/F,6.5/4.5/G,6.5/2.5/H}{\node at (\i,\j) {\small $\l$};}

  \end{tikzpicture}
\caption {Two rows and two columns on the grid.}
\label{fig:rowscolumns}
\end{figure}

\begin{figure}
  \centering
  \begin{tikzpicture}[scale=3]

    \begin{scope}[y={(-1.8cm,0.1cm)},x={(0.8cm,0.6cm)},z={(0cm,2.2cm)}]
    \def\a{0.3}
    \def\t{0.4}
    \pgfmathsetmacro\s{1-\a+4*\a*\t}
    \def\fake{-0.15}
    
    \coordinate (c-d-AB) at (0,1);
    \coordinate (CD-AB) at (0,0);
    \coordinate (CD-a-b) at (1,0);
    \coordinate (c-d-a-b) at (1,1);

    \coordinate (AB-c-d) at (0,1,1);
    \coordinate (AB-CD) at (0,0,1);
    \coordinate (a-b-CD) at (1,0,1);
    \coordinate (a-b-c-d) at (1,1,1);

    \coordinate (c-AB-d) at (0,1+\a,0.5);
    \coordinate (a-CD-b) at (1+\a,0,0.5);

    \coordinate(c-a-d-b) at (1+2*\a*\t,1+2*\a*\t,\t);
    \coordinate(a-c-b-d) at (1+2*\a*\t,1+2*\a*\t,1-\t);

    \coordinate (c-a-b-d) at (\s,1+\a,0.5);
    \coordinate (a-c-d-b) at (1+\a,\s,0.5);
    
    \draw (c-d-AB)-- coordinate[midway] (cd-AB) (CD-AB)--(CD-a-b);
    \draw (AB-c-d)-- coordinate[midway] (AB-cd) (AB-CD)--(a-b-CD)--(a-b-c-d)--(AB-c-d);
    \draw (AB-c-d)-- coordinate[midway] (cAB-d) (c-AB-d)-- coordinate[midway] (c-dAB)(c-d-AB) (a-b-CD)--(a-CD-b)--(CD-a-b);
    \draw[dashed] (CD-a-b)--(c-d-a-b)--(c-d-AB);
    \draw[dashed](c-a-d-b)--(c-d-a-b) (a-b-c-d)--(a-c-b-d);
    \draw[dashed] (c-AB-d)--(c-a-b-d) (a-CD-b)--(a-c-d-b);
    \draw[dashed] (c-a-b-d)--(c-a-d-b)--(a-c-d-b)--(a-c-b-d)--(c-a-b-d);
    \draw[thick,green!70!black] (AB-CD)-- coordinate[midway] (ABCD) (CD-AB);

    \foreach \coord in {c-d-AB,CD-a-b,c-d-a-b,AB-c-d,a-b-CD,a-b-c-d,c-AB-d,a-CD-b,c-a-d-b,a-c-b-d,c-a-b-d,a-c-d-b}{      
      \node at (\coord) {$\bullet$};
    }
    \foreach \coord in {AB-CD,CD-AB}{      
      \node[green!70!black] at (\coord) {\large$\bullet$};
    }

    \coordinate (front) at (0,1,0.45);
    \coordinate (right) at (1,0,0.8);
    \coordinate (bottom) at (0.5,0.5,0);
    \coordinate (bottomfake) at (0.5,0.5,\fake);
    \coordinate (top) at (0.5,0.5,1);

    \draw[latex-] (front) --++(-0.9,0,0) coordinate[pos=1] (ABcd) node[pos=0.6,anchor=south east] {\small front};
    \draw[latex-] (right) --++(0,-0.3,0) coordinate[pos=1] (CDab) node[pos=0.7,anchor=south] {\small right};
    \draw[latex-] (top) --++(0,0,0.3) coordinate[pos=1] (ab-cd) node[pos=0.7,anchor=west] {\small top};
    \draw[latex-,dashed] (bottom)--(bottomfake);
    \draw (bottomfake) --++(0,0,-0.2) coordinate[pos=1] (cd-ab) node[pos=0.6,anchor=west] {\small bottom};

    
  \end{scope}

    \foreach \coord/\pos/\x/\y/\z/\w in {
      a-b-c-d/south/A/B/C/D,
      a-c-b-d/south/A/C/B/D,
      c-a-b-d/east/C/A/B/D,
      a-c-d-b/west/A/C/D/B,
      c-a-d-b/north/C/A/D/B,
      c-d-a-b/north/C/D/A/B}{
      \node[anchor=\pos,inner sep=0,outer sep=2pt] at (\coord){\begin{tikzpicture}[scale=0.4]
          \draw (0,0)to[out=10,in=-10]++(0,1)to[out=10,in=-10]++(0,1)to[out=10,in=-10]++(0,1)to[out=10,in=-10]++(0,1);
          \draw (0,0)to[out=170,in=-170]++(0,1)to[out=170,in=-170]++(0,1)to[out=170,in=-170]++(0,1)to[out=170,in=-170]++(0,1);
          \node[anchor=center] at (0,3.5) {\tiny $\x$};
          \node[anchor=center] at (0,2.5) {\tiny $\y$};
          \node[anchor=center] at (0,1.5) {\tiny $\z$};
          \node[anchor=center] at (0,0.5) {\tiny $\w$};
        \end{tikzpicture}};
      }

    \foreach \coord/\pos/\x/\y/\z in {
      a-b-CD/south/A/B/CD,
      c-d-AB/north/C/D/AB}{
      \node[anchor=\pos,inner sep=0,outer sep=2pt] at (\coord){\begin{tikzpicture}[scale=0.4]
          \draw (0,2)to[out=10,in=-10]++(0,1)to[out=10,in=-10]++(0,1);
          \draw (0,2)to[out=170,in=-170]++(0,1)to[out=170,in=-170]++(0,1);
          \node[anchor=center] at (0,3.5) {\tiny $\x$};
          \node[anchor=center] at (0,2.5) {\tiny $\y$};
          \draw[fill=black!50] (0,0) to[out=70,in=-70] coordinate[midway] (r) ++(0,2) to[out=-110,in=110] coordinate[midway] (l) ++(0,-2);
          \draw (l) arc (0:360:0.5);
          \node[anchor=center] at ($(l)+(-0.5,0)$) {\tiny $\mathit{\z}$};
        \end{tikzpicture}};
      }

    \foreach \coord/\pos/\x/\y/\z in {
      AB-c-d/south/AB/C/D,
      CD-a-b/north/CD/A/B}{
      \node[anchor=\pos,inner sep=0,outer sep=2pt] at (\coord){\begin{tikzpicture}[scale=0.4]
          \draw (0,0)to[out=10,in=-10]++(0,1)to[out=10,in=-10]++(0,1);
          \draw (0,0)to[out=170,in=-170]++(0,1)to[out=170,in=-170]++(0,1);
          \node[anchor=center] at (0,1.5) {\tiny $\y$};
          \node[anchor=center] at (0,0.5) {\tiny $\z$};
          \draw[fill=black!50] (0,2) to[out=70,in=-70] coordinate[midway] (r) ++(0,2) to[out=-110,in=110] coordinate[midway] (l) ++(0,-2);
          \draw (l) arc (0:360:0.5);
          \node[anchor=center] at ($(l)+(-0.5,0)$) {\tiny $\mathit{\x}$};
        \end{tikzpicture}};
      }

    \foreach \coord/\pos/\x/\y/\z in {
      a-CD-b/west/A/CD/B,
      c-AB-d/east/C/AB/D}{
      \node[anchor=\pos,inner sep=0,outer sep=2pt] at (\coord){\begin{tikzpicture}[scale=0.4]
          \draw (0,0)to[out=10,in=-10]++(0,1)++(0,2)to[out=10,in=-10]++(0,1);
          \draw (0,0)to[out=170,in=-170]++(0,1)++(0,2)to[out=170,in=-170]++(0,1);
          \node[anchor=center] at (0,3.5) {\tiny $\x$};
          \node[anchor=center] at (0,0.5) {\tiny $\z$};
          \draw[fill=black!50] (0,1) to[out=70,in=-70] coordinate[midway] (r) ++(0,2) to[out=-110,in=110] coordinate[midway] (l) ++(0,-2);
          \draw (l) arc (0:360:0.5);
          \node[anchor=center] at ($(l)+(-0.5,0)$) {\tiny $\mathit{\y}$};
        \end{tikzpicture}};
      }

    \foreach \coord/\pos/\x/\y in {
      AB-CD/south/AB/CD,
      CD-AB/north/CD/AB}{
      \node[anchor=\pos,inner sep=0,outer sep=2pt] at (\coord){\begin{tikzpicture}[scale=0.4]
          \draw[fill=black!50] (0,0) to[out=70,in=-70] coordinate[midway] (r) ++(0,2) to[out=-110,in=110] coordinate[midway] (l) ++(0,-2);
          \draw (l) arc (0:360:0.5);
          \node[anchor=center] at ($(l)+(-0.5,0)$) {\tiny $\mathit{\y}$};
          \draw[fill=black!50] (0,2) to[out=70,in=-70] coordinate[midway] (r) ++(0,2) to[out=-110,in=110] coordinate[midway] (l) ++(0,-2);
          \draw (l) arc (0:360:0.5);
          \node[anchor=center] at ($(l)+(-0.5,0)$) {\tiny $\mathit{\x}$};
        \end{tikzpicture}};
      }

      \node[anchor=north,inner sep=0,outer sep=2pt] at (cd-AB){\begin{tikzpicture}[scale=0.4]
          \draw (0,2)to[out=10,in=-10]++(0,2);
          \draw (0,2)to[out=170,in=-170]++(0,2);
          \node[anchor=center] at (0,3) {\tiny $\mathit{CD}$};
          \draw[fill=black!50] (0,0) to[out=70,in=-70] coordinate[midway] (r) ++(0,2) to[out=-110,in=110] coordinate[midway] (l) ++(0,-2);
          \draw (l) arc (0:360:0.5);
          \node[anchor=center] at ($(l)+(-0.5,0)$) {\tiny $\mathit{AB}$};
        \end{tikzpicture}};

      \node[anchor=south,inner sep=0,outer sep=2pt] at (AB-cd){\begin{tikzpicture}[scale=0.4]
          \draw (0,0)to[out=10,in=-10]++(0,2);
          \draw (0,0)to[out=170,in=-170]++(0,2);
          \node[anchor=center] at (0,1) {\tiny $\mathit{CD}$};
          \draw[fill=black!50] (0,2) to[out=70,in=-70] coordinate[midway] (r) ++(0,2) to[out=-110,in=110] coordinate[midway] (l) ++(0,-2);
          \draw (l) arc (0:360:0.5);
          \node[anchor=center] at ($(l)+(-0.5,0)$) {\tiny $\mathit{AB}$};
        \end{tikzpicture}};

      \node[anchor=south east,inner sep=0,outer sep=2pt] at (cAB-d){\begin{tikzpicture}[scale=0.4]
          \draw (0,0)to[out=10,in=-10]++(0,1)to[out=10,in=-10]  coordinate[midway] (r)++(0,1);
          \draw (0,0)to[out=170,in=-170]++(0,1)to[out=170,in=-170] coordinate[midway] (l)++(0,1);
          \node[anchor=center] at (0,0.5) {\tiny $D$};
          \node[anchor=center] at (0,1.5) {\tiny $C$};
          \draw (l) arc (0:360:0.5);
          \node[anchor=center] at ($(l)+(-0.5,0)$) {\tiny $\mathit{AB}$};
        \end{tikzpicture}};

      \node[anchor=north east,inner sep=0,outer sep=2pt] at (c-dAB){\begin{tikzpicture}[scale=0.4]
          \draw (0,0)to[out=10,in=-10]coordinate[midway] (r)++(0,1)to[out=10,in=-10]  ++(0,1);
          \draw (0,0)to[out=170,in=-170] coordinate[midway] (l)++(0,1)to[out=170,in=-170] ++(0,1);
          \node[anchor=center] at (0,0.5) {\tiny $D$};
          \node[anchor=center] at (0,1.5) {\tiny $C$};
          \draw (l) arc (0:360:0.5);
          \node[anchor=center] at ($(l)+(-0.5,0)$) {\tiny $\mathit{AB}$};
        \end{tikzpicture}};
      
      \node[anchor=west,inner sep=0,outer sep=2pt] at (ABCD){\begin{tikzpicture}[scale=0.4]
          \draw[fill=black!50] (0,0) to[out=70,in=-70] ++(0,3) to[out=-110,in=110] coordinate[pos=0.3] (l1) coordinate[pos=0.7] (l2) ++(0,-3);
          \draw (l1) arc (0:360:0.5);
          \draw (l2) arc (0:360:0.5);
          \node[anchor=center] at ($(l1)+(-0.5,0)$) {\tiny $\mathit{AB}$};
          \node[anchor=center] at ($(l2)+(-0.5,0)$) {\tiny $\mathit{CD}$};
        \end{tikzpicture}};

    \foreach \coord/\pos/\x/\y in {
      ABcd/east/AB/CD,
      CDab/west/CD/AB}{
      \node[anchor=\pos,inner sep=0,outer sep=2pt] at (\coord){\begin{tikzpicture}[scale=0.4]
          \draw (0,0)to[out=10,in=-10]++(0,2);
          \draw (0,0)to[out=170,in=-170] coordinate[midway] (l)++(0,2);
          \node[anchor=center] at (0,1) {\tiny $\mathit{\y}$};
          \draw (l) arc (0:360:0.5);
          \node[anchor=center] at ($(l)+(-0.5,0)$) {\tiny $\mathit{\x}$};
        \end{tikzpicture}};
      }

    \foreach \coord/\pos/\x/\y in {
      ab-cd/south/AB/CD,
      cd-ab/north/CD/AB}{
      \node[anchor=\pos,inner sep=0,outer sep=2pt] at (\coord){\begin{tikzpicture}[scale=0.4]
          \draw (0,0)to[out=10,in=-10]++(0,2)to[out=10,in=-10]++(0,2);
          \draw (0,0)to[out=170,in=-170]++(0,2)to[out=170,in=-170]++(0,2);
          \node[anchor=center] at (0,3) {\tiny $\mathit{\x}$};
          \node[anchor=center] at (0,1) {\tiny $\mathit{\y}$};
        \end{tikzpicture}};
      }

  \end{tikzpicture}
\caption {For the picture in Figure~\ref{fig:rowscolumns}, we show the moduli space for $H_i+H_j=A+B+C+D$. This is a polyhedron with $9$ facets; two of these (the front and the right facet) form the special boundary. We show the configurations corresponding to each vertex, and to some of the facets (the top, bottom, right, and front one). We also show the configurations for the five edges along the front facet. (In particular, note that the green edge corresponds to two disk degenerations.) The configurations that correspond to the remaining edges and facets can be easily deduced.}
\label{fig:rowrow}
\end{figure}

\begin{figure}
  \centering
  \begin{tikzpicture}[scale=3]

    \begin{scope}[y={(-1.8cm,0.1cm)},x={(0.8cm,0.6cm)},z={(0cm,2.2cm)}]
      \def\curve{0.2}

    \coordinate (c-d-AB) at (0,1);
    \coordinate (CD-AB) at (0,0);
    \coordinate (CD-a-b) at (1,0);

    \coordinate (AB-c-d) at (0,1,1);
    \coordinate (AB-CD) at (0,0,1);
    \coordinate (a-b-CD) at (1,0,1);

    \coordinate (CD=AB) at (0,0,0.3);
    \coordinate (AB=CD) at (0,0,0.7);

    \coordinate (center) at (0,0,0.5);
    
    \draw (c-d-AB)-- (CD-AB)--(CD-a-b);
    \draw (AB-c-d)-- (AB-CD)--(a-b-CD);

    \draw[thick,green!70!black] (AB-CD)-- coordinate[midway] (ABCD) (AB=CD) (CD=AB)-- coordinate[midway] (CDAB) (CD-AB);
    \fill[green!30] (AB=CD) ..controls ($(AB=CD)+(\curve,-\curve,0)$) and ($(CD=AB)+(\curve,-\curve,0)$).. (CD=AB) ..controls ($(CD=AB)+(-\curve,\curve,0)$) and ($(AB=CD)+(-\curve,\curve,0)$).. (AB=CD);
    \draw[thick,green!70!black] (AB=CD) ..controls ($(AB=CD)+(\curve,-\curve,0)$) and ($(CD=AB)+(\curve,-\curve,0)$).. coordinate[midway] (AB--CD) (CD=AB);
    \draw[thick,green!70!black] (AB=CD) ..controls ($(AB=CD)+(-\curve,\curve,0)$) and ($(CD=AB)+(-\curve,\curve,0)$).. coordinate[midway] (CD--AB) (CD=AB);

    \foreach \coord in {c-d-AB,CD-a-b,AB-c-d,a-b-CD}{
      \node at (\coord) {$\bullet$};
    }
    \foreach \coord in {AB-CD,CD-AB,AB=CD,CD=AB}{      
      \node[green!70!black] at (\coord) {\large$\bullet$};
    }

    \coordinate (right) at (1,0,0.5); 
    \coordinate (front)  at (0,1,0.5);
    \coordinate (top)  at (0.8,0.8,1);

    
  \end{scope}

  \node[sloped,yslant=0.7,xscale=0.9] at (right) {\small right};
  \node[sloped,yslant=-0.05] at (front)  {\small front};
  \node[sloped,rotate=-55,yslant=1.3,yscale=1.2,xscale=0.7] at (top)  {\small top};

  \foreach \coord/\pos/\x/\y in {
    AB-CD/south/AB/CD,
    CD-AB/north/CD/AB}{
    \node[anchor=\pos,inner sep=0,outer sep=2pt] at (\coord){\begin{tikzpicture}[scale=0.4]
        \draw[fill=black!50] (0,0) to[out=70,in=-70] coordinate[midway] (r) ++(0,2) to[out=-110,in=110] coordinate[midway] (l) ++(0,-2);
        \draw (l) arc (0:360:0.5);
        \node[anchor=center] at ($(l)+(-0.5,0)$) {\tiny $\mathit{\y}$};
        \draw[fill=black!50] (0,2) to[out=70,in=-70] coordinate[midway] (r) ++(0,2) to[out=-110,in=110] coordinate[midway] (l) ++(0,-2);
        \draw (l) arc (0:360:0.5);
        \node[anchor=center] at ($(l)+(-0.5,0)$) {\tiny $\mathit{\x}$};
      \end{tikzpicture}};
  }

  \foreach \coord/\pos/\x/\y/\sign in {
    AB=CD/east/AB/CD/1,
    CD=AB/east/CD/AB/-1}{
    \draw[latex-] ($(\coord)+(-0.04cm,0.01*\sign cm)$) -- ++(-0.4cm,0.1*\sign cm) 
    node[anchor=\pos,inner sep=0,outer sep=2pt] {\begin{tikzpicture}[scale=0.4]
        \draw[-,fill=black!50] (0,0) to[out=70,in=-70] ++(0,2) to[out=-110,in=110] coordinate[midway] (l) ++(0,-2);
        \draw[-,fill=black!50] (l) arc (0:360:0.5);
        \draw[-] ($(l)+(-0.5,0)+(120:1)$) circle (0.5);
        \draw[-] ($(l)+(-0.5,0)+(-120:1)$) circle (0.5);
        \node[anchor=center] at ($(l)+(-0.5,0)+(120:1)$) {\tiny $\mathit{\x}$};
        \node[anchor=center] at ($(l)+(-0.5,0)+(-120:1)$) {\tiny $\mathit{\y}$};
      \end{tikzpicture}};
  }

      
  \foreach \coord/\pos/\x/\y in {
    ABCD/west/AB/CD,
    CDAB/west/CD/AB}{
    \node[anchor=\pos,inner sep=0,outer sep=2pt] at (\coord){\begin{tikzpicture}[scale=0.4]
        \draw[fill=black!50] (0,0) to[out=70,in=-70] ++(0,3) to[out=-110,in=110] coordinate[pos=0.3] (l1) coordinate[pos=0.7] (l2) ++(0,-3);
        \draw (l1) arc (0:360:0.5);
        \draw (l2) arc (0:360:0.5);
        \node[anchor=center] at ($(l1)+(-0.5,0)$) {\tiny $\mathit{\x}$};
        \node[anchor=center] at ($(l2)+(-0.5,0)$) {\tiny $\mathit{\y}$};
      \end{tikzpicture}};
  }

  \foreach \coord/\pos/\x/\y in {
    AB--CD/west/AB/CD,
    CD--AB/east/CD/AB}{
    \node[anchor=\pos,inner sep=0,outer sep=2pt] at (\coord) {\begin{tikzpicture}[scale=0.4]
        \draw[fill=black!50] (0,0) to[out=70,in=-70] ++(0,2) to[out=-110,in=110] coordinate[midway] (l) ++(0,-2);
        \draw (l) arc (0:360:0.5);
        \node[anchor=center] at ($(l)+(-0.5,0)$) {\tiny $\mathit{\x}$};
        \draw ($(l)+(-1.5,0)$) circle (0.5);
        \node[anchor=center] at ($(l)+(-1.5,0)$) {\tiny $\mathit{\y}$};
      \end{tikzpicture}};
  }


    \node[anchor=center,inner sep=0,outer sep=2pt] at (center) {\begin{tikzpicture}[scale=0.4]
        \draw[fill=black!50] (0,0) to[out=70,in=-70] ++(0,3) to[out=-110,in=110] coordinate[midway] (l) ++(0,-3);
        \draw[fill=white] (l) arc (0:360:1);
        \node[anchor=center] at ($(l)+(-1,0)$) {\tiny $\mathit{ABCD}$};
      \end{tikzpicture}};

  \end{tikzpicture}
\caption {If we consider the actual Gromov compactification of the moduli space, then the green edge from Figure~\ref{fig:rowrow} gets replaced by more complicated spaces of bubble trees, as shown here. We have shown the labels of the green vertices, the green edges, and the green facet. (The other labels are same as in Figure~\ref{fig:rowrow}.)}
\label{fig:gromov}
\end{figure}

\begin{figure}
\centering
  \begin{tikzpicture}[scale=3]

    \begin{scope}[y={(-1.8cm,0.1cm)},x={(0.8cm,0.6cm)},z={(0cm,2.2cm)}]
    \def\a{0.3}
    \def\t{0.4}
    \pgfmathsetmacro\s{1-\a+4*\a*\t}
    
    \coordinate (c-d-AB) at (0,1);
    \coordinate (CD-AB) at (0,0);
    \coordinate (CD-a-b) at (1,0);
    \coordinate (c-d-a-b) at (1,1);

    \coordinate (AB-c-d) at (0,1,1);
    \coordinate (AB-CD) at (0,0,1);
    \coordinate (a-b-CD) at (1,0,1);
    \coordinate (a-b-c-d) at (1,1,1);

    \coordinate (c-AB-d) at (0,1+\a,0.5);
    \coordinate (a-CD-b) at (1+\a,0,0.5);

    \coordinate(c-a-d-b) at (1+2*\a*\t,1+2*\a*\t,\t);
    \coordinate(a-c-b-d) at (1+2*\a*\t,1+2*\a*\t,1-\t);

    \coordinate (c-a-b-d) at (\s,1+\a,0.5);
    \coordinate (a-c-d-b) at (1+\a,\s,0.5);
    
    \draw (a-b-CD)--(a-b-c-d)--(AB-c-d);
    \draw (AB-c-d)-- coordinate[midway] (cAB-d) (c-AB-d)-- coordinate[midway] (c-dAB)(c-d-AB) (a-b-CD)--(a-CD-b)--(CD-a-b);
    \draw[dashed] (CD-a-b)--(c-d-a-b)--(c-d-AB);
    \draw[dashed](c-a-d-b)--(c-d-a-b) (a-b-c-d)--(a-c-b-d);
    \draw[dashed] (c-AB-d)--(c-a-b-d) (a-CD-b)--(a-c-d-b);
    \draw[dashed] (c-a-b-d)--(c-a-d-b)--(a-c-d-b)--(a-c-b-d)--(c-a-b-d);

    \draw[black!50,thin] (c-d-AB) ..controls (CD-AB) .. (CD-a-b);
    
    \foreach \coord in {c-d-AB,CD-a-b,c-d-a-b,AB-c-d,a-b-CD,a-b-c-d,c-AB-d,a-CD-b,c-a-d-b,a-c-b-d,c-a-b-d,a-c-d-b}{      
      \node at (\coord) {$\bullet$};
    }

    \coordinate (front) at (0,1,0.45);
    \coordinate (right) at (1,0,0.8);

    \draw[latex-] (front) --++(-0.9,0,0) coordinate[pos=1] (ABcd) node[pos=0.6,anchor=south east] {\small front};
    \draw[latex-] (right) --++(0,-0.3,0) coordinate[pos=1] (CDab) node[pos=0.7,anchor=south] {\small right};

    \coordinate (c) at (0.1,0.1,0.5);
    \draw (AB-c-d) ..controls ($(AB-c-d)+(0,-0.5,0)$) and ($(c)+(-0.2,0.2,0.5)$) .. (c);
    \draw (a-b-CD) ..controls ($(a-b-CD)+(-0.5,0,0)$) and ($(c)+(0.2,-0.2,0.5)$) .. (c);
    \draw (c-d-AB) ..controls ($(c-d-AB)+(0,-0.5,0)$) and ($(c)+(-0.2,0.2,-0.5)$) .. (c);
    \draw (CD-a-b) ..controls ($(CD-a-b)+(-0.5,0,0)$) and ($(c)+(0.2,-0.2,-0.5)$) .. (c);
    \coordinate (cc) at (0.7,0.7,0.5);
    \draw[thick,green!70!black] (c)--(cc) node[pos=0,green!70!black] {\large$\bullet$} node[pos=1,green!70!black] {\large$\bullet$} coordinate[midway] (chalf);
    
  \end{scope}

  \node[anchor=east,inner sep=0,outer sep=2pt] at (ABcd){\begin{tikzpicture}[scale=0.4]
      \draw (0,0)to[out=10,in=-10]coordinate[midway] (r)++(0,2);
      \draw (0,0)to[out=170,in=-170] coordinate[midway] (l)++(0,2);
      \node[anchor=center] at (0,1) {\tiny $\mathit{CD}$};
      \draw (l) arc (0:360:0.5);
      \node[anchor=center] at ($(l)+(-0.5,0)$) {\tiny $\mathit{AB}$};
    \end{tikzpicture}};

  \node[anchor=west,inner sep=0,outer sep=2pt] at (CDab){\begin{tikzpicture}[scale=0.4]
      \draw (0,0)to[out=10,in=-10]coordinate[midway] (r)++(0,2);
      \draw (0,0)to[out=170,in=-170] coordinate[midway] (l)++(0,2);
      \node[anchor=center] at (0,1) {\tiny $\mathit{AB}$};
      \draw (r) arc (-180:180:0.5);
      \node[anchor=center] at ($(r)+(0.5,0)$) {\tiny $\mathit{CD}$};
    \end{tikzpicture}};

  \node[anchor=east,inner sep=0,outer sep=2pt] at (cc){\begin{tikzpicture}[scale=0.4]
      \draw[fill=black!50] (0,0) to[out=70,in=-70] coordinate[midway] (r)++(0,3) to[out=-110,in=110] coordinate[midway] (l)  ++(0,-3);
      \draw (l) arc (0:360:0.5);
      \draw (r) arc (-180:180:0.5);
      \node[anchor=center] at ($(l)+(-0.5,0)$) {\tiny $\mathit{AB}$};
      \node[anchor=center] at ($(r)+(0.5,0)$) {\tiny $\mathit{CD}$};
    \end{tikzpicture}};
  
  \node[anchor=north,inner sep=0,outer sep=2pt] at (chalf){\begin{tikzpicture}[scale=0.4]
      \draw[fill=black!50] (0,0) to[out=70,in=-70] coordinate[pos=0.3] (r)++(0,3) to[out=-110,in=110] coordinate[pos=0.3] (l)  ++(0,-3);
      \draw (l) arc (0:360:0.5);
      \draw (r) arc (-180:180:0.5);
      \node[anchor=center] at ($(l)+(-0.5,0)$) {\tiny $\mathit{AB}$};
      \node[anchor=center] at ($(r)+(0.5,0)$) {\tiny $\mathit{CD}$};
    \end{tikzpicture}};

    \node[anchor=west,inner sep=0,outer sep=2pt] at (c){\begin{tikzpicture}[scale=0.4]
      \draw[fill=black!50] (0,0) to[out=70,in=-70] coordinate[midway] (r)++(0,2) to[out=-110,in=110]  ++(0,-2);
      \draw[fill=black!50] (0,2) to[out=70,in=-70] ++(0,2) to[out=-110,in=110] coordinate[midway] (l)  ++(0,-2);
      \draw[fill=white] (l) arc (0:360:0.5);
      \draw[fill=white] (r) arc (-180:180:0.5);
      \node[anchor=center] at ($(l)+(-0.5,0)$) {\tiny $\mathit{AB}$};
      \node[anchor=center] at ($(r)+(0.5,0)$) {\tiny $\mathit{CD}$};

      \node[anchor=center] at (1.5,2) {\tiny $=$};
      
      \draw[fill=black!50] (3,2) to[out=70,in=-70] coordinate[midway] (r)++(0,2) to[out=-110,in=110]  ++(0,-2);
      \draw[fill=black!50] (3,0) to[out=70,in=-70] ++(0,2) to[out=-110,in=110] coordinate[midway] (l)  ++(0,-2);
      \draw[fill=white] (l) arc (0:360:0.5);
      \draw[fill=white] (r) arc (-180:180:0.5);
      \node[anchor=center] at ($(l)+(-0.5,0)$) {\tiny $\mathit{AB}$};
      \node[anchor=center] at ($(r)+(0.5,0)$) {\tiny $\mathit{CD}$};

    \end{tikzpicture}};

  \end{tikzpicture}
\caption {The moduli space for $H_i+V_i=A+B+C+D$, where $H_i$ and
  $V_i$ are as in Figure~\ref{fig:row}. This is obtained from the
  polyhedron in Figure~\ref{fig:rowrow} by folding the green edge in
  half. (This can be visualized by pushing the midpoint of the green
  edge of Figure~\ref{fig:rowrow} towards the interior until it folds
  into half.) The front, bottom, top and right facets from
  Figure~\ref{fig:rowrow} now meet at a single point (the right green
  dot); moreover, the two edges on the top facet and the two edges on
  the bottom facet enter this vertex from exactly opposite
  directions. For simplicity, we only show the configurations for the
  green edge, its two endpoints, and the two facets that form
  the special boundary (the front and the right facet, which meet
  along the green edge from exactly opposite directions). The other
  labels are just as in Figure~\ref{fig:rowrow}, except that the $CD$
  disk degenerations are now to the right of the strips. Note that the
  moduli space shown here is not a convex polyhedron, but rather a
  stratified space, where the local picture near the left green dot is
  the Whitney umbrella from Figure~\ref{fig:Whitney}.}
\label{fig:rowcolumn}
\end{figure}

\begin{figure}
\centering
  \begin{tikzpicture}[scale=3]

    \begin{scope}[y={(-1.8cm,0.1cm)},x={(0.8cm,0.6cm)},z={(0cm,2.2cm)}]
    \def\s{0.3}
    \def\fakebottom{-0.15}
    \def\fakeright{0.55}
    \def\fakeleft{0.15}
    
    \coordinate (a-b-AB) at (0,1);
    \coordinate (AB-AB) at (0,0);
    \coordinate (AB-a-b) at (1,0);
    \coordinate (a-b-a-b) at (1,1);
    \coordinate (a-AB-b) at (1,1,1);

    \coordinate (center) at (0.5-\s,0.5-\s,0.5);
    
    \draw (a-b-AB)--(AB-AB)--(AB-a-b);
    \draw (a-b-AB)--(a-AB-b)--(AB-a-b);
    \draw[dashed] (AB-a-b)--(a-b-a-b)--(a-b-AB);
    \draw[dashed] (a-b-a-b)--(a-AB-b);

    \draw[thick,green!70!black] (AB-AB)-- coordinate[midway] (ABAB) node[pos=0, green!70!black] {\large $\bullet$} node[pos=1, green!70!black] {\large $\bullet$} (center);

    \foreach \coord in {a-b-AB,AB-a-b,a-b-a-b,a-AB-b}{      
      \node at (\coord) {$\bullet$};
    }

    \coordinate (front) at (0.2,0.9,0.2);
    \coordinate (rightback) at (1,0.9,0.5);
    \coordinate (rightfake) at (1+\fakeright,0.9,0.5);
    \coordinate (leftback) at (0.9,1,0.5);
    \coordinate (leftfake) at (0.9,1+\fakeleft,0.5);
    \coordinate (bottom) at (0.5,0.5,0);
    \coordinate (bottomfake) at (0.5,0.5,\fakebottom);

    \draw[latex-] (front) --++(-0.7,0,0.35) coordinate[pos=1] (ABab) node[pos=0.4,anchor=north east] {\small front};

    \draw[latex-,dashed] (rightback) --(rightfake);
    \draw (rightfake) --coordinate[pos=0.8] (aab-b) node[pos=0.1,anchor=west] {\small right back} ++(0.5,0,0);

    \draw[latex-,dashed] (leftback) --(leftfake);
    \draw (leftfake) --coordinate[pos=1] (a-abb) node[pos=0.1,anchor=north east] {\small left back} ++(0,0.4,0);
    
    \draw[latex-,dashed] (bottom)--(bottomfake);
    \draw (bottomfake) --++(0,0,-0.1) coordinate[pos=1] (ab-ab) node[pos=0.6,anchor=east] {\small bottom};

  \end{scope}

  \node[anchor=north,inner sep=0,outer sep=2pt] at (a-b-a-b){\begin{tikzpicture}[scale=0.4]
      \draw (0,0)to[out=10,in=-10]++(0,1)to[out=10,in=-10]++(0,1)to[out=10,in=-10]++(0,1)to[out=10,in=-10]++(0,1);
      \draw (0,0)to[out=170,in=-170]++(0,1)to[out=170,in=-170]++(0,1)to[out=170,in=-170]++(0,1)to[out=170,in=-170]++(0,1);
      \node[anchor=center] at (0,3.5) {\tiny $A$};
      \node[anchor=center] at (0,2.5) {\tiny $B$};
      \node[anchor=center] at (0,1.5) {\tiny $A$};
      \node[anchor=center] at (0,0.5) {\tiny $B$};
    \end{tikzpicture}};

  \node[anchor=north,inner sep=0,outer sep=2pt] at (a-b-AB){\begin{tikzpicture}[scale=0.4]
      \draw (0,2)to[out=10,in=-10]++(0,1)to[out=10,in=-10]++(0,1);
      \draw (0,2)to[out=170,in=-170]++(0,1)to[out=170,in=-170]++(0,1);
      \node[anchor=center] at (0,3.5) {\tiny $A$};
      \node[anchor=center] at (0,2.5) {\tiny $B$};
      \draw[fill=black!50] (0,0) to[out=70,in=-70] coordinate[midway] (r) ++(0,2) to[out=-110,in=110] coordinate[midway] (l) ++(0,-2);
      \draw (l) arc (0:360:0.5);
      \node[anchor=center] at ($(l)+(-0.5,0)$) {\tiny $\mathit{AB}$};
    \end{tikzpicture}};

  \node[anchor=north,inner sep=0,outer sep=2pt] at (AB-a-b){\begin{tikzpicture}[scale=0.4]
      \draw (0,0)to[out=10,in=-10]++(0,1)to[out=10,in=-10]++(0,1);
      \draw (0,0)to[out=170,in=-170]++(0,1)to[out=170,in=-170]++(0,1);
      \node[anchor=center] at (0,1.5) {\tiny $A$};
      \node[anchor=center] at (0,0.5) {\tiny $B$};
      \draw[fill=black!50] (0,2) to[out=70,in=-70] coordinate[midway] (r) ++(0,2) to[out=-110,in=110] coordinate[midway] (l) ++(0,-2);
      \draw (l) arc (0:360:0.5);
      \node[anchor=center] at ($(l)+(-0.5,0)$) {\tiny $\mathit{AB}$};
    \end{tikzpicture}};

  \node[anchor=south,inner sep=0,outer sep=2pt] at (a-AB-b){\begin{tikzpicture}[scale=0.4]
      \draw (0,0)to[out=10,in=-10]++(0,1)++(0,2)to[out=10,in=-10]++(0,1);
      \draw (0,0)to[out=170,in=-170]++(0,1)++(0,2)to[out=170,in=-170]++(0,1);
      \node[anchor=center] at (0,3.5) {\tiny $A$};
      \node[anchor=center] at (0,0.5) {\tiny $B$};
      \draw[fill=black!50] (0,1) to[out=70,in=-70] coordinate[midway] (r) ++(0,2) to[out=-110,in=110] coordinate[midway] (l) ++(0,-2);
      \draw (l) arc (0:360:0.5);
      \node[anchor=center] at ($(l)+(-0.5,0)$) {\tiny $\mathit{AB}$};
    \end{tikzpicture}};

  \node[anchor=north,inner sep=0,outer sep=2pt] at (AB-AB){\begin{tikzpicture}[scale=0.4]
      \draw[fill=black!50] (0,0) to[out=70,in=-70] coordinate[midway] (r) ++(0,2) to[out=-110,in=110] coordinate[midway] (l) ++(0,-2);
      \draw (l) arc (0:360:0.5);
      \node[anchor=center] at ($(l)+(-0.5,0)$) {\tiny $\mathit{AB}$};
      \draw[fill=black!50] (0,2) to[out=70,in=-70] coordinate[midway] (r) ++(0,2) to[out=-110,in=110] coordinate[midway] (l) ++(0,-2);
      \draw (l) arc (0:360:0.5);
      \node[anchor=center] at ($(l)+(-0.5,0)$) {\tiny $\mathit{AB}$};
    \end{tikzpicture}};

  \node[anchor=east,inner sep=0,outer sep=2pt] at (center){\begin{tikzpicture}[scale=0.4]
      \draw[fill=black!50] (0,0) to[out=70,in=-70] ++(0,2) to[out=-110,in=110] coordinate[midway] (l)  ++(0,-2);
      \draw (l) to[out=90,in=90] ($(l)+(-1.5,0.5)$) to[out=-90,in=180] (l);
      \draw (l) to[out=-90,in=-90] ($(l)+(-1.5,-0.5)$) to[out=90,in=180] (l);
      \node[anchor=center] at ($(l)+(-0.9,0.4)$) {\tiny $\mathit{AB}$};
      \node[anchor=center] at ($(l)+(-0.9,-0.4)$) {\tiny $\mathit{AB}$};
    \end{tikzpicture}};

  \node[anchor=east,inner sep=0,outer sep=2pt] at (ABAB){\begin{tikzpicture}[scale=0.4]
      \draw[fill=black!50] (0,0) to[out=70,in=-70] ++(0,3) to[out=-110,in=110] coordinate[pos=0.3] (l1) coordinate[pos=0.7] (l2) ++(0,-3);
      \draw[fill=white] (l1) arc (0:360:0.5);
      \draw[fill=white] (l2) arc (0:360:0.5);
      \node[anchor=center] at ($(l1)+(-0.5,0)$) {\tiny $\mathit{AB}$};
      \node[anchor=center] at ($(l2)+(-0.5,0)$) {\tiny $\mathit{AB}$};
    \end{tikzpicture}};
  
  \node[anchor=east,inner sep=0,outer sep=2pt] at (ABab){\begin{tikzpicture}[scale=0.4]
      \draw (0,0)to[out=10,in=-10]++(0,2);
      \draw (0,0)to[out=170,in=-170] coordinate[midway] (l)++(0,2);
      \node[anchor=center] at (0,1) {\tiny $\mathit{AB}$};
      \draw (l) arc (0:360:0.5);
      \node[anchor=center] at ($(l)+(-0.5,0)$) {\tiny $\mathit{AB}$};
    \end{tikzpicture}};

  \node[anchor=north,inner sep=0,outer sep=2pt] at (ab-ab){\begin{tikzpicture}[scale=0.4]
      \draw (0,0)to[out=10,in=-10]++(0,2)to[out=10,in=-10]++(0,2);
      \draw (0,0)to[out=170,in=-170]++(0,2)to[out=170,in=-170]++(0,2);
      \node[anchor=center] at (0,3) {\tiny $\mathit{AB}$};
      \node[anchor=center] at (0,1) {\tiny $\mathit{AB}$};
    \end{tikzpicture}};

  \node[anchor=east,inner sep=0,outer sep=2pt] at (a-abb){\begin{tikzpicture}[scale=0.4]
      \draw (0,0)to[out=10,in=-10]++(0,3)to[out=10,in=-10]++(0,1);
      \draw (0,0)to[out=170,in=-170]++(0,3)to[out=170,in=-170]++(0,1);
      \node[anchor=center] at (0,3.5) {\tiny $\mathit{A}$};
      \node[anchor=center] at (0,1.5) {\tiny $\mathit{ABB}$};
    \end{tikzpicture}};

    \node[anchor=south west,inner sep=0,outer sep=2pt] at (aab-b){\begin{tikzpicture}[scale=0.4]
      \draw (0,0)to[out=10,in=-10]++(0,1)to[out=10,in=-10]++(0,3);
      \draw (0,0)to[out=170,in=-170]++(0,1)to[out=170,in=-170]++(0,3);
      \node[anchor=center] at (0,0.5) {\tiny $\mathit{B}$};
      \node[anchor=center] at (0,2.5) {\tiny $\mathit{AAB}$};
    \end{tikzpicture}};

  \end{tikzpicture}
\caption {The moduli space for $2H_i=2A+2B$, where $H_i$ is as in
  Figure~\ref{fig:row}. This is obtained from a convex pyramid (with a
  quadrilateral base) by smoothing along the top half of the front
  edge, and pulling the midpoint of that edge (the top green dot)
  outwards, so that the local model of the top green dot inside the space is like
  $Z(0,2,0;(2))\subset \bZ(2,0,0)$ from Figure~\ref{fig:Whitney}. We labeled the
  configurations corresponding to each vertex and facet, as well as
  that for the green edge. The labels on the other edges can be easily
  deduced.}
\label{fig:2row}
\end{figure}

As alluded to in the beginning of Section~\ref{sec:new}, the space
$\bModuli_{\vN, \vlambda}(D)$ will be a model for the compactified
moduli space of pseudo-holomorphic strips in $\Sym^n(T^2)$ relative to
$\Ta=\alpha_1\times\dots\times\alpha_n,
\Tb=\beta_1\times\dots\times\beta_n$, modulo translation by $\R$, such
that:
\begin{itemize}
\item the strips have domain $D$;
\item each strip is equipped with $|\vN|= N_2 + \dots + N_{n}$ marked
  points on the $\alpha$ and $\beta$-boundaries, combined into groups
  of $N_j$ points, $j=2, \dots, n$. The $N_j$ points in the $j\th$
  group are meant to be the places where a holomorphic $\alpha$ or
  $\beta$ disk-degeneration with domain $H_j$ or $V_j$ has bubbled
  off;
\item for each $j=2, \dots, n$, the $N_j$ points in the $j\th$ group
  are partitioned according to $\lambda_j$. Points in the same part of
  $\lambda_j$ are supposed to be at the same height on the boundary of
  the strip.
\end{itemize}
See Figure~\ref{fig:bubble-intuition}.

There is a special case that we will not discuss in this paper, namely
when $D$ is trivial (that is, equal to $c_x$ for some $x \in \S$) and
$\vN=\vec{0}$. (In that case, the moduli space should be a point
divided by a trivial $\R$ action.) From now on we will always assume
that at least one of $D$ and $\vN$ is nonzero. Then, the dimension $k$
of $\bModuli_{\vN, \vlambda}(D)$ will be given by
\begin{equation}
\label{eq:dimension}
k=\mu(D)-1+|\vlambda|=\operatorname{gr}(D,\vN,\vlambda)-1
\end{equation}
where $\operatorname{gr}(D,\vN,\vlambda)$ is the homological grading
of $(D,\vN,\vlambda)$, as an element of the chain complex $\CDP_*$,
from Equation~\eqref{eq:CDP-grading}.

As we shall see in later sections, the strata of
$\bModuli_{\vN, \vlambda}(D)$ will be products of lower-dimensional
moduli spaces, corresponding to trajectory breaking or bubbling
off further $\alpha$- and $\beta$-degenerations. There will be a
single codimension-zero stratum in $\bModuli_{\vN, \vlambda}(D)$,
denoted $\Moduli_{\vN, \vlambda}(D)$. The strata that correspond to
some bubbles will comprise what we call the {\em special boundary} of
$\bModuli_{\vN, \vlambda}(D)$.

For simplicity, when $\vN=\vec{0}$, we will write
\[
  \bModuli_0(D)\defeq \bModuli_{\vec{0}, \vec{0}}(D).
\]

Let us put an equivalence relation on domains by 
\begin{equation}
\label{eq:eqD}
D\sim D' \ \iff \ (D-D'\in\periodic \text{ and } \coefficients{D}=\coefficients{D'}) \ \iff \ D-D'\in\ZZ\basis{H_2-V_2,\dots,H_n-V_n}.
\end{equation}

Note that an equivalence class of domains is specified by the initial and final points of the domain (call them $x$ and $y$), as well as the vector $\coefficients{D} = (m_2, \dots, m_{n})$. Our construction will ensure that the moduli spaces $\bModuli_0(D)$, over all $D$ in the same equivalence class, will glue together along their special boundaries, to produce a single $\langle k \rangle$-manifold, which we will denote 
\begin{equation}
\label{eq:modd}
 \bModuli([D])= \bModuli(x, U_2^{m_2} \dots U_{n}^{m_{n}} y).
 \end{equation}
 
 The construction of the stratified spaces
 $\bModuli_{\vN, \vlambda}(D)$ will be given in
 Sections~\ref{sec:construction}. For now, to help the reader get some
 intuition, we will discuss a few simple cases, and present some
 examples of spaces that could {\em potentially} play the role of
 $\bModuli_{\vN, \vlambda}(D)$ in those cases. We emphasize that these
 spaces are not actually what the later constructions will
 produce. The final constructions will be inductive and hard to make
 explicit. Rather, the spaces we describe in the examples below are
 merely some explicit spaces that satisfy the formal properties of
 $\bModuli_{\vN, \vlambda}(D)$.  Specifically, they have the right
 dimension, their strata have the correct local model and are indexed
 on the different possibilities for trajectory breaking and bubbles,
 and the spaces corresponding to domains in the same equivalence class
 can be glued together to form $\langle k \rangle$-manifolds.


\begin{example}
\label{ex:Dzero}
Suppose $D=c_x \in \pdomains(x,x)$ is the trivial domain. Then,
\[
  \bModuli_{\vN, \vlambda}(c_x)=\overline{\Bigl( \prod_{j} \Sym^{\ell(\lambda_{j})}(\R) \Bigr)/\R},
\]
where the compactification is induced from the compactification of $\R$ by $\overline{\R}=\{-\infty \} \cup \R \cup \{+\infty\}.$ 
\end{example}

\begin{example}
  Suppose $D\in\rectangles(x,y)$ is a rectangle. Then,
  \[
    \bModuli_{\vN, \vlambda}(D)=\overline{\prod_{j} \Sym^{\ell(\lambda_{j})}(\R)},
  \]
  with the compactification again induced from the same
  compactification of $\R$ by $\overline{\R}$. In particular, if
  $\vN=0$, $\bModuli_0(D)$ is a point.
\end{example}

\begin{example}
\label{ex:RR}
Let $D$ be a positive domain of index two on the grid, that is,
either: (a) the union of two rectangles or (b) an L-shaped hexagon, as
discussed in Item~(\ref{item:maslov-index-2}) in
Section~\ref{sec:grid-background}. Then $\bModuli_0(D)$ is an
interval, which can be viewed as a $1$-dimensional
$\langle 1 \rangle$-manifold. The two ends correspond to the different
ways of splitting $D$ into two domains of index one (trajectory
breaking). See Figure~\ref{fig:Dind2}.
\end{example}

\begin{example}
\label{ex:nodegs}
More generally, suppose $D$ is a positive domain of index $k+1$ that
does not contain any horizontal annulus $H_i$ or any vertical annulus
$V_j$, so that $\alpha$- and $\beta$-degenerations are
impossible. Then $\bModuli_0(D)$ is a $k$-dimensional
$\langle k \rangle$-manifold, with the boundary corresponding to
trajectory breaking. The $i$-colored multifacet
$\del_i(\bModuli_0(D))$ (for $i=1, \dots, k$) corresponds to
splittings of $D$ the form $D^1 \ast D^2$, where $\mu(D^1)=i$ and
$\mu(D^2)=k+1-i$. See Figure~\ref{fig:Dind3} for a picture of
$\bModuli_0(D)$ for an index three domain. In general, to an index $k+1$
domain made of $k+1$ disjoint rectangles one can associate the
$k$-dimensional permutohedron (cf. Example~\ref{ex:permuto}). Other
types of domains yield other $\langle k \rangle$-manifolds.
\end{example}

\begin{example}
\label{ex:HV}
When $H_i$ is a full row and $\vN=0$, we let $\bModuli_0(H_i)$ be an
interval, where one end corresponds to the decomposition into two
rectangles and the other end is the special boundary, corresponding to
an $\alpha$-degeneration. For the column $V_i$ that contains the same
$O_i$ marking as $H_i$, $\bModuli_0(V_i)$ is another interval, also
with one special boundary point, this time corresponding to a
$\beta$-degeneration. Gluing $\bModuli_0(H_i)$ to $\bModuli_0(V_i)$
along their special boundaries yields the $\langle 1 \rangle$-manifold
$\bModuli([H_i])= \bModuli([V_i])$. See Figure~\ref{fig:row}.
\end{example}

\begin{example}
  Figure~\ref{fig:rowplus} shows the spaces $\bModuli_0$ for two
  domains of index three: the column $V_i=C+D$ plus a rectangle $C$
  contained it, and the row $H_i=A+B$ plus the disjoint rectangle
  $C$. These domains would be glued together to produce the
  $\langle 2 \rangle$-manifold $\bModuli([A+B+C])= \bModuli([2C+D])$.
\end{example}

\begin{example}
  Suppose we have rows $H_i=A+B$, $H_j=C+D$, as well as columns
  $V_i=E+F$, $V_j=G+H$ as in Figure~\ref{fig:rowscolumns}. Then, the
  moduli space $\bModuli_0(H_i+H_j)$ is shown in
  Figure~\ref{fig:rowrow}, with special boundary consisting of the
  front and the right facet; those for the domains $H_i+V_j$,
  $V_i+H_j$ and $V_i+V_j$ are very similar. These four spaces glue
  together along their special boundaries to form the
  $\langle 3 \rangle$-manifold $\bModuli_0([H_i+H_j])$, which is a
  three-dimensional permutohedron. Note that in this gluing, the
  special edge drawn in green in Figure~\ref{fig:rowrow} is common to
  all four polyhedra.
\end{example}

\begin{remark} \label{rem:gromov} In symplectic geometry we encounter
  moduli spaces of bubble trees that we do not consider here.  For
  example, in Figure~\ref{fig:rowrow} the green edge corresponds to
  two disk degenerations, and as we move along the edge we change the
  relative heights where these two degenerations take place. In
  particular, there is a point in the middle that corresponds to the
  two degenerations happening at the same height. If we were to
  actually consider the Gromov compactification from symplectic
  geometry, instead of that point we would have a whole new
  (two-dimensional) facet, corresponding to degenerating an index four
  disk with domain $H_i+H_j=A+B+C+D$, as in
  Figure~\ref{fig:gromov}. Thus, our moduli
  space compactifications are different from the usual
  compactifications in symplectic geometry---ours is a quotient of the
  usual one, while the usual one is a blowup of ours. We do not delve into this issue further since in this
  paper we do not construct the actual moduli spaces from holomorphic geometry,
  but rather construct their models inductively by obstruction theory.
\end{remark}

\begin{example}
\label{ex:RK}
Let us return to the domains from
Figure~\ref{fig:row}, where $H_i=A+B$ is a row and $V_i=C+D$ is a
column. The space $\bModuli_0$ for the domain $H_i+V_i = A+B+C+D$ made
of a row and a column is shown in Figure~\ref{fig:rowcolumn}. In fact,
it is almost the same polyhedron as in Figure~\ref{fig:rowrow}, except
that the green edge is folded in half. The folding is due to the fact
that since $H_i$ and $V_i$ go through the same $O_i$ marking, we want
to identify the $AB$ and $CD$ disk degenerations. Thus, the point on
the green edge where the $AB$ degeneration is at a certain distance up
from the $CD$ degeneration, is identified with the point where the
$CD$ degeneration is on top of $AB$, at the same distance.

To be more precise, with the notation from Example~\ref{ex:Dzero}, the
green line in Figure~\ref{fig:rowrow} is the space
$\bModuli_{\vN, \vlambda}(0)$ with $\vN$ being the vector with $1$'s
in positions $i$ and $j$ (and $0$ otherwise), and $\vlambda$ the
unique possible vector of partitions. This space is the
compactification of
\[
  (\Sym^1(\R) \times \Sym^1(\R))/\R,
\]
which is just $\overline{\R}=\{-\infty \} \cup \R \cup \{+\infty\}.$

On the other hand, the (folded) green line in
Figure~\ref{fig:rowcolumn} is $\bModuli_{\vN, \vlambda}(0)$ with $\vN$
being the vector with a single $2$ in position $i$, and $\vlambda$
consisting of trivial partitions except for $\lambda_i = (1,1)$. This
is the compactification of
\[
  \Sym^2(\R)/\R.
\]

In particular, there is a special point (the left green dot in
Figure~\ref{fig:rowcolumn}) where the $AB$ and $CD$ degenerations
happen at the same height. The local model of special point inside the
space $\bModuli_0(H_i+V_i)$ is
\[
  Z(0,2,0;(2)) \subset \bZ(1,0,1)
\]
from Figure~\ref{fig:Whitney}. The green line corresponds to the thickened line in Figure~\ref{fig:Whitney}, and the front and right facets in Figure~\ref{fig:rowcolumn} meet along the green line, forming a Whitney umbrella.
\end{example}

\begin{example}
\label{ex:three}
Reusing the same configuration from Figure~\ref{fig:row}, we now consider
the domain $2H_i=2A+2B$ (a row with multiplicity two). The
corresponding space $\bModuli_0(2H_i)$ is pictured in
Figure~\ref{fig:2row}. This is glued with a similar space
$\bModuli_0(2V_i)$, as well as with the space $\bModuli_0(H_i+V_i)$
from Example~\ref{ex:RK}, to yield a single
$\langle 3 \rangle$-manifold $\bModuli_0([2H_i])$. Around the top
green dot, the gluing is modeled on the Whitney umbrella from
Figure~\ref{fig:Whitney}, with $\bModuli_0(2H_i)$,
$\bModuli_0(H_i+V_i)$ and $\bModuli_0(2V_i)$ playing the roles of
$Z(2,0,0)$, $Z(1,0,1)$ and $Z(0,0,2)$, respectively.
\end{example}

\section{The stratification}
\label{sec:strata}
We now describe the intended stratification of the spaces
$\bModuli_{\vN, \vlambda}(D)$, where
$(D,\vN,\vlambda)\in\DNlambda$. We will ensure that these spaces have
a unique codimension zero stratum, which we denote
$\Moduli_{\vN,\vlambda}(D)$, and we let $\bdy\Moduli_{\vN,\vlambda}(D)$
denote $\bModuli_{\vN,\vlambda}(D)\setminus
\Moduli_{\vN,\vlambda}(D)$. (These spaces are not topological manifolds
in general, and so this is not the boundary in any usual sense.)

\subsection{Enumeration of strata} \label{sec:enum}
We ask that $\bModuli_{\vN, \vlambda}(D)$ has the following strata:
  \begin{equation}\label{eq:moduli-bdy}
  \Moduli_{\vN^1+\coefficients{E^1}+\coefficients{F^1}, \vlambda^1}(D^1) \times \dots \times
   \Moduli_{\vN^r+\coefficients{E^r}+\coefficients{F^r}, \vlambda^r}(D^r),
   \end{equation}
   with $r \geq 1$ and generators $x=w_0,w_1,\dots,w_{r-1},w_r=y$, such that for each  $1\leq i \leq r,$
   \begin{gather*}
     D^i \in\pdomains(w_{i-1},w_{i}),\\
     E^i = \sum_{j=2}^{n} O_j(E^i) H_j \in\pperiodic \text{ is a sum of rows},\\
     F^i= \sum_{j=2}^{n} O_j(F^i) V_j \in\pperiodic \text{ is a sum of columns},
  \end{gather*}
satisfying
\[
  \sum_i (D^i+E^i+F^i)=D.
\]
Further,
\[
  \vN^i=(N^i_2,\dots,N^i_{n})\in\NN^{n-1},\ \ \sum_i \vN^i=\vN,
\]
and
\[
  \vlambda^i= (\lambda^i_2, \dots, \lambda^i_{n}), \ \ \ \lambda^i_j \in \Part(N^i_j+ O_j(E^i) +  O_j(F^i) )
\]
are such that there exist some other partitions
\begin{equation}\label{eq:veta-partition}
  \veta^i= (\eta^i_2, \dots, \eta^i_{n}), \ \ \eta^i_j \in \Part(N^i_j)
\end{equation}
with $\eta^i_j \geq \lambda^i_j$ (in the notation of Section~\ref{sec:local}) and
\[
  \eta^1_j * \dots * \eta^r_j = \lambda_j, \ j=2, \dots, n.
\]
Here, $*$ is the concatenation of partitions defined in
Equation~\eqref{eq:concatenate}. Note that, if $\veta^i$ exist, then they are
unique. This is because an ordered partition can be uniquely (if at
all) decomposed as a concatenation of partitions of specified sizes.

A few explanations are in order. In the description of the strata, the
$D^i$'s are the pieces in the trajectory breaking, the $E^i$'s  correspond to $\alpha$-boundary degenerations, and the
$F^i$'s to $\beta$-boundary degenerations. The points where the
boundary degenerations through $O_j$ are attached were originally
partitioned according to $\lambda_j$. When the trajectory breaks into
$r$ pieces, these points get split into $r$ groups, where the
$i\th$ group is partitioned according to $\eta^i_j$. Since we also
pick up extra boundary degenerations from the $E^i$ and $F^i$, we
add more points. We could also join some of the parts,
to make the partition $\eta^i_j$ coarser, since this is what happens
in lower dimensional strata; compare Equations~\eqref{eq:INcompare}
and~\eqref{eq:ZN}. The result of this process is the partition
$\lambda^i_j\leq \eta^i_j$.


The strata of $\bModuli_{\vN, \vlambda}(D)$ are required to satisfy
the following coherence relations with respect to their
closures. Given a stratum as in Equation~\eqref{eq:moduli-bdy}, its closure in
$\bModuli_{\vN, \vlambda}(D)$ should be the product of the closures of
its factors:
\begin{equation}
\label{eq:bproduct}
 \bModuli_{\vN^1+\coefficients{E^1}+\coefficients{F^1}, \vlambda^1}(D^1) \times \dots \times
   \bModuli_{\vN^r+\coefficients{E^r}+\coefficients{F^r}, \vlambda^r}(D^r).
   \end{equation}
Further, if for $i=1, \dots, r$ we have strata
\[
  \prod_{k=1}^{m_i}  \Moduli_{\vN^{i,k}+\coefficients{E^{i, k}}+\coefficients{F^{i,k}}, \vlambda^{i, k}}(D^{i, k}) \subset \bModuli_{\vN^i+\coefficients{E^i}+\coefficients{F^i}, \vlambda^i}(D^i)
\]
we ask that the inclusion of the product stratum
\[
  \prod_{k=1}^{m_1}  \Moduli_{\vN^{1,k}+\coefficients{E^{1, k}}+\coefficients{F^{1,k}}, \vlambda^{1, k}}(D^{1, k}) \times \dots \times \prod_{k=1}^{m_r}  \Moduli_{\vN^{r,k}+\coefficients{E^{r, k}}+\coefficients{F^{r,k}}, \vlambda^{r, k}}(D^{r, k})
\]
into $\bModuli_{\vN, \vlambda}(D)$ factors through Equation~\eqref{eq:bproduct}.

Let us also check that Equation~\eqref{eq:bproduct} ensures that
$\bModuli_{\vN,\vlambda}(D)$ has a unique codimension zero
stratum. Using the dimension formula~\eqref{eq:dimension}, the
codimension of the stratum described in Equation~\eqref{eq:moduli-bdy}
is
\begin{equation}
\label{eq:codimension}
r-1+ \sum_{i=1}^r \mu(E^i) + \sum_{i=1}^r \mu(F^i) + \sum_{j=2}^{n} \Bigl(\ell(\lambda_j) - \sum_{i=1}^r \ell(\lambda^i_j) \Bigr).
\end{equation}
We can write each of the last summands as 
\[
  \ell(\lambda_j)-  \sum_{i=1}^r \ell(\lambda^i_j) = \sum_{i=1}^r (\ell(\eta^i_j) - \ell(\lambda^i_j)).
\]
 Using the co-length inequality \eqref{eq:colength}, we have
 \[
   \ell(\eta^i_j) - \ell(\lambda^i_j) \geq - O_j(E^i) - O_j(F^i).
 \]
 Since $\mu(E^i)= 2\sum_j O_j(E^i)$ and $\mu(F^i)= 2\sum_j O_j(F^i)$, we deduce that the codimension given in Equation~\eqref{eq:codimension} is at least
 \[
   r-1+ \sum_{i=1}^r \sum_{j=2}^n \bigl( O_j(E^i) + O_j(F^i) \bigr) \geq r-1.
 \]
In particular, $\Moduli_{\vN, \vlambda}(D)$ appears as the unique codimension zero stratum, with $r=1$, $D^1=D$, $E^1 =F^1=0,\vlambda^1=\veta^1=\vlambda$.

\begin{example}
  If we consider the domain $2H_i=2A+2B$ from Example~\ref{ex:three}, the moduli space $\bModuli_0(2H_i)$ has the following (open) strata:
  \begin{itemize}
  \item Dimension 3: $\Moduli_0(2H_i)$;
  \item Dimension 2: $\Moduli_0(A)\times\Moduli_0(A+2B)$, $\Moduli_0(2A+B)\times\Moduli_0(B)$, $\Moduli_0(H_i)\times\Moduli_0(H_i)$, $\Moduli_{\coefficients{H_i},(1)_i}(H_i)$;
  \item Dimension 1: $\Moduli_0(A)\times\Moduli_0(B)\times\Moduli_0(H_i)$, $\Moduli_0(A)\times\Moduli_0(H_i)\times\Moduli_0(B)$, $\Moduli_0(H_i)\times\Moduli_0(A)\times\Moduli_0(B)$, $\Moduli_0(H_i)\times \Moduli_{\coefficients{H_i},(1)_i}(c_x)$,  $\Moduli_{\coefficients{H_i},(1)_i}(c_x)\times\Moduli_0(H_i)$, $\Moduli_0(A)\times\Moduli_{\coefficients{H_i},(1)_i}(B)$, $\Moduli_{\coefficients{H_i},(1)_i}(A)\times\Moduli_0(B)$, $\Moduli_{\coefficients{2H_i},(1,1)_i}(c_x)$;
  \item Dimension 0: $\Moduli_0(A)\times\Moduli_0(B)\times\Moduli_0(A)\times\Moduli_0(B)$, $\Moduli_0(A)\times\Moduli_0(B)\times\Moduli_{\coefficients{H_i},(1)_i}(c_x)$, $\Moduli_0(A)\times\Moduli_{\coefficients{H_i},(1)_i}(c_y)\times\Moduli_0(B)$, $\Moduli_{\coefficients{H_i},(1)_i}(c_x)\times \Moduli_0(A)\times\Moduli_0(B)$, $\Moduli_{\coefficients{H_i},(1)_i}(c_x)\times \Moduli_{\coefficients{H_i},(1)_i}(c_x)$, $\Moduli_{\coefficients{2H_i},(2)_i}(c_x)$.
  \end{itemize}
  See also Figure~\ref{fig:2row}.
\end{example}

\begin{remark}
\label{rmk:lambdaj}
In the above example, we denoted by $(\lambda)_i$ the vector
consisting of a partition $\lambda$ in position $i$, and trivial
partitions elsewhere.
\end{remark}

Using this stratification on the moduli spaces $\bModuli(D,\vN,\vlambda)$, we get the following partial order on the set $\DNlambda(\Grid)$. Declare
\begin{equation}\label{eq:DNlamba-poset}
  (D',\vN',\vlambda')\leq (D,\vN,\vlambda)
\end{equation}
if $\bModuli_{\vN,\vlambda}(D)$ has an open stratum
$\prod_i\Moduli_{\vN^i+\Os(\E^i)+\Os(F^i),\vlambda^i}(D^i)$ with
$(D',\vN',\vlambda')=(D^i,\vN^i+\Os(\E^i)+\Os(F^i),\vlambda^i)$ for
some $i$. This is easily seen to be a partial order: for
anti-reflexivity, note that if $(D',\vN',\vlambda')<(D,\vN,\vlambda)$,
then $\gr(D',\vN',\vlambda')<\gr(D,\vN,\vlambda)$. (This partial order
is related to the chain complex $\CDP_*$ from
Section~\ref{sec:new}. If $(D',\vN',\vlambda')$ appears in
$\delta(D,\vN,\vlambda)$ in any of the terms $\dI,\dII,\dIII,\dIV$
(from Equation~\eqref{eq:delta1}--\eqref{eq:delta4}), then
$(D',\vN',\vlambda')<(D,\vN,\vlambda)$ with
$\gr(D',\vN',\vlambda')=\gr(D,\vN,\vlambda)-1$; indeed, the partial
order from Equation~\eqref{eq:DNlamba-poset} is generated by these
atomic relations.)

\subsection{Codimension one strata} \label{sec:codim1} For future
reference, let us also describe the codimension-one strata of
$\bModuli_{\vN, \vlambda}(D)$. Using Formula~\eqref{eq:codimension}, a
similar analysis shows that this consists of strata of three possible
types, corresponding to trajectory breaking, boundary degeneration
(Type II), and coarsening of partitions (Type III); in order to stress
the correspondence with the different types of terms in the
differential $\delta$ on $\CDP_*$ (cf.~Remark~\ref{rem:1234}), we
have subdivided trajectory breaking into two further subtypes (Type I
and IV).

{\bf Type I} codimension-one strata correspond to $r=2$ and pure
trajectory breaking into two non-constant domains, with no additional
disk degenerations ($E^i=F^i=0$), and no coarsening of the
partitions:
\[
 \Moduli_{ \vN^1, \vlambda^1}(D^1) \times \Moduli_{ \vN^2, \vlambda^1}(D^2),
\]
where 
\begin{itemize}
\item $w$ is an intermediate generator, 
\item $D^1\in\pdomains(x,w)$ and $D^2\in\pdomains(w,y)$ are such that $D^1 + D^2 = D,$ and neither $D^1$ nor $D^2$ is trivial,
\item $ \vN^1, \vN^2 \in \NN^{n-1}$ are such that  $\vN^1 + \vN^2 = \vN,$
\item $\vlambda^i= (\lambda^i_2, \dots, \lambda^i_{n}), i=1,2$, are vectors of partitions such that 
$ \lambda^1_j * \lambda^2_j = \lambda_j$ for all $j=2, \dots, n$.
\end{itemize}

Note that, among these strata, the terms where one of the two factors is
zero-dimensional are when either $D^1$ or $D^2$ is a rectangle, and the partition corresponding to that rectangle is empty:
\begin{align*}
 \Moduli_{\vN, \vlambda}(D^1) \times \Moduli_{0}(R), & \ \ \text{with }D^1 \in \pdomains(x,w), R\in\rectangles(w,y), D^1+R = D,\\
  \Moduli_{0}(R) \times \Moduli_{\vN, \vlambda}(D^2), & \ \ \text{with }R\in\rectangles(x,w), D^2 \in \pdomains(w,y),  R+D^2 = D.
\end{align*}

\medskip {\bf Type II} codimension-one strata correspond to no
trajectory breaking ($r=1$) and a single boundary degeneration (with
domain row $H_j$ or column $V_j$ for some $2\leq j\leq n$), and no
coarsening of the partitions:
\begin{align*}
 \Moduli_{\vN + \ve_j, \vlambda'}(D^1), & \text{ with } D^1 + H_j = D, \\
  \Moduli_{\vN + \ve_j, \vlambda'}(D^1), & \text{ with } D^1 + V_j = D, 
\end{align*}
where
\[
  \vlambda'=(\lambda'_2, \dots, \lambda'_{n})
\]
is such that $\lambda_j' \in \UE(\lambda_j)$, and
$ \lambda'_s = \lambda_s$ for all $s \neq j$. Here, $\UE(\lambda_j)$
is the set of unit enlargements of $\lambda_j$
(cf. Definition~\ref{def:UE}).

\medskip {\bf Type III} codimension-one strata correspond to no
trajectory breaking ($r=1$) and no boundary degenerations, but rather
an elementary coarsening of a partition $\lambda_j$ (for some $j$):
\[
 \Moduli_{\vN, \vlambda'}(D), \text{ with } \vlambda'=(\lambda'_2, \dots, \lambda'_{n}),
\]
 where $\lambda_j' \in \EC(\lambda_j)$, and $ \lambda'_s = \lambda_s$
 for all $s \neq j$. Here, $\EC(\lambda_j)$ is the set of elementary
 coarsenings (cf. Definition~\ref{def:EC}).

\medskip {\bf Type IV} codimension-one strata correspond to pure
trajectory breaking $(r=2$) where one of the domains is constant, but
with no additional disk degenerations ($E^i=F^i=0$), and no coarsening
of the partitions:
\begin{gather*}
 \Moduli_{ \vN^1, \vlambda^1}(c_x) \times \Moduli_{ \vN^2, \vlambda^1}(D),\\
 \Moduli_{ \vN^1, \vlambda^1}(D) \times \Moduli_{ \vN^2, \vlambda^1}(c_y),
\end{gather*}
where 
\begin{itemize}
\item $ \vN^1, \vN^2 \in \NN^{n-1}$ are such that  $\vN^1 + \vN^2 = \vN,$
\item $\vlambda^i= (\lambda^i_2, \dots, \lambda^i_{n}), i=1,2$, are vectors of partitions such that 
$ \lambda^1_j * \lambda^2_j = \lambda_j$ for all $j=2, \dots, n$.
\end{itemize}

Note that, among these strata, the terms where one of the two factors is
zero-dimensional are when the vector of partitions corresponding to the constant domain consists of empty partitions and a single length-one partition:
\begin{gather*}  
 \Moduli_{N^1\ve_j,  (N^1)_j}(c_x) \times \Moduli_{\vN^2, \vlambda^2}(D), \\
  \Moduli_{\vN^1, \vlambda^1}(D) \times \Moduli_{N^2\ve_j, (N^2)_j}(c_y).
\end{gather*}
Here, $N^1,N^2$ (when written without the vector symbols) are natural
numbers, and $(N^1)_j,(N^2)_j$ denote vectors of partitions as in
Remark~\ref{rmk:lambdaj}.

\begin{remark}
\label{rem:1234}
The different types of strata correspond to different kinds of terms
in the differential $\delta$ on the complex $\CDP_*$;
cf.~Section~\ref{sec:new}. Type I strata, where one of the factors is
zero dimensional, correspond to terms of $\delta$ of Type I; Type II
corresponds to Type II; Type III to Type III; and Type IV, where one
of the factors is zero dimensional, to Type IV.
\end{remark}


\begin{definition}
\label{def:dels}
Let $X =\bModuli_{\vN, \vlambda}(D)$ be of dimension $k$. Define its
\emph{thick dimension} by the equation
\begin{equation}
  \label{eq:dimtilde}
  \tdim \bModuli_{\vN, \vlambda}(D) = \mu(D)-1 + 2|\vN|.
\end{equation}
(This definition is justified in the next section.) Let
$l=\tdim \bModuli_{\vN, \vlambda}(D)$. By analogy with the notation
for $\langle n \rangle$-manifolds in Section~\ref{sec:nmflds}, for
$i=1, \dots, l$, we let $\del_iX$ be the closure of the union of all
codimension-one strata of Types I and IV of the form
\[
  \Moduli_{ \vN^1, \vlambda^1}(D^1) \times \Moduli_{ \vN^2, \vlambda^1}(D^2)
\]
with
\[
  \tdim \Moduli_{ \vN^1, \vlambda^1}(D^1) = i-1.
\]
We also let the {\em special boundary} of $X$, denoted $\del_sX$, be the closure of the union of all codimension-one strata of Types II and III. 
\end{definition}

It is easy to see that every higher codimension stratum is contained in the closure of a codimension-one stratum. Therefore, altogether, the boundary of $X$ is
\[
  \del X = (\del_1 X \cup \dots \cup \del_{l}X) \cup \del_sX.
\]

\subsection{Local models} \label{sec:modelsmoduli} Let us describe the local models for how the strata  from Equation~\eqref{eq:moduli-bdy} should live inside the moduli spaces $\bModuli_{\vN, \vlambda}(D)$. Note that every $\bModuli_{\vN, \vlambda}(D)$ is itself a stratum of a space of the form
\[
  \bModuli_{0}(\tD),\]
with  $\tD =D+\tE+\tF, \tE\in\pperiodic\text{ a sum of rows},\tF\in\pperiodic\text{ a sum of columns},\Os(\tE) +\Os(\tF)=\vN.$

There are several possible choices of such $\tD$, depending on the
choice of $\tE,\tF$; indeed, if $\vN=(N_2,\dots,N_n)$, there are
exactly $(N_2+1)(N_3+1)\cdots(N_n+1)$ such choices. (For example, for
the green line $\Moduli_{2\ve_i,(1,1)_i}(c_x)$ in the Whitney umbrella
from Figures~\ref{fig:rowcolumn} and \ref{fig:2row} we have three such
choices, $2H_i,H_i+V_i,2V_i$.)

The dimension of each $\bModuli_{0}(\tD)$ is given by the thick
dimension $l=\tdim \bModuli_{\vN, \vlambda}(D)$. The different possible
$\tD$ are in the same equivalence class $[\tD]$. As mentioned in
Equation~\eqref{eq:modd}, we will glue these $\bModuli_{0}(\tD)$ along
their special boundaries to form an $l$-dimensional
$\langle l\rangle$-manifold
\[
  \bModuli([\tD])=\bigcup \bModuli_0(\tD),
\]
with $\bdy_i\bModuli([\tD])=\bigcup \bdy_i\bModuli_0(\tD)$, for each
$i=1,2,\dots,l$. That is, after gluing, $\bModuli([\tD])$ no longer has any special boundary, and $\bdy_i$ corresponds precisely to trajectory breaking into $(i-1)$ and $(l-i)$-dimensional pieces:
\[
  \bdy_i\bModuli([\tD])=\coprod_{\substack{[\tD^1]+[\tD^2]=[\tD]\\\tdim \Moduli([\tD^1]) = i-1}} \bModuli([\tD^1])\times \bModuli([\tD^2]).
\]

Recall from Section~\ref{sec:TM} that if we specify a tubular
neighborhood $T_X$ of a stratum $X$ inside a stratified space, the
tubular neighborhoods $\ol{T}_{X, Y}$ of $X$ inside other closed strata $\ol{Y}$ are
just given by intersecting $T_X$ with $\ol{Y}$. Thus, to understand the
local model of a stratum inside $\bModuli_{\vN, \vlambda}(D)$, it
suffices to consider its local model inside the bigger space
$\bModuli_{0}(\tD)$.

Using the notation of Section~\ref{sec:generalLM}, we ask that the local model in the normal directions for the stratum
\[
  X=\Moduli_{\vN^1+\coefficients{E^1}+\coefficients{F^1}, \vlambda^1}(D^1) \times \dots \times
  \Moduli_{\vN^r+\coefficients{E^r}+\coefficients{F^r}, \vlambda^r}(D^r)
\]
inside $\bModuli_0(\tD)$ is the same as the local model for 
\[
  \R^{a} \times \{0\} \times  Z(0,\vN^1+\Os(E^1)+ \Os(F^1), 0; \vlambda^1) \times \cdots \times Z(0, \vN^r+\Os(E^r)+ \Os(F^r), 0; \vlambda^r)
\]
inside the union of 
\[
  \R^{a} \times \R_+^{r-1} \times \bZ(\Os(E^1) + \Os(\tE^1), 0, \Os(F^1)+\Os(\tF^1)) \times  \dots \times  \bZ(\Os(E^r)+ \Os(\tE^r), 0, \Os(F^r)+ \Os(\tF^r)),
\]
over all possible choices of $\tE^i\in\pperiodic$ a sum of rows and $\tF^i\in\pperiodic$ a sum of columns satisfying
\[
  \Os(\tE^i) + \Os(\tF^i)=\vN^i, \ \ \sum_i \tE^i = \tE,\ \ \sum_i \tF^i = \tF.
\]
Here $a$ is given by
\begin{align*}
  a&=\dim X-\sum_i \dim Z(0,\vN^i+\Os(E^i)+\Os(F^i),0;\vlambda^i)\\
   &=\sum_i \mu(D^i) + \#\{i\mid \vN^i+\Os(E^i)+\Os(F^i)=0\}.
\end{align*}

Since $\bModuli([\tD])=\bigcup \bModuli_0(\tD)$, we will require that the local model in the normal directions for the stratum $X$ in $\bModuli([\tD])$ is the union of the above local models, which is same as the local model for
\[
  \R^{a} \times \{0\} \times  Z(0,\vN^1+\Os(E^1)+ \Os(F^1), 0; \vlambda^1) \times \cdots \times Z(0, \vN^r+\Os(E^r)+ \Os(F^r), 0; \vlambda^r)
\]
inside
\[
  \R^{a} \times \R_+^{r-1} \times Z_{\vN^1+\Os(E^1)+\Os(F^1)}\times\dots \times Z_{\vN^1+\Os(E^1)+\Os(F^1)}.
\]

We will come back to these local models in Section~\ref{sec:neat},
when we will describe neat embeddings of our moduli spaces. In the
meantime, we present the following illuminating example.

\begin{example}\label{ex:local-model-example}
  Continuing from Examples~\ref{ex:RK} and~\ref{ex:three}
  (Figures~\ref{fig:rowcolumn} and~\ref{fig:2row}), the local model of
  the green edge $\Moduli_{2\ve_i,(1,1)_i}(c_x)$ inside
  $\bModuli([2H_i])$ is same as the local model of
  $Z(0,2\ve_i,0;(1,1)_i)$ inside $Z_{2\ve_i}$, which in turn, is same
  as the local model of $\RR\times\{0\}$ inside
  $\RR\times Z_1\times Z_1$ (cf.~Example~\ref{ex:Z020-inside-Z2}); the
  latter space has four codimension zero closed strata, of which
  $\RR\times \bZ(1,0,0)\times\bZ(1,0,0)$ is the local model inside
  $\bModuli_0(2H_i)$, $\RR\times \bZ(0,0,1)\times\bZ(0,0,1)$ is the
  local model inside $\bModuli_0(2V_i)$, and
  $\RR\times \big((\bZ(1,0,0)\times\bZ(0,0,1))\cup
  (\bZ(0,0,1)\times\bZ(1,0,0))\big)$ is the local model inside
  $\bModuli_0(H_i+V_i)$.

  The local model of the green vertex $\Moduli_{2\ve_i,(2)_i}(c_x)$
  inside $\bModuli([2H_i])$ is same as the local model of the point
  $Z(0,2\ve_i,0;(2)_i)$ inside $Z_{2\ve_i}$, and its three codimension zero
  closed strata $Z(2\ve_i,0,0)$, $Z(\ve_i,0,\ve_i)$, $Z(0,0,2\ve_i)$
  are the local models inside $\bModuli_0(2H_i)$, $\bModuli_0(H_i+V_i)$,
  $\bModuli_0(2V_i)$, respectively.

  Finally, the local model of the other green vertex
  $\Moduli_{\ve_i,(1)_i}(c_x)\times \Moduli_{\ve_i,(1)_i}(c_x)$ inside
  $\bModuli([2H_i])$ is same as the local model of the point
  $\{0\}\times Z(0,\ve_i,0;(1)_i)\times Z(0,\ve_i,0;(1)_i)$ inside
  $\RR_+\times Z_{\ve_i}\times Z_{\ve_i}$; the latter space also has
  four codimension zero closed strata, of which
  $\RR_+\times \bZ(\ve_i,0,0)\times\bZ(\ve_i,0,0)$ is the local model
  inside $\bModuli_0(2H_i)$,
  $\RR_+\times \bZ(0,0,\ve_i)\times\bZ(0,0,\ve_i)$ is the local model
  inside $\bModuli_0(2V_i)$, and
  $\RR_+\times \big((\bZ(\ve_i,0,0)\times\bZ(0,0,\ve_i))\cup
  (\bZ(0,0,\ve_i)\times\bZ(\ve_i,0,0))\big)$ is the local model inside
  $\bModuli_0(H_i+V_i)$.

  See Figure~\ref{fig:local-models-example}.
\end{example}

\begin{figure}
  \centering
  \begin{tikzpicture}

    \foreach \i in {0,1,2}{
    
  \begin{scope}[xshift=5*\i cm,x={(1cm,0)},z={(0,2cm)},y={(0.3cm,0.5cm)}]

    \coordinate (top) at (0,0,1);
    \coordinate (bottom) at (0,0);

    \coordinate (topleft) at (-1,-1,1);
    \coordinate (topright) at (1,1,1);

    \node (mixed) at (-0.7,0.5,1.3) {\small $\bModuli_0(H_i+V_i)$};

    \ifnum\i=1

    \path (-2,0) ..controls(-2,0,0.5) and ($(topleft)+(-0.5,0.5)$).. coordinate[pos=0.8] (fakeleft) (topleft);
    \path (topright) ..controls($(topright)+(-0.5,0.5)$) and (0,2,0.5).. coordinate[pos=0.2] (fakeright)(0,2);
    
    \fill[black!50] ($(fakeleft)+(0,0,-0.2)$)--(fakeleft)--(top)--(fakeright)--++(0,0,-0.2)--cycle;
    \fill[black!20] (-2,0) ..controls(-2,0,0.5) and ($(topleft)+(-0.5,0.5)$)..(topleft) -- (top)--(bottom)--cycle;
    \fill[black!50] (2,0) ..controls(2,0,0.5) and ($(topright)+(0.5,-0.5)$).. (topright)..controls($(topright)+(-0.5,0.5)$) and (0,2,0.5)..(0,2)--cycle;
    \fill[black!50] (topleft) ..controls ($(topleft)+(0.5,-0.5)$) and (0,-2,0.5).. (0,-2) -- (bottom) -- (2,0) ..controls(2,0,0.5) and ($(topright)+(0.5,-0.5)$).. (topright)--cycle;

    \draw (-2,0) ..controls(-2,0,0.5) and ($(topleft)+(-0.5,0.5)$)..(topleft) ..controls ($(topleft)+(0.5,-0.5)$) and (0,-2,0.5).. (0,-2);
    \draw (2,0) ..controls(2,0,0.5) and ($(topright)+(0.5,-0.5)$).. (topright);
    \draw (0,-2)--(bottom)--(2,0) (-2,0)--(-0.7,0);
    \draw (fakeright)--(top)--(fakeleft);
    \draw[dashed] (-0.7,0)--(bottom)--(0,2) .. controls(0,2,0.5) and ($(topright)+(-0.5,0.5)$)..(topright);

    \node at (1,1) {\small $\bModuli_0(2V_i)$};
    \node at (-1,-1) {\small $\bModuli_0(2H_i)$};

    \node[green!70!black] at (top) {\large $\bullet$};

    \else

    \ifnum\i=2
    \fill[black!35] (-2,-2)--(2,-2)--(2,2)--(-2,2)--cycle;
    \fi
    
    \fill[black!50] (0,2)--(bottom)--(top)--(0,2,1)--cycle;
    \fill[black!20] (-2,0)--(2,0)--(2,0,1)--(-2,0,1)--cycle;
    \fill[black!50] (0,-2)--(bottom)--(top)--(0,-2,1)--cycle;

    \draw (-2,0)--(-2,0,1)--(2,0,1)--(2,0);
    \draw (0,-2)--(0,-2,1)--(0,2,1)--(0,2,0.5);
    \draw[dashed] (0,2,0.5)--(0,2);

    \ifnum\i=2

    \draw (-2,0)--(-2,-2)--(2,-2)--(2,2)--(1.4,2);
    \draw[dashed] (-2,0)--(-2,2)--(1.4,2);

    \draw[ultra thick] (-0.6,0)--(-2,0) (0,-2)--(bottom)--(2,0);
    \draw[dashed,ultra thick] (-0.6,0)--(bottom)--(0,2);

    \node[green!70!black] at (bottom) {\large $\bullet$};

    \else
    \draw (-0.6,0)--(-2,0) (0,-2)--(bottom)--(2,0);
    \draw[dashed] (-0.6,0)--(bottom)--(0,2);
    \fi

    \node at (1,1,0.9) {\small $\bModuli_0(2V_i)$};
    \node at (-1,-1,0.9) {\small $\bModuli_0(2H_i)$};

    \draw[->] (mixed) -- (-0.7,0.5,1);
    \draw[->] (mixed) to[out=-45,in=90] (1,-1,1);
    
    \fi
    
    \draw[ultra thick,green!70!black] (bottom)--(top);

  \end{scope}

  }
    
  \end{tikzpicture}
  \caption{The local models of $\Moduli_{2\ve_i,(1,1)_i}(c_x)$ (left), $\Moduli_{2\ve_i,(2)_i}(c_x)$ (middle), and $\Moduli_{\ve_i,(1)_i}(c_x)\times \Moduli_{\ve_i,(1)_i}(c_x)$ (right) inside $\bModuli([2H_i])$, from Example~\ref{ex:local-model-example}.}\label{fig:local-models-example}
\end{figure}
\section{Embeddings and framings}
\label{sec:neat}
The moduli spaces $\bModuli_{\vN, \vlambda}(D)$ will come equipped with suitable embeddings in
\[
  \E^d_l \defeq \R^d \times \R_+ \times \R^d \times \R_+ \times \cdots \times \R_+ \times \R^d \; \cong \; \R_+^{l}\times \R^{d(l+1)}
\]
and they will also be framed. Here, $d$ is some sufficiently large
integer depending only on the grid $\Grid$, whereas
$l = \tdim \bModuli_{\vN, \vlambda}(D)$ is the thick dimension given
by the Formula~\eqref{eq:dimtilde}; for convenience, we will
henceforth also assume that $d$ is even, so that the above isomorphism
is orientation-preserving. Note that $\E^d_l$ is $\E(l,d(l+1))$ in the
notation of Section~\ref{sec:neatk}.

\subsection{Neat embeddings of stratified spaces} \label{sec:neats}
In this section we will describe the required properties for the embedding of $X=\bModuli_{\vN, \vlambda}(D) \hookrightarrow \E^d_l$ that we plan to construct. By analogy with Section~\ref{sec:neatk}, an embedding with these properties will be called {\em neat}. We assume that the strata of $X$ are as described in Section~\ref{sec:strata}. 

\begin{definition}
  \label{def:neatmoduli}
  Let $X=\bModuli_{\vN,\vlambda}(D)$ with thick dimension $\tdim
  X=l$. A \emph{neat embedding} of $X$ consists of the following data:
  \begin{enumerate}[label=(NE-\arabic*),ref=(NE-\arabic*)]
  \item An equivalence class of a subspace
    $U=U_{\vN,\vlambda}(D)\subset \E^d_l$, called the
    \emph{thickening} of $X$, such that $U$ is an $l$-dimensional
    $\langle l\rangle$-manifold and the inclusion $U\into \E^d_l$ is a
    neat embedding; the equivalence relation is described below.
  \item A topological embedding
    $\iota=\iota_X\from X\into U$---which is a smooth
    embedding when restricted to each open stratum---satisying
    $\iota^{-1}(\bdy_i U)=\bdy_i X$ for all $i=1,\dots,l$.
  \item\label{item:local-model-identification} A stratification of $U$ with $\iota(X)$ as a closed stratum, and an
    identification of the local model of each open stratum of $X$ of the form
    \[
      Y= \Moduli_{\vN^1+\coefficients{E^1}+\coefficients{F^1}, \vlambda^1}(D^1) \times \dots \times
      \Moduli_{\vN^r+\coefficients{E^r}+\coefficients{F^r}, \vlambda^r}(D^r)
    \]
    inside $U$, with the local model of
    \[
      A_Y=\R^{a_Y} \times \{0\} \times  Z(0,\vN^1+\Os(E^1)+ \Os(F^1), 0; \vlambda^1) \times \cdots \times Z(0, \vN^r+\Os(E^r)+ \Os(F^r), 0; \vlambda^r)
    \]
    inside
    \[
      B_Y=\R^{a_Y} \times \R_+^{r-1} \times Z_{\vN^1+\Os(E^1)+\Os(F^1)}\times\dots \times Z_{\vN^1+\Os(E^1)+\Os(F^1)}.
    \]
    (Here
    $a_Y=\dim Y-\sum_i\dim Z(0,\vN^1+\Os(E^1)+ \Os(F^1), 0;
    \vlambda^1)$.) In more detail, for any point $y\in Y$, let $f_y$
    be the isomorphism from the local model of $\iota(Y)$ inside $U$
    to the local model of $A_Y$ inside $B_Y$; then this identification may
    be recorded by just remembering the pullback (under $f_y$) of the
    standard frame of $A_Y$ inside $B_Y$. This framing of the normal
    bundle of $Y$ inside $U$ is called the \emph{internal framing}. 
    Note that the internal framings are defined separately on the open
    strata, since different strata have different dimensions, and
    hence a different number of vectors in their internal framings.
  \item The equivalence relation on these subspaces $U$ is generated
    by the following atomic relation: $U\sim U'$ if $U\subset U'$, the
    embedding $X\into U'$ is induced from the embedding $X\into U$,
    the stratification of $U$ is induced from the stratification of
    $U'$, and the local models of each open stratum $Y$ inside $U$ and
    $U'$ are identified.
  \end{enumerate}
\end{definition}

\begin{remark}\label{rem:same-local-models}
  Definition~\ref{def:neatmoduli} is inspired from the local models
  presented in Section~\ref{sec:modelsmoduli}. The thickening $U$
  corresponds to a ``tubular'' neighborhood of $X$ in the larger space
  $\bModuli([\tD])$, where $\tD=D +\tE+\tF$ is as in
  Section~\ref{sec:modelsmoduli}. For each open stratum $Y$, the local
  model of $Y$ in $U$ is identified with that of $A_Y$ inside $B_Y$,
  which is the same as the local model of $Y$ inside
  $\bModuli([\tD])$.
\end{remark}

We will also require these neat embeddings to satisfy various
compatibility relations.
\begin{definition}\label{def:compatible-neat}
  Let $S$ be a downward closed subset of $\DNlambda$, with respect to
  the partial order from Equation~\eqref{eq:DNlamba-poset}.  A
  \emph{coherent neat embedding} of $S$ consists of neat embeddings of
  each $(D,\vN,\vlambda)\in S$, satisfying the following
  compatibility conditions. Fix $(D,\vN,\vlambda)\in S$, and
  let $X=\bModuli_{\vN,\vlambda}(D)$, $l=\tdim X$, and
  $U=U_{\vN,\vlambda}(D)$.  For any open stratum $Y$ of of the form
  \[
    Y= \Moduli_{\vN^1+\coefficients{E^1}+\coefficients{F^1}, \vlambda^1}(D^1) \times \dots \times
    \Moduli_{\vN^r+\coefficients{E^r}+\coefficients{F^r}, \vlambda^r}(D^r),
  \]
  set
  $Y_i=\Moduli_{\vN^i+\coefficients{E^i}+\coefficients{F^i},
    \vlambda^i}(D^i)$, $l_i=\tdim Y_i$, and
  $U_i=U_{\vN^i+\coefficients{E^i}+\coefficients{F^i},\vlambda^i}(D^i)$. Then
  \begin{itemize}
  \item There exists thickening $V_i$ of $\ol{Y}_i$, equivalent to
    $U_i$, and an open subset $V$ of $U$ containing $\ol{Y}$, such
    that
    \begin{equation}\label{eq:thickening-compatibility}
      V_1\times[0,\epsilon)\times V_2\times[0,\epsilon)\times\dots\times[0,\epsilon)\times V_r=V
    \end{equation}
    (for some small $\epsilon>0$), as subsets of
    $\E^d_{l_1}\times\RR_+\times\E^d_{l_2}\times\RR_+\times\dots\times\RR_+\times\E^d_{l_r}=\E^d_l$;
    moreover, the stratification of $V$ induced from that on $U$ 
    agrees with the product stratification of
    $V_1\times[0,\epsilon)\times\dots\times[0,\epsilon)\times V_r$.
  \item The embedding $\iota_X$, restricted to $\ol{Y}$, is the product embedding
    \(
      \big(\iota_{\ol{Y}_1},\dots,\iota_{\ol{Y}_r}\big),
    \)
    composed with the inclusion
    \[
      \E^d_{l_1}\times\E^d_{l_2}\times\dots\times\E^d_{l_r}\cong \E^d_{l_1}\times\{0\}\times\E^d_{l_2}\times\{0\}\times\dots\times\{0\}\times\E^d_{l_r}\into \E^d_l.
    \]
  \item For any point $y=(y_1,\dots,y_r)$ in $Y$, let $f_y$ be the
    identification of the local model of $Y$ inside $V$ with that of
    $A_Y$ inside $B_Y$ from Item~\ref{item:local-model-identification}
    of Definition~\ref{def:neatmoduli}.
    \begin{itemize}
    \item The restriction of $f_y$ to the local model of $Y$ inside
      $X$ should be an identification with the local model of $A_Y$ inside
      \[
        \RR^{a_Y}\times\R_+^{r-1}\times\prod_i\bZ(\Os(E^i),\vN^i,\Os(F^i);\veta^i).
      \]
      (Here $\veta^i$ is the vector of partitions from Equation~\eqref{eq:veta-partition}.)
    \item Identifying $V$ with
      $V_1\times[0,\epsilon)\times\dots\times[0,\epsilon)\times V_r$,
      the restriction of $f_y$ to the local model of $Y$ inside
      $V_1\times\{0\}\times\dots\times\{0\}\times V_r$ should be an
      identification with the local model of $A_Y$ inside
      \[
        \R^{a_Y} \times\{0\}\times \prod_iZ_{\vN^i+\Os(E^i)+ \Os(F^i)};
      \]
      moreover, this identification is required to agree with the
      product identification of $f_{y_i}$ from the local model of
      $Y_i$ inside $V_i$ with that of $A_{Y_i}$ inside $B_{Y_i}$.  In
      other words, the internal framing of $Y$ inside $V$ is obtained
      by concatenating the $(r-1)$ standard unit vectors in the
      $[0,\epsilon)$ factors of
      $V_i\times[0,\epsilon)\times\dots\times[0,\epsilon)\times V_r$
      and the internal framings of $Y_i$ inside $V_i$.
    \end{itemize}
  \end{itemize}
\end{definition}

In addition to internal framings (which is part of the data of neat
embeddings), we also have a notion of external framings for neat
embeddings.

\begin{definition}
  Suppose we have a neat embedding of $X=\bModuli_{\vN, \vlambda}(D)$
  into $\E^d_l$, with associated thickening $U$. An {\em external
    framing} of $X$ is a framing of the $d(l+1)$-dimensional normal bundle
  to $U$ in $\E^d_l$; in other words, a smoothly varying, ordered
  basis for a complement of $TU$ in $T\E^d_l$
  (see Convention~\ref{conv:framings}).
\end{definition}

Once again, we have a notion of compatibility of external framings.
\begin{definition}\label{def:compatible-ext-framing}
  Let $S$ be a downward closed subset of $\DNlambda$, along with a
  coherent neat embedding. A \emph{coherent external framing} of $S$
  consists of external framings of $\bModuli_{\vN,\vlambda}(D)$ for
  each $(D,\vN,\vlambda)\in S$ satisfying the following. Fix
  $(D,\vN,\vlambda)\in S$, and let
  $X=\bModuli_{\vN,\vlambda}(D)$ and $U=U_{\vN,\vlambda}(D)$.  For any
  open stratum
  $Y=\prod_i\Moduli_{\vN^i+\coefficients{E^i}+\coefficients{F^i},
    \vlambda^i}(D^i)$ of $X$, let
  $Y_i=\Moduli_{\vN^i+\coefficients{E^i}+\coefficients{F^i},
    \vlambda^i}(D^i)$ and
  $U_i=U_{\vN^i+\coefficients{E^i}+\coefficients{F^i},\vlambda^i}(D^i)$. Let
  $V_i$ be a thickening of $\ol{Y}_i$ equivalent to $U_i$,
  and let $V$ be the open subset of $U$ containing $\ol{Y}$ satisfying
  Equation~\eqref{eq:thickening-compatibility}. The the framing of the
  normal bundle of $U$, restricted to $V$, is the product framing of
  the normal bundles of $V_i$.
\end{definition}

\begin{example}
\label{ex:neathex2}
If $X = \Moduli_0(D)$ and $D$ does not contain any row or column (as
in Example~\ref{ex:nodegs}), then $X$ is a
$\langle k \rangle$-manifold, the thickening $U$ is just $X$ itself,
and the notion of neat embedding coincides with that for
$\langle k \rangle$-manifolds given in Section~\ref{sec:neatk}. For
example, the hexagon moduli space from Figure~\ref{fig:Dind3} is a
$\langle 2 \rangle$-manifold, and Figure~\ref{fig:neathex} shows a
neat embedding of that hexagon. The procedure we will use in
Section~\ref{sec:construction} to construct such a neat embedding with
external framings will be as follows. We start by choosing embeddings
of the zero-dimensional moduli spaces $\Moduli_0(A)$, $\Moduli_0(B)$,
and $\Moduli_0(C)$ in $\RR^d$ along with framings of their normal
bundles. The six black dots on the $\RR^{3d}$ line are products of
these moduli spaces (with product normal framings), corresponding to
permutations in the order in which they appear as vertices in
Figure~\ref{fig:Dind3}:
\[
  \Moduli_0(A) \times \Moduli_0(B) \times \Moduli_0(C),\ \ \Moduli_0(A) \times \Moduli_0(C) \times \Moduli_0(B), \dots
\]
We will then construct neat embeddings of the one-dimensional moduli
spaces $\Moduli_0(A+B)$, $\Moduli_0(B+C)$ and $\Moduli_0(A+C)$ in
$\RR_+\times\RR^{2d}$, along with a normal framing that extends the
framing on the boundary. By taking products of the zero- and
one-dimensional moduli spaces we get neat embeddings (with product
normal framings) of the edges of the hexagon. Finally, we extend to an
embedding of the entire hexagon, and frame its normal bundle. (This
procedure automatically ensures that the neat embeddings and external
framings are coherent.)
\end{example}

\begin{example}
  In Figure~\ref{fig:neatrow}, we show neat embeddings for the moduli
  spaces from Figure~\ref{fig:row}, corresponding to the row $H_i=A+B$
  and the column $V_i=C+D$, along with their external framings. The
  two moduli spaces are glued at the blue point
  $\bModuli_{1, (1)}(c_x)$ where we also specify a thickening of that
  point (the orange interval) and the internal frame at that point
  (the red vector). The two halves of the orange interval (or an
  equivalent smaller thickening) are the tubular neighborhoods of the
  blue point inside the two moduli spaces.
\end{example}
\begin{figure}
  \centering
  \begin{tikzpicture}[scale=2.5,x={(1cm,0)},z={(0,1cm)},y={(-0.4cm,-0.6cm)}]
    \draw[->] (-0.5,0)--(2,0) node[pos=1,anchor=west] {\small $\RR^d$};
    \draw[->] (0,-0.5)--(0,2)  node[pos=1,anchor=north east] {\small $\RR^d$};
    \draw[->] (0,0,0)--(0,0,1) node[pos=1,anchor=south] {\small $\RR_+$};

    \coordinate (A) at (0,0.4);
    \coordinate (C) at (0,1.8);
    \coordinate (B) at (0.4,0);
    \coordinate (D) at (1.8,0);
    \coordinate (AB) at (0.4,0.4);
    \coordinate (CD) at (1.8,1.8);

    \coordinate (O) at (1.1,1.1,1.5);

    \node[anchor=south east] at (A) {\small $A$};
    \node[anchor=south east] at (C) {\small $C$};
    \node[anchor=south] at (B) {\small $B$};
    \node[anchor=south] at (D) {\small $D$};

    \draw[dashed] (A)--(AB)--(B) (C)--(CD)--(D);

    \draw[thick] (AB) ..controls($(AB)+(0,0,0.3)$) and ($(O)+(-0.3,-0.3,0)$).. coordinate[pos=0.5] (H) ($(O)+(-0.1,-0.1,0)$)--++(0.2,0.2,0) ..controls ($(O)+(0.3,0.3,0)$) and ($(CD)+(0,0,0.3)$) ..  coordinate[pos=0.5] (V) (CD);

    \draw[line width=3pt,orange,opacity=0.7] (O)++(-0.3,-0.3)--($(O)+(0.4,0.4)$);    
    
    \draw[ultra thick,-latex,black!50] (AB)--++(0,0.3) node[pos=1,anchor=north east,inner sep=0,outer sep=0] {\small $1$}; 
    \draw[ultra thick,-latex,black!50] (AB)--++(0.3,0) node[pos=1,anchor=west,inner sep=0,outer sep=0] {\small $2$};

    \draw[ultra thick,-latex,black!50] (O)--++(0,0.3) node[pos=1,anchor=north east,inner sep=0,outer sep=0] {\small $1$}; 
    \draw[ultra thick,-latex,black!50] (O)--++(0,0,-0.3) node[pos=1,anchor=north,inner sep=0,outer sep=0] {\small $2$};

    \draw[ultra thick,-latex,black!50] (CD)--++(0,0.3) node[pos=1,anchor=north east,inner sep=0,outer sep=0] {\small $1$}; 
    \draw[ultra thick,-latex,black!50] (CD)--++(-0.3,0) node[pos=1,anchor=east,inner sep=0,outer sep=0] {\small $2$};

    \draw[ultra thick,-latex,black!50] (H)--++(0,0.3) node[pos=1,anchor=north east,inner sep=0,outer sep=0] {\small $1$}; 
    \draw[ultra thick,-latex,black!50] (H)--++(0.15,0.15,-0.15) node[pos=1,anchor=north,inner sep=0,outer sep=0] {\small $2$};

    \draw[ultra thick,-latex,black!50] (V)--++(0,0.3) node[pos=1,anchor=north east,inner sep=0,outer sep=0] {\small $1$}; 
    \draw[ultra thick,-latex,black!50] (V)--++(-0.3,-0.3,-0.3) node[pos=1,anchor=east,inner sep=0,outer sep=0] {\small $2$};

    \draw[ultra thick,-latex,red] (O)--++(0.3,0.3,0);
    \node[blue!60] at (O) {\Large $\bullet$};
   
    \foreach \a in {A,B,C,D,AB,CD}{\node at (\a) {$\bullet$};}

  \end{tikzpicture}
  \caption {Neat embeddings for the moduli spaces in
    Figure~\ref{fig:row}. The orange interval is the thickening of the
    blue dot, with the internal frame shown by the red arrow. The gray
    arrows (ordered $1$ and $2$) indicate the external framings.}
\label{fig:neatrow}
\end{figure}

\begin{example}
  In Figure~\ref{fig:neatfig} we show a neat embedding for the moduli
  space $\bModuli_0(2C+D)$ from Figure~\ref{fig:rowplus}, together
  with an external framing. We also show the thickening of the special
  (blue) boundary. A tubular neighborhood of this blue edge inside the
  moduli space $\bModuli_0(2C+D)$ is half of this thickening (or an
  equivalent smaller one). The other half of the thickening would be a
  tubular neighborhood of the blue edge inside the moduli space
  $\bModuli_0(A+B+C)$ from Figure~\ref{fig:rowplus}.
\end{example}

\begin{figure}
  \centering
  \begin{tikzpicture}[scale=2.5]

    \begin{scope}[x={(1cm,0)},z={(0,1cm)},y={(-0.4cm,-0.6cm)}]

      \coordinate (A) at (0,1,1.2);
      \coordinate (Ashift) at ($(A)+(0,0,-0.3)$);
      \coordinate (Aother) at ($(A)+(0,0,0.3)$);
      \coordinate (B) at (0,1.6);
      \coordinate (C) at (1.3,1.4);
      \coordinate (Cshift) at ($(C)+(-0.2,0.2)$);
      \coordinate (Cother) at ($(C)+(0.2,-0.2)$);
      \coordinate(ABC) at ($(B)+(0.4,-0.3,0.3)$);

      \draw[->] (0,0)--(1,0) node[pos=1,anchor=west] {\small $\RR_+$};
      \draw[->] (0,-0.5)--(0,2)  node[pos=1,anchor=north east] {\small $\RR^{3d}$};
      \draw[->] (0,0,0)--(0,0,1) node[pos=1,anchor=south] {\small $\RR_+$};

      \fill[black!25,opacity=0.5] (A)--(Ashift) ..controls ($(Ashift)+(0,0,-0.5)$) and ($(B)+(0,0,0.3)$).. coordinate[midway] (AB) (B)
      ..controls($(B)+(0.3,0)$) and ($(Cshift)+(-0.2,0.2)$).. coordinate[midway] (BC) (Cshift)--(C)
       ..controls($(C)+(0,0,0.5)$)and ($(A)+(0.5,0)$).. coordinate[midway] (AC) (A);

      \fill[orange,opacity=0.7] (Cshift)--(Cother) ..controls ($(Cother)+(0,0,0.5)$) and ($(Aother)+(0.5,0)$)..(Aother)
      --(Ashift) ..controls($(Ashift)+(0.5,0)$) and ($(Cshift)+(0,0,0.5)$).. coordinate[midway] (ACshift) (Cshift);

      \draw[thick] (A)--(Ashift) ..controls ($(Ashift)+(0,0,-0.5)$) and ($(B)+(0,0,0.3)$) ..(B)
      ..controls($(B)+(0.3,0)$) and ($(Cshift)+(-0.2,0.2)$).. (Cshift)--(C);

      \draw[ultra thick,blue!60] (C) ..controls($(C)+(0,0,0.5)$)and ($(A)+(0.5,0)$)..(A);

      \draw[ultra thick,-latex,black!60] (B)--++(0,0.3);      
      \draw[ultra thick,-latex,black!60] (A)--++(0,0.3);      
      \draw[ultra thick,-latex,black!60] (C)--++(0.25,0.25);      
      \draw[ultra thick,-latex,black!60] (AB)--++(0,0.4,0.3);
      \draw[ultra thick,-latex,black!60] (BC)--++(0.15,0.3);
      \draw[ultra thick,-latex,black!60] (AC)--++(0.1,0.4);
      \draw[ultra thick,-latex,black!60]  (ABC)--++(0.2,0.85,0.5);
      
      \draw[ultra thick,-latex,red] (A)--++(0,0,-0.2);
      \draw[ultra thick,-latex,red] (C)--++(-0.15,0.15);
      \draw[ultra thick,-latex,red] (AC)--++(-0.16,0,-0.16);

      \node at (B) {$\bullet$};

      \node[blue!60] at (A) {\Large $\bullet$};
      \node[blue!60] at (C) {\Large $\bullet$};

    \end{scope}    
    
  \end{tikzpicture}
\caption {A neat embedding for the triangular moduli space from the
  right hand side of Figure~\ref{fig:rowplus}. The gray arrows
  indicate the external framing. The thickening of the special (blue)
  boundary is shown in orange, with internal framing indicated by the
  red arrows.}
\label{fig:neatfig}
\end{figure}

\begin{example}
  In Figure~\ref{fig:neatrowcolumn} we show a neat embedding for the
  front facet of Figure~\ref{fig:2row}, that is, the moduli space for
  the domain $AB$ with the disk $AB$ attached. The thickening of the
  green edge, along with its internal frames, is also shown. (The
  internal frames have two or three vectors at the $1$- and
  $0$-dimensional strata of the green edge, respectively.) The external frames
  cannot be shown due to severe dimension reduction.
\end{example}

\begin{figure}
  \centering
    \begin{tikzpicture}[scale=3, x={(1cm,-0.2cm)},z={(0,1cm)},y={(0.6cm,0.66cm)}]

      \coordinate (e1) at (0.8,0);
      \coordinate (e2) at (0,0.4);
      \coordinate (e3) at (0,0,0.8);

      \coordinate (p) at ($(e1)+(e2)$);
      \coordinate (p2) at ($(p)+(-0.2,0)$);
      \coordinate (p1) at ($(p)+(0,-0.2)$);
      \coordinate (pm2) at ($(p)+(0.2,0)$);
      \coordinate (pm1) at ($(p)+(0,0.2)$);
      \coordinate (p12) at ($(p)+(210:0.2)$);
      \coordinate (pm12) at ($(p)+(30:0.2)$);

      \coordinate (q) at ($(p)+(0,0,0.7)$);
      \coordinate (q12) at ($(q)+(235:0.2)$);
      \coordinate (qm12) at ($(q)+(45:0.2)$);

      \coordinate (pq) at ($(p)!0.5!(q)$);

      \fill[black!25] (e2) ..controls ($(e2)+(0,0,0.2)$) and ($(e3)+(0,0.1,-0.1)$).. (e3) ..controls ($(e3)+(0.2,0.2,0.1)$) and ($(q12)+(-0.15,-0.15)$) .. (q12)--(q) -- (p)--(e2);
      \fill[black!40] (e1) ..controls ($(e1)+(0,0,0.3)$) and ($(e3)+(0.5,0,0.1)$).. (e3) ..controls ($(e3)+(0.2,0.2,0.1)$) and ($(q12)+(-0.15,-0.15)$) .. (q12)--(q) -- (p)--(e1);

      \node[anchor=east] at (0,0) {\small $\RR^{4d}$};
      \draw[->] (0,0)--(1,0) node[pos=1,anchor=west] {\small $\RR_+$};
      \draw[->] (0,0)--(0,1.4)  node[pos=1,anchor=south] {\small $\RR_+$};
      \draw[->] (0,0,0)--(0,0,1) node[pos=1,anchor=south] {\small $\RR_+$};

      \fill[orange,opacity=0.7] (p12) arc (210:390:0.2) --(qm12) ..controls($(qm12)+(0,0,0.2)$) and ($(q12)+(0,0,0.4)$).. (q12) --(p12);

      \draw[ultra thin,orange!70!black] (p12) -- (q12) ..controls ($(q12)+(0,0,0.4)$) and ($(qm12)+(0,0,0.2)$).. (qm12) -- (pm12);
      \draw[ultra thin, orange!70!black] (p) circle (0.2);
      \draw[ultra thin] (p1) ..controls ($(p1)+(0,0,0.3)$) and ($(q12)+(0.1,-0.1)$).. (q12);
      \draw[ultra thin] (p2) ..controls ($(p2)+(0,0,0.3)$) and ($(q12)+(-0.05,0.05)$).. (q12);
      \draw[ultra thin] (p)--(pm1) ..controls ($(pm1)+(0,0,0.3)$) and ($(qm12)+(-0.1,0.1)$).. coordinate[pos=0.35] (pq1) (qm12);
      \draw[ultra thin] (p)--(pm2) ..controls ($(pm2)+(0,0,0.3)$) and ($(qm12)+(0.05,-0.05)$).. coordinate[pos=0.35] (pq2) (qm12);

      \draw[ultra thin, black!60] (e3) ..controls ($(e3)+(0.2,0.2,0.1)$) and ($(q12)+(-0.15,-0.15)$) .. (q12)--(q) -- (qm12);
      
      \draw[thick] (e1)--(p);
      \draw[thick] (e2)--(p);
      \draw[thick] (e1) ..controls ($(e1)+(0,0,0.3)$) and ($(e3)+(0.5,0,0.1)$).. (e3);
      \draw[thick] (e2) ..controls ($(e2)+(0,0,0.2)$) and ($(e3)+(0,0.1,-0.1)$).. (e3);

      \draw[green!70!black, ultra thick] (p)--(q);

      \draw[ultra thick,-latex,red] (p)--(pm1)  node[pos=1,outer sep=0,inner sep=0,anchor=south west] {\small 2};
      \draw[ultra thick,-latex,red] (p)--(pm2)  node[pos=1,outer sep=0,inner sep=2pt,anchor=west] {\small 3};
      \draw[ultra thick,-latex,red] (p)++(0.01,0.01)-- node[pos=1,outer sep=0,inner sep=0,anchor=south west] {\small 1} ++(0,0,0.2);

      \draw[ultra thick,-latex,red] (pq)--(pq1) node[pos=1,outer sep=0,inner sep=0,anchor=south west] {\small 1};
      \draw[ultra thick,-latex,red] (pq)--(pq2) node[pos=1,outer sep=0,inner sep=2pt,anchor=west] {\small 2};
      
      \draw[ultra thick,-latex,red] (q)--++(0,0,0.2) node[pos=1,outer sep=0,inner sep=2pt,anchor=south] {\small 2};
      \draw[ultra thick,-latex,red] (q)--++($(pq1)-(pq)$) node[pos=1,outer sep=0,inner sep=0,anchor=south west] {\small 1};
      \draw[ultra thick,-latex,red] (q)--++($(pq2)-(pq)$) node[pos=1,outer sep=0,inner sep=2pt,anchor=west] {\small 3};

      \node at (e1) {$\bullet$};
      \node at (e2) {$\bullet$};
      \node at (e3) {$\bullet$};
      \node[green!70!black] at (p) {\Large $\bullet$};
      \node[green!70!black] at (q) {\Large $\bullet$};

    \end{tikzpicture}
    \caption {A neat embedding for the front facet of the moduli space
      shown in Figure~\ref{fig:2row}. The thickening of the green edge
      is shown in orange, along with its stratification (which is
      identical to the stratification from
      Figure~\ref{fig:local-models-example}). The internal frames
      along the green edge are shown by the (numbered) red arrows.}
    \label{fig:neatrowcolumn}
\end{figure}

\section{The embedded framed cobordism group}
\label{sec:group}
The obstruction classes that we will define during the construction
will naturally live inside a group
\[
  \tOmega^k_\fr=\colim_m \tOmega^k_{\fr, m},
\]
which we call the
\emph{embedded framed cobordism group}. 

Recall that the usual framed cobordism group
\[
  \Omega^k_\fr=\colim_m\Omega^k_{\fr,m}
\]
is defined as follows: the elements of $\Omega^k_{\fr, m}$ are the
equivalence classes of closed $k$-dimensional manifolds $M$ embedded
in $\RR^m$, together with a framing of the normal bundle; the
equivalence relation is given by framed cobordisms in
$\RR^m\times[0,1]$; and the group structure is
$[M_1]+[M_2]=[M \amalg M'_2]$, where $M'_2$ is a sufficiently large
translation of $M_2$. There is a natural map
$\sigma\from \Omega^k_{\fr,m}\to\Omega^k_{\fr,m+1}$, and
$\Omega^k_\fr$ is the colimit.

The group $\tOmega^k_{\fr, m}$ is defined similarly to $\Omega^k_{\fr,m}$,
except we require the framed cobordisms to also be embedded in
$\RR^m$. More precisely:
\begin{enumerate}[leftmargin=*,label=($\Omega$-\arabic*),ref=($\Omega$-\arabic*)]
\item \label{item:1}The elements of $\tOmega^k_{\fr, m}$ are the
  equivalence classes of closed $k$-dimensional manifolds $M$ embedded
  in $\RR^m$, together with a vector field $\vv$ (in $\RR^m$) along
  $M$ which is everywhere transverse to $TM$, and a framing of an
  $(m-k-1)$-dimensional complement of $TM\oplus\langle \vv\rangle$. We
  assume $m\geq 2k+3$. (Also, we will always follow
  Convention~\ref{conv:framings}: framings are not necessarily
  orthonormal, and complements are not necessarily orthogonal.)
\item\label{item:framed-cob-equivalence} The equivalence relation
  stipulates $(M_1,\vv_1)\sim (M_2,\vv_2)$ if there is an embedded
  framed cobordism in $\RR^m$ from $M_1$ to $M'_2$, which starts in
  the direction of $\vv_1$ and ends in the direction of $-\vv_2$.  Here,
  $M'_2$ is a translation of $M_2$ in a generic direction so that
  $M_1\cap M'_2=\varnothing$. We call a direction $\vec{e}\in S^{m-1}$
  {\em generic for} $(M,\vv)$ if  the projection $\pi\from \RR^m\to \RR^{m-1}$
  to the hyperplane perpendicular to $\vec{e}$ sends $M$ diffeomorphically unto an
  embedded submanifold of $\RR^{m-1}$ and $\vv$ to a vector field in
  $\RR^{m-1}$ along $\pi(M)$ which is everywhere transverse to the
  tangent space of $\pi(M)$. (A standard application of Sard's lemma
  shows that if $m\geq 2k+2$, then non-generic directions constitute a
  measure zero subset of $S^{m-1}$.) 
\item The group structure on $\tOmega^k_\fr$ is given by
  $[(M_1,\vv_1)]+[(M_2,\vv_2)]=[(M_1,\vv_1) \amalg (M'_2,\vv_2)]$, where $M'_2$
  is a translation of $M_2$ in a generic direction, as above.
\item \label{item:4}The zero element is the empty submanifold, and negation is given
  by reversing $\vv$, that is, $-[(M,\vv)]=[(M,-\vv)]$.
\end{enumerate}

The above definition deserves some justification, specifically to
show that $\sim$ defines a well-defined equivalence relation, and
that $(M,-\vv)$ is the inverse of $(M,\vv)$. The following lemmas are key.
\begin{lemma}\label{lemma:framed-cobordism-pushoff}
  Consider a framed $(M,\vv)$ as above and let $\vec{e}$ be a generic
  direction for $(M, \vv)$. Let $M'$ denote a pushoff in the direction of
  $\vec{e}$. Then there is an embedded framed cobordism (as
  described in Item~\ref{item:framed-cob-equivalence}) from $(M,\vv)$
  to $(M',\vec{e})$, for some normal framing of $(M',\vec{e})$.
\end{lemma}

\begin{proof}
  By rescaling if necessary, we can assume the pushoff $M'$ of $M$ is by the unit vector $\ve$. 
  
  Fix a smooth embedding $\gamma\from [0, 1] \to \R^2$ such that
  \[
    \gamma(0)=(0, 0), \ \ \gamma(1)=(1, 0), \ \ \gamma'(0)=(0, 1), \ \ \gamma'(1)=(1,0)
  \]
  and such that the image of $\gamma$ is contained in the strip $S_{\epsilon} = [0,1] \times [0,\epsilon)$, for $\epsilon > 0$ small. 

  For every $p \in M$, let $V_p = \operatorname{Span} (\ve, \vv_p)$ and let $A_p\from \R^2 \to V_p$ be the linear isomorphism that takes $(1, 0)$ to $\ve$ and $(0, 1)$ to $\vv_p$. If $\epsilon$ is sufficiently small, the genericity condition on $\ve$ guarantees that the union of all $p+A_p(S_{\epsilon})$ forms a smoothly embedded bundle over $M$, with fiber $S_{\epsilon}$. Then, the map 
  \[
    f\from [0, 1] \times M \to \R^{m}, \ \ f(t, p) = p+ A_p \circ \gamma(t)
  \]
  describes a smoothly embedded cobordism $S$ from $(M, \vv)$ to $(M', \ve)$. We
  can also choose a normal framing on this cobordism, which agrees
  with the given framing at $(M,\vv)$.  Figure~\ref{fig:framed-cobordism-pushoff}
  illustrates the proof.
\end{proof}

\begin{figure}
  \centering

  \begin{tikzpicture}[scale=2]

    \begin{scope}[xshift=-4cm]

      \coordinate (O) at (0,0);
      \coordinate (One) at (1,0);
      \coordinate (mid) at (0.5,0);

      \fill[black!30] (O) rectangle ($(One)+(0,0.4)$) node[pos=1,anchor=north east,color=black,inner sep=0,outer sep=1] {$S_{\epsilon}$};
      
      \draw[->] (-0.2,0) -- (1.2,0);
      \draw[->] (0,-0.2) -- (0,1.2);
      
      \node at (O) {$\bullet$};
      \node at (One) {$\bullet$};
      \node[anchor=north east] at (O) {$(0,0)$};
      \node[anchor=north] at (One) {$(1,0)$};
      
      \draw[thick] (O) ..controls (0,0.4) and (mid) .. node[pos=0.7,anchor=south,inner sep=0,outer sep=1] {$\gamma$} (One);

      \draw[->,thick] (2,0.2) -- ++(1,0) node[midway,anchor=south] {$A_p$};
    \end{scope}

    \coordinate (O) at (0,0);
    \coordinate (One) at (1,0);
    \coordinate (mid) at (0.5,0);

    \draw[thick,-latex,dashed] (One) --++ (1,0) node[pos=1,anchor=west] {$\vec{e}$};
    \draw[thick,-latex,dashed] (O)--(-1*0.6,2*0.6) node[pos=1,anchor=east] {$\vv$};

    \draw[thick] (O) ..controls (-1,2) and (mid).. coordinate[pos=0.33] (first) coordinate[pos=0.67] (last) (One);

    \draw[thick,-latex] (O) --++(210:0.4);
    \draw[thick,-latex] (first) --++(90:0.4);
    \draw[thick,-latex] (last) --++(60:0.4);   
    \draw[thick,-latex] (One)--++(90:0.4);
    
    \node at (O) {\Large $\bullet$};
    \node at (One) {\Large $\bullet$};

    \node[anchor=north] at (O) {$M$};
    \node[anchor=north] at (One) {$M'$};

  \end{tikzpicture}
  \caption{A picture illustrating the proof of
    Lemma~\ref{lemma:framed-cobordism-pushoff}; the notation is same
    from the lemma. The vector fields $\vv$ and $\vec{e}$ (at $M$ and
    $M'$) are shown by dashed arrows, and the normal framings are shown
    by solid arrows.}\label{fig:framed-cobordism-pushoff}
\end{figure}

\begin{lemma}\label{lemma:framed-cobordism-pushoff2}
  As in Lemma~\ref{lemma:framed-cobordism-pushoff}, consider $(M,\vv)$,
  a generic direction $\vec{e}$ for $(M, \vv)$, and a pushoff $M'$ in the direction
  of $\vec{e}$. Then there is an embedded framed cobordism from
  $(M,\vv)$ to $(M',\vv)$ (as described in
  Item~\ref{item:framed-cob-equivalence}), with the normal framing
  on $(M',\vv)$ the same as the given normal framing on $(M,\vv)$.
\end{lemma}

\begin{proof}
  Let $N$ be another pushoff of $M$ in the direction of $\vec{e}$; assume
  this pushoff is much smaller compared to the given pushoff
  $M'$. Let $N'$ be the symmetric pushoff of $M'$ in the direction
  of $-\vec{e}$. By Lemma~\ref{lemma:framed-cobordism-pushoff},
  there is an embedded framed cobordism $F$ from $(M,\vv)$ to
  $(N,\vec{e})$, for some normal framing of $(N,\vec{e})$. Consider
  the symmetric cobordism $F'$ from $(M',-\vv)$ to $(N',-\vec{e})$,
  and view it as a cobordism from $(N',\vec{e})$ to $(M',\vv)$. The
  framing of $F$ induces a framing of $F'$ by symmetry; in
  particular, the normal framings on $(N,\vec{e})$ and 
  $(N',\vec{e})$ agree, and the normal framings on $(M,\vv)$ and
  $(M',\vv)$ agree.

  Simply by translating along the $\vec{e}$ direction, we
  get an embedded framed cobordism $S$ from $(N,\vec{e})$ to
  $(N',\vec{e})$. Then the union $F\cup S\cup F'$ is a framed
  cobordism from $(M,\vv)$ to $(M',\vv)$ as required.
\end{proof}

\begin{lemma} \label{lemma:minus} Given a framed $(M, \vv)$ with a
  generic direction $\ve$, let $M'$ be a pushoff of $M$ in the
  direction of $\ve$, and let $(-M', -\vv)$ be obtained from
  $(M', -\vv)$ by changing the sign of one of the framing
  vectors. Then, there exists an embedded framed cobordism from
  $(M, \vv)$ to $(-M', -\vv)$.
\end{lemma}

\begin{proof}
  Note that when flowing $M$ we are allowed to continuously deform its
  framing. Thus, without loss of generality, we can assume that one of
  the framing vectors of $M$, call it $\vw$, lies in the plane spanned
  by $\ve$ and $\vv$; in fact, we can assume it to be perpendicular to
  $\vv$ in that plane.

  Let $M''$ be a smaller pushoff of $M$ in the direction of $\ve$, so
  that $M''$ is intermediate between $M$ and $M'$. Consider the
  cobordism $S$ from $(M, \vv)$ to $(M'', \ve)$ defined in
  Lemma~\ref{lemma:framed-cobordism-pushoff}, and compose it with the
  reverse of the cobordism from $(M', \vv)$ to $(M'', -\ve)$, provided
  by the same lemma. Altogether, we get a cobordism from $(M, \vv)$ to
  $(M', -\vv)$. Furthermore, we can choose one of the vector fields in
  the framing to be perpendicular to $S$ in the plane spanned by $\vv$
  and $\ve$, and let the other vectors stay constant. Following the
  framing, we see that the distinguished framing vector $\vw$ gets
  turned into $-\vw$; see Figure~\ref{fig:doublecob}. Thus, when
  taking into account the framing, the end of the cobordism is
  $(-M', -\vv)$.
\end{proof}

\begin{figure}
  \centering
  \begin{tikzpicture}[scale=2,xscale=1.2]
    \coordinate (O) at (0,0);
    \coordinate (One) at (1,0);
    \coordinate (Two) at (2,0);
    
    \draw[thick,-latex,dashed] (One) --++ (0.5,0) node[pos=0.8,anchor=north] {$\vec{e}$};
    \draw[thick,-latex,dashed] (O)--++(-1*0.5,2*0.5) node[pos=1,anchor=east] {$\vv$};
    \draw[thick,-latex,dashed] (Two)--++(1*0.5,-2*0.5) node[pos=1,anchor=west] {$-\vv$};

    \draw[thick] (O) ..controls (-1,2) and ($(O)!0.5!(One)$).. (One);
    \draw[thick] (Two) ..controls ($(Two)+(-1,2)$) and ($(Two)!0.5!(One)$).. (One);

    \draw[thick,-latex] (O) --++(210:0.4) node[anchor=east] {$\vw$};
    \draw[thick,-latex] (One)--++(90:0.4);
    \draw[thick,-latex] (Two)--++(30:0.4) node[anchor=west] {$-\vw$};
    
    \node at (O) {\Large $\bullet$};
    \node at (One) {\Large $\bullet$};
    \node at (Two) {\Large $\bullet$};

    \node[anchor=north west] at (O) {$M$};
    \node[anchor=north] at (One) {$M''$};
    \node[anchor=north east] at (Two) {$M'$};
  \end{tikzpicture}\\
  \caption {The cobordism in Lemma~\ref{lemma:minus}. Following the
    same convention from Figure~\ref{fig:framed-cobordism-pushoff},
    the vector fields $\vv$, $-\vv$, $\ve$ (at $M$, $M'$, $M''$) are
    drawn with dashed arrows, and the normal vector field is shown by
    solid arrows.}
  \label{fig:doublecob}
\end{figure}

Armed with these lemmas, we can prove:

\begin{proposition}
  Items \ref{item:1}-\ref{item:4} make $\tOmega^k_{\fr,m}$ into a well-defined Abelian group.  
\end{proposition}

\begin{proof}
  Let us start by showing that the relation $\sim$ is well-defined, i.e., it does not depend on which translation we choose in \ref{item:framed-cob-equivalence}. Consider $(M,\vv)$ and $(N,\vw)$, and assume there is an embedded framed
  cobordism $S$ from $(M,\vv)$ to $(N',\vw)$ for some generic pushoff
  $N'$. If $N''$ is another generic pushoff, then by
  Lemma~\ref{lemma:framed-cobordism-pushoff2}, there are embedded
  framed cobordisms $F$ from $(N',\vw)$ to $(N,\vw)$ and $F'$ from $(N,\vw)$
  to $(N'',\vw)$. The union $S\cup F\cup F'$ then is an immersed framed
  cobordism from $(M,\vv)$ to $(N'',\vw)$. However, since we assumed
  $m\geq 2k+3$, by perturbing the cobordism in the interior, we may
  assume it is embedded. 
  
  The proof that the relation $\sim$ is transitive is similar to the above argument. The statement that $\sim$ is reflexive is
  same as the statement that $(M,-\vv)$ is the inverse of $(M,\vv)$, and
  it is Lemma~\ref{lemma:framed-cobordism-pushoff2}. To see that $\sim$ is symmetric, note that if $(M, \vv)\sim (N, \vw)$, by reversing the cobordism and its framing we get that $(-N, -\vw)\sim (-M, -\vv)$; applying Lemma~\ref{lemma:minus}, we deduce that $(N, \vw) \sim (M, \vv)$.
  
  It is also not 
  hard to check that the group operation
  $[(M_1,\vv_1)]+[(M_2,\vv_2)]=[(M_1,\vv_1) \amalg (M'_2,\vv_2)]$ is
  well-defined, and commutative.
\end{proof}

There is a natural stabilization map 
\begin{equation}
  \label{eq:stabomega}
  \sigma\from  \tOmega^k_{\fr,m}\to \tOmega^k_{\fr,m+1}
\end{equation}
defined as follows. Given $(M,\vv)$ inside $\RR^m$ along with a
normal framing $\langle \vw_1,\dots,\vw_{m-k-1}\rangle$, consider it as
an element of $\tOmega^k_{\fr,m+1}$ by considering $M\times\{0\}$
inside $\RR^m\times\{0\}\subset \RR^{m+1}$, using the same vector
field $\vv$, and using the normal framing
$\langle \vw_1,\dots,\vw_{m-k-1},\vec{e}\rangle$, where $\vec{e}$ is the
positive unit normal vector in the new $\RR$ direction. 

We define $\tOmega^k_\fr$ to be the colimit of the groups $\tOmega^k_{\fr,m}$ under the
maps $\tOmega^k_{\fr,m}\to \tOmega^k_{\fr,m+1}$. It is worth comparing this new
group $\tOmega^k_\fr=\colim_m \tOmega^k_{\fr,m}$ with
the usual framed cobordism group
$\Omega^k_\fr=\colim_m \Omega^k_{\fr,m}$. 

\begin{proposition}
  \label{prop:compareomegas}
  The groups $\tOmega^k_\fr$ and $\Omega^k_\fr$ are isomorphic.
\end{proposition}

First we need another lemma.

\begin{lemma}
  \label{lem:rotation}
  Consider a framed $(M, \vv)$ as in the definition of
  $\tOmega^k_{\fr, m}$, and let $\vw$ be one of the vector fields in
  the framing of $M$. Let $(M, -\vw)$ be framed by replacing the
  vector field $\vw$ with $\vv$. Then, $(M, \vv)$ and $(M, -\vw)$ map
  to the same element under the stabilization map
  $\tOmega^k_{\fr, m} \to \tOmega^{k}_{\fr, m+1}$.  Hence, $(M, \vv)$
  and $(M, -\vw)$ represent the same element in $\tOmega^k_{\fr}$.
\end{lemma}

\begin{proof}
  Let $\ve$ be the new unit coordinate vector in $\R^{m+1}$, normal to
  $\R^m$. Under the stabilization map, we identify $M \subset \R^m$
  with $M \times \{0\} \subset \R^{m+1}$, and we add $\ve$ to the
  normal framings of $(M, \vv)$ and $(M, -\vw)$. Let $M'$ be the
  pushoff of $M$ in the direction $\ve$. Note that $\ve$ is a generic
  vector for $(M, \vv)$ in $\R^{m+1}$. Thus, it suffices to construct
  an embedded framed cobordism from $(M, \vv)$ to $(M', -\vw)$ in
  $\R^{m+1}$, which we do as follows.

  The argument is similar to that in the proof of Lemma~\ref{lemma:framed-cobordism-pushoff}. Fix a smooth embedding $\gamma\from [0, 1] \to \R^2$ such that
  \[
    \gamma(0)=\gamma(1)=(0, 0), \ \ \gamma'(0)=(1, 0), \ \ \gamma'(1)=(0, -1)
  \]
  and such that the image of $\gamma$ is contained in the ball $B(\epsilon)$ of radius $\epsilon$ around the origin, for $\epsilon > 0$ small. We let $\gamma^{\perp}$ be the normal vector field to the image of $\gamma$, obtained from $\gamma'$ by a counterclockwise rotation by $90^\circ$. For example, $\gamma^{\perp}(0)=(0, 1)$ and $\gamma^{\perp}(1)=(1,0)$.

  Fix also a smooth map $\zeta\from [0, 1] \to [0, 1]$ with 
  \[
    \zeta(0)=0, \ \ \zeta(1)=1, \ \ \zeta'(0)=\zeta'(1)= 0 \text{ and }  \zeta'(t) > 0 \text{ for } t \in (0, 1).
  \]

  For every $p \in M$, let $V_p = \operatorname{Span} (\vv_p, \vw_p)$
  and let $A_p\from \R^2 \to V_p$ be the linear isomorphism that takes
  $(1, 0)$ to $\vv_p$ and $(0, 1)$ to $\vw_p$. If $\epsilon$ is
  sufficiently small, the union of all $p+A_p(B(\epsilon))$ forms a
  smoothly embedded disk bundle over $M$ in $\R^m$. Then, the map
  \[
    f\from [0, 1] \times M \to \R^{m+1}, \ \ f(t, p) = p+ A_p \circ \gamma(t) + \zeta(t)\cdot \ve
  \]
  describes a smoothly embedded cobordism $S$ from $(M, \vv)$ to
  $(M, -\vw)$. For the normal framing on $S$, we use the pushforward
  of $\gamma^{\perp}$ under $A_p$ to interpolate between $\vw$ and
  $\vv$ as $t$ goes from $0$ to $1$. We also keep $\ve$ as part of the
  normal framing throughout the cobordism.
\end{proof}


\begin{proof}[Proof of Proposition~\ref{prop:compareomegas}]
  There is a natural map
  $f\from \Omega^k_{\fr, m} \to \tOmega^k_{\fr, m+1}$. Given a framed
  manifold $M\subset \RR^m$,  we let $f(M)$ be the same manifold, viewed
  inside $\RR^m\times\{0\}\subset \RR^{m+1}$, the vector field $\vv$ be 
  the constant positive unit vector field in the new $\RR$ direction,
  and the normal framing be the original framing of $M$, multiplied by $(-1)^m$. It is
  immediate from the definitions that if $M\sim N$ in
  $\Omega^k_{\fr,m}$, then $f(M)\sim f(N)$ in $\tOmega^k_{\fr,m+1}$. 
  
  Next, consider the following diagram:
  \begin{equation}
    \label{eq:hey}
    \vcenter{\hbox{\xymatrix{
      \Omega^k_{\fr, m} \ar[r]^f \ar[d]_{\sigma} &\tOmega^k_{\fr, m+1} \ar[d]^{\sigma}\\
      \Omega^k_{\fr, m+1}\ar[r]^f & \tOmega^k_{\fr, m+2}
    }}}
  \end{equation}
  where the vertical arrows are stabilization maps. We claim that the
  diagram \eqref{eq:hey} commutes after one more stabilization, i.e.,
  $\sigma \circ \sigma \circ f = \sigma \circ f \circ \sigma$. Indeed,
  suppose we have a manifold $M \subset \RR^m$ framed by the sequence
  of vectors $( \vw_1,\dots,\vw_{m-k})$. Let $\ve_{m+1}$ and
  $\ve_{m+2}$ dente the two new unit vectors when we stabilize from
  $\R^m$ to $\R^{m+2}$. The images of
  $[M, ( \vw_1,\dots,\vw_{m-k})] \in \Omega^k_{\fr, m}$ under the two
  possible compositions in \eqref{eq:hey} are
  \[
    (-1)^m[(M, ( \vw_1,\dots,\vw_{m-k}, \ve_{m+1})), \ve_{m+2}]\ \ \text{ and }  \ \ (-1)^{m+1}[(M, ( \vw_1,\dots,\vw_{m-k}, \ve_{m+2})), \ve_{m+1}].
  \]
  These become identical after one more stabilization, as proved in
  Lemma~\ref{lem:rotation}. From here it follows that the maps $f$
  induce a well-defined map
  \[
    \Omega^k_{\fr} \to \tOmega^k_{\fr}
  \]
  on the colimits. 

  There is also a natural map
  \begin{equation}\label{eq:tomega-to-omega}
    g\from \tOmega^k_{\fr, m} \to \Omega^k_{\fr, m}
  \end{equation}
  defined as follows. Given $(M,\vv)$ inside $\RR^m$ along with a
  normal framing $( \vw_1,\dots,\vw_{m-k-1})$, we map it to $M$ with the normal framing 
  \[
    ( \vw_1,\dots,\vw_{m-k-1}, (-1)^{m+1}\vv).
  \]
  (In $\Omega^k_{\fr,m}$, this is equivalent to the normal framing
  \[
    ( (-1)^k\vv,\vw_1,\dots,\vw_{m-k-1})
  \]
  of $M$.)  To see that the map $g$ is well-defined, let us consider
  $(M, \vv)$ as an element of $\tOmega^k_{\fr,m+1}$ as in the
  definition of the stabilization map \eqref{eq:stabomega}.  Let
  $M'=M\times\{1\}\subset \RR^{m+1}$ be the unit pushoff in the new
  $\vec{e}$ direction. By Lemma~\ref{lemma:framed-cobordism-pushoff},
  there is an embedded framed cobordism $S(M, \vv)$ from
  $(M\times\{0\},\vv)$ to $(M\times\{1\},\vec{e})$ in $\RR^{m+1}$;
  indeed, the proof of the lemma shows that the cobordism lies inside
  $\RR^m\times[0,1]$. Furthermore, the induced normal framing of
  $(M\times\{1\},\vec{e})$ in $\RR^{m+1}$ is
  $( \vw_1, \dots, \vw_{m-k-1}, -\vv)$.  Now, if we have a framed
  cobordism $W$ from $(M_1,\vv_1)$ to $(M_2,\vv_2)$ in $\RR^m$, we can
  treat it as a cobordism inside $\RR^m\times\{\frac{1}{2}\}$, and
  compose with the reverse cobordism $S(M_1, \vv_1)^r$ from
  $(M_1,\vec{e})$ to $(M_1,\vv_1)$ (viewed inside
  $\RR^m\times[0,\frac{1}{2}]$) and with $S(M_2, \vv_2)$ from
  $(M_2,\vv_2)$ to $(M_2,\vec{e})$ (viewed inside
  $\RR^m\times[\frac{1}{2},1]$); see Figure~\ref{fig:tricob}. This
  produces a framed cobordism in $\RR^m\times[0,1]$ from
  $M_1 \times \{0\}$ to $M_2 \times \{1\}$, where the last framing
  vectors are $-\vv_1$ and $-\vv_2$, respectively. After multiplying
  the framing on this cobordism by $(-1)^m$, we get a framed cobordism
  from $g(M_1,\vv_1)$ to $g(M_2,\vv_2)$. Thus, $g$ is
  well-defined. The presence of the $(-1)^{m+1}$ factor in the
  definition of $g$ ensures that it commutes with the stabilization
  maps, producing a map $\tOmega^k_{\fr} \to \Omega^k_{\fr}$ in the
  colimit.


  \begin{figure}
    \centering
    \begin{tikzpicture}[scale=4,x={(0.4cm,0.2cm)},y={(0,1cm)},z={(1cm,0)}]

      \coordinate (m1left) at (0.3,0.3);
      \coordinate (m1mid) at ($(m1left)+(0,0,1)$);
      \coordinate (m2right) at (0.7,0.7,2);
      \coordinate (m2mid) at ($(m2right)+(0,0,-1)$);
      \coordinate (mid) at (0.5,0.5,1);

      \draw[gray] (0,0)-- node[midway,anchor=north,text=black,outer sep=2ex] {\small$\RR^m\times\{0\}$} ++(1,0)--++(0,1)--++(-1,0)--++(0,-1);
      \draw[gray] (0,0,1)-- node[midway,anchor=north,text=black,outer sep=2ex] {\small$\RR^m\times\{1/2\}$} ++(1,0)--++(0,1)--++(-1,0)--++(0,-1);
      \draw[gray] (0,0,2)-- node[midway,anchor=north,text=black,outer sep=2ex] {\small$\RR^m\times\{1\}$} ++(1,0)--++(0,1)--++(-1,0)--++(0,-1);

      \draw[thick] (m1mid) ..controls ($(m1mid)+(0.2,-0.2)$) and ($(m1left)+(0,0,0.2)$).. node[pos=0.5,anchor=south] {\small $S(M_1,\vv_1)^r$} (m1left);
      \draw[thick] (m1mid) ..controls ($(m1mid)+(-0.2,0.2)$) and ($(mid)+(-0.2,0.3)$)..(mid).. controls ($(mid)+(0.2,-0.3)$) and ($(m2mid)+(0,-0.2)$) .. node[pos=0.7,anchor=west,inner sep=0] {\small $W$} (m2mid);
      \draw[thick] (m2mid) ..controls ($(m2mid)+(0,0.2)$) and ($(m2right)+(0,0,-0.2)$).. node[pos=0.5,anchor=south] {\small $S(M_2,\vv_2)$} (m2right);

      \draw[thick,-latex,dashed] (m1left)-- node[pos=0.8,anchor=south] {$\ve$} ++(0,0,0.25);
      \draw[thick,-latex,dashed] (m2right)-- node[pos=0.8,anchor=south] {$\ve$} ++(0,0,0.25);
      \draw[thick,-latex,dashed] (m1mid)-- node[pos=1,anchor=south] {$\vv_1$} ++(-0.2,0.2);
      \draw[thick,-latex,dashed] (m2mid)-- node[pos=1,anchor=south] {$\vv_2$} ++(0,0.2);

      \draw[ultra thick,-latex] (m1left) -- node[anchor=west,pos=0.8] {$-\vv_1$} ++(0.2,-0.2);
      \draw[ultra thick,-latex] (m1mid) -- node[anchor=south,pos=0.8] {$\ve$} ++(0,0,0.25);
      \draw[ultra thick,-latex] (m2mid) -- node[anchor=north,pos=0.8] {$\ve$} ++(0,0,0.25);
      \draw[ultra thick,-latex] (m2right) -- node[anchor=north,pos=1] {$-\vv_2$} ++(0,-0.2);

      \draw[thick,-latex] (m1left) -- ++(-0.15,-0.15);
      \draw[thick,-latex] (m1mid) --  ++(-0.15,-0.15);
      \draw[thick,-latex] (m2mid) -- ++(-0.4,0);
      \draw[thick,-latex] (m2right) -- ++(-0.4,0);

      \node at (m1left) {\Large $\bullet$};
      \node at (m1mid) {\Large $\bullet$};
      \node at (m2mid) {\Large $\bullet$};
      \node at (m2right) {\Large $\bullet$};

      \node[anchor=south] at (m1left) {\small $M_1$};
      \node[anchor=north west] at (m1mid) {\small $M_1$};
      \node[anchor=south east] at (m2mid) {\small $M_2$};
      \node[anchor=south] at (m2right) {\small $M_2$};

    \end{tikzpicture}\\
    \caption {The composed cobordism from
      Proposition~\ref{prop:compareomegas}. Following earlier
      conventions, the vectors $\ve,\vv_1,\vv_2,\ve$ (at
      $M_1\subset\RR^m\times\{0\},M_1\subset\RR^m\times\{1/2\},M_2\subset\RR^m\times\{1/2\},M_2\subset\RR^m\times\{1\}$,
      respectively) are drawn dashed, and the normal vectors are drawn
      solid, of which, the last normal vectors
      ($-\vv_1,\ve_1,\ve_1,-\vv_2$, respectively) are drawn thicker.}
    \label{fig:tricob}
  \end{figure}

  It is immediate that the composition $g\circ f$ is the stabilization 
  $\Omega^k_{\fr,m}\to\Omega^k_{\fr,m+1}$. In the other direction,
  $(f\circ g)(M,\vv)$ is equivalent in $(M,\vv)$ in $\tOmega^k_{\fr,m+1}$
  using the framed cobordism $S(M, \vv)$ from $(M\times\{0\},\vv)$ to
  $(M\times\{1\},\vec{e})$ in $\RR^m\times[0,1]$. It follows that the maps induced by $f$ and $g$ on the colimits are inverse to each other.
\end{proof}

\section{Constructing the moduli spaces}
\label{sec:construction}
We will construct the stratified spaces
$\bModuli_{\vN, \vlambda}(D)$, along with their embeddings and
framings, inductively by dimension. We will usually denote their
dimension and thick dimension by $k$ and $l$; and these moduli
spaces will be embedded in $\E^d_l$. For the reader's convenience,
we first outline the procedure in Subsection~\ref{sec:outline}, and
then give more details in the following subsections.

\subsection{Outline}
\label{sec:outline}
Let $\DNlambdad\subset\DNlambda$ be the downward closed subset
consisting of all triples $(c_{\xid},\vN,\vlambda)$ where $\vN$ is
just made of $0$'s and $1$'s, and let
$\DNlambda'=\DNlambda\setminus\DNlambdad$.  Recall from
Proposition~\ref{prop:CDPhomology} that $\DNlambdad$ generates a
subcomplex $\CDPd \subset \CDP_*$ such that the quotient complex
$\CDP'_*$ is acyclic.

As a first step, we will construct the moduli spaces
$ \bModuli_{\vN, \vlambda}(c_{\xid})$, for all
$(c_{\xid},\vN,\vlambda)\in\DNlambdad$. This construction is
independent of what we do in this section, so we will describe it later (in Section~\ref{sec:permutohedron}).

After this, we will construct the remaining spaces
$\bModuli_{\vN, \vlambda}(D)$ inductively on their dimension $k=\gr(D,\vN,\vlambda)-1$. 

Let
$\DNlambda'_{\leq k}=\{(D,\vN,\vlambda)\in\DNlambda'\mid
\gr(D,\vN,\vlambda)-1\leq k\}$ and
$\DNlambda'_{k}=\DNlambda'_{\leq k}\setminus \DNlambda'_{\leq k-1}$.

\begin{enumerate}[leftmargin=*,label=(C-\arabic*)]
\item \label{item:emb} Assume the moduli spaces
  $\bModuli_{\vN,\vlambda}(D)$ have already been constructed for all
  $(D,\vN,\vlambda)\in \DNlambdad\cup\DNlambda'_{\leq k-1}$, along with
  coherent neat embeddings and coherent external framings.
\item \label{item:exponentiate} For any
  $(D,\vN,\vlambda)\in\DNlambda'_{k}$, its $(k-1)$-dimensional
  boundary $\bdy\bModuli_{\vN, \vlambda}(D)$ has already been
  constructed. This has a stratified thickening, and we use it to
  construct an open neighborhood $V$ of
  $\bdy\bModuli_{\vN, \vlambda}(D)$ inside
  $\bModuli_{\vN, \vlambda}(D)$ (with a coherent neat embedding and
  external framing).
\item \label{item:smooth} The neighborhood $V$ is a Whitney stratified space
  with compact boundary $\del V =\bdy\bModuli_{\vN, \vlambda}(D)$. By
  smoothing a smaller open neighborhood $V''$, we construct a new open
  neighborhood $V'$ of
  $\bdy\bModuli_{\vN, \vlambda}(D)$ in $\bModuli_{\vN,\vlambda}(D)$,
  such that $\ol{V'}\setminus V'$ is a smooth $(k-1)$-dimensional manifold, denoted
  $\bdy'\bModuli_{\vN, \vlambda}(D)$.
\item \label{item:framedboundary} 
  Frame $\bdy'\bModuli_{\vN, \vlambda}(D)$ using the internal and external framings of $V$. We also equip
  $\bdy'\bModuli_{\vN, \vlambda}(D)$ with a vector field $\vv$, which is the
  outer normal to $\ol{V}'$. Thus, we obtain an element
  $[\bdy'\bModuli_{\vN, \vlambda}(D)]$ in the framed cobordism group
  $\tOmega^{k-1}_{\fr}$.
\item \label{item:ok} This element
  $[\bdy'\bModuli_{\vN, \vlambda}(D)]$ is the obstruction to filling
  in $\bdy'\bModuli_{\vN,\vlambda}$ to construct
  $\bModuli_{\vN, \vlambda}(D)$ (with coherent neat embedding and
  coherent external framing).  Putting together all
  $(D,\vN,\vlambda)$'s, we have an obstruction class in the form of a
  cochain
  \[
    \mf{o}_k\in\Hom(\CDP'_{k+1},\tOmega^{k-1}_\fr),\qquad \mf{o}_k(D, \vN, \vlambda) = [\bdy'\bModuli_{\vN, \vlambda}(D)].
  \]
\item \label{item:cocycle} We prove that $\mf{o}_k$ is a cocycle.
\item \label{item:change} Since $\CDP'$ is acyclic, it follows that
  $\mf{o}_k$ is a coboundary of some element
  $\mf{b}\in\Hom(\CDP'_{k},\tOmega^{k-1}_\fr)$. Change the
  $(k-1)$-dimensional spaces by $-\mf{b}$. (Note that we don't change
  any lower dimensional spaces.) After this change, the new
  $(k-1)$-dimensional boundaries $\bdy'\bModuli_{\vN, \vlambda}(D)$
  are framed null-cobordant.
\item \label{item:fill} Fill in each
  $\bdy'\bModuli_{\vN, \vlambda}(D)$ arbitrarily to obtain the desired
  moduli space $\bModuli_{\vN, \vlambda}(D)$, embedded in $\E^d_l$,
  with a normal framing.
\item \label{item:newinternal} We split the normal framings to the
  moduli spaces $\Moduli_{\vN, \vlambda}(D)$ into internal and
  external framings, and construct thickenings using the internal
  framings.
\item \label{item:continue} This finishes the induction step. Now
  $\bModuli_{\vN,\vlambda}(D)$ have been constructed for all
  $(D,\vN,\vlambda)\in \DNlambdad\cup\DNlambda'_{\leq k}$, along with
  coherent neat embeddings and coherent external framings.  (The next
  step of the induction---when we are constructing the
  $(k+1)$-dimensional moduli spaces---might require modifying the
  just-constructed $k$-dimensional spaces, but none of the smaller
  dimensions.)
\end{enumerate}

\subsection{The base case} \label{sec:base}  Let us recall the 
formulas~\eqref{eq:dimension} and \eqref{eq:dimtilde} for the
dimension $k$ and the thick dimension $l$ of the moduli spaces
$\Moduli_{\vN, \vlambda}(D)$:
\begin{equation}
\label{eq:dimk}
k = \mu(D) - 1 + \sum \ell(\lambda_j),
\end{equation}
\begin{equation}
\label{eq:diml}
l  = \mu(D)-1 + 2\sum N_j. 
\end{equation}
The base case in the induction corresponds to moduli spaces with
$k=0$. (For the base case we only need to do Steps~\ref{item:fill}
and~\ref{item:newinternal} from the outline.) From the above formula
we see that there are two kinds of such moduli spaces:
\begin{itemize}
\item those with $\mu(D)=1$ and trivial $\vlambda$ (that is,
  $\vN = \vec{0}$); then $D$ must be a rectangle $R$ on the grid, and
  we are looking at the moduli spaces $\Moduli_0(R)$;
\item those with $\mu(D)=0$ and $\lambda_j = (N)$ for some $j$, where
  $N$ denotes $N_j$, and we have $N_i = 0$ for all $i \neq j$; then
  $D$ is the constant domain $c_x$ for some $x \in \S$, and we write
  the moduli spaces as $\Moduli_{N\ve_j, (N)_j}(c_x)$, with the
  notation from Remark~\ref{rmk:lambdaj}.
\end{itemize}

The moduli spaces of the first kind have thick dimension $0$. We
define them to be single points, embedded in $\E^d_0 = \R^d$ in any
way, and framed so that the resulting element in
$\Omega^0_{\fr} \cong \Z$ is the sign $s(R) \in \{\pm 1\}$ from
\eqref{item:signa}. For instance, if $\ve_1,\dots,\ve_d$ denote the
standard unit vectors in $\RR^d$, we may frame the point
$\Moduli_0(R)$ by $[(-1)^{s(R)}\ve_1,\ve_2,\dots,\ve_d]$. Since the
thick dimension is zero, the point itself is its own thickening, so
this completes both Steps~\ref{item:fill} and \ref{item:newinternal}.

The moduli spaces of the second kind have thick dimension $l=2N-1$. We
define them to be single points as well, embedded arbitrarily in the
interior of $\E^d_l$, and framed positively. For instance, if
$\ve^+_1,\dots,\ve^+_l$ denote the standard unit vectors in $\R_+^l$,
and $\ve_1,\dots,\ve_{d(l+1)}$ denote the standard unit vectors in
$\R^{d(l+1)}$ (as before), then we may choose to frame the point
$\Moduli_{N\ve_j, (N)_j}(c_x)\in\mathring{\E}^d_l\cong\mathring{\R}_+^l\times\R^{d(l+1)}$ by
$[\ve^+_1,\dots,\ve^+_l,\ve_1,\dots,\ve_{d(l+1)}]$. This completes
Step~\ref{item:fill}.

For Step~\ref{item:newinternal}, declare the first $l$ vectors to be
the internal frame, and the rest to be the external frame. For the
thickening, we choose an open embedding of the local model $Z_N$ in
the interior of $\E^d_l$, with the origin in $Z_N$ mapped to
$\Moduli_{N\ve_j, (N)_j}(c_x)$, such that the standard frame at the
origin maps to the internal frame at $\Moduli_{N\ve_j, (N)_j}(c_x)$.
For instance, we may choose to embed a neighborhood of the origin in
$Z_N$ to the interior of $\E^d_l$ by an affine map, with the origin
mapping to $\Moduli_{N\ve_j, (N)_j}(c_x)$, and the standard frame
mapping to the internal frame. Declare the thickening
$U_{N\ve_j,(N)_j}(c_x)$ to be the image of this embedding, with the
induced stratification, and extend the external framing at
$\Moduli_{N\ve_j, (N)_j}(c_x)$ to a normal framing of this entire
thickening.

Note that $\DNlambdad$ contains exactly $(n-1)$ domains
$(c_{\xid},\ve_j,(1)_j)$, $j=2,\dots,n$ with $\gr=1$. Their
$0$-dimensional moduli spaces have already been constructed (although
that construction is described in the next section), along with neat
embeddings and external framings. When the reader compares the
construction here and the one from Section~\ref{sec:permutohedron}, it
should be clear that they agree.

\subsection{Boundaries and their neighborhoods}
\label{sec:bdries}
We now give more detailed explanations for some of the steps in the
outline of the induction above. In this subsection we discuss
Steps~\ref{item:emb}---\ref{item:smooth}.

Step~\ref{item:emb} just states the induction hypothesis. We assume we
have constructed moduli spaces $\bModuli_{\vN,\vlambda}(D)$ for
all $(D,\vN,\vlambda)\in \DNlambdad\cup\DNlambda'_{\leq k-1}$, along with
coherent neat embeddings (Definition~\ref{def:compatible-neat}) and
coherent external framings
(Definition~\ref{def:compatible-ext-framing}).

For Step~\ref{item:exponentiate}, recall that strata of $\bdy\bModuli_{\vN, \vlambda}(D)$ are of the form
\begin{equation}
  \label{eq:str3}
  Y= \Moduli_{\vN^1+\coefficients{E^1}+\coefficients{F^1}, \vlambda^1}(D^1) \times \dots \times
  \Moduli_{\vN^r+\coefficients{E^r}+\coefficients{F^r}, \vlambda^r}(D^r).
\end{equation}
Since $\dim Y\leq k-1$, by our induction hypothesis, each
$\bModuli_{\vN^i+\coefficients{E^i}+\coefficients{F^i},
  \vlambda^i}(D^i)$ already comes with a neat embedding (including a
stratified thickening $U_i$) and internal and external framings in
$\E^d_{l_i}$, where $l_i$ is its thick dimension. Altogether, we
obtain an embedding of the product
\[ U_1 \times U_2 \times \dots \times U_r\ \hookrightarrow \ \E^d_{l_1} \times \{0\} \times \E^d_{l_2} \times \{0\} \times \dots \times   \{0\} \times \E^d_{l_r} \subset \E^d_l,\]
where 
\[
  l =\tdim \bModuli_{\vN, \vlambda}(D) = l_1 + \dots + l_r + (r-1)
\]
and we identify
\begin{equation}
\label{eq:edli}
 \E^d_l \cong  \E^d_{l_1} \times \R_+ \times \E^d_{l_2} \times \R_+ \times \dots \times   \R_+ \times \E^d_{l_r}.
 \end{equation}
Let
\[
  U(\ol{Y}) = U_1 \times [0, \epsilon_Y) \times U_2 \times  [0, \epsilon_Y)\times \dots \times  [0, \epsilon_Y) \times U_r \subset \E^d_l,
\]
with the product stratification, and define the {\em thickening of the
  boundary $\bdy\bModuli_{\vN, \vlambda}(D)$} to be
\[
  U=U(\bdy\bModuli_{\vN, \vlambda}(D)) = \bigcup_{Y \leq \bdy\bModuli_{\vN, \vlambda}(D)} U(\ol{Y}),
\]
where $\epsilon_Y > 0$ are chosen so that $\epsilon_Y \ll \epsilon_Z$
for $Z < Y$. Since the stratifications of the thickenings of all
moduli spaces up to dimension $k-1$ are coherent, they induce a
stratification of $U$, with each $Y$ as an open stratum, whose local
model in this thickening is same as the local model $Y$ in
$\bModuli([\tD])$ (where $\tD =D+\tE+\tF$, with $\tE,\tF$ a sum of
rows, columns satisfying $\Os(\tE) +\Os(\tF)=\vN$),
cf.~Definition~\ref{def:neatmoduli}
Item~\ref{item:local-model-identification} and
Remark~\ref{rem:same-local-models}. Similarly, since the external
framings of all these $U(\ol{Y})$ are coherent, they induce an
external framing of $U$ in $\E^d_l$.

Thus, we have constructed a stratified thickening
\[
  \bdy\bModuli_{\vN, \vlambda}(D) \hookrightarrow U \subset \E^d_l
\]
with an external framing.
Let $V=V(\bdy\bModuli_{\vN, \vlambda}(D))$ be the union of the closed
strata of $U$ that correspond to $\bModuli_{\vN,\vlambda}(D)$. Then
$V$ is a stratified open neighborhood of
$\bdy\bModuli_{\vN,\vlambda}(D)$ inside
$\bModuli_{\vN,\vlambda}(D)$. This neighborhood $V$ also has the same
thickening $U$ and the same external framing. That is, we have constructed
$\bModuli_{\vN,\vlambda}(D)$ near its boundary, thus completing
Step~\ref{item:exponentiate}.

Next, for Step~\ref{item:smooth}, the stratified space $U$ is Whitney
and hence Thom-Mather stratified. Therefore, each stratum
$Y\leq \bdy\bModuli_{\vN,\vlambda}(D)$ has an open tubular
neighborhood of the form $\rho_Y^{-1}([0,\epsilon_Y))$ inside $U$
(similar to Lemma~\ref{lemma:Nbhd}). Taking their union over $Y$ (and
ensuring $\epsilon_Y\ll \epsilon_Z$ for $Z<Y$) produces an open
neighborhood $U''$ of $\bdy\bModuli_{\vN,\vlambda}(D)$ inside $U$. Its
complement $U\setminus U''$ is an $l$-dimensional
$\langle l\rangle$-manifold with compact boundary, and has a smoothing
$\smo[U\setminus U'']$ (similar to
Definition~\ref{def:smoothstrata}). The complement of the smoothing,
\[
  U'=U\setminus \smo[U\setminus U''],
\]
is a new open neighborhood of $\bdy\bModuli_{\vN,\vlambda}(D)$ inside
$U$. (By construction, we have $\ol{U}''\subset U'\subset\ol{U}'\subset U$.)

Define $V''=U''\cap V$ and $V'=U'\cap V$. Then $\ol{V}''$ is exactly
the closed tubular neighborhood of
$\bdy V=\bdy\bModuli_{\vN,\vlambda}(D)$ inside $V$ from
Lemma~\ref{lemma:Nbhd}, and $V'$ is the complement of its smoothing
from Definition~\ref{def:smoothstrata}. See Figure~\ref{fig:smooth}
once again, with the whole square playing the role of $V$, $\bdy X$
playing the role of $\bdy V=\bdy\bModuli_{\vN,\vlambda}(D)$,
$\mathcal{N}$ playing the role of $\ol{V}''$, and the complement of
$\smo[X]$ playing the role of $V'$. By construction,
$\bdy'\bModuli_{\vN,\vlambda}(D)=\ol{V}'\setminus V'$ is a smooth
$(k-1)$-dimensional manifold. This completes Step~\ref{item:smooth}.

\subsection{Obtaining a cochain}\label{sec:obtaining-cochain}
In this subsection, we discuss Steps~\ref{item:framedboundary} and \ref{item:ok}.




For Step~\ref{item:framedboundary}, note that
$\bdy'\bModuli_{\vN, \vlambda}(D)$ is a smooth $(k-1)$-dimensional
submanifold of $\inte(\E^d_l)$, and we can identify $\inte(\E^d_l)$
with a Euclidean space. We want $\bdy'\bModuli_{\vN, \vlambda}(D)$ to
give an element of $\tOmega^{k-1}_{\fr}$. For this, we equip it with $\vv$
(the outer normal to $\ol{V}'$), and we are left to specify a normal framing
to
\[
  T(\bdy'\bModuli_{\vN, \vlambda}(D)) \oplus \langle  \vv \rangle = TV|_{\bdy'\bModuli_{\vN, \vlambda}(D)}.
\]
Consider the thickening $U$ of $\bModuli_{\vN,\vlambda}(D)$ near its
boundary, constructed in Step~\ref{item:exponentiate} (in
Section~\ref{sec:bdries}). The normal bundle of $V$ inside $U$ is
equipped with an internal frame (cf.~Definition~\ref{def:neatmoduli}
Item~\ref{item:local-model-identification}), and the normal bundle of
$U$ has an external frame. Concatenating the internal frame and the
external frame, we get a framing of the normal bundle of $V$ inside
$\E^d_l$, thus completing Step~\ref{item:framedboundary}.

For Step~\ref{item:ok}, by collecting all these elements
$[\bdy'\bModuli_{\vN,\vlambda}(D)]\in\tOmega^{k-1}_{\fr}$ for all
$(D,\vN,\vlambda)\in\DNlambda'_{k}$, we get an cochain
$\mf{o}_k\in\Hom(\CDP'_{k+1},\tOmega^{k-1}_\fr)$. We remark that
$\mf{o}_k(D,\vN,\vlambda)$ is precisely the obstruction to filling in
$\bdy'\bModuli_{\vN,\vlambda}(D)$. (This will discussed in more detail
for Steps~\ref{item:fill} and \ref{item:newinternal}, after we make
the obstruction class vanish, and actually fill in
$\bdy'\bModuli_{\vN,\vlambda}(D)$.)

\subsection{The cocycle condition}\label{sec:cocycle-condition}
This subsection is devoted to Step \ref{item:cocycle}. We split the
discussion into two cases, according to whether $k=1$ or $k \geq
2$. In the case $k=1$, we obtain a stronger conclusion:
\begin{proposition}
\label{prop:cocycle1}
We have $\mf{o}_1 = 0 \in\Hom(\CDP'_{2},\tOmega^0_\fr)$.
\end{proposition}

\begin{proof}
  We seek to show that for every $1$-dimensional moduli space
  $\bModuli_{\vN, \vlambda}(D)$, the smoothed boundary
  $\bdy'\bModuli_{\vN, \vlambda}(D)$ represents the zero element in
  $\tOmega^0_{\fr}$. Using
  Formula~\eqref{eq:dimk}, we find that $1$-dimensional moduli spaces
  are of one of the following kinds (using the notation from
  Remark~\ref{rem:1234}):
\begin{enumerate}
\item\label{item:signcheck-case1} $\bModuli_0(D)$, where $D$ is a positive domain of index $2$ on
  the grid, which is either a disjoint union of two rectangles or an L-shaped
  hexagon (cf.~Item~(\ref{item:maslov-index-2}) from
  Section~\ref{sec:grid-background}). Then, $D$ has two distinct
  representations as concatenations of two rectangles, and therefore
  the boundary $\del \bModuli_0(D)$ consists of two Type I strata of
  the form $\Moduli_0(R_1) \times \Moduli_0(R_2)$;
  cf.~Example~\ref{ex:RR} and Figure~\ref{fig:Dind2};

\item\label{item:signcheck-case2} $\bModuli_0(D)$, where $D$ is either a horizontal annulus $H_j$
  or a vertical annulus $V_j$. Then, the boundary $\del \bModuli_0(D)$
  consists of a Type I stratum $\Moduli_0(R_1) \times \Moduli_0(R_2)$
  and a Type II stratum $\Moduli_{\ve_j, (1)_j}(c_x)$;
  cf.~ Example~\ref{ex:HV} and Figure~\ref{fig:row};

\item\label{item:signcheck-case3} $\bModuli_{N\ve_j, (N)_j}(R)$, where $R$ is a rectangle (of
  index $1$) from $x$ to $y$, $N>0$, and
  $j \in \{2, \dots, n\}$. Then, $\bdy\bModuli_{\vN, \vlambda}(D)$
  consists of two strata of Type IV, namely
  $\Moduli_{N\ve_j, (N)_j}(c_x) \times \Moduli_0(R)$ and
  $\Moduli_0(R) \times \Moduli_{N\ve_j, (N)_j}(c_y)$;

\item\label{item:signcheck-case4} $\bModuli_{N\ve_i + M\ve_j, (N)_{i} + (M)_{j}} (c_x)$, where
  $c_x$ is a constant domain and $N, M>0$, $i \neq j$. Then,
  $\bdy\bModuli_{\vN, \vlambda}(D)$ has two strata of Type IV, namely
  $\Moduli_{N\ve_i, (N)_{i} } (c_x) \times \Moduli_{M\ve_j, (M)_{j} }
  (c_x)$ and $\Moduli_{M\ve_j, (M)_{j} } (c_x)\times\allowbreak \Moduli_{N\ve_i, (N)_{i} } (c_x)$;

\item\label{item:signcheck-case5} $\bModuli_{(N+M)\ve_j, (N,M)_j}(c_x)$, where
  $c_x$ is a constant domain and $N, M>0$, Then,
  $\bdy\bModuli_{\vN, \vlambda}(D)$ consists of the stratum
  $ \Moduli_{N\ve_i, (N)_{i} } (c_x) \times \Moduli_{M\ve_i, (M)_{i} }
  (c_x)$ of Type IV, and the stratum
  $ \Moduli_{(N+M)\ve_i, (N+M)_{i} } (c_x)$ of Type III.
\end{enumerate}

In all the above situations, the boundaries
$\bdy\bModuli_{\vN, \vlambda}(D)$ consist of two $0$-dimensional
strata, which are points according to the construction in
Section~\ref{sec:base}. Hence, the smoothings
$\bdy'\bModuli_{\vN, \vlambda}(D)$ also consist of two points.

By a case by case analysis, using the definitions in
Section~\ref{sec:base} and the properties \eqref{eq:sRS},
\eqref{eq:RSh}, \eqref{eq:RSv} of the sign assignment on rectangles,
we will check that the two points in
$\bdy'\bModuli_{\vN, \vlambda}(D)$ come with opposite signs in
$\tOmega^0_\fr \cong \Omega^0_{\fr} \cong \Z$, so they sum up to
$0$. Recall from Step~\ref{item:framedboundary}
(Section~\ref{sec:obtaining-cochain}) that each point is framed by
concatenating the internal and external framings. Further recall from
Equation~\eqref{eq:tomega-to-omega} that the isomorphism
$\tOmega^0_\fr \to \Omega^0_\fr$ is given by concatenating the vector
$\vec{v}$ with the framing of the point. That is, each point, as an
element in $\Omega^0_\fr$, is framed by concatenating $\vec{v}$, its
internal frame, and its external frame. For convenience, assume the
$0$-dimensional moduli spaces are framed explicitly as in
Section~\ref{sec:base}.

In Case~(\ref{item:signcheck-case1}), that the two points in
$\R_+\times\R^{2d}$ come with opposite signs is a consequence of
Equation~\eqref{eq:sRS} (with $\vec{v}=\ve_1^+$ in both points).

 In
Case~(\ref{item:signcheck-case2}), the Type I stratum is positively
(respectively negatively) oriented if $D=H_j$ (respectively $D=V_j$),
according to Equation~\eqref{eq:RSh} (respectively~\eqref{eq:RSv}),
with $\vec{v}=\ve_1^+$; for the Type II stratum, there is no internal
frame, the external frame is the standard (positive) frame for
$\R^{2d}$, and the vector $\vec{v}$ comes from the inward normal
vector of the point $Z(0,1,0)$ in $\bZ(1,0,0)$ (respectively
$\bZ(0,0,1)$), which is the opposite (respectively same) as the
internal frame $\delta v_{1,1}$ of $Z(0,1,0)$
(cf.~Section~\ref{sec:modelinternal}), and hence maps to the frame
$-\ve^+_1$ (respectively $\ve^+_1$) of $\R_+$. 

In
Case~(\ref{item:signcheck-case3}), for either point in
$\R_+^{2N}\times\R^{(2N+1)d}$, its external frame is a frame for
$\R^{(2N+1)d}$ with sign $s(R)$; the point near the boundary
$\Moduli_{N\ve_j, (N)_j}(c_x) \times \Moduli_0(R)$ (respectively
$\Moduli_0(R) \times \Moduli_{N\ve_j, (N)_j}(c_y)$) has internal frame
$[\ve^+_1,\dots,\ve^+_{2N-1}]$ and $\vec{v}=\ve^+_{2N}$ (respectively
internal frame $[\ve^+_2,\dots,\ve^+_{2N}]$ and $\vec{v}=\ve^+_{1}$)
and hence represents $-s(R)$ (respecively $s(R)$) in
$\Omega^0_\fr$.

In Case~(\ref{item:signcheck-case4}), for either point
in $\R_+^{2N+2M-1}\times\R^{(2N+2M)d}$, its external frame is the
standard (positive) frame for $\R^{(2N+2M)d}$; assuming $i<j$, the
point near the boundary
$\Moduli_{N\ve_i, (N)_{i} } (c_x) \times \Moduli_{M\ve_j, (M)_{j} }
(c_x)$ (respectively $\Moduli_{M\ve_j, (M)_{j} } (c_x)$
$\times \Moduli_{N\ve_i, (N)_{i} } (c_x)$) has internal frame
$[\ve^+_1,\dots,\ve^+_{2N-1},\allowbreak \ve^+_{2N+1},\dots,\ve^+_{2N+2M-1}]$ and
$\vec{v}=\ve^+_{2N}$ (respectively
$[\ve^+_{2M+1},\dots,\ve^+_{2N+2M-1},\ve^+_{1},\dots,\ve^+_{2M-1}]$
and $\vec{v}=\ve^+_{2M}$) and hence represents $-1$ (respectively $1$)
in $\Omega^0_\fr$. 

Finally in Case~(\ref{item:signcheck-case5}), for
either point in $\R_+^{2N+2M-1}\times\R^{(2N+2M)d}$, the external
frame is the standard (positive) frame for $\R^{(2N+2M)d}$ and
the internal frame is
$[\ve^+_1,\dots,\ve^+_{2N-1},\ve^+_{2N+1},\allowbreak\dots,\allowbreak\ve^+_{2N+2M-1}]$
(similar to the above case); for the the point near
$ \Moduli_{N\ve_i, (N)_{i} } (c_x) \times \Moduli_{M\ve_i, (M)_{i} }$,
$\vec{v}=\ve^+_{2N}$, but for the point near
$\Moduli_{(N+M)\ve_i,(N+M)_i}(c_x)$, the vector $\vec{v}$ comes from
the inward normal vector of the point $Z(0,N+M,0;(N+M))$ in
$\bZ(0,N+M,0;(N,M))$, which is the negative of the internal frame
$-\delta \Delta_{1,N}$ of $Z(0,N+M,0;(N+M))$, and hence maps to the
vector $-\ve^+_{2N}$ in $\RR_+^{2N+2M-1}$.
\end{proof}

For $k \geq 2$, we have:
\begin{proposition}
\label{prop:cocycle}
The element $\mf{o}_k\in\Hom(\CDP'_{k+1},\tOmega^{k-1}_\fr)$ is a cocycle.
\end{proposition}

\begin{proof}
We need to show that $\delta \mf{o}_k$ evaluates to zero on any generator $(E, \vM, \vmu)\in \DNlambda'_{k+1}$. This is equivalent to 
\[
  \mf{o}_k(\delta (E, \vM, \vmu)) = 0.
\]
Write
\[
  \delta (E, \vM, \vmu) = \sum s_{D, \vN, \vlambda} (D, \vN, \lambda),
\]
where $s_{D, \vN, \vlambda} \in \{\pm 1\}$ and
$(D, \vN, \vlambda) \in \DNlambda'_{k}$. (This is the differential
in the quotient complex $\CDP'_*=\CDP_*/\CDPd$, so if
Formula~\eqref{eq:delta123} applied on $(E,\vM,\vmu)$ produces any
triples from $\DNlambdad$, we suppress them.  Also, here and later,
when we sum over $(D, \vN, \vlambda)$, we consider only the triples in
$\DNlambda'_{k}$ that appear in $\delta (E, \vM, \vmu)$.)  We aim 
to prove:
\[
\sum_{(D, \vN, \vlambda)}s_{D, \vN, \vlambda}\cdot \mf{o}_k(D, \vN, \lambda)  = 0,
\]
that is,
\begin{equation}
\label{eq:sDNL}
 \sum_{(D, \vN, \vlambda)} s_{D, \vN, \vlambda} [\bdy'\bModuli_{\vN, \vlambda}(D)] = 0.
\end{equation}
(It is possible certain triples $(D,\vN,\vlambda)$ appear twice in
$\delta (E,\vM,\vmu)$---for instance $(c_x,2\ve_j,(1,1)_j)$ appears
twice in $\delta(H_j,\ve_j,(1)_j)$---but we consider each appearance
as a separate instance.)

Consider the $(k+1)$-dimensional moduli space
$\bModuli_{\vM, \vmu}(E)$. Of course, this has not yet been
constructed in our inductive procedure. Nevertheless, as mentioned in
Section~\ref{sec:codim1}, we know that its codimension-1 strata are
supposed to be of three types: products (Type I and IV) and single
moduli spaces (Type II and III). Furthermore, let us distinguish
between the products where one of the factors is zero-dimensional, and
those were both factors are positive dimensional. When a factor is
zero-dimensional, it must be a single point (see
Section~\ref{sec:base} and Remark~\ref{rem:zero} below), and therefore
the product can be identified with the other factor, which is some
$(k-1)$-dimensional moduli space $\Moduli_{\vN, \vlambda}(D)$. The
triples $(D, \vN, \vlambda)$ that appear in $\delta (E, \vM, \vmu)$
come from products (Type I and IV) where one factor is $0$-dimensional
and the other is $\Moduli_{\vN, \vlambda}(D)$ with
$(D,\vN,\vlambda)\notin\DNlambdad$, as well as from Type II and III
strata $\Moduli_{\vN, \vlambda}(D)$ (also with
$(D,\vN,\vlambda)\notin\DNlambdad$); see Remark~\ref{rem:1234}.

Recall that all the moduli spaces of dimension up to $(k-1)$ have
already been constructed, as well as moduli spaces of
$(D,\vN,\vlambda)\in\DNlambdad$. Let us define the {\em old boundary}
of $\bModuli_{\vM, \vmu}(E)$, denoted $\dold \bModuli_{\vM, \vmu}(E)$,
to be the union of all strata of $\bModuli_{\vM, \vmu}(E)$ of
codimension $2$ or higher (that is, dimension $(k-1)$ or lower),
together with the Types I and IV strata where neither factor of the
product is $0$-dimensional (and therefore both factors are of
dimension $(k-1)$ or lower), together with the Types I and IV strata
where one factor is $0$-dimensional and the other factor is
$\Moduli_{\vN, \vlambda}(D)$ for some $(D,\vN,\vlambda)\in\DNlambdad$,
together with Types II and III strata $\Moduli_{\vN,\vlambda}(D)$ for
some $(D,\vN,\vlambda)\in\DNlambdad$. By repeating what we did for
Step~\ref{item:exponentiate} in Section~\ref{sec:bdries}, we see that
$\dold \bModuli_{\vM, \vmu}(E)$ has already been constructed in our
inductive procedure, together with its embedding in $\E^d_l$ (where
$l=\tdim \bModuli_{\vM, \vmu}(E)$), its stratified thickening $U$, and
internal and external framings. Indeed, if we repeat the construction
for Step~\ref{item:smooth}, by taking the union of closed strata of
$U$ that correspond to $\bModuli_{\vM,\vmu}(E)$, we get a neighborhood
$W$ of $\dold \bModuli_{\vM, \vmu}(E)$ inside
$\bModuli_{\vM, \vmu}(E)$. As before, we take a smaller Thom-Mather
neighborhood $U''$, smooth its complement $U\setminus U''$, thereby
producing another neighborhood $U'$ (with
$\ol{U}''\subset U'\subset\ol{U}'\subset U$). By intersecting these
with $W$, we get similar neighborhoods $W',W''$ of
$\dold\bModuli_{\vM,\vmu}(E)$ inside $\bModuli_{\vM,\vmu}(E)$.

Consider any Type I or IV stratum of the form
$\pt\times \bModuli_{\vN,\vlambda}(D)$ with
$(D,\vN,\vlambda)\notin\DNlambdad$. Then the neighborhoods
$\pt\times V,\pt\times V',\pt\times V''$ of
$\pt\times\bdy\bModuli_{\vN,\vlambda}(D)$ inside
$\pt\times \bModuli_{\vN,\vlambda}(D)$ (as constructed in
Section~\ref{sec:bdries}) are given by intersecting this stratum with
$W,W',W''$, respectively. Denote these neighborhoods
$\wt{V},\wt{V}',\wt{V}''$, and let
$\wt{\bdy}'\bModuli_{\vN,\vlambda}(D)=\pt\times\bdy'\bModuli_{\vN,\vlambda}(D)=\ol{\wt{V}}'\setminus\wt{V}'$.  A similar discussion holds for any Type I or IV stratum of the form
$\bModuli_{\vN,\vlambda}(D)\times\pt$. Similarly, for any Type II or
III stratum $\bModuli_{\vN,\vlambda}(D)$, the neighborhoods $V,V',V''$
of $\bdy\bModuli_{\vN,\vlambda}(D)$ inside
$\bModuli_{\vN,\vlambda}(D)$ are given by intersecting this stratum
with $W,W',W''$, respectively. For consistency, denote these
neighborhoods $\wt{V},\wt{V}',\wt{V}''$ as well, and let
$\wt{\bdy}'\bModuli_{\vN,\vlambda}(D)=\bdy'\bModuli_{\vN,\vlambda}(D)$.

Define $\bdy''\bModuli_{\vM,\vmu}(E)=\ol{W}'\setminus W'$. Then
$\bdy''\bModuli_{\vM,\vmu}(E)$ is a compact $k$-dimensional smooth manifold in $\inte(\E^d_l)$ with boundary
\begin{equation}\label{eq:bdy-of-bdy''}
 \bdy\big(\bdy''\bModuli_{\vM,\vmu}(E)\big)= \coprod_{(D,\vN,\vlambda)}\wt{\bdy}'\bModuli_{\vN,\vlambda}(D),
\end{equation}
see Figure~\ref{fig:second}. 
\begin{figure}
  \centering
  \begin{tikzpicture}[scale=1.2]
    \begin{scope}[z={(0.6cm,0.3cm)},y={(0,-1cm)},x={(0.8cm,-0.2cm)}]

      \draw[thick,line join=round,fill=black!60] (-1,2,0.5) -- (-0.4,0.8,0.2)
      ..controls (-1.6,0.8,0.8) and (-0.8,0,0.1).. (-0.8,0) ..controls
      (-0.8,0,-0.1) and (-0.7,0.8,-0.35).. (-0.4,0.8,-0.2)--(-1,2,-0.5) -- (0,2)--(-1,2,0.5);
      \draw[thick] (-1,2,-0.5) ..controls (-1.2,2,0) .. (-1,2,0.5);
      \draw[thick] (-0.4,0.8,-0.2) ..controls (-0.5,0.8,0) .. (-0.4,0.8,0.2);

      \fill[pattern=north east lines] (-1,2,0.5) -- (-0.4,0.8,0.2)
      ..controls (-1.6,0.8,0.8) and (-0.8,0,0.1).. (-0.8,0) -- (0,0) -- (0,2)--(-1,2,0.5);

      \fill[red!50,opacity=0.5] (-2,2,1) ..controls (-1,1,0.5) and
      (-1.2,0,0.1).. (-1.2,0) ..controls (-1.2,0,-0.1) and
      (-1,1,-0.5).. (-2,2,-1) .. controls (-2.2,2,0).. (-2,2,1);
      \draw[ultra thick,line join=round,red!50!black] (-2,2,1) ..controls (-1,1,0.5) and
      (-1.2,0,0.1).. (-1.2,0) ..controls (-1.2,0,-0.1) and
      (-1,1,-0.5).. (-2,2,-1);
      \draw[thin,red!50!black] (-2,2,-1) .. controls (-2.2,2,0).. (-2,2,1);

      \draw (2,2,-1) to[out=90,in=-100,looseness=0.5] (1,2,0.5);
      \draw (2,2,-1) -- (0,2);
      \draw (0,2) -- (-2,2,-1) ..controls(-2,1,-0.8) and (-2,0,-0.25).. coordinate[midway] (z110) (-2,0) -- (2,0);
      \draw (2,0) ..controls(2,0,0.25) and (2,1,0.8).. coordinate[midway] (z011) (2,2,1) -- (0,2) -- (-2,2,1) ..controls(-2,1,0.8) and (-2,0,0.25).. (-2,0);

      \draw[ultra thick] (0,0)--(0,2) coordinate[midway] (z02011);
      \node at (0,0) {\Large $\bullet$};

    \end{scope}

    \begin{scope}[xshift=6cm,z={(0.6cm,0.3cm)},y={(0,-1cm)},x={(0.8cm,-0.2cm)}]

      \fill[red!50,opacity=0.5] (-2,2,-1) ..controls (-1,1,-0.5) and
      (-1.2,0,-0.1).. (-1.2,0) .. controls (-1.2,-1) and (1.2,-1).. (1.2,0,0)
      ..controls (1.2,0,-0.1) and (1,1,-0.5)..(2,2,-1)
      .. controls (0,2,-2).. (-2,2,-1);
      \draw[ultra thick,line join=round,red!50!black] (1.2,0) ..controls (1.2,0,-0.1) and
      (1,1,-0.5).. (2,2,-1);
      \draw[thin,red!50!black] (-1.2,0,0) .. controls (0,0,-1.2).. (1.2,0,0);

      \fill[white]  (0,0)--(-0.8,0) ..controls
      (-0.8,0,-0.1) and (-0.7,0.8,-0.35).. (-0.4,0.8,-0.2)--(-1,2,-0.5) -- (0,2)--(0,0);
      \fill[white] (1,2,-0.5)--(0,0)--(0,2)--cycle;
      \fill[draw,thick,pattern=north east lines] (0,0)--(-0.8,0) ..controls
      (-0.8,0,-0.1) and (-0.7,0.8,-0.35).. (-0.4,0.8,-0.2)--(-1,2,-0.5) -- (0,2)--(0,0);
      \fill[draw,thick,pattern=north east lines] (1,2,-0.5)--(0,0)--(0,2)--cycle;
      
      \draw[thick,line join=round,fill=black!60] (-1,2,0.5) -- (-0.4,0.8,0.2)
      ..controls (-1.6,0.8,0.8) and (-0.8,0,0.1).. (-0.8,0) ..controls (-0.8,-0.8) and (0.8,-0.8,-0.1).. (0.8,0)
      ..controls (0.8,0,0.05) and (0.7,0.8,0.35)..
      (0.4,0.8,0.2)--(1,2,0.5);

      \draw[thick,fill=black!60] (-1,2,0.5) ..controls (0,2,1) .. (1,2,0.5)--(0,2)--cycle;
      \draw[thick,fill=black!60] (-1,2,-0.5) ..controls (0,2,-1) .. (1,2,-0.5)--(0,2)--cycle;
      \draw[thick] (-0.4,0.8,0.2) ..controls (0,0.8,0.4)..(0.4,0.8,0.2);
      \draw[thick] (-0.8,0) ..controls (0,0,0.4)..(0.8,0);

      \fill[red!50,opacity=0.5] (-2,2,1) ..controls (-1,1,0.5) and
      (-1.2,0,0.1).. (-1.2,0) .. controls (-1.2,-1) and (1.2,-1).. (1.2,0,0)
      ..controls (1.2,0,0.1) and (1,1,0.5)..(2,2,1)
      .. controls (0,2,2).. (-2,2,1);

      \draw[ultra thick,line join=round,red!50!black] (2,2,1) ..controls (1,1,0.5) and
      (1.2,0,0.1).. (1.2,0);      
      \draw[ultra thick,line join=round,red!50!black] (-2,2,1) ..controls (-1,1,0.5) and
      (-1.2,0,0.1).. (-1.2,0) ..controls (-1.2,0,-0.1) and
      (-1,1,-0.5).. (-2,2,-1);

      \draw[thin,red!50!black] (-2,2,-1) .. controls (0,2,-2).. (2,2,-1);
      \draw[thin,red!50!black] (-2,2,1) .. controls (0,2,2).. (2,2,1);
      \draw[thin,red!50!black] (-1.2,0,0) .. controls (0,0,1.2).. (1.2,0,0);
      \draw[thin,red!50!black] (-1.2,0,0) .. controls (-1.2,-1) and (1.2,-1).. (1.2,0,0);

      \draw (2,2,-1) to[out=90,in=-100,looseness=0.5] (1,2,0.5);
      \draw (2,2,-1) -- (0,2);
      \draw (0,2) -- (-2,2,-1) ..controls(-2,1,-0.8) and (-2,0,-0.25).. coordinate[midway] (z110) (-2,0) -- (-0.8,0) (0.8,0)-- (2,0);
      \draw (2,0) ..controls(2,0,0.25) and (2,1,0.8).. coordinate[midway] (z011) (2,2,1) -- (0,2) -- (-2,2,1) ..controls(-2,1,0.8) and (-2,0,0.25).. (-2,0);


    \end{scope}

  \end{tikzpicture}\\
  \caption {This is the Whitney umbrella from
    Figure~\ref{fig:Whitney}. On the left, we assume locally
    $\bModuli_{\vM, \vmu}(E)=Z(2, 0, 0)$,
    $\bModuli_{\vN, \vlambda}(D)=Z(1,1,0)$, and
    $\bdy \bModuli_{\vN, \vlambda}(D),\allowbreak\dold \bModuli_{\vM,
      \vmu}(E)=Z(0, 2, 0;(1, 1))$ (drawn in thick).  The neighborhood
    $W''$ is shown in gray, and its intersection with
    $\bModuli_{\vN, \vlambda}(D)$ is the neighborbood $\wt{V}''$
    (shown striped); $\wt{\bdy}'\bModuli_{\vN,\vlambda}(D)$ is the thick
    red curve, and $\bdy''\bModuli_{\vM,\vmu}(E)$ is the transparent
    pink surface. On the right, we assume locally
    $\bModuli_{\vM, \vmu}(E)=Z(1, 0, 1)$,
    $\bModuli_{\vN_1, \vlambda_1}(D_1)=Z(1,1,0)$,
    $\bModuli_{\vN_2, \vlambda_2}(D_2)=Z(0,1,1)$, and
    $\bdy \bModuli_{\vN_1, \vlambda_1}(D_1),\allowbreak\bdy
    \bModuli_{\vN_2, \vlambda_2}(D_2),\allowbreak\dold \bModuli_{\vM,
      \vmu}(E)=Z(0, 2, 0;(1, 1))$ (not visible); the colors are
    same as before.}
\label{fig:second}
\end{figure}

Similarly to Step~\ref{item:framedboundary}, we frame
$\bdy''\bModuli_{\vM,\vmu}(E)$ by concatenating the outward normal
vector (of $\bdy''\bModuli_{\vM,\vmu}(E)$ in $\ol{W}'$), the internal
frame (of $W$ inside $U$) and the external frame (of $U$ inside
$\E^d_l$). Define
$[\wt{\bdy}'\bModuli_{\vN,\vlambda}(D)]\in\tOmega^{k-1}_\fr$ be to be
manifold $\wt{\bdy}'\bModuli_{\vN,\vlambda}(D)$, embedded in
$\inte(\E^d_l)$, with frame induced from that of
$\bdy''\bModuli_{\vM,\vmu}(E)$, and the vector field $\vec{v}$ the
inward normal vector of $\wt{\bdy}'\bModuli_{\vN,\vlambda}(D)$ in
$\bdy''\bModuli_{\vM,\vmu}(E)$. Equation~\eqref{eq:bdy-of-bdy''}
implies
\[
\sum_{(D,\vN,\vlambda)} [\wt{\bdy}'\bModuli_{\vN,\vlambda}(D)]=0.
\]
Equation~\eqref{eq:sDNL} now follows from a sign check
\begin{equation}\label{eq:cocycle-sign-check}
  [\wt{\bdy}'\bModuli_{\vN,\vlambda}(D)]=(-1)^{k+1}s_{D,\vN,\vlambda}[\bdy'\bModuli_{\vN,\vlambda}(D)],
\end{equation}
which occupies the rest of the proof.

Let $\vw$ denote the inward normal vector field of
$\wt{\bdy}'\bModuli_{\vN,\vlambda}(D)$ in
$\bdy''\bModuli_{\vM,\vmu}(E)$, $\vv$ denote the outward normal vector
field of $\wt{\bdy}'\bModuli_{\vN,\vlambda}(D)$ in $\ol{\wt{V}}'$. Let
$l_D=\tdim\bModuli_{\vN,\vlambda}(D)$ and $U_D$ be the
$l_D$-dimensional thickening of $\bdy'\bModuli_{\vN,\vlambda}(D)$.
Apply the isomorphism $\tOmega^{k-1}_\fr \to \Omega^{k-1}_\fr$ from
Equation~\eqref{eq:tomega-to-omega}, and we treat
$(-1)^{k-1}[\wt{\bdy}'\bModuli_{\vN,\vlambda}(D)]\in\Omega^{k-1}_{\fr,l+d(l+1)}$
as the manifold $\wt{\bdy}'\bModuli_{\vN,\vlambda}(D)\subset\E^d_l$
framed as
\[
[\vw,\vv,\text{ internal frame of $W\subset U$},\text{ external frame of $U\subset\E^d_l$}],
\]
equivalently as
\[
[\vv,-\vw,\text{ internal frame of $W\subset U$},\text{ external frame of $U\subset\E^d_l$}].
\]
Similarly, treat
$(-1)^{k-1}[\bdy'\bModuli_{\vN,\vlambda}(D)]\in\Omega^{k-1}_{\fr,l_D+d(l_D+1)}$
as the manifold $\bdy'\bModuli_{\vN,\vlambda}(D)\subset\E^d_{l_D}$ framed as
\[
[\vv,\text{ internal frame of $V\subset U_D$},\text{ external frame of $U_D\subset \E^d_{l_D}$}],
\]
or by stabilizing, as the manifold $\bdy'\bModuli_{\vN,\vlambda}(D)\times 0$ in $\E^d_{l_D}\times(\R_+\times\E^d_{l-l_D})\cong\E^d_l$ framed as
\[
[\vv,\text{ internal frame of $V\subset U_D$},\text{ external frame of $U_D\subset \E^d_{l_D}$},\text{ standard frame of $\E^d_{l_D}\subset\E^d_{l}$}].
\]
Also recall that we had assumed that $d$ is even, so there is an even
number of vectors in each external frame, and we will freely reshuffle
their orders without affecting any signs.

The sign check is fairly easy for Type II and Type III strata. In
those cases, $l=l_D$, $U=U_D$ and $\pm\vw$ is one of the vectors of the
internal frame of $V\subset U$, and rest of the vectors constitute the
frame of $W\subset U$ (cf.~Remark~\ref{rem:pause-to-justify}), so we
just need to calculate the position $P+1$ where $-\vw$ appears in the
the internal frame of $V\subset U$ and with what sign $S\in\{\pm
1\}$. As in the proof of Proposition~\ref{prop:cocycle1}, the outward
vector $-\vw$ appears with sign $S=-1$ if it is Type II stratum coming
some vertical annulus $V_j$, and otherwise appears with sign $S=1$. By
inspecting Formulas~\eqref{eq:delta2} and~\eqref{eq:delta3}, we see
$(-1)^PS=s_{D,\vN,\vlambda}(-1)^{|\vmu|+\mu(E)+1}$ in both cases; and
we have $|\vmu|+\mu(E)-1=k+1$ from Formula~\eqref{eq:dimension}.

For Type I and Type IV strata of the form
$\bModuli_{\vN,\vlambda}(D)\times\pt$, there are embeddings
\[
  \wt{\bdy}'\bModuli_{\vN,\vlambda}(D)=\bdy'\bModuli_{\vN,\vlambda}(D)\times 0\times\pt\into \E^d_{l_D}\times\R_+\times\E^d_{l-l_D-1}.
\]
In Type I, $l-l_D=1$, locally $U\cong U_D\times\R_+\times\pt$,
$\vlambda=\vmu$, and the point is framed with sign $s(R)$ in $\R^d$;
in Type IV, $l-l_D=2|\vM-\vN|$, locally
$U\cong U_D\times\R_+\times(\R_+^{l-l_D-1})$ (assuming the thickening
of $\pt$ inside $\E^d_{l-l_D-1}\cong\R_+^{l-l_D-1}\times\R^{(l-l_D)d}$
is $\R_+^{l-l_D-1}\times0$), $|\vlambda|=|\vmu|-1$, and the point is
framed positively. The outward vector $-\vw$ is the negative of the
unit vector of the $\R_+$ factor, and to bring it to the second
position requires a sign of $(-1)^{|\vlambda|}$; for Type IV, there is
an additional sign of $(-1)^{|(\lambda_{j+1},\dots,\lambda_n)|}$ to
bring the $l-l_D-1$ odd number of vectors of the internal frame of
$\pt\subset\E^d_{l-l_D-1}$ to the correct position. In either case,
after inspecting Formulas~\eqref{eq:delta1} and~\eqref{eq:delta4}, one
gets the sign $s_{D,\vN,\vlambda}(-1)^{|\vmu|+\mu(E)+1}$.

For Type I and Type IV strata of the form
$\pt\times\bModuli_{\vN,\vlambda}(D)$, we need to recall the following
fact: If $M^m\subset \R^a, N^n\subset \R^b$ are framed manifolds, then
the framed manifolds $M\times N,N\times M\subset\R^{a+b}$---each with
product framing---are related by the sign $(-1)^{mn+mb+na}$;
therefore, for framed manifolds
$M^m\subset \R^a, N^n\subset \R^b,P^p\subset\R^c$, $M\times N\times P$
and $P\times N\times M$ are related by the sign
$(-1)^{mn+np+pm+m(b+c)+n(c+a)+p(a+b)}$. Now for a Type I stratum, we have
\[
  \wt{\bdy}'\bModuli_{\vN,\vlambda}(D)=\Moduli_0(R)\times 0\times\bdy'\bModuli_{\vN,\vlambda}(D)\into \R^d\times\R_+\times\E^d_{l-1},
\]
and we may assume, up to a sign $s(R)$, that it is framed as
\[
  [\text{standard frame in $\R^d$},\vw=e^+_1,\vv,\text{ internal frame of $V\subset U_D$},\text{ external frame of $U_D\subset\E^d_{l-1}$}].
\]
By the above discussion, the latter is same as
\[
  \bdy'\bModuli_{\vN,\vlambda}(D)\times 0\times\Moduli_0(R)\into \E^d_{l-1}\times\R_+\times\R^d,
\]
framed as
\[
  [\vv,\text{ internal frame of $V\subset U_D$},\text{ external frame of $U_D\subset\E^d_{l-1}$},\ve^+_{l-1},\text{ standard frame in $\R^d$}],
\]
up to a further sign of $(-1)^{\dim\bdy'\bModuli_{\vN,\vlambda}(D)}=(-1)^{k-1}$. Instead for a Type IV stratum, 
\[
  \wt{\bdy}'\bModuli_{\vN,\vlambda}(D)=\Moduli_{N\ve_j,(N)_j}(c_x)\times 0\times\bdy'\bModuli_{\vN,\vlambda}(D)\into \E^d_{2N-1}\times\R_+\times\E^d_{l_D},
\]
framed as
\[
  [\text{standard frame in $\E^d_{2N-1}$},\vw=e^+_{2N},\vv,\text{ internal frame of $V\subset U_D$},\text{ external frame of $U_D\subset\E^d_{l_D}$}],
\]
up to a sign of
$(-1)^{|(\mu_1,\dots,\mu_{j-1})|}=s_{D,\vN,\vlambda}(-1)^{|\vmu|+\mu(E)+1}$
which comes from bringing the odd-dimensional internal frame of $\pt\in \E^d_{2N-1}$
to the front.  We switch the order of the three factors again, but
this time, do not incur any additional sign.
\end{proof}

\subsection{Concluding the induction}
This section carries out the remaining steps of the induction. In
Step~\ref{item:change}, we get a cochain
$\mf{b}\in\Hom(\CDP'_{k},\tOmega^{k-1}_\fr)$ satisfying
$\delta\mf{b}=\mf{o}_k$. For each $(F,\vP,\vnu)\in\DNlambda'_{k-1}$, we
pick a framed manifold $M(F,\vP,\vnu)\subset \inte(\E^d_l)$ (where
$l=\tdim\bModuli_{\vP,\vnu}(F)$), such that $[M(F,\vP,\vnu)]\in\Omega^{k-1}_\fr$ is
the image of $-\mf{b}(F,\vP,\vnu)\in\tOmega^{k-1}_\fr$ under the
isomorphism $\tOmega^{k-1}_\fr\to\Omega^{k-1}_\fr$. Change all these
$(k-1)$-dimensional moduli spaces by taking disjoint union with
$M(F,\vP,\vnu)$, that is, define
\[
\bModuli^{\text{new}}_{\vP,\vnu}(F)=\bModuli_{\vP,\vnu}(F)\amalg M(F,\vP,\vnu).
\]
This new piece $M(F,\vP,\vnu)$ is equipped with a framing of its
normal bundle, but does not yet have a stratified thickening with
internal framings, nor an external framing. We declare the first
$l-k+1$ vectors of its frame to be the internal frame, and the
remaining $d(l+1)$ vectors to be the external frame. To construct the
stratified thickening, we reuse our strategy from
Section~\ref{sec:base}; fix a point $p$ in the open stratum
$Z(\vec{0},\vP,\vec{0};\vnu)\subset Z_{\vP}$ (from
Equation~\eqref{eq:ZvN}), and a small open disk $D_p$ around $p$ in
the normal direction (that is, in the affine subspace spannned by the
standard frame at $p$); for each point $q\in M(F,\vP,\vnu)$, embed
$D_p$ in $\inte(\E^d_l)$ by an affine map sending $p$ to $q$ and the
standard frame at $p$ to the internal frame at $q$; assuming $D_p$ is
small enough, the union of these affine embeddings of $D_p$ forms the
stratified thickening of $M(F,\vP,\vnu)$.  (This construction of the
thickening is similar to---but simpler than---what we will do in
Step~\ref{item:newinternal} below.)

This modification changes $\bdy'\bModuli_{\vN,\vlambda}(D)$ for
$(D,\vN,\vlambda)\in\DNlambdad_{k+1}$. If we write the differential in $\CDP'_*$ as
\[
  \delta(D,\vN,\vlambda)=\sum s_{F,\vP,\vnu}(F,\vP,\vnu),
\]
with $s_{F,\vP,\vnu}\in\{\pm 1\}$, then the terms that appear in the
sum correspond to certain strata of $\bdy\bModuli_{\vN,\vlambda}(D)$ (as
in proof of Proposition~\ref{prop:cocycle}). Changing such a stratum
by taking disjoint union with $M(F,\vP,\vnu)$ has the effect of also
changing $\bdy'\bModuli_{\vN,\vlambda}(D)$ by taking disjoint union with
$M(F,\vP,\vnu)$ (if Type II, III) or by taking disjoint union with
$\pt\times M(F,\vP,\vnu)$ or $M(F,\vP,\vnu)\times\pt$ (if Type I, IV);
let us denote it $\wt{M}(F,\vP,\vnu)$ for uniformity. Then
\[
  \bdy^{\prime,\text{new}}\bModuli_{\vN,\vlambda}(D)=\bdy'\bModuli_{\vN,\vlambda}(D)\amalg\coprod_{(F,\vP,\vnu)}\wt{M}(F,\vP,\vnu).
\]
A sign check (left to the reader) similar to that in
Proposition~\ref{prop:cocycle} shows that $[\wt{M}(F,\vP,\vnu)]$,
viewed as an element in $\tOmega^{k-1}_\fr$, differs from
$[M(F,\vP,\vnu)]$, viewed as an element of $\Omega^{k-1}_\fr$, by
precisely the sign $s_{F,\vP,\vnu}$. Therefore in $\tOmega^{k-1}_\fr$ we get
\[
  \mf{o}_{k}^{\text{new}}(D,\vN,\vlambda)=\mf{o}_k(D,\vN,\vlambda)-\sum_{(F,\vP,\vnu)}s_{F,\vP,\vnu}\mf{b}(F,\vP,\vnu)=0.
\]

\begin{remark}
\label{rem:zero}
In view of Proposition~\ref{prop:cocycle1}, we see that Step \ref{item:change} is unnecessary when $k=1$. Thus, the $0$-dimensional moduli spaces are not changed in the process, and they will always remain single points, as they were defined in Section~\ref{sec:base}. 
\end{remark}

In Step~\ref{item:fill}, once we have that
$\bdy'\bModuli_{\vN, \vlambda}(D)$ is framed null-cobordant, we choose
a filling $\Moduli'_{\vN, \vlambda}(D) \subset \inte(\E^d_l)$ and
define
\[
  \bModuli_{\vN, \vlambda}(D) \defeq \ol{V}' \cup \Moduli'_{\vN, \vlambda}(D).
\]

Finally, for Step~\ref{item:newinternal}, note that, by construction,
$\bModuli_{\vN, \vlambda}(D)$ comes equipped with a framing of its
normal bundle in $\E^d_l$. As before, declare the first $l-k$ vectors
to be the internal frame and the remaining $d(l+1)$ vectors to be the
external frame. For the stratified thickening, observe that the
thickening $U$ is already defined around
$\ol{V}'\subset\bModuli_{\vN,\vlambda}(D)$.  Fix a point $p$ in the
open stratum $Z(\vec{0},\vN,\vec{0};\vlambda)\subset Z_{\vN}$, and let
$D_p$ be a small open disk around $p$ in the normal direction (spanned
by the standard frame). By construction, for any point in
$\ol{V}'\setminus\bdy\bModuli_{\vN,\vlambda}(D)$, the local model in
the normal direction inside $U$ agrees with $D_p$. For every point
$q\in\bdy'\bModuli_{\vN,\vlambda}(D)$, fix an embedding $D_p\into U$
sending $p$ to $q$ that identifies $D_p$ with the local model at
$q$. Since the stratified thickening $U$ is already defined around $q$
by induction, we may not be able to ensure that this is an affine
embedding; but we may choose these embeddings smoothly over
$\bdy'\bModuli_{\vN,\vlambda}(D)$ to get a codimension-one embedding
$\iota\from\bdy'\bModuli_{\vN,\vlambda}(D)\times D_p\into U$,
respecting the stratifications. We may assume $\image(\iota)$ divides
$U$ into two pieces, and that $U'$ is one the pieces.

Push forward the $(l-k)$-dimensional standard frame (in the affine
spaces $D_p$) to get $(l-k)$ pairwise commuting linearly independent
vector fields on $\image(\iota)$. In some neighborhood
$N\subset\inte(\E^d_l)$ of
$\Moduli'_{\vN,\vlambda}(D)\cup_{\bdy'\bModuli_{\vN,\vlambda}(D)}\image(\iota)$,
extend the internal frame of $\Moduli'_{\vN,\vlambda}(D)$ and these
vector fields on $\image(\iota)$ to $(l-k)$ pairwise commuting vector
fields; if $N$ is small enough, the extended vector fields are also
linearly independent. By integrating along these vector fields, we get
an embedding $\Moduli'_{\vN,\vlambda}(D)\times D_p\into N$, and define
the stratified thickening $U_{\vN,\vlambda}(D)$ to be union of
$\ol{U}'$ and the image of this embedding (along $\image(\iota)$).

Step~\ref{item:continue} just states the conclusion of the induction step and needs no additional explanation.

\subsection{Gluing moduli spaces}
\label{sec:gluing}
Recall the equivalence relation on domains given by Equation~\eqref{eq:eqD}:
\[
  D\sim D' \ \iff \ (D-D'\in\periodic \text{ and }
  \coefficients{D}=\coefficients{D'} ).
\]
Now that we have constructed all the framed moduli spaces
$\bModuli_{\vN, \vlambda}(D)$, we glue $\bModuli_0(D)$ for all $D$ in
the same equivalence class. The results are the required moduli spaces
$\bModuli([D])$ as in Equation~\eqref{eq:modd}, which are
$\langle l \rangle$-manifolds neatly embedded in $\E^d_l$, and
equipped with normal (external) framings there. The fact that we can
glue the different $\bModuli_0(D)$ is automatic since we constructed
them along with their thickenings $U_0(D)\subset\E^d_l$, and the local
models of $\bModuli_0(D)\subset U_0(D)$ were defined to be the same as
our required local models of $\bModuli_0(D)\subset\bModuli([D])$,
cf.~Remark~\ref{rem:same-local-models}.

\section{Embedding and framing the permutohedra}
\label{sec:permutohedron}

In this section we construct the moduli spaces
$\bModuli_{\vec{N},\vec{\lambda}}(D)$ (along with coherent neat
embeddings and external framings), where $D=c_{\xid}$, the constant
domain from the fixed generator $\xid$ to itself, and each entry in
$\vec{N}$ is $0$ or $1$ (so $\vec{\lambda}$ is a trivial
partition). These correspond to the triples
$(D, \vN, \vlambda)\in\DNlambdad$ generating the subcomplex
$\CDPd \subset \CDP$ that carries all the homology of $\CDP$;
cf. Proposition~\ref{prop:CDPhomology}.

Suppose $\vN$ is as above and let $|\vec{N}|=n$. (In this section, $n$ no longer denotes the grid index.) We will use the
notation
\begin{equation}\label{eq:XI-subset}
  I= \set{i}{N_i=1}=\{p_1<\dots<p_n\}
\end{equation}
and
\begin{equation}\label{eq:XI-moduli}
  X_I\defeq\bModuli_{\vec{N},\vec{\lambda}}(c_{\xid}).
\end{equation}
Note that $X_I$ is supposed to be an $(n-1)$-dimensional
$\codim{n-1}$-manifold. Its thick dimension is $2n-1$.

We will define $X_I$ to be the permutohedron $\Pi_n$, which is the
convex hull of the $n!$ points in $\R^n$ obtained by permuting the
coordinates of $(1, 2, \dots, n)$; cf.~Example~\ref{ex:permuto}. We
will then neatly embed and frame $\Pi_n$ inside
\[
  \RR^d\times\mathring{\RR}_+\times\RR^d\times\RR_+\times\RR^d\times\dots\times\RR_+\times\RR^d\times\mathring{\RR}_+\times\RR^d\cong\mathring{\RR}_+^n\times\RR_+^{n-1}\times\RR^{2dn} \subset \E^d_{2n-1},
\]
coherently with respect to the neat embeddings and framings of the
lower dimensional strata. We will do this for the case $d=0$,
so we do not actually need to construct the external frames. For
larger $d$, we can then simply compose with the standard inclusion
\[
  \mathring{\RR}_+^n\times\RR_+^{n-1}\into
\mathring{\RR}_+^n\times\RR_+^{n-1}\times\RR^{2dn}, \ \ x\mapsto
(x,0),
\]
and define the external frame to be the standard frame of the new
$\R^{2dn}$.

We will also only construct the internal frames (coherently), and will
not construct the thickenings. This is because the local model (in the
normal direction) of $X_I$ inside its thickening is the product
$\prod_I Z_1$, where $Z_1$ is a line, cf.~Equation~\eqref{eq:locmodZN} and Example~\ref{ex:Z1}. Therefore, if we are
given the internal frames at each point $x\in X_I$, we can define the
thickening by simply taking the union of small open disks in the
affine planes spanned by the internal frames.

All that remains to do is to embed $X_I=\Pi_n$ inside
$\mathring{\RR}_+^n\times\RR_+^{n-1}$, along with its internal frame,
which consists of $n$ vectors fields $(v_{p_1}, \dots,v_{p_n})$
indexed by elements of $I$, and everything should be done coherently with respect to the data for
the lower-dimensional strata. For this we will use the fact that the
quotient of $\RR^n\times\Pi_n$ by the symmetric group $S_n$ is
diffeomorphic to $\mathring{\RR}_+^n\times\RR^{n-1}_+$. We will then
embed $X_I$ inside $\RR^n\times\Pi_n$ by the map
$x\mapsto ((p_1,\dots,p_n),x)$ and then quotient by the symmetric
group action. The construction of $X_I$ is described in
Section~\ref{sec:permutahedron-as-moduli-space}, and its embedding and
framing, as outlined above, is described in detail in
Section~\ref{sec:final-step-for-permutahedron}. The main work is
in Section~\ref{sec:quotient-by-sn}, which is devoted to proving
$(\RR^n\times\Pi_n)/S_n\cong\mathring{\RR}_+^n\times\RR_+^{n-1}$.

\subsection{The permutohedron}\label{sec:permutohedron-basic}

In this section, we will collect well-known facts about the permutohedron
$\Pi_n$. See for instance \cite[Example~0.10]{Ziegler} and
\cite[Section~3.3]{LLS}.
\begin{enumerate}[leftmargin=*,label=($\Pi$-\arabic*),ref=($\Pi$-\arabic*)]
\item Letting $S_n$ denote the group of permutations of
  $\{1,2,\dots,n\}$, the permutohedron $\Pi_n$ is
  the convex hull of the $n!$ points
  $v_\sigma=(\sigma^{-1}(1),\dots,\sigma^{-1}(n))$ in $\RR^n$, for
  $\sigma\in S_n$.
\item The permutohedron $\Pi_n$ is $(n-1)$-dimensional and lies in
  the hyperplane
  \[
    \Affine^{n-1}=\{(x_1,\dots,x_n)\in\RR^n\mid \sum_j
    x_j=n(n+1)/2\}
  \]
  and the points $v_\sigma$ are its vertices.
\item\label{item:permuto-halfspace} For any non-empty proper subset $S\subset\{1,2,\dots,n\}$ of
  cardinality say $k$, let $\HalfSpace_S\subset \Affine^{n-1}$ denote
  the half-space $\{(x_1,\dots,x_n)\in\Affine^{n-1}\mid \sum_{j\in S}
  x_j\geq k(k+1)/2\}$. Then $\Pi_n$ is also the intersection of
  the $2^n-2$ half-spaces $\HalfSpace_S$, and the facets of 
  $\Pi_n$ are $F_S=\Pi_n\cap \bdy \HalfSpace_S$. 
  \item\label{item:permuto-facet-vertices} The vertices in the facet $F_S$ are precisely the $v_\sigma$ so
  that $\{\sigma(1),\sigma(2),\dots,\sigma(k)\}=S$.
\item\label{item:permuto-codim-n-manifold} The permutohedron carries the structure of $\codim{n-1}$-manifold by declaring
  \[
    \bdy_k
    \Pi_n=\displaystyle\bigcup_{ \{S,\ |S|=k \}} F_S.
  \]
\item\label{item:permuto-facet-identify} Each of the facets $F_S\subset \bdy_k\Pi_n$ can be
  identified with products of lower dimensional permutohedra
  $\Pi_k\times \Pi_{n-k}$. Identify $\RR^n$ with 
  $\prod_{ j\in\{1,\dots,n\}}\RR$, and using the
  linear ordering of the elements of $S$ and
  $S^c=\{1,2,\dots,n\}\setminus S$, identify $\RR^k$ and $\RR^{n-k}$
  with $\prod_{ j\in S}\RR$ and
  $\prod_{ j\in S^c}\RR$, respectively. Then the map
  \[
  \begin{tikzpicture}[xscale=7]
  \node (left) at (0,0) {$\displaystyle\prod_{ j\in S}\RR\times\displaystyle\prod_{ j\in S^c}\RR$}; 
  \node (right) at (1,0) {$\displaystyle\prod_{ j\in S}\RR\times\prod_{ j\in S^c}\RR
    \cong\prod_{ j\in\{1,\dots,n\}}\RR$};
  \draw[->] (left)--(right) node[pos=0.5,anchor=south] {\tiny$+(0,\dots,0,k,\dots,k)$};
  \end{tikzpicture}
  \]
  identifies $\Pi_k\times\Pi_{n-k}\subset\RR^k\times\RR^{n-k}$
  with the facet $F_S\subset \bdy_k\Pi_n\subset\RR^n$.
\item We will also need the action of $S_n$ on $\Pi_n$. Consider
  the left action of $S_n$ on $\RR^n$ given by:
  \[
    \sigma\cdot(x_1,\dots,x_n)=(x_{\sigma^{-1}(1)},\dots,x_{\sigma^{-1}(n)}).
  \]
  This restricts to an action on $\Affine^{n-1}$ and $\Pi_n$. On
  the vertices of the permutohedron, $S_n$ acts by
  $\sigma\cdot v_\tau=v_{\sigma\tau}$.
\end{enumerate}
See Figure~\ref{fig:permuto-basics} for an illustration of some of
these concepts.

\def\theta{22}
\pgfmathsetmacro{\costheta}{cos(\theta)}%
\pgfmathsetmacro{\sintheta}{sin(\theta)}%
\begin{figure}
  \centering
  \begin{tikzpicture}[scale=0.8,%
    x={(-0.4cm,-0.6cm)},
    y={(1cm,0)},
    z={(0cm,1cm)}]

    \draw[thick,gray,->] (0,0,0)--(4,0,0) node[pos=1,anchor=north east] {$x_1$};
    \draw[thick,gray,->] (0,0,0)--(0,4,0) node[pos=1,anchor=west] {$x_2$};
    \draw[thick,gray,->] (0,0,0)--(0,0,4) node[pos=1,anchor=south] {$x_3$};

    \foreach\i in {1,2,...,6}{
      \foreach\j in {1,2,...,6}{
      }}

    \foreach \sigma/\x/\y/\z/\pos in {
      123/1/2/3/south,
      213/2/1/3/south east,
      132/1/3/2/west,
      321/3/2/1/north,
      231/3/1/2/east,
      312/2/3/1/north west}{

      \coordinate (v\sigma) at (\x,\y,\z);
      \node at (v\sigma) {$\bullet$};
      \node[anchor=\pos] at (v\sigma) {$v_{[\sigma]}$};
    }

    
    \foreach \s/\sigma\tau/\pos/\sep in {
      1/123/132/south west/0,
      2/213/231/east/2,
      3/312/321/north west/0}{

      \draw[thick] (v\sigma) -- (v\tau) node[midway,anchor=\pos,inner sep=\sep pt] {$F_{\{\s\}}$};
      }

    \foreach \s/\t/\sigma\tau/\pos/\sep in {
      1/2/123/213/south/4,
      1/3/132/312/west/2,
      2/3/231/321/north east/0}{

      \draw[ultra thick] (v\sigma) -- (v\tau) node[midway,anchor=\pos,inner sep=\sep pt] {$F_{\{\s,\t\}}$};
      }

      \draw[ultra thin,dashed] ($(v213)!0.5!(v123)$)--($(v321)!0.5!(v312)$);
      
      \node[gray] at (2,2,2) {$\bullet$};
    
  \end{tikzpicture}
  \caption{The permutohedron $\Pi_3\subset\RR^3$. Here the
    permutations $\sigma\in S_3$ are denoted as
    $[\sigma(1)\sigma(2)\sigma(3)]$. The facets in $\bdy_1\Pi_3$ are
    drawn thin and the facets in $\bdy_2\Pi_3$ are drawn thick. The
    action of $S_3$ is also shown; $S_3$ is generated by the
    transposition $[213]$ and the $3$-cycle $[231]$, and they act on
    $\Pi_3$ by reflection across the dashed line, and by positive
    rotation by $120^\circ$ around the gray center,
    respectively.}\label{fig:permuto-basics}
\end{figure}

\subsection{Construction of the moduli spaces}\label{sec:permutahedron-as-moduli-space}
We define the moduli spaces $X_I$ from
Equation~\eqref{eq:XI-moduli} as permutohedra: 
\begin{equation}\label{eq:XI-definition}
  X_I=\Pi_n\subset\RR^n\cong \prod_I \RR,
\end{equation}
where for convenience, we have identified the ambient space $\RR^n$
with $\prod_I\RR$ using the linear ordering of the
elements of $I$ (cf.~Equation~\eqref{eq:XI-subset}), that is, using
the bijection $\{1,\dots,n\}\to I$, $i\mapsto p_i$. (It is understood
that for distinct subsets $I,I'$ of cardinality $n$, the spaces
$X_{I}$ and $X_{I'}$ are \emph{different} copies of $\Pi_n$; for
this it might be useful to regard them as living in different ambient
spaces $\prod_I\RR$ and $\prod_{I'}\RR$.)

We need to check that the stratification on $X_I$ is as described in
Section~\ref{sec:enum}. Indeed, it follows from
Items~\ref{item:permuto-codim-n-manifold}
and~\ref{item:permuto-facet-identify} that $X_I$ is a
$\codim{n-1}$-manifold with $\bdy_k X_I$ identified with
\[
  \coprod_{\substack{J\subset I\\|J|=k}} X_J\times X_{I\setminus J}.
\]
Moreover, these identifications are coherent, that is, for any
$k<\ell$, the two identifications of $\bdy_{\{k,\ell\}}X_I$ with
\[
  \coprod_{\substack{K\subset J\subset I\\|K|=k,|J|=\ell}} X_K\times
  X_{J\setminus K}\times X_{I\setminus J}
\]
are the same. See for instance \cite[Lemma~3.17]{LLS}.

\subsection{Quotienting by the symmetric group}\label{sec:quotient-by-sn}

In this section, we will consider the diagonal actions by the
symmetric group on $\RR^n\times\RR^n$, $\RR^n\times\Affine^{n-1}$, and
$\RR^n\times\Pi^{n-1}$.  We will prove that the quotient
$\big(\RR^n\times\Pi_n\big)/S_n$ is diffeomorphic to
$\RR^n\times\RR^{n-1}_+$; this is stated more precisely as
Proposition~\ref{prop:homeo} below.

Let $\pi\from \RR^n\times\RR^n\to\big(\RR^n\times\RR^n\bigr)/S_n$
denote the projection to the quotient.  Consider the well-known
homeomorphism
\[
\psi_n\from \bigl( \R^n \times \R^{n} \bigr) / S_n \to \R^{2n}
\]
given by the Vi{\`e}te relations. In more detail, let $\alpha_1,
\dots, \alpha_n, \beta_1, \dots, \beta_{n}$ be the coordinates on $
\RR^n \times \RR^n$, which we group into complex variables
$\alpha_j+i\beta_j$, and let $a_1,\dots,a_n,b_1,\dots,b_n$ be the
coordinates of the target $\RR^{2n}$, which we also group into complex
variables $a_j+ib_j$. Then the function $\psi_n$ is induced by the
smooth function $\psi_n\circ\pi\from \RR^n\times\RR^n\to\RR^{2n}$, which is given by
equating the coefficients of the polynomial
\begin{equation}\label{eq:Pz}
 \prod_{j=1}^n(z-\alpha_j-i\beta_j) = z^n - (a_1+ib_1)z^{n-1} + (a_2+ib_2)z^{n-2} - \dots + (-1)^n ( a_n + ib_n).
\end{equation}

Note that $\psi_n$ restricts to a homeomorphism
$\big(\RR^n\times\Affine^{n-1}\big)/S_n\to\RR^{2n-1}$, which we also
denote by $\psi_n$. To wit, $\RR^n\times\Affine^{n-1}$ is the subspace
given by $\beta_1 + \dots + \beta_n = n(n+1)/2$, and so its image is
given by the subspace $b_1 = n(n+1)/2$.

\begin{proposition}
\label{prop:homeo}
The image of $\big(\RR^n\times\Pi_n\bigr)/S_n$ under $\psi_n$ is a
smooth submanifold with corners of $\RR^{2n-1}$, and there is a
diffeomorphism
\begin{equation}\label{eq:Psi}
\Psi_n\from\RR^n\times\RR^{n-1}_+\to\psi_n(\big(\RR^n\times\Pi_n\bigr)/S_n).
\end{equation}
Moreover, the composition
$\iota\circ\Psi_n\from\RR^n\times\RR_+^{n-1}\to\RR^{2n-1}$ (where $\iota$ is the inclusion) is a proper
smooth embedding and the composition
$\Psi_n^{-1}\circ\psi_n\circ\pi\from\RR^n\times\Pi_n\to\RR^n\times\RR^{n-1}_+$
is a smooth map of $\codim{n-1}$ manifolds which is a smooth covering
map away from the big diagonal (i.e., the subset on which the
$S_n$-action is not free). That is, we have the following diagram
\begin{equation}\label{eq:Psi-diagram}
\vcenter{\hbox{\begin{tikzpicture}[xscale=4]
  \node (rr) at (0,0) {$\RR^n\times\RR^n$};
  \node (ra) at (0,1) {$\RR^n\times\Affine^{n-1}$};
  \node (rp) at (0,2) {$\RR^n\times\Pi_n$};
  \node (rrq) at (1,0.5) {$\bigl(\RR^n\times\RR^n\bigr)/S_n$};
  \node (raq) at (1,1.5) {$\bigl(\RR^n\times\Affine^{n-1}\bigr)/S_n$};
  \node (rpq) at (1,2.5) {$\bigl(\RR^n\times\Pi_n\bigr)/S_n$};
  \node (r2n) at (2,0) {$\RR^{2n}$};
  \node (r2n1) at (2,1) {$\RR^{2n-1}$};
  \node (psirp) at (2,2) {$\psi_n(\bigl(\RR^n\times\Pi_n\bigr)/S_n)$};
  \node (rnrn1) at (3,2) {$\RR^n\times\RR^{n-1}_+$};

  \draw[->,thin] (rr) -- (rrq) node[midway,anchor=south] {\tiny $\pi$};
  \draw[->,thin] (ra) -- (raq) node[midway,anchor=south] {\tiny $\pi$};
  \draw[->,thin] (rp) -- (rpq) node[midway,anchor=south] {\tiny $\pi$};

  \draw[->,thick] (rr) -- (r2n);
  \draw[->,thick] (ra) -- (r2n1);
  \draw[->,thick] (rp) -- (psirp);

  \draw[-latex,thin] (rrq) -- ($(r2n)!0.25!(rrq)$) node[midway,anchor=south] {\tiny $\psi_n$};
  \draw[-latex,thin] (raq) -- ($(r2n1)!0.25!(raq)$) node[midway,anchor=south] {\tiny $\psi_n$};
  \draw[-latex,thin] (rpq) -- ($(psirp)!0.4!(rpq)$) node[midway,anchor=south] {\tiny $\psi_n$};
  \draw[-latex,thick] (rnrn1) -- (psirp) node[midway,anchor=south] {\tiny $\Psi_n$};

    \draw[right hook->,thick] (rp) -- (ra);
  \draw[right hook->,thick] (ra) -- (rr);

  \draw[right hook->,thin,preaction={draw=white, -, line width=6pt}] (rpq) -- (raq);
  \draw[right hook->,thin,preaction={draw=white, -, line width=6pt}] (raq) -- (rrq);

  \draw[right hook->,thick] (psirp) -- (r2n1);
  \draw[right hook->,thick] (r2n1) -- (r2n);
  
\end{tikzpicture}}}
\end{equation}
Here, the thick arrows are smooth and the solid-head arrows are
homeomorphisms (and the solid-head thick arrow is the diffeomorphism
$\Psi_n$) and all the embeddings are proper.

Moreover, these maps will be compatible with the identifications of
facets $F_S\subset \bdy_k\Pi_n$ with $\Pi_k\times\Pi_{n-k}$
from Item~\ref{item:permuto-facet-identify}. Specifically, the
following diagram will commute:
\begin{equation}\label{eq:Psi-compatible}
\vcenter{\hbox{\begin{tikzpicture}[xscale=4,yscale=-1]
  \node (fs) at (1,1) {$\RR^n\times F_S$};
  \node (piemb) at (1,3) {$\RR^n\times\RR^{n-1}_+$};

  \node (pi) at (2,1) {$\RR^n\times\Pi_n$};
  \node (pimod) at (2,2) {$\big(\RR^n\times\Pi_n\big)/S_n$};
  \node (pimodemb) at (2,3) {$\RR^{2n-1}$};

  \node (fssplit) at (0,1) {$(\RR^k\times \Pi_k)\times(\RR^{n-k}\times \Pi_{n-k})$};
  \node (fsemb) at (0,2) {$(\RR^k\times \RR^{k-1}_+)\times(\RR^{n-k}\times \RR^{n-k-1}_+)$};
  \node (fsembemb) at (0,3) {$\RR^n\times (\RR^{k-1}_+\times\{0\}\times \RR^{n-k-1}_+)$};

  \draw[->] (fs) -- (fssplit) node[midway,anchor=south] {\tiny$\cong$};
  \draw[right hook->] (fs) -- (pi);
  \draw[->] (pi) -- (pimod) node[midway,anchor=west] {\tiny$\pi$};
  \draw[->] (pimod) -- (pimodemb) node[midway,anchor=west] {\tiny$\psi_n$};
  \draw[->] (fssplit) -- (fsemb) node[midway,anchor=east] {\tiny$(\Psi_k^{-1}\psi_k\pi,\Psi_{n-k}^{-1}\psi_{n-k}\pi)$};
  \draw[->] (fsemb) -- (fsembemb) node[midway,anchor=east] {\tiny$\cong$};
  \draw[right hook->] (fsembemb) -- (piemb);
  \draw[right hook->] (piemb) -- (pimodemb) node[midway,anchor=south] {\tiny$\Psi_n$};
  
\end{tikzpicture}}}
\end{equation}
Here, the identification
$\RR^k\times\RR^{n-k}\stackrel{\cong}{\longrightarrow}\RR^n$ in the bottom-left vertical
arrow is the usual one. However, the identification
$\RR^n\stackrel{\cong}{\longrightarrow}\RR^k\times\RR^{n-k}$ on the top-left horizontal arrow is
similar to the one from Item~\ref{item:permuto-facet-identify}; that
is, we identify $\RR^n$, $\RR^k$, $\RR^{n-k}$ with
$\prod_{ j\in\{1,\dots,n\}}\RR$,
$\prod_{ j\in S}\RR$ and $\prod_{
  j\in S^c}\RR$, respectively, and then identify
$\prod_{ j\in\{1,\dots,n\}}\RR$ with
$\prod_{ j\in S}\RR\times\prod_{
  j\in S^c}\RR$.
\end{proposition}

\begin{proof}
The construction of the diffeomorphism from Equation~\eqref{eq:Psi} is
inductive. For the base case $n=1$, the maps $\pi$ and $\psi_1$ are
identity (after naturally identifying the relevant spaces with $\RR$),
and we may define $\Psi_1$ to be the identity as well.

We explicitly also do the next case $n=2$ to help build
intuition. (Also, a portion of the proof for $n>2$ does not generalize
to $n=2$). Recall that $\RR^2\times\Pi_2$ is the subset of
$\RR^2\times\RR^2$ given by
\[
\beta_1\geq 1,\qquad\beta_2\geq 1,\qquad\beta_1+\beta_2=3,
\]
and $\RR^2\times\bdy\Pi_2$ is given by $\beta_1=1$ (and hence
$\beta_2=2$) or $\beta_2=1$ (and hence $\beta_1=2$). We want to
satisfy the compatibility from Equation~\eqref{eq:Psi-compatible}, so
we need to analyse the image of $\bigl(\RR^2\times\bdy\Pi_2\bigr)/S_2$
under $\psi_2$. Since we are quotienting by the action of $S_2$, we may
assume $\beta_1=1$ and $\beta_2=2$.

From Equation~\eqref{eq:Pz}, the image of the boundary is given by the
coefficients of the polynomial
\[
z^2 - (a_1+ib_1)z + (a_2+ib_2)= (z-\alpha_1-i)(z-\alpha_2-2i) =
z^2-(\alpha_1+\alpha_2+3i)z+(\alpha_1\alpha_2-2)+i(2\alpha_1+\alpha_2).
\]
As we already observed, this lies in the subspace $\RR^3\subset\RR^4$
given by $b_1=3$. The image is a parametrized surface in this $\RR^3$, given by
\begin{equation}\label{eq:parametrized-surface-n2}
a_1=\alpha_1+\alpha_2,\qquad a_2=\alpha_1\alpha_2-2,\qquad b_2=2\alpha_1+\alpha_2.
\end{equation}
Here, $a_1,a_2,b_2$ are the coordinates of $\RR^3$ and
$\alpha_1,\alpha_2$ are the parameters. We may solve for 
$\alpha_1,\alpha_2$ in terms of $a_1,b_2$, and so this surface can
also be written as $\{a_2=3a_1b_2-2a_1^2-b_2^2-2\}$. This is a graph
of 2-variable function in $a_1,b_2$, and therefore represents a
properly embedded $\RR^2$ via the map
$(a_1,b_2)\mapsto(a_1,3a_1b_2-2a_1^2-b_2^2-2,b_2)$. Furthermore, its
complement has two components, say $A=\{a_2<3a_1b_2-2a_1^2-b_2^2-2\}$
and $B=\{a_2>3a_1b_2-2a_1^2-b_2^2-2\}$, and the closure of each is a
properly embedded $\RR^2\times\RR_+$ via the maps
$(a_1,b_2,t)\mapsto(a_1,3a_1b_2-2a_1^2-b_2^2-2-t,b_2)$ and
$(a_1,b_2,t)\mapsto(a_1,3a_1b_2-2a_1^2-b_2^2-2+t,b_2)$,
respectively.

On the ambient $\RR^3$, define the continuous function
\[
\smim_1\from\RR^3\to\RR
\]
as the minimum of the two imaginary parts of the two roots of $z^2 -
(a_1+3i)z + (a_2+ib_2)$. Then the given surface
$\psi_2(\bigl(\RR^2\times\bdy\Pi_2\bigr)/S_2)$ is precisely the subspace
$\{\smim_1=1\}$, while $\psi_2(\bigl(\RR^2\times\Pi_2\bigr)/S_2)$ is the
subspace $\{\smim_1\geq 1\}$. By the intermediate value theorem, the
latter is one of the closures $\bar{A}$ or $\bar{B}$, which is indeed diffeomorphic
to $\RR^2\times\RR_+$. Indeed, we can determine that it is $\bar{A}$
since the value of $\psi_2$ at any specific point of
$\bigl(\RR^2\times\mathring{\Pi}_1\bigr)/S_2$ (say
$\alpha_1=\alpha_2=0,\beta_1=\beta_2=3/2$) lies in $A$ (at the point
$a_1=b_2=0,a_2=-9/4$).

Although in this case we were explicitly able to see the parametrized
surface as a properly embedded $\RR^2\subset\RR^3$, there is a more
abstract argument which also works. The abstract argument is also
needed to check compatibility.  Let us take $S=\{1\}$ (respectively,
$S=\{2\}$), and start with the point
$((\alpha_1,\alpha_2),(1,2))\in\RR^2\times F_S$ (respectively,
$((\alpha_1,\alpha_2),(2,1))\in\RR^2\times F_S$) in the top-middle
vertex of Equation~\eqref{eq:Psi-compatible}. Starting left and
following four arrows, we get
\[
((\alpha_1,\alpha_2),(1,2))\mapsto ((\alpha_1,1),(\alpha_2,1))\mapsto
((\alpha_1),(\alpha_2))\mapsto
((\alpha_1,\alpha_2),(0))\mapsto((\alpha_1,\alpha_2),(0))\in\RR^2\times\RR^1_+
\]
(respectively,
\[
((\alpha_1,\alpha_2),(2,1))\mapsto ((\alpha_2,1),(\alpha_1,1))\mapsto
((\alpha_2),(\alpha_1))\mapsto
((\alpha_2,\alpha_1),(0))\mapsto((\alpha_2,\alpha_1),(0))\in\RR^2\times\RR^1_+),
\]
while starting right, and following three arrows, we get
\[
((\alpha_1,\alpha_2),(1,2))\mapsto((\alpha_1,\alpha_2),(1,2))\mapsto[((\alpha_1,\alpha_2),(1,2))]\mapsto(\alpha_1+\alpha_2,\alpha_1\alpha_2-2,2\alpha_1+\alpha_2)
\]
(respectively,
\[
((\alpha_1,\alpha_2),(2,1))\mapsto((\alpha_1,\alpha_2),(2,1))\mapsto[((\alpha_2,\alpha_1),(1,2))]\mapsto(\alpha_2+\alpha_1,\alpha_2\alpha_1-2,2\alpha_2+\alpha_1)).
\]
Therefore, in either case, the compatibility condition tells us that
the map $\Psi_2\from\RR^2\times\RR_+\to\RR^3$ must restrict on the
boundary to the given parametrization from
Equation~\eqref{eq:parametrized-surface-n2}. Letting $\Psi_|$ denote
the map $\RR^2\to\RR^3$ given by
Equation~\eqref{eq:parametrized-surface-n2}, we then have to prove the
following:
\begin{enumerate}[leftmargin=*]
\item \textbf{Properly embedded:} We have to prove
  $\Psi_|\from\RR^2\to\RR^3$ is a proper smooth embedding, that is, it is a
  proper map which is an immersion and a diffeomorphism onto its image.

  For injectivity, note that image of $\Psi_|$ is the subspace
  $\{\smim_1=1\}$ of $\RR^3$, which consists of monic quadratic
  polynomials whose one root has imaginary part $1$ and the other root
  has imaginary part $2$; and such polynomials \emph{uniquely}
  factorize as $(z-\alpha_1-i)(z-\alpha_2-2i)$.

  To show it is an immersion, consider the map
  $\psi_2\circ\pi\from\RR^2\times\RR^2\to\RR^4$ given by the Vi\`ete
  relations, as shown in the bottom row of
  Equation~\eqref{eq:Psi-diagram}. It is well-known that away from the
  diagonal this is a local diffeomorphism (that is, roots of a
  polynomial are smooth functions of its coefficients), so
  $d(\psi_2\circ\pi)$ has rank $4$ at each point. Our map $\Psi_|$ is
  the restriction to the $2$-dimensional affine subspace
  $\beta_1=1,\beta_2=2$ (which lies in the complement of the
  diagonal), and so $d\Psi_|$ has rank $2$ at each point.

  The statement that the inverse map is smooth is again just the
  statement that the roots of a polynomial are smooth functions of its
  coefficients.

  Finally, $\Psi_|$ is automatically proper since it is an embedding
  with closed image $\{\smim_1=1\}\subset\RR^3$.
\item \textbf{Extendable:} We have to prove the proper smooth embedding
  $\Psi_|\from\RR^2\to\RR^3$ extends to a proper smooth embedding
  $\Psi_2\from\RR^2\times\RR_+\to\RR^3$ with image
  $\bar{A}=\{\smim_1\geq
  1\}=\psi_2(\bigl(\RR^2\times\Pi_2\bigr)/S_2)$. This is the part where
  the general proof for $n>2$ does not work, so we have to fall back
  to our explicit computation from before and simply set
  \begin{equation}\label{eq:Psi2-explicit-map}
    \Psi_2(\alpha_1,\alpha_2,t)=(\alpha_1+\alpha_2,\alpha_1\alpha_2-2-t,2\alpha_1+\alpha_2).
  \end{equation}
\end{enumerate}

Let us now do the general case for $n \geq 3$. We will use the
following.
\begin{proposition}\label{prop:proper-h-cobordism}
  Assume $m \geq 4$.  Let $i\from\RR^m \hookrightarrow \RR^{m+1}$ be a
  proper smooth embedding. Then the embedding splits $\RR^{m+1}$ into
  two pieces, each the image of a proper smooth embedding of $\RR^m \times \RR_+$; moreover,
  the embeddings of $\RR^m\times\RR_+$ may be chosen to agree with $i$
  on the boundary.
\end{proposition}

\begin{proof}
The Jordan-Brouwer theorem says that the complement of $i(\RR^m)$ has
two connected components. Let $A$ be the closure of one of these
components. Then $A$ is a smooth manifold with boundary, and we need
to show it is a properly embedded $\RR^m \times \RR_+$.

By taking one-point compactifications, we can extend $i$ to an
embedding of $S^m$ into $S^{m+1}$, smooth away from a point. By a
result of Kirby \cite{KirbyLF}, since $m \geq 4$, the embedding cannot
fail to be locally flat at exactly one point. Therefore, it is locally
flat. By the topological Sch\"onflies theorem \cite{Brown, Mazur,
  Morse}, $S^m$ splits $S^{m+1}$ into two pieces, each homeomorphic to
$B^m$. After removing the point at infinity, we get that $A$ is
homeomorphic to $\RR^m \times \RR_+.$ We will show that $A$ is
diffeomorphic to $\RR^m\times\RR_+$.

We also know that $A$ is smooth, and $\del A$ is a properly embedded
$\RR^m$. Using the tubular neighborhood theorem for proper smooth embeddings,
choose a tubular neighborhood $V$ of $\del A$ in $A$, with closure
$\bar V$, such that we have a diffeomorphism
\[
\phi_1\from \bar V
\to \RR^m \times [0, \epsilon].
\]
By the same argument as for $A$ (taking one-point compactifications),
we find that $A \setminus V$ is homeomorphic to $\RR^m \times
\RR_+$. Let
\[
\phi_2\from (A\setminus V) \to \RR^m \times [\epsilon, \infty)
\]
be a homeomorphism.

Let us identify $\RR^m \times \{\epsilon\}$ with $\RR^m$, and view the
restriction of $\phi_1 \circ \phi_2^{-1}$ as a self-homeomorphism $h\from \RR^m
\to \RR^m$. This can be extended to a self-homeomorphism
\[
  \tilde h\from\RR^m \times [\epsilon, \infty) \to \RR^m
  \times[\epsilon,\infty), \qquad \tilde h(x, t) = (h(x), t).
\]
By replacing $\phi_2$ with $\tilde h\circ \phi_2$, we can assume
without loss of generality that $\phi_1$ and $\phi_2$ coincide on the
boundary $\del (A \setminus V)$. Let
$\bar{\phi}\from A\to \RR^m\times[0,\infty)$ be the homeomorphism
obtained by gluing $\phi_1$ and $\phi_2$.

Consider the product $A \times \RR_+$, which is homeomorphic to
$\RR^m\times\RR_+\times\RR_+$ by
$\bar{\phi}\times\operatorname{id}$. This is an $(m+1)$-dimensional
smooth manifold with corners. Fix a strictly increasing smooth
function $f\from[0,\epsilon)\to\RR_+$ with $f(0)=1$, $f'(0)=0$, and
$\lim_{t\to\epsilon}f(t)=\infty$, for instance
$f(t)=\sec\big(\frac{\pi t}{2\epsilon}\big)$. Inside the half-strip
$V\times\RR_+$, consider the graph $\Gamma$ of the smooth composition
function
\[
  V\stackrel{\phi_1}{\longrightarrow}
\RR^m\times[0,\epsilon)\stackrel{\pi_2}{\longrightarrow}[0,\epsilon)
\stackrel{f}{\longrightarrow}\RR_+.
\]
Note that $\Gamma$ is diffeomorphic to $\RR^m\times\RR_+$ and splits
the collar $V \times \R_+$ into two pieces $B$ (above $\Gamma$) and
$C$ (below $\Gamma$), as in Figure~\ref{fig:ProperHCob}.

\begin{figure}
  \centering
  \begin{tikzpicture}[scale=1.3,xscale=0.75]
    \draw[thick] (0,0) -- (4,0) node[pos=1,anchor=north
    east] {$A$} node[pos=0.25,anchor=north,color=black]
    {$V$};

    \draw[->] (0,0) -- (0,3) node[pos=1,anchor=east] {$\RR_+$};

    \draw (2,0)--++(0,3);

    \foreach \i in {0,1}{
      \node[anchor=east] at (0,\i) {\i};
    }

    \node[anchor=south west] at (0,1) {$\Gamma$};
    
    \draw[thick] (0,1) ..controls (1,1) and (1.95,1).. (1.95,3);

    \foreach \i/\j/\l in {1/2/B,1/0.5/C,3/1.25/D}{
      \node at (\i,\j) {$\l$};
      }
    
    \foreach \i/\j in {0/0,2/0,0/1}{
      \node at (\i,\j) {$\bullet$};
      }
  \end{tikzpicture}
  \caption {The smooth manifold (with corners) $A\times\RR_+$. The
    region $C \cup D$ is a proper h-cobordism with boundary, from $A$
    to $\Gamma$.}
  \label{fig:ProperHCob}
\end{figure}

Let $D = (A \setminus V) \times \R_+$. Then $C \cup D$ is a proper
cobordism with boundary, from $A$ to $\Gamma.$ In fact, the
homeomorphism $\bar{\phi} \times \operatorname{id}$ sends $C\cup D$ to
$\RR^m\times E$, where $E\subset\RR^2_+$ is the region in the first
quadrant that lies below and to the right of the graph of the function
$f(t),0\leq t<\epsilon$. Since $E$ is diffeomorphic to
$\RR_+\times[0,1]$, we find that $C \cup D$ is homeomorphic to
$\RR^m \times \R_+ \times [0,1]$. Thus, $C \cup D$ is a proper h-cobordism with
boundary, with inclusion of either end a simple
equivalence, and has dimension $m+2 \geq 6$.  Siebenmann's
proper h-cobordism theorem \cite{Siebenmann} (applied to cobordisms of
manifolds with boundary, where the boundary cobordism is a cylinder,
cf.~\cite[Footnote~1 on pp.~484]{Siebenmann}) implies that $C \cup D$
is diffeomorphic to a cylinder. 
Therefore, $A$ is diffeomorphic to $\Gamma\cong\RR^m \times \RR_+$.

If $j\from \RR^m\times\RR_+\stackrel{\cong}{\longrightarrow}A$ is any
diffeomorphism, then let $g\from \RR^m\to \RR^m$ denote the
self-diffeormorphism $j^{-1}\circ i$; extend this to a
self-diffeomorphism
$\tilde{g}\from \RR^m\times\RR_+\to\RR^m\times\RR_+$ by
$\tilde{g}(x,t)=(g(x),t)$. Then
$j\circ \tilde{g}\from \RR^m\times\RR_+\to A$ is a diffeomorphism that
agrees with given embedding $i$ on the boundary.  Moreover, the
composition
$\RR^m\times\RR_+\stackrel{j\circ \tilde{g}}{\longrightarrow}
A\into\RR^{m+1}$ is a smooth embedding with closed image, so it is a proper
embedding.
\end{proof}

We will also need the following in order to glue diffeomorphisms on
the boundary.
\begin{proposition}\label{prop:collaring-unique}
  Let $i\from\RR^m\times\RR_+ \into \RR^{m+1}$ be a proper
  smooth embedding. Let $j\from\RR^m\times[0,1]\into\RR^{m+1}$ be a collar
  neighborhood of $i(\RR^m\times\{0\})$ inside $i(\RR^m\times\RR_+)$
  which agrees with $i$ on $\RR^m\times\{0\}$. Then there exists a
  proper smooth embedding $i'\from\RR^m\times\RR_+\into \RR^{m+1}$ with the
  same image as $i$ which agrees with $j$ on some open neighborhood of
  $\RR^m\times\{0\}$ inside $\RR^m\times [0,1]$.
\end{proposition}

\begin{proof}
  By pulling back by the diffeomorphism $i$, we may assume $i$ is the
  identity map. That is, we have a collar neighborhood
  $j\from\RR^m\times[0,1]\to\RR^m\times\RR_+$ with $j(x,0)=(x,0)$ for
  all $x\in\RR^m$, and we want to construct a diffeomorphism
  $i'\from\RR^m\times\RR_+\to\RR^m\times\RR_+$ which agrees with $j$
  on some open set around $\RR^m\times\{0\}$.

  Apply the local collaring uniqueness theorem \cite[Theorem A.1]{KS-book} with $M=\RR^m\times\{0\}$,
  $W=\RR^m\times\RR_+$, $C=\varnothing$, $D=M$,
  $f\from\RR^m\times[0,1]\to\RR^m\times\RR_+$ the given collar
  neighborhood $j$, and $g\from\RR^m\times[0,1]\to\RR^m\times\RR_+$
  the standard inclusion. The required diffeomorphism $i'$ is then the
  map $h_1\from\RR^m\times\RR_+\to\RR^m\times\RR_+$.
\end{proof}


  
  

We will now construct a proper smooth embedding
$\Psi_n\from\RR^n\times\RR_+^{n-1}\to\RR^{2n-1}$ with image
$\psi_n(\bigl(\RR^n\times\Pi_n\bigr)/S_n)$, as in
Equation~\eqref{eq:Psi-diagram}. Moreover, it will respect the
$\codim{n-1}$-manifold structures, so it will map
$\RR^n\times\bdy_J\RR_+^{n-1}$ to
$\psi_n(\bigl(\RR^n\times\bdy_J\Pi_n\bigr)/S_n)$ for any
$J\subset\{1,2,\dots,n-1\}$, where $\bdy_J$ denotes
$\displaystyle\bigcap_{j\in J}\bdy_j$.

The map $\Psi_n$ is required to satisfy the compatibility condition from
Equation~\eqref{eq:Psi-compatible}; therefore, it is already defined
on $\RR^n\times\bdy\RR^{n-1}_+$ respecting the stratification given by
the closed strata $\RR^n\times\bdy_J\RR^{n-1}_+$. This statement
deserves further details.

Define continuous functions $\smim_k\from\RR^{2n-1}\to\RR$, $k=1,\dots,n-1$,
as follows: $\smim_k(a_1,\dots,a_n,b_2,\dots,b_n)$ is the sum of the
$k$ smallest imaginary parts among the $n$ roots of the polynomial
\[
  z^n - (a_1+in(n+1)/2)z^{n-1} + (a_2+ib_2)z^{n-2} - \dots + (-1)^n ( a_n +
  ib_n).
\]
Recall from Item~\ref{item:permuto-halfspace} that the permutohedron
$\Pi_n\subset\Affine^{n-1}$ is the intersection of the half-spaces
$\HalfSpace_S=\{(x_1,\dots,x_n)\in\Affine^{n-1}\mid \sum_{j\in S}
x_j\geq k(k+1)/2\}$; therefore,
$\psi_n(\big(\RR^n\times\Pi_n\big)/S_n)$ is precisely the subspace
\[
  \{\smim_1\geq 1,\smim_2\geq 3,\dots,\smim_k\geq
  k(k+1)/2,\dots,\smim_{n-1}\geq (n-1)n/2\}\subset \RR^{2n-1}.
\]
Moreover, if $|S|=k$, then $F_S=\Pi_n\cap\bdy \HalfSpace_S$, and
therefore $\psi_n(\big(\RR^n\times\bdy_k\Pi_n\big)/S_n)$ is precisely the
subspace
\[
  \{\smim_1\geq 1,\smim_2\geq 3,\dots,\smim_k=
  k(k+1)/2,\dots,\smim_{n-1}\geq (n-1)n/2\}\subset \RR^{2n-1}.
\]
More tersely, define the continuous function $\bar\smim\from\RR^{2n-1}\to\RR$ by
\begin{equation}\label{eq:bar-smim}
  \bar\smim=\min\big\{\smim_1-1,\smim_2-3,\dots,\smim_k-k(k+1)/2,\dots,\smim_{n-1}-(n-1)n/2\big\}.
\end{equation}
Then $\psi_n(\big(\RR^n\times\bdy\Pi_n\big)/S_n)=\{\bar\smim=0\}$ and
$\psi_n(\big(\RR^n\times\Pi_n\big)/S_n)=\{\bar\smim\geq 0\}$.

Consider the subset $S=\{1,2,\dots,k\}$, and
the corresponding facet $F_S\subset\bdy_k\Pi_n$. Fix any point
$(\alpha_1,\dots,\alpha_n,\beta_1,\dots,\beta_n)\in\RR^n\times F_S$; so we have
\[
  \beta_1+\dots+\beta_k=k(k+1)/2,\qquad\beta_{k+1}+\dots+\beta_n=\big(n(n+1)-k(k+1)\big)/2.
\]
Starting from this point on the top-middle vertex of
Equation~\eqref{eq:Psi-compatible}, and following three arrows on the
right, we get the point
$(a_1,\dots,a_n,b_2,\dots,b_n)\in\RR^{2n-1}$ by equating the
coefficients of the polynomial---call it $P(z)$---from
Equation~\eqref{eq:Pz}. However, starting left, the first arrow takes
it to the point
\[
  ((\alpha_1,\dots,\alpha_k,\beta_1,\dots,\beta_k),(\alpha_{k+1},\dots,\alpha_n,\beta_{k+1}-k,\dots,\beta_n-k))\in (\RR^k\times\Pi_k)\times(\RR^{n-k}\times\Pi_{n-k}).
\]
Under the map $\psi_k\circ\pi$, the point
$(\alpha_1,\dots,\alpha_k,\beta_1,\dots,\beta_k)$ maps to
$(c_1,\dots,c_k,d_2,\dots,d_k)\in\RR^{2k-1}$ given by equating the
coefficients of the polynomial $Q_0(z)$, again from
Equation~\eqref{eq:Pz} (but with $k$ instead of $n$, $c_j$ instead of
$a_j$, and $d_j$ instead of $b_j$). Under the map
$\psi_{n-k}\circ\pi$, the point
$(\alpha_{k+1},\dots,\alpha_n,\beta_{k+1}-k,\dots,\beta_n-k)$ maps to
the point
$(\tilde{e}_1,\dots,\tilde{e}_{n-k},\tilde{f}_2,\dots,\tilde{f}_{n-k})\in\RR^{2n-2k-1}$
given by equating the coefficients of the polynomial
\[
\tilde{Q}_{1}(z)=\prod_{j=k+1}^n(z-\alpha_j-i(\beta_j-k))=z^{n-k}-(\tilde{e}_1+i\tilde{f}_1)z^{n-k-1}+\dots+(-1)^{n-k}(\tilde{e}_{n-k}+i\tilde{f}_{n-k}),
\]
with $\tilde{f}_1=\frac{(n-k)(n-k+1)}{2}$.  Let
$(e_1,\dots,e_{n-k},f_2,\dots,f_{n-k})\in\RR^{2n-2k-1}$ be given by
the polynomial $Q_{1}(z)$ from Equation~\eqref{eq:Pz} (but with
$n-k$ instead of $n$, $e_j$ instead of $a_j$, $f_j$ instead of $b_j$,
and roots $\alpha_j+i\beta_j$ for $k<j\leq n$):
\[
Q_1(z)=\prod_{j=k+1}^n(z-\alpha_j-i\beta_j)=z^{n-k}-(e_1+if_1)z^{n-k-1}+\dots+(-1)^{n-k}(e_{n-k}+if_{n-k}),
\]
with ${f}_1=\frac{n(n+1)-k(k+1)}{2}$.
Then $(e_1,\dots,e_{n-k},f_2,\dots,f_{n-k})$ can be obtained from $(\tilde{e}_1,\dots,\tilde{e}_{n-k},\allowbreak\tilde{f}_2,\dots,\tilde{f}_{n-k})$ by applying a diffeomorphism
\begin{equation}\label{eq:Xi-n-k}
  \xi_{n,k}\from\RR^{2n-2k-1}\stackrel{\cong}{\longrightarrow}\RR^{2n-2k-1},
\end{equation}
namely, by equating the coefficients of the polynomial $Q_{1}(z)=\tilde{Q}_{1}(z-ik)$:
\begin{align*}
  &z^{n-k}-(e_1+i\frac{n(n+1)-k(k+1)}{2})z^{n-k-1}+\dots+(-1)^{n-k}(e_{n-k}+if_{n-k})\\
  &\qquad\qquad=(z-ik)^{n-k}-(\tilde{e}_1+i\frac{(n-k)(n-k+1)}{2})(z-ik)^{n-k-1}+\dots+(-1)^{n-k}(\tilde{e}_{n-k}+i\tilde{f}_{n-k}).
\end{align*}

Assume
\begin{align*}
  \Psi^{-1}_k(c_1,\dots,c_k,d_2,\dots,d_{k})&=(s_1,\dots,s_k,t_1,\dots,t_{k-1})\in\RR^k\times\RR^{k-1}_+\\
  \Psi^{-1}_{n-k}(\tilde{e}_1,\dots,\tilde{e}_{n-k},\tilde{f}_2,\dots,\tilde{f}_{n-k})&=(s_{k+1},\dots,s_n,t_{k+1},\dots,t_{n-1})\in\RR^{n-k}\times\RR^{n-k-1}_+.
\end{align*}
Then the compatibility condition from
Equation~\eqref{eq:Psi-compatible} states that $\Psi_n$ must be
defined on $\RR^n\times\bdy_k\RR^{n-1}_+$ as follows:
\[
  \Psi_n((s_1,\dots,s_n),(t_1,\dots,t_{k-1},0,t_{k+1},\dots,t_{n-1}))=(a_1,\dots,a_n,b_2,\dots,b_{n}).
\]
This map has an explicit description in terms of $\Psi_k$ and
$\Psi_{n-k}$. We have
\begin{align*}
  \Psi_k(s_1,\dots,s_k,t_1,\dots,t_{k-1})&=(c_1,\dots,c_k,d_2,\dots,d_{k})\in\RR^{2k-1}\\
  \xi_{n,k}\circ\Psi_{n-k}(s_{k+1},\dots,s_n,t_{k+1},\dots,t_{n-1})&=({e}_1,\dots,{e}_{n-k},{f}_2,\dots,{f}_{n-k})\in\RR^{2n-2k-1},
\end{align*}
and since $P(z)=Q_0(z)Q_1(z)$, the point $(a_1,\dots,a_n,b_2,\dots,b_n)$
can be obtained from $(c_1,\dots,c_k,\allowbreak d_2,\dots,d_{k})$ and
$(e_1,\dots,e_{n-k},f_2,\dots,f_{n-k})$ by equating the coefficients
of
\begin{align*}
  &z^{n}-(a_1+i\frac{n(n+1)}{2})z^{n-1}+\dots+(-1)^{n}(a_n+ib_n)\\
  &\qquad=\big(z^{k}-(c_1+i\frac{k(k+1)}{2})z^{k-1}+\dots+(-1)^{k}(c_n+id_n)\big)\\
  &\qquad\qquad\times\big(z^{n-k}-(e_1+i\frac{n(n+1)-k(k+1)}{2})z^{n-k-1}+\dots+(-1)^{n-k}(e_{n-k}+if_{n-k})\big).
\end{align*}

\begin{figure}
  \centering
  \begin{tikzpicture}[scale=1]
    \coordinate (o) at (0,0);
    \coordinate (x) at (4,0);
    \coordinate (y) at (0,4);

    \fill[black!25] ($(y)+(-1,0)$) to[out=-45,in=90] (210:1) to[out=-90,in=210] (-90:1) to[out=30,in=110] ($(x)+(0,-1)$) -- ($(x)+(0,1)$) to[out=110,in=-45] (53:2.3) to[out=135,in=-45] ($(y)+(1,0)$) -- cycle;

    \fill[pattern=north east lines, pattern color=black!70] (y) to[out=-45,in=90] (o) to[out=30,in=110] (x) -- ($(x)+(0,0.5)$) to[out=110,in=-45] (53:2) to[out=135,in=-45] ($(y)+(0.5,0)$) -- cycle;
    
    \draw[thick] (y) to[out=-45,in=90] node[midway,anchor=south,rotate=90,outer sep=2ex] {\tiny $\Psi_|(\RR^3\times 0\times\RR_+)$} (o) to[out=30,in=110] node[pos=0] {$\bullet$} node[pos=0,anchor=east] {\tiny $\Psi_|(\RR^3\times(0,0))$} node[midway,anchor=north,outer sep=2ex] {\tiny $\Psi_|(\RR^3\times\RR_+\times 0)$} (x);

    \draw[->] (2,3) -- (1,3) node[pos=0,anchor=west] {\tiny$\smim_1>1,\smim_2> 3$};
    \draw[->] (-1.2,3.5) -- (-0.2,3.5) node[pos=0,anchor=east] {\tiny$\smim_1<1,\smim_2> 3$};
    \draw[->] (4.8,-0.3) -- (3.8,-0.3) node[pos=0,anchor=west] {\tiny$\smim_1>1,\smim_2< 3$};
    \draw[->] (-1.5,-0.7) -- (-0.5,-0.7) node[pos=0,anchor=east] {\tiny$\smim_1<1,\smim_2< 3$};
    
    \draw[thick,dashed] (-90:1) --(o)--(210:1);
  \end{tikzpicture}
  \caption{The embedding $\Psi_|\from\RR^3\times\bdy\RR_+^2\to\RR^5$
    (for the case $n=3$). We have indicated the signs of the functions
    $\smim_1-1$ and $\smim_2-3$ in a neighborhood (shown in gray) of
    $\Psi_|(\RR^3\times\bdy\RR_+^2)$. A collar extension (drawn
    striped) is also shown, which will be used to extend $\Psi_|$ to
    $\Psi_3\from\RR^3\times\RR^2_+\to\RR^5$.}\label{fig:Psi-induction-step}
\end{figure}

As before, let $\Psi_|$ denote this map
$\RR^n\times\bdy\RR^{n-1}_+\to\RR^{2n-1}$ with image
$\psi_n(\big(\RR^n\times\bdy\Pi_n\big)/S_n)$. We have to
extend it to a map $\Psi_n\from\RR^n\times\RR^{n-1}_+\to\RR^{2n-1}$
which is a diffeomorphism onto its image
$\psi_n(\big(\RR^n\times\Pi_n\big)/S_n)$. Therefore, we need to
check the following.
\begin{enumerate}[leftmargin=*]
\item \textbf{Well-defined:} For well-definedness we have to prove the
  following couple of things.

  The map $\Psi_|$ on $\RR^n\times\bdy_k\RR^{n-1}_+$ was defined from
  the subset $S=\{1,2,\dots,k\}$ via
  Equation~\eqref{eq:Psi-compatible}. We need to show that for any
  other subset $S'$ with $|S'|=k$, Equation~\eqref{eq:Psi-compatible}
  would have produced the same map. Fix any point
  $p'=(\alpha'_1,\dots,\alpha'_n,\beta'_1,\dots,\beta'_n)\in\RR^n\times
  F_{S'}$. There exists some permutation $\sigma\in S_n$ which maps
  this point to some point
  $p=(\alpha_1,\dots,\alpha_n,\beta_1,\dots,\beta_n)\in\RR^n\times
  F_{S}$. Since we quotient by $S_n$, if we start at either $p$ or
  $p'$ at the top-middle vertex of Equation~\eqref{eq:Psi-compatible}
  and proceed rightwards by three arrows, we will end up at the same
  point $(a_1,\dots,a_n,b_2,\dots,b_n)\in\RR^{2n-1}$. However, if we
  proceed leftwards by one arrow, we will end up at the point
  \[
    ((\alpha_1,\dots,\alpha_k,\beta_1,\dots,\beta_k),(\alpha_{k+1},\dots,\alpha_n,\beta_{k+1}-k,\dots,\beta_n-k))\in(\RR^k\times\Pi_k)\times(\RR^{n-k}\times\Pi_{n-k});
  \]
  this is due to way the identification
  $\RR^n\stackrel{\cong}{\longrightarrow}\RR^k\times\RR^{n-k}$ is set up
  in Equation~\eqref{eq:Psi-compatible}. Therefore, the map $\Psi_|$
  defined using the subset $S'$ will be same as the map $\Psi_|$
  defined using the subset $S=\{1,2,\dots,k\}$.

  We also need to check that the maps $\Psi_|$, as defined on
  $\RR^n\times\bdy_k\RR^{n-1}_+$ and
  $\RR^n\times\bdy_\ell\RR^{n-1}_+$, agree on their common boundary
  $\RR^n\times\bdy_{\{k,\ell\}}\RR^{n-1}_+$. Let $k<\ell$, and let
  $S=\{1,2,\dots,k\}$ and $T=\{1,2,\dots,\ell\}$. Using repeated
  applications of Equation~\eqref{eq:Psi-compatible}, it is not hard
  to show that in either case, the map
  $\Psi_|\from \RR^n\times\bdy_{\{k,\ell\}}\RR^{n-1}_+\to\RR^{2n-1}$
  is given as follows. For any point
  $p=((s_1,\dots,s_n),(t_1,\dots,t_{k-1},0,t_{k+1},\dots,t_{\ell-1},0,t_{\ell+1},\dots,t_{n-1}))\in\RR^n\times\bdy_{\{k,\ell\}}\RR^{n-1}_+$,
  let
  \begin{align*}
    (c_1,\dots,c_k,d_2,\dots,d_k)&=\Psi_k(s_1,\dots,s_k,t_1,\dots,t_{k-1})\\
    ({e}_1,\dots,{e}_{\ell-k},{f}_2,\dots,{f}_{\ell-k})&=\xi_{\ell,k}\circ\Psi_{\ell-k}(s_{k+1},\dots,s_{\ell},t_{k+1},\dots,t_{\ell-1})\\
    ({g}_1,\dots,{g}_{n-\ell},{h}_2,\dots,{h}_{n-\ell})&=\xi_{n,\ell}\circ\Psi_{n-\ell}(s_{\ell+1},\dots,s_{n},t_{\ell+1},\dots,t_{n-1}).
  \end{align*}
  where $\xi_{\ell,k}\from\RR^{2\ell-2k-1}\to\RR^{2\ell-2k-1}$ and
  $\xi_{n,\ell}\from\RR^{2n-2\ell-1}\to\RR^{2n-2\ell-1}$ are as defined
  in Equation~\eqref{eq:Xi-n-k}. Then $\Psi_|(p)=(a_1,\dots,a_n,b_2,\dots,b_n)$ is obtained by equating the coefficients of
  \begin{align*}
  &z^{n}-(a_1+i\frac{n(n+1)}{2})z^{n-1}+\dots+(-1)^{n}(a_n+ib_n)\\
  &\qquad=\big(z^{k}-(c_1+i\frac{k(k+1)}{2})z^{k-1}+\dots+(-1)^{k}(c_n+id_n)\big)\\
  &\qquad\qquad\times\big(z^{\ell-k}-(e_1+i\frac{\ell(\ell+1)-k(k+1)}{2})z^{\ell-k-1}+\dots+(-1)^{\ell-k}(e_{\ell-k}+if_{\ell-k})\big)\\
  &\qquad\qquad\times\big(z^{n-\ell}-(g_1+i\frac{n(n+1)-\ell(\ell+1)}{2})z^{n-\ell-1}+\dots+(-1)^{n-\ell}(g_{n-\ell}+ih_{n-\ell})\big).
  \end{align*}
\item \textbf{Properly embedded:} We next have to prove
  $\Psi_|\from\RR^n\times\bdy\RR^{n-1}_+\to\RR^{2n-1}$ is a proper
  smooth embedding, that is, it is a proper map which is an immersion and a
  diffeomorphism onto its image.

  For injectivity, we first check that the image of different open
  strata of $\RR^n\times\bdy\RR^{n-1}_+$ are disjoint. This follows
  from the observation that for any non-empty subset
  $J\subset\{1,2,\dots,n-1\}$, the image of the interior of
  $\RR^n\times\bdy_J\RR^{n-1}_+$ is given by
  \begin{equation}\label{eq:image-of-open-stratum}
    \{p\in\RR^{2n-1}\mid \smim_{k}(p)=\frac{k(k+1)}{2},\,\forall k\in
    J\text{ and }\smim_k(p)>\frac{k(k+1)}{2},\,\forall k\notin J\}.
  \end{equation}
  So it is now enough to show that $\Psi_|$, restricted to any open
  stratum, is injective.  Fix any stratum $\RR^n\times\bdy_J(\RR^n)$
  with $J=\{k_1<k_2<\dots<k_m\}$, and fix any point
  \[
    p=((s_1,\dots,s_n),(t_1,\dots,t_{k_1-1},0,t_{k_1+1},\dots,t_{k_2-1},0,t_{k_2+1},\dots,t_{k_m-1},0,t_{k_m+1},\dots,t_{n-1}))
  \]
  in its interior. Let
  $\Psi_|(p)=(a_1,\dots,a_n,b_2,\dots,b_n)\in\RR^{2n-1}$ and consider the polynomial
  \begin{equation}\label{eq:Pz-again}
    P(z)=  z^{n}-(a_1+ib_1)z^{n-1}+\dots+(-1)^{n}(a_n+ib_n)
  \end{equation}
  with $b_1=\frac{n(n+1)}{2}$.  In the same vein as the definition of
  $\Psi_|$ (and also the second part of the above argument for
  well-definedness), for $\ell=0,1,\dots,m$, define
  \[
    (c^\ell_1,\dots,c^\ell_{k_{\ell+1}-k_\ell},d^\ell_2,\dots,d^\ell_{k_{\ell+1}-k_\ell})=\xi_{k_{\ell+1},k_\ell}\circ\Psi_{k_{\ell+1}-k_\ell}(s_{k_\ell+1},\dots,s_{k_{\ell+1}},t_{k_\ell+1},\dots,t_{k_{\ell+1}-1})
  \]
  (with the understanding that $k_0=0$ and $k_{m+1}=n$), and consider the polynomial
  \begin{equation}\label{eq:Q-ell-poly}
    Q_\ell(z)=z^{k_{\ell+1}-k_\ell}-(c^\ell_1+id^\ell_1)z^{k_{\ell+1}-k_\ell-1}+\dots+(-1)^{k_{\ell+1}-k_\ell}(c^\ell_{k_{\ell+1}-k_\ell}+id^\ell_{k_{\ell+1}-k_\ell})
  \end{equation}
  with $d^\ell_1=\frac{k_{\ell+1}(k_{\ell+1}+1)-k_\ell(k_\ell+1)}{2}$. Then we have
  \begin{equation}\label{eq:Pz-factorized}
    P(z)=Q_0(z)Q_1(z)\cdots Q_m(z)
  \end{equation}
  and expanding the coefficients, we get
  $(a_1,\dots,a_n,b_2,\dots,b_n)$ as a function of
  $(c^1_1,\dots,c^m_{n-k_m},\allowbreak d^1_2,\dots,d^m_{n-k_m})$. The
  maps $\Psi_{k_{\ell+1}-k_\ell}$ are embeddings and the maps
  $\xi_{k_{\ell+1},k_\ell}$ are diffeomorphisms, so the only place
  non-injectivity might arise is in this final map. Let
  $\alpha_1+i\beta_1,\dots,\alpha_n+i\beta_n$ be the roots of $P(z)$
  arranged in increasing order of their imaginary parts, that is,
  $\beta_1\leq\beta_2\leq\cdots\leq\beta_n$. From
  Equation~\eqref{eq:image-of-open-stratum}, we know:
  \[
    \sum_{j=1}^{k}\beta_j=\frac{k(k+1)}{2},\,\forall k\in J\cup\{n\},\qquad
    \sum_{j=1}^k \beta_j> \frac{k(k+1)}{2},\,\forall k\in\{1,2,\dots,n-1\}\setminus J
  \]
  Therefore, for every $k\in J$,
  \begin{equation}\label{eq:sum-beta}
    \begin{split}
      \beta_{k}&=\sum_{j=1}^{k}\beta_j-\sum_{j=1}^{k-1}\beta_j\leq \frac{k(k+1)}{2}-\frac{(k-1)k}{2}=k\\
      \beta_{k+1}&=\sum_{j=1}^{k+1}\beta_j-\sum_{j=1}^{k}\beta_j\geq \frac{(k+1)(k+2)}{2}-\frac{k(k+1)}{2}=k+1.
    \end{split}
  \end{equation}
  Therefore, the set of $k$ smallest $\beta_j$'s is well-defined, for
  every $k\in J$. Note that for every $1\leq\ell\leq m$, the polynomial
  $Q_0(z)Q_1(z)\cdots Q_{\ell-1}(z)$ has roots $\alpha_j+i\beta_j$,
  for the $k_\ell$ smallest $\beta_j$'s. Therefore, the factorization
  from Equation~\eqref{eq:Pz-factorized} is the unique one of that
  form, thus completing the proof of injectivity.

  Being an immersion is a local condition. Since
  $\RR^n\times\bdy\RR^{n-1}_+$ is not a smooth manifold, rather the
  boundary of a manifold with corners, let us clarify what we mean by
  an immersion. For any point $p\in\RR^n\times\bdy\RR^{n-1}_+$, we
  will construct a neighborhood $U$ of $p$ inside
  $\RR^n\times\RR^{n-1}$ and an extension of $\Psi_|$ to $U$ which is
  a smooth embedding. Continuing from the proof of injectivity, fix a point
  $p$ in some open stratum $\RR^n\times\mathring{\bdy}_J\RR^{n-1}_+$,
  and let us reuse the same notation.  Instead of parametrizing
  $\Psi_|(\RR^n\times\mathring{\bdy}_J(\RR^n)$ by the parameters $s_j$
  ($1\leq j\leq n$) and $t_j$ ($j\in\{1,\dots,n-1\}\setminus J$), let
  us parametrize it by the variables $c^\ell_j$ ($0\leq \ell\leq m$,
  $1\leq j\leq k_{\ell+1}-k_\ell$) and $d^\ell_j$ ($0\leq \ell\leq m$,
  $2\leq j\leq k_{\ell+1}-k_\ell$), which is a valid reparametrization
  since the maps $\Psi_{k_{\ell+1}-k_\ell}$ are smooth embeddings and the
  maps $\xi_{k_{\ell+1},k_\ell}$ are diffeomorphisms. As before,
  create and set additional variables
  $d^\ell_1=\frac{k_{\ell+1}(k_{\ell+1}+1)-k_\ell(k_\ell+1)}{2}$, for
  $0\leq\ell\leq m$. Define the polynomials $Q_\ell(z)$ as in
  Equation~\eqref{eq:Q-ell-poly}, and if we define the polynomial
  $P(z)$ using Equations~\eqref{eq:Pz-again}
  and~\eqref{eq:Pz-factorized}, this determines the variables
  $a_1,\dots,a_n,b_1,\dots,b_n$ as smooth functions of the variables
  $(c^\ell_j,d^\ell_j)$ (with $b_1=\frac{n(n+1)}{2}$). We now let the
  variables $d^\ell_1$ vary in small neighborhoods, and apply the same
  function to get the variables $a_j,b_j$ (of course, now $b_1$ also
  varies). Since the inequalities from Equation~\eqref{eq:sum-beta}
  still hold up to some small $\epsilon>0$, we still have
  $\beta_k<\beta_{k+1}$ for all $k\in J$, and therefore, the
  factorization from Equation~\eqref{eq:Pz-factorized} is still
  well-defined, and so the function
  $(c^\ell_j,d^\ell_j)\mapsto(a_j,b_j)$ is still injective. Since the 
  roots of a polynomial are smooth functions of its coefficients, this
  is a local diffeomorphism, and therefore its linearization has rank
  $2n$ near $p$. If we restrict to the affine subspace
  $\sum_\ell d^\ell_1=\frac{n(n+1)}{2}$ (corresponding to the affine
  subspace $b_1=\frac{n(n+1)}{2}$ in the target), we get a local
  diffeomorphism from some neighborhood $U$ of $p$ in
  $\RR^n\times \RR^{n-1}$ to some neighborhood of $\Psi_|(p)$ in
  $\RR^{2n-1}$.

  To finish the argument that $\Psi_|$ is a smooth embedding, we need to
  prove that it is a diffeomorphism onto its image. Since it is already
  an injective immersion, we simply have to show that the inverse map
  is continuous. This is again the statement that roots of a
  polynomial are continuous functions of its coefficients.
  
  The statement that $\Psi_|\from\RR^n\times\bdy\RR^{n-1}_+\to\RR^{2n-1}$
  is proper is automatic since it is an embedding with closed image
  $\{\bar\smim=0\}$, where $\bar\smim$ is the function defined in
  Equation~\eqref{eq:bar-smim}.
\item\label{item:permu-extendable-collar} \textbf{Extendable:} Finally, we have to extend $\Psi_|$ to a
  proper smooth embedding $\Psi_n\from\RR^n\times\RR_+\to\RR^{2n-1}$ with
  image $\psi_n(\big(\RR^n\times\Pi_n\big)/S_n)$. Let $F\from\RR^n\times\RR_+^{n-1}\to\RR_+$ be the function
  \[
    (s_1,\dots,s_n,t_1,\dots,t_{n-1})\mapsto t_1t_2\cdots t_{n-1}.
  \]
  Let $\bK^{n-1}=\RR^{n-1}\setminus\RR^{n-1}_*$ be the union of
  the coordinate axes.  Since $\Psi_|$ is a proper embedding, by the
  collar neighborhood theorem, $\Psi_|$ extends to an embedding---call
  it $\wt{\Psi}$---of some neighborhood $U$ of
  $\RR^n\times\bdy\RR^{n-1}_+$ inside $\RR^n\times\RR^{n-1}$. (Such an
  extension can be obtaining by patching together the local
  extensions---as constructed during the immersion proof
  earlier---using partitions of unity.) By rescaling if necessary,
  we may further assume $U$ contains the subspace
  $F^{-1}([0,\epsilon])$ for some small $\epsilon>0$, and $\wt{\Psi}$
  restricts to a proper smooth embedding on that subspace.  See
  Figure~\ref{fig:Psi-induction-step} for the case $n=3$, which shows
  $\wt{\Psi}(U)$ (inside $\RR^5$) in gray,
  $\wt{\Psi}(U\cap\big(\RR^3\times\bK^2\big))$ as black lines
  (which are solid or dashed depending on whether they are in
  $\Psi_|(\RR^3\times\bdy\RR^2_+)$ or not), and
  $\wt{\Psi}(F^{-1}([0,\epsilon]))$ striped. Then
  $\wt{\Psi}(F^{-1}(\epsilon/2))$ is a smoothly properly embedded
  $\RR^{2n-2}$ inside $\RR^{2n-1}$; let $A$ denote the component of
  its complement that \emph{does not} contain
  $\wt{\Psi}(U\cap\big(\RR^n\times\bK^{n-1}\big))$. Using
  Propositions~\ref{prop:proper-h-cobordism}
  and~\ref{prop:collaring-unique} with $m=2n-2$, we get a proper smooth embedding
  $F^{-1}([\epsilon/2,\infty))\to \RR^{2n-1}$ with image $A$, which
  agrees with $\wt{\Psi}$ on some neighborhood of $F^{-1}(\epsilon/2)$
  inside $F^{-1}([\epsilon/2,\epsilon])$; therefore, it glues with
  $\wt{\Psi}$ on $F^{-1}([0,\epsilon/2])$ and produces a
  proper smooth embedding $\Psi_n\from\RR^n\times\RR^{n-1}_+\to\RR^{2n-1}$.

  The only thing left to check is that $\Psi_n$ has the correct image, namely the subspace
  \[
    \psi_n(\big(\RR^n\times\Pi_n\big)/S_n)= \{\smim_1\geq
    1,\dots,\smim_k\geq k(k+1)/2,\dots,\smim_{n-1}\geq
    (n-1)n/2\}=\{\bar\smim\geq 0\}\subset \RR^{2n-1},
  \]
  where $\bar{\smim}$ is the function defined in
  Equation~\eqref{eq:bar-smim}. By the Jordan-Brouwer theorem,
  $\RR^{2n-1}\setminus\Psi_|(\RR^n\times\bdy\RR^{n-1}_+)$ has two
  components. Let $B$ be the component containing
  $\wt{\Psi}(F^{-1}([0,\epsilon]))$, and therefore,
  $B=\Psi_n(\RR^n\times\RR^{n-1}_+)$. Let
  $C=\RR^{2n-1}\setminus\bar{B}$. The construction of $\Psi_|$ ensures
  that $\bar\smim=0$ precisely on $\Psi_|(\RR^n\times\RR_+^{n-1})$,
  hence $\bar\smim\neq 0$ on $B\cup C$. So we have to show
  $\bar\smim>0$ somewhere on $B$ and $\bar\smim<0$ somewhere on
  $C$. Consider the point
  \[
    p=((0,\dots,0),(1,2,\dots,n))\in\RR^n\times\Pi_n\subset \RR^n\times\Affine^{n-1}.
  \]
  It is contained in the affine subspaces
  $\RR^n\times\bdy \HalfSpace_S\subset \RR^n\times\Affine^{n-1}$ from
  Item~\ref{item:permuto-halfspace} for $S=\{1,2,\dots,k\}$,
  $1\leq k<n$. Consider a small open ball $V$ around $p$ in
  $\RR^n\times\Affine^{n-1}$. These $(n-1)$ hyperplanes cut $V$ into
  $2^{n-1}$ regions, which are distinguished by the signs of the
  following $(n-1)$ functions:
  \[
    \beta_1-1,\beta_1+\beta_2-3,\dots,\sum_{j=1}^k\beta_j-\frac{k(k+1)}{2},\dots,\sum_{j=1}^{n-1}\beta_j-\frac{(n-1)n}{2}.
  \]
  The unique region where all the signs are positive is the one that
  contains $\RR^n\times\mathring{\Pi}_{n-1}$; therefore, the image of
  that region under the map $\psi_n\circ\pi$ (from
  Equation~\eqref{eq:Psi-diagram}) has $\bar\smim>0$, and the image of
  the other $2^{n-1}-1$ regions has $\bar\smim<0$. From the immersion
  proof from before, the map $\wt{\Psi}^{-1}\circ\psi_n\circ\pi$ is a
  diffeomorphism from the small open ball $V$ around $p$ to a (not
  necessarily round) small open ball $W$ around
  $((0,\dots,0),(0,\dots,0))$ inside $U$ inside
  $\RR^n\times\RR^{n-1}$. Under this local map, the union of the
  hyperplanes
  \[
    \bigcup_{\substack{S=\{1,\dots,k\}\\1\leq
        k<n}}V\cap\big(\RR^n\times\bdy\HalfSpace_S\big)
  \]
  maps to
  $W\cap\big(\RR^n\times\bK^{n-1}\big)$. Therefore, one of the
  $2^{n-1}$ regions maps into $\wt{\Psi}^{-1}(B)$, while the other
  $2^{n-1}-1$ regions map into $\wt{\Psi}^{-1}(C)$. Since
  $\bar\smim$ has the same sign on $C$, on these latter $2^{n-1}-1$
  regions, $\bar\smim$ must have the same sign, which then must be
  negative. (We are using $n\geq 3$, so we can distinguish the numbers
  $2^{n-1}-1$ and $1$.) Consequently on $B$, the function $\bar\smim$
  must be positive, thus completing the proof. See also
  Figure~\ref{fig:Psi-induction-step}.\qedhere
\end{enumerate}
\end{proof}

\subsection{Embedding and framing the moduli spaces}\label{sec:final-step-for-permutahedron}
In this section, we will smoothly embed and frame the moduli spaces $X_I$,
which were defined to be permutohedra in
Equation~\eqref{eq:XI-definition}. As in
Equation~\eqref{eq:XI-subset}, let $I=\{p_1<\dots<p_n\}$. Consider
the smooth embedding
\begin{equation}\label{eq:jmath-basic-emb}
  \jmath_I\from X_I\into \RR^n\times\Pi_n,\qquad x\mapsto
  ((p_1,p_2,\dots,p_n),x),
\end{equation}
which also respects the $\codim{n-1}$-manifold structure. As in the
previous sections, using the linear ordering of the elements of $I$,
it will useful to identify the first factor $\RR^n$ with
$\prod_I\RR$ and to treat the second factor $\Pi_n$
as embedded in $\prod_I\RR$.

Compose with the map $\Psi_n^{-1}\circ\psi_n\circ\pi$ from
Proposition~\ref{prop:homeo} to get a smooth map of
$\codim{n-1}$-manifolds
\begin{equation}\label{eq:main-embedding}
  \Psi_n^{-1}\circ\psi_n\circ\pi\circ\jmath_I\from X_I\to \RR^n\times\RR_+^{n-1}.
\end{equation}
Since the points $p_i$ are distinct, every non-trivial element of the
symmetric group $S_n$ sends the subset
$\jmath_I(X_I)\subset \RR^n\times\Pi_n$ to a disjoint subset, and
therefore, the map from Equation~\eqref{eq:main-embedding} is a smooth
embedding. To fit our requirements regarding embeddings of moduli
spaces, we need to replace $\RR$ with $\mathring\RR_+$. Fix a
diffeomorphism $f\from\RR\to\mathring\RR_+$, and consider the
diffeomorphism
\[
  F\from\RR^n\times\RR^{n-1}_+\to\mathring\RR_+^n\times\RR_+^{n-1}, \ \ 
(s_1,\dots,s_n,t_1,\dots,t_{n-1})\mapsto
(f(s_1),\dots,f(s_n),t_1,\dots,t_{n-1}).
\]
Then we embed the moduli spaces by the map
\begin{equation}\label{eq:main-embedding-log}
  \iota_I\defeq F\circ\Psi_n^{-1}\circ\psi_n\circ\pi\circ\jmath_I\from X_I\into \mathring\RR_+^n\times\RR_+^{n-1}.
\end{equation}

We will also need to choose a framing of the normal bundle of this
embedding. At each point $x\in X_I$, this will consist of an internal
frame with $n$ vectors which will span a complement to the tangent
space to $\iota_I(X_I)$ in $\mathring\RR_+^n\times\RR_+^{n-1}$ at
$\iota_I(x)$. The vectors in the frame are indexed by the set $I$; so
we will let $(v_{p_1}(x),\dots,v_{p_n}(x))$ denote the frame at $x$ so
that the vector $v_{p_i}(x)$ corresponds to the point $p_i\in I$. Let
$(e_{p_1},\dots,e_{p_n})$ be the unit vectors in $\prod_I\RR$; they
frame the normal bundle of the embedding $\jmath_I$ from
Equation~\eqref{eq:jmath-basic-emb}. Then define
\begin{equation}\label{eq:XI-framing}
  v_{p_i}(x)=(d\wh{\iota})_x(e_{p_i})\qquad 1\leq i\leq n,x\in X_I,
\end{equation}
where $\wh\iota=F\circ\Psi_n^{-1}\circ\psi_n\circ\pi$.

All that remains is to check that these embeddings and framings
satisfy the required coherence conditions on their lower-dimensional
strata. Fix any non-empty proper subset $J\subset I$ of cardinality
$k$. Then there is a corresponding facet in $\bdy_kX_I$ which is
identified with $X_J\times X_{I\setminus J}$. To show that the
embeddings are coherent, we need to check the following diagram
commutes,
\begin{equation}\label{eq:XI-coherent-embed}
  \vcenter{\hbox{
  \begin{tikzpicture}[xscale=7,yscale=-1.3]

    \node (XJ) at (0,0) {$X_J\times X_{I\setminus J}$};
    \node (XI) at (0,1) {$\bdy_kX_I$};

    \node (Rn) at (1,2) {$\mathring\RR_+^n\times\RR^{k-1}_+\times\{0\}\times\RR^{n-k-1}_+$,};
    \node (Rnfull) at (0,2) {$\mathring\RR_+^n\times\RR^{n-1}$};
    \node (Rk) at (1,0) {$\mathring\RR_+^k\times\RR^{k-1}_+\times\mathring\RR_+^{n-k}\times\RR^{n-k-1}_+$}; 
    \node (Rkswitched) at (1,1) {$\mathring\RR_+^k\times\mathring\RR_+^{n-k}\times\RR^{k-1}_+\times\RR^{n-k-1}_+$}; 

    \draw[right hook->] (XJ) -- (XI);
    \draw[left hook->] (Rn) -- (Rnfull);
    \draw[right hook->] (XJ) -- (Rk) node[midway,anchor=south] {\small $(\iota_J,\iota_{I\setminus J})$};
    \draw[right hook->] (XI) -- (Rnfull) node[midway,anchor=east] {\small $\iota_I$};
    \draw[->] (Rk) -- (Rkswitched) node[midway,anchor=west] {\small $\cong$};
    \draw[->] (Rkswitched) -- (Rn) node[midway,anchor=west] {\small $\cong$};
    
  \end{tikzpicture}}}
\end{equation}
where the identifications on the rightmost column are the usual ones
by rearranging the factors. This follows from the commutativity of the following diagram: 
\[
\vcenter{\hbox{\begin{tikzpicture}[xscale=5,yscale=-1.5]
      
      \node (XI) at (2,0) {$X_I$};
      \node (pi) at (2,1) {$\displaystyle\prod_I\RR\times X_I$};
      \node (piemb) at (2,2) {$\RR^n\times\RR^{n-1}_+$};

      \node (XJ) at (0,0) {$X_J\times X_{I\setminus J}$};
      \node (fssplit) at (0,1) {$(\displaystyle\prod_J\RR\times X_J)\times(\displaystyle\prod_{I\setminus J}\RR\times X_{I\setminus J})$};
      \node (fsemb) at (0,2) {$(\RR^k\times \RR^{k-1}_+)\times(\RR^{n-k}\times \RR^{n-k-1}_+)$};
      \node (fsembemb) at (1.2,2) {$\RR^n\times (\RR^{k-1}_+\times\{0\}\times \RR^{n-k-1}_+)$};

      \node (piemb1) at (2,3) {$\mathring\RR_+^n\times\RR^{n-1}_+$};
      \node (fsemb1) at (0,3) {$(\mathring\RR_+^k\times \RR^{k-1}_+)\times(\mathring\RR_+^{n-k}\times \RR^{n-k-1}_+)$};
      \node (fsembemb1) at (1.2,3) {$\mathring\RR_+^n\times (\RR^{k-1}_+\times\{0\}\times \RR^{n-k-1}_+)$};

      \draw[right hook->] (XJ) -- (XI);
      \draw[right hook->] (XJ) -- (fssplit) node[midway,anchor=east] {\tiny $(\jmath_J,\jmath_{I\setminus J})$};
      \draw[right hook->] (XI) -- (pi) node[midway,anchor=west] {\tiny $\jmath_I$};
      
  \draw[right hook->] (fssplit) -- (pi);

  \draw[->] (fssplit) -- (fsemb) node[midway,anchor=east] {\tiny$(\Psi_k^{-1}\psi_k\pi,\Psi_{n-k}^{-1}\psi_{n-k}\pi)$};
  \draw[->] (pi) -- (piemb) node[midway,anchor=west] {\tiny$\Psi_n^{-1}\psi_n\pi$};

  \draw[->] (fsemb) -- (fsembemb) node[midway,anchor=south] {\tiny$\cong$};
  \draw[right hook->] (fsembemb) -- (piemb);

  \draw[->] (fsemb1) -- (fsembemb1) node[midway,anchor=south] {\tiny$\cong$};
  \draw[right hook->] (fsembemb1) -- (piemb1);

  \draw[->] (fsemb) -- (fsemb1) node[midway,anchor=east] {\tiny $(F,F)$};
  \draw[->] (fsembemb) -- (fsembemb1) node[midway,anchor=east] {\tiny $F$};
  \draw[->] (piemb) -- (piemb1) node[midway,anchor=west] {\tiny $F$};
  
\end{tikzpicture}}}
\]
The central pentagon commutes by Equation~\eqref{eq:Psi-compatible};
the top rectangle commutes by definition (Equation~\eqref{eq:jmath-basic-emb});
and the bottom two rectangles commute since the map $F$ was defined
using the map $f$ on each $\RR$ component.

To see that the framings are coherent, we have to show that the normal
framing of $\iota_I$ from Equation~\eqref{eq:XI-framing} agrees with
the product framing on the subspace $X_J\times X_{I\setminus
  J}$. However, the normal framings of the embeddings $\jmath_I$ from
Equation~\eqref{eq:jmath-basic-emb} are given by the unit vectors, and
so they are indeed coherent. (Recall that the vectors in the frames are
indexed by the elements of $I$.) Since the normal framings for
$\iota_I$ are defined using those unit vectors via
Equation~\eqref{eq:XI-framing}, the commutativity of
Diagram~\eqref{eq:XI-coherent-embed} implies that the framings are
coherent.

\begin{example}
  Let us fix the diffeomorphism $f\from \RR\to\mathring{\RR}_+$ to be
  the exponential map. For $n=1$, the map $\iota_I$ embeds the point
  $X_I$ as $e^{p_1}\in\mathring{\RR}_+$, with normal frame
  $v_{p_1}=e^{p_1}\ve^+_1$, where $\ve^+_1$ is the unit vector
  in $\mathring{\RR}_+$ (reusing the notation from
  Section~\ref{sec:base}).

  For $n=2$, let $(\alpha_1,\alpha_2,\beta_1)$ be the coordinates on
  $\RR^2\times\Pi_2$ (with $\beta_2=3-\beta_1$), let $(a_1,a_2,b_2)$
  be the coordinates on $\RR^3$, and let $(s_1,s_2,t)$ be the
  coordinates on $\RR^2\times\RR_+$, as in the proof of
  Proposition~\ref{prop:homeo}. The map $\jmath_I$ embeds the interval
  $X_I$ in $\RR^2\times\Pi_2$ by the map
  \[
    x\mapsto (\alpha_1=p_1,\alpha_2=p_2,\beta_1=1+x),\qquad x\in[0,1],
  \]
  with normal frame given by $(e_{p_1}=\tfrac{\del}{\del\alpha_1},e_{p_2}=\tfrac{\del}{\del\alpha_2})$.
  The map $\psi_2\circ\pi\from \RR^2\times\Pi_2\to \RR^3$ is given by
  \[
    (\alpha_1,\alpha_2,\beta_1)\mapsto
    (a_1=\alpha_1+\alpha_2,a_2=\alpha_1\alpha_2-\beta_1(3-\beta_1),b_2=\alpha_1(3-\beta_1)+\alpha_2\beta_1),
  \]
  and the map $\Psi_2\from \RR^2\times\RR_+\to\RR^3$ from Equation~\eqref{eq:Psi2-explicit-map} is given by
  \[
    (s_1,s_2,t)\mapsto(a_1=s_1+s_2,a_2=s_1s_2-2-t,b_2=2s_1+s_2).
  \]
  Therefore, the composition $\wh\iota=F\circ\Psi_2^{-1}\circ\psi_2\circ\pi\from\RR^2\times\Pi_2\to\mathring{\RR}_+^2\times\RR_+$ is given by
  \[
    (\alpha_1,\alpha_2,\beta_1)\mapsto \big(e^{(2-\beta_1)\alpha_1+(\beta_1-1)\alpha_2}, e^{(\beta_1-1)\alpha_1+(2-\beta_1)\alpha_2},(\beta_1-1)(2-\beta_1)((\alpha_1-\alpha_2)^2+1)\big),
  \]
  and hence the embedding $\iota_I=\wh\iota\circ\jmath_I$ of the
  interval $X_I\cong[0,1]$ in $\mathring{\RR}_+^2\times\RR_+$ is given
  by
  \[
    x\mapsto \big(e^{(1-x)p_1+xp_2},e^{xp_1+(1-x)p_2},x(1-x)((p_1-p_2)^2+1)\big),\qquad x\in[0,1].
  \]
  Let $(\ve^+_1,\ve^+_3,\ve^+_2)$ denote the standard frame in
  $\mathring{\RR}_+^2\times\RR_+$ (as in
  Section~\ref{sec:base}). Then the normal framing to $\iota_I$ is given by
  \begin{align*}
    v_{p_1}(x)&=(d\wh\iota)_x(\tfrac{\del}{\del\alpha_1})=(1-x)e^{(1-x)p_1+xp_2}\ve^+_1+xe^{xp_1+(1-x)p_2}\ve^+_3+2x(1-x)(p_1-p_2)\ve^+_2,\\
    v_{p_2}(x)&=(d\wh\iota)_x(\tfrac{\del}{\del\alpha_2})=xe^{(1-x)p_1+xp_2}\ve^+_1+(1-x)e^{xp_1+(1-x)p_2}\ve^+_3+2x(1-x)(p_2-p_1)\ve^+_2.
  \end{align*}
  Note that the endpoints $x=0$ and $x=1$ produce the product embeddings
  (and product framings) of $X_{\{p_1\}}\times X_{\{p_2\}}$ and
  $X_{\{p_2\}}\times X_{\{p_1\}}$, respectively, in
  $\mathring{\RR}^2_+\cong\mathring{\RR}^2_+\times\{0\}\subset\mathring{\RR}^2_+\times\RR_+$.
\end{example}

\section{The Cohen-Jones-Segal construction} 
\label{sec:cjs}

We have now constructed all moduli spaces $\bModuli([D])$, as
$l$-dimensional $\langle l\rangle$-manifolds, along with neat
embeddings in $\E^d_l$, as well as their their (external) framings. It
remains to put them together into a framed flow category, and then run
the Cohen-Jones-Segal construction to obtain the knot Floer stable
homotopy types, following the set-up in \cite{LipshitzSarkar}.

\subsection{Framed flow categories} We review here some definitions
from \cite[Section 3.2]{LipshitzSarkar}. One slight difference is that
\cite{LipshitzSarkar} uses the composition order, while we are using
the concatenation order, and therefore, our factors will be ordered and indexed
differently. The other main difference is
that we allow our categories to have infinitely many objects.

\begin{definition}
  A {\em flow category} $\Cat$ is a category with objects $\Cat=\Ob(\Cat)$ 
  equipped with a function $\gr\from \Cat \to \Z$ (called the grading),
  such that:
  \begin{itemize}
  \item $\Hom(x, x) = \{\Id\}$ for all $x \in \Cat$;
  \item For all distinct $x, y \in \Cat$ with $\gr(x)-\gr(y)=k$, the morphism space
    $\Hom(x, y)$ is a compact $(k-1)$-dimensional
    $\langle k-1 \rangle$-manifold; in particular, it is
    empty for $\gr(x) \leq \gr(y)$;
  \item For distinct $x, y, z \in \Cat$ with $\gr(x)-\gr(y)=m$, the composition
    \[
      \circ\from \Hom(x, y) \times \Hom (y, z) \to \Hom(x, z)
    \]
    is an embedding into $\del_m \Hom(x, z)$. Moreover,
    \[
      \circ^{-1}(\del_i \Hom(x, z)) = \begin{cases}
        \del_i\Hom(x, y) \times \Hom(y, z) &\text{for} \ i <m \\
        \Hom(x, y) \times  \del_{i-m}\Hom(y,z) &\text{for} \ i >m; 
      \end{cases}
    \]
    (Since we are using the concatenation order, for
    $p\in\Hom(x,y),q\in\Hom(y,z)$, we will denote their composition to
    be $p*q$, in a similar vein to
    Item~(\ref{item:domain-concatenation-grid}) from
    Section~\ref{sec:grid-background}.)
  \item For distinct $x, z \in \Cat$, the composition $\circ$ induces a diffeomorphism
    \[
      \del_m \Hom(x, z) \cong \coprod_{\{y \mid \gr(y) =\gr(x)-i\}} \Hom(x, y) \times \Hom(y,z).
    \]
  \end{itemize} 
\end{definition}

Given a flow category $\Cat$ and $x, y \in \Cat$, we define the {\em
  compactified moduli space from $x$ to $y$} as
\[
  \bM(x, y) = \begin{cases}
    \emptyset & \text{if } x=y,\\
    \Hom(x, y)& \text{if } x\neq y.
  \end{cases}
\]
Furthermore, for $i \in \Z$, we let
$\Cat_i=\{ x \in \Cat \mid \gr(x) = i\}$, topologized as a discrete
space. Then, for $i, j \in \Z$, we define
\[
  \bM(i, j) = \coprod_{x \in \Cat_i, y \in \Cat_j}  \bM(x, y).
\]
We call a flow category \emph{compact} if $\bM(i,j)$ is compact for all $i,j$.

Next we will define neat embeddings of flow categories. For this,
recall the space
$\E^d_l=\RR^d\times\RR_+\times\dots\times\RR_+\times\RR^d\cong
\RR_+^l\times\RR^{d(l+1)}$ from Section~\ref{sec:neat}. It is convenient to set
\[
  \E^d_{i,j}=\E^d_{i-j-1}
\]
for any pair of integers $i>j$.
\begin{definition}
  \label{def:necat}
  A {\em neat embedding} $\iota$ of a flow category $\Cat$ relative $d \in \NN$ is a collection of neat embeddings 
  \[
    \iota_{x, y}\from\bM(x, y) \hookrightarrow \E^d_{\gr(x),\gr(y)},
  \]
  defined for every $x, y \in \Cat$, such that
  \begin{itemize}
  \item For all $i, j \in \Z$, the union of all $\iota_{x, y}$ for $x \in \Cat_i, y \in \Cat_j$ induces a neat embedding of $\bM(i, j)$;
  \item For all $x, y, z \in \Cat$ and for all $(p, q)\in \bM(x, y) \times \bM(y,z)$, we have
    \[
      \iota_{x, z} (p*q) = (\iota_{x, y}(p),0,\iota_{y,z} (q))\in \E^d_{\gr(x),\gr(y)}\times \R_+ \times \E^d_{\gr(y),\gr(z)} = \E^d_{\gr(x),\gr(z)}.
    \]
  \end{itemize}
\end{definition}

Given a neat embedding $\iota$ of $\Cat$, and objects $x, y \in \Cat$,
we let $\nu_{x, y}$ denote the normal bundle to $\bM(x, y)$ under the
embedding $\iota_{x, y}$.
\begin{definition}
  \label{def:ffc}
  A {\em framed flow category} is a flow category $\Cat$
  together with a neat embedding $\iota$ (relative some $d$), and also
  equipped with framings for the normal bundles $\nu_{x, y}$ for all
  $x, y \in \Cat$, such that the product framing of
  $\nu_{x,y}\times \nu_{y,z}$ equals the pullback framing of
  $\circ^* \nu_{x, z}$ for all $x, y, z.$
\end{definition}

We introduce the following notion of a framed flow subcategory, which is analogous to a  subcomplex of a chain complex.

\begin{definition}\label{def:flowcat-sub-quotient}
  Let $\Cat$ be a framed flow category. Suppose we have a full subcategory $\Dat \subseteq \Cat$ such that
  $\Hom(x, y) =\emptyset$ for all $x \in \Dat$ and
  $y \in \Cat \setminus \Dat$. We let the embedding and framings of the morphism sets in $\Dat$ be 
  obtained by restriction from $\Cat$. We then call $\Dat$ a {\em framed
    flow subcategory} of $\Cat$.
\end{definition}

We can similarly define a {\em framed flow quotient category} of
$\Cat$, starting from some $\Dat \subseteq \Cat$ such that
$\Hom(x, y) =\emptyset$ for all $x \in \Cat \setminus \Dat$ and
$y \in \Dat$.

\begin{remark}
  In \cite[Definition 3.29]{LipshitzSarkar}, framed flow subcategories
  and quotient categories are called downward closed and upward closed
  subcategories, respectively.
\end{remark}

\subsection{Spectra}\label{sec:spectra-background}
In this section, we review notions related to classical (sequential)
spectra. For details, see \cite{Adams,BousfieldFriedlander,AbramKriz}.

A {\em spectrum} $X$ is a collection of based spaces $\{X_n\}_{n \geq 0}$ along with
structure maps $\sigma^X_n=\sigma_n\from X_n\smas S^1\to X_{n+1}$. The
map $\sigma_n$ induces maps $\wt{H}_k(X_n)\to\wt{H}_{k+1}(X_{n+1})$
and $\pi_k(X_n)\to\pi_{k+1}(X_{n+1})$; the homology and homotopy
groups of a spectrum $X$ are defined to be
\[
  H_k(X)=\colim_n \wt{H}_{k+n}(X_n)\qquad \pi_k(X)=\colim_n\pi_{k+n}(X_n).
\]
There are shift functors on spectra, given by $\sh(X)_n=X_{n+1}$ and
$\sh^{-1}(X)_n=X_{n-1}$ if $n>0$, and the basepoint if $n=0$.  

Given a spectrum $X$ and a based space $S$, the spectrum $S\smas X$ is
defined as
\[
  (S\smas X)_n\defeq S\smas X_n,\qquad \sigma^{S\smas X}_n\defeq \text{id}_S\smas \sigma^X_n\from S\smas X_n\smas S^1\to S\smas X_{n+1};
\]
and the spectrum $\Hom(S,X)$ is defined as
\begin{gather*}
  \Hom(S,X)_n\defeq \Hom(S,X_n),\text{ the based mapping space, and}\\
  \sigma^{\Hom(S,X)}_n\from \Hom(S,X_n)\smas S^1\to\Hom(S,X_{n+1}),\text{ given by }f\smas p\mapsto \big(s\mapsto \sigma^X_n(f(s)\smas p)\big).
\end{gather*}
If $S=S^k$, the functor $S\smas \cdot$ is called the $k\th$ suspension
functor $\Sigma^k$, and the the functor $\Hom(S,\cdot)$ is called the
$k\th$ loop functor $\Omega^k$.  The simplest example of a spectrum is
the sphere spectrum $\SphereS$ with
\[
  \SphereS_n\defeq S^n=\underbrace{S^1\smas\dots\smas S^1}_n,\qquad \sigma_n\defeq\Id\from S^n\smas S^1\to S^{n+1}.
\]
More generally, the {\em suspension spectrum} of a based space $X$ is the
spectrum $X\smas \SphereS$; its homology and homotopy groups are simply the reduced
homology and stable homotopy groups of the original space $X$:
\[
  H_*(X\smas\SphereS)\cong \wt{H}_*(X)\qquad \pi_*(X\smas\SphereS)\cong\pi_*^s(X).
\]
For any $k\in\ZZ$, define the spectrum $\SphereS^k$ as
$\Sigma^k\SphereS$ if $k\geq 0$, and $\Omega^{-k}\SphereS$ if
$k\leq 0$.

A {\em map of spectra}, $f\from X\to Y$, is a collection of maps
$f_n\from X_n\to Y_n$ such that
$f_{n+1}\circ\sigma^X_n=\sigma^Y_{n+1}\circ(f_n\smas S^1)\from
X_n\smas S^1\to Y_{n+1}$. The map is called a {\em weak
equivalence} if it induces
an isomorphism $\pi_*(X)\to\pi_*(Y)$; we will usually suppress the
adjective `weak'. If both $X$ and $Y$ are $n$-connected (that is, have
$\pi_k=0$ for all $k\leq n$) for some $n\in\ZZ$, then this is
equivalent to inducing isomorphism on homology $H_*(X)\to H_*(Y)$. The
stable homotopy category of spectra is obtained by inverting all the
equivalences. The functors $H_*$ and $\pi_*$ factor through the stable
homotopy category. Spectra also have a notion of homotopy colimits,
defined levelwise; and homotopy colimits also preserve
equivalences. In the stable homotopy category, there are isomorphisms
\[
  \Sigma^kX\simeq \sh^k(X)\qquad \Omega^kX \simeq \sh^{-k}(X).
\]
(Note that these isomorphisms are not induced by the natural-looking
levelwise maps $(S^k\smas X)_n=S^k\smas X_n\to (\sh^k(X))_n=X_{n+k}$
and $(\sh^{-k}(X))_n=X_{n-k}\to(\Omega^kX)_n=\Omega^kX_n$, since they
do not induce maps on spectra.) In particular, the $k\th$ suspension
functor $\Sigma^k$ and loop functor $\Omega^k$ are inverses to one
another: $X\cong \Omega^k\Sigma^kX$.  This isomorphism in the stable
homotopy category is realized by the equivalence
$X\to\Omega^k\Sigma^k X$, given levelwise
$X_n\to \Hom(S^k,S^k\smas X_n)$ by $x\mapsto (p\mapsto p\smas x)$;
therefore, we will be justified in referring to $\Omega^k$ as the
$k\th$ desuspension functor, and sometimes abusing notation to write
$\Sigma^{-k}$ instead of $\Omega^k$.

\subsection{From framed flow categories to spectra} 
\label{sec:tospectra}
We now review how to build a CW complex, and then a spectrum, from a
framed flow category (under some mild restrictions). We follow
\cite[Section 3.3]{LipshitzSarkar} (with the ordering of factors
reversed, and some additional minor modifications), which is in turn
inspired from \cite{CJS}.

Let $(\Cat, \iota, \phi)$ be a compact framed flow category, with
$\iota$ denoting the neat embedding, and $\phi$ the normal
framings. For now, we assume that the grading function
$\gr\from \Cat \to \Z$ is bounded, with image in some interval
$[B+1 , A]$ with $A, B \in \Z$. Let us further impose $B<0$. Let
\begin{equation}
  \label{eq:Cdba}
  C_d(A,B) \defeq (A-B)d-B-1.
\end{equation}

We construct a CW complex $|\Cat|_{\iota, \phi, A,B}$. We start with a single $0$-cell---the
basepoint---and then for each $x \in \Cat$, we attach a cell
$\Cell(x)$, inductively on the grading $\gr(x)=m$. The cell $\Cell(x)$
will have dimension $C_d(A,B) +m$.

Let us choose $\epsilon > 0$ sufficiently small so that for all $i$
and $j$, the embedding $\iota_{i, j}$ of $\bM_{i, j}$ into
$\E^d_{i,j}$ extends to an embedding of
$\bM_{i, j}\times [-\epsilon, \epsilon]^{(i-j)d}$ using the normal
framings. Choose $R$ sufficiently large so that for all $i$ and $j$,
the image
$\iota_{i,j}(\bM(i, j) \times [-\epsilon, \epsilon]^{(i-j)d})$ lies in
\[
  [-R, R]^d \times [0, R] \times \dots \times [0, R] \times [-R, R]^d \subset \RR^d\times\RR_+\times\dots\times\RR_+\times\RR^d=\E^d_{i,j}.
\]

Let us suppose we attached all the lower dimensional cells and we want to attach $\Cell(x)$, where $\gr(x)=m$. Define
\[
  \Cell(x)=\underbrace{[-\epsilon,\epsilon]^d\times
    \{0\}\times\dots\times\{0\}\times[-\epsilon,\epsilon]^d}_{\subset\E^d_{A,m}}\times
  \{0\}\times \underbrace{[-R, R]^d \times [0, R] \times \dots \times
    [0, R] \times [-R, R]^d}_{\subset\E^d_{m,B}} \subset \E^d_{A,B}.
\]
(There are $A-m$ instances of $\{0\}$ and $[-\epsilon,\epsilon]^d$,
$m-B$ instances of $[-R,R]^d$, and $m-B-1$ instances of $[0,R]$.)  For
any $1\leq i<m-B$, define $\bdy_i \Cell(x)\subset \Cell(x)$ to be the
subset where the coordinate in the $i\th$ $[0,R]$ factor is $0$. (This
is a slight abuse of notation since we are not defining $\Cell(x)$ as
a $\langle m-B-1\rangle$-manifold; however, had we replaced all
instances of $[-R,R]$, $[0,R]$, and $[-\epsilon,\epsilon]$ with $\RR$,
$\RR_+$, and $(-\epsilon,\epsilon)$, respectively, then $\Cell(x)$
would have had a natural $\langle m-B-1\rangle$-manifold structure,
and $\bdy_i\Cell(x)$ would have been the correct definition.)

To see how to attach $\Cell(x)$ to a lower cell $\Cell(y)$ where
$\gr(y)=l$, consider the neat embedding $\iota_{x, y}$ (extended using
the framing $\phi$):
\[
  \bM(x, y) \times [-\epsilon,\epsilon]^d \times\{0\}\times\dots\times
\{0\} \times [-\epsilon,\epsilon]^d\into[-R, R]^d \times [0, R] \times \cdots \times [0, R] \times [-R,
R]^d\subset\E^d_{m,l}.
\]
Extending by $\Id$ on
$[-\epsilon,\epsilon]^d\times\dots\times\{0\}\times[-\epsilon,\epsilon]^d\subset\E^d_{A,m}$
and $[-R,R]^d\times\dots\times[0,R]\times[-R,R]^d\subset\E^d_{l,B}$,
we get an embedding of 
\begin{gather*}
  \bM(x,y)\times\Cell(y)=\bM(x,y)\times \underbrace{[-\epsilon,\epsilon]^d\times\dots\times\{0\}\times[-\epsilon,\epsilon]^d}_{\subset\E^d_{A,m}}\times
  \{0\}\times\underbrace{[-\epsilon,\epsilon]^d\times\dots\times\{0\}\times[-\epsilon,\epsilon]^d}_{\subset\E^d_{m,l}}\\
  \times\{0\}\times \underbrace{[-R, R]^d \times \dots \times
    [0, R] \times [-R, R]^d}_{\subset\E^d_{l,B}}
\end{gather*}
into 
\begin{gather*}
  \bdy_{m-l}\Cell(y)=\underbrace{[-\epsilon,\epsilon]^d\times\dots\times\{0\}\times[-\epsilon,\epsilon]^d}_{\subset\E^d_{A,m}}\times
  \{0\}\times\underbrace{[-R,R]^d\times\dots\times[0,R]\times[-R,R]^d}_{\subset\E^d_{m,l}}\\
  \times\{0\}\times \underbrace{[-R, R]^d \times \dots \times
    [0, R] \times [-R, R]^d}_{\subset\E^d_{l,B}}.
\end{gather*}
Let $\Cell_y(x)\subset\bdy_{m-l}\Cell(x)$ be the image of the
embedding. We then define the attaching map from $\del \Cell(x)$ to
the lower skeleton to be the projection to $\Cell(y)$ on each
$\Cell_y(x) \cong \bM(x, y)\times \Cell(y)$, and the basepoint map on
$\bdy\Cell(x)\setminus\bigcup_y\Cell_y(x)$.

After attaching all the cells $\Cell(x)$, we obtain the desired CW
complex $|\Cat|_{\iota, \phi, A,B}$. Its dependence on $A$ and $B$ is
explained in \cite[Lemma 3.26]{LipshitzSarkar}. It is proved there
that, if we have $B' \leq B$ and $A ' \geq A$, then there is a
homotopy equivalence
\begin{equation}
  \label{eq:Sp}
  \Sigma^{C_d(A', B')-C_d(A,B)} |\Cat|_{\iota, \phi, A,B} \xrightarrow{\sim} |\Cat|_{\iota, \phi, A',B'}.
\end{equation}
In more detail, consider
\[
  \E^d_{A,B}=\RR^d\times\RR_+\times\RR^d\times\dots\times\RR_+\times\RR^d,
\]
and let $T_d(A,B)$ be the set of the $d(A-B)$ $\RR$-factors appearing
in the above formula, union the set of the $(-B-1)$ $\RR_+$-factors
appearing in position below $0$, that is, in positions
$-1,-2,\dots,B+1$. Then $T_d(A,B)$ is a set of cardinality $C_d(A,B)$,
and let
\begin{equation}\label{eq:SdAB-sphere}
  S_d(A,B)=\bigwedge_{T_d(A,B)} S^1
\end{equation}
be the corresponding $C_d(A,B)$-dimensional sphere. The homotopy equivalence from Equation~\eqref{eq:Sp} is actually
\begin{equation}
  \label{eq:Sp-again}
  \big(\bigwedge_{\substack{T_d(A',B')\\\setminus T_d(A,B)}}S^1\big) \smas |\Cat|_{\iota, \phi, A,B} \xrightarrow{\sim} |\Cat|_{\iota, \phi, A',B'}.
\end{equation}

This observation allows us to
define the following spectrum canonically:
\[
  \Omega^{C_d(A,B)}(|\Cat|_{\iota, \phi, A,B}\smas\SphereS)=\Hom(S_d(A,B),|\Cat|_{\iota, \phi, A,B}\smas\SphereS).
\]
Given $B'\leq B$ and $A'\geq A$, we get an equivalence
\begin{equation}\label{eq:Cat-Catprime-spectral-relation}
\begin{split}
  \Omega^{C_d(A,B)}|\Cat|_{\iota, \phi, A,B}\smas\SphereS&\xrightarrow{\sim} \Omega^{C_d(A,B)}\Omega^{C_d(A',B')-C_d(A,B)}\Sigma^{C_d(A',B')-C_d(A,B)}|\Cat|_{\iota, \phi, A,B}\smas\SphereS\\
  &\xrightarrow{\sim} \Omega^{C_d(A',B')}|\Cat|_{\iota,\phi,A',B'}\smas\SphereS.
\end{split}
\end{equation}
Here, the first map is induced from the map $X\to \Omega^k\Sigma^k X$ from Section~\ref{sec:spectra-background}, with $k=C_d(A',B')-C_d(A,B)$, and $\Omega^k,\Sigma^k$
being the functors $\Hom(S,\cdot)$, and $S\smas\cdot$, with
$S=\bigwedge_{T_d(A',B')\setminus T_d(A,B)}S^1$, and the second map is
induced by Equation~\eqref{eq:Sp-again}. Now define the spectrum
\begin{equation}
  \label{eq:SDef}
  \Sp(\Cat, \iota, \phi) \defeq \hocolim_{A,B} \Omega^{C_d(A,B)}|\Cat|_{\iota, \phi, A,B}\smas\SphereS,
\end{equation}
where the homotopy colimit is taken using the maps from
Equation~\eqref{eq:Cat-Catprime-spectral-relation}, as $A \to \infty$
and $B \to -\infty$.

So far we have worked under the assumption that the grading function
$\gr$ is bounded. Let us relax this assumption by requiring only that
$\gr$ is bounded below, by some constant $B+1$. For every $K\in \Z$,
there is a compact framed flow subcategory $\Cat_{\leq K}$ of $\Cat$,
whose objects are those $x \in \Cat$ with $\gr(x) \leq K$. For any
$K\leq A$, we get spaces $|\Cat_{\leq K}|_{\iota,\phi,A,B}$, and for
any $K \leq K' \leq A$, we have inclusions
\[
  |\Cat_{\leq K}|_{\iota, \phi, A,B} \hookrightarrow  |\Cat_{\leq K'}|_{\iota, \phi, A,B}.
\]
Taking the homotopy colimit as $A,K\to \infty$ and $B\to-\infty$ subject to the restriction $K\leq A$, we define
\begin{equation}\label{eq:triple-hocolim}
  \Sp(\Cat, \iota, \phi) \defeq \hocolim_{A,B,K}  \Omega^{C_d(A,B)}|\Cat_{\leq K}|_{\iota, \phi, A,B}\smas\SphereS.
\end{equation}
Thus, we have spectra associated to compact framed flow categories
even when $\gr$ is only bounded below. 

Given a compact framed flow category $(\Cat,\iota,\phi)$, there is an \emph{associated
chain complex $C_*(\Cat,\iota,\phi)$} whose chain group is freely generated by the
objects with homological grading given by the $\gr$ function, and
whose differential is given by
\begin{equation}\label{eq:flowcat-to-chaincx}
  \bdy x=\sum_{\substack{y\in\Cat\\\gr(y)=\gr(x)-1}}\#\bM(x,y) y.
\end{equation}
(Here $\#\bM(x,y)$ is the signed count of points in the compact framed
$0$-manifold $\bM(x,y)$.) When $\gr$ is bounded below, it follows
from~\cite[Lemma 3.24]{LipshitzSarkar} that this chain complex is isomorphic to the shifted reduced cellular chain complex of $|\Cat|_{\iota,\phi,A,B}$,
\begin{equation}\label{eq:flowcat-to-chaincx}
  C_*(\Cat,\iota,\phi)\cong\Sigma^{-C_d(A,B)}\wt{C}^{\mathrm{CW}}_*(|\Cat|_{\iota,\phi,A,B}),
\end{equation}
and therefore, its homology is isomorphic to the
homology of the spectrum $\Sp(\Cat,\iota,\phi)$:
\begin{equation}\label{eq:flowcat-to-homology}
  H_*(\Cat,\iota,\phi)\cong H_*(\Sp(\Cat,\iota,\phi)).
\end{equation}

\begin{remark}
  If $\gr$ is not bounded below, we have compact framed flow categories
  $\Cat_{\geq L}$ with objects those $x \in \Cat$ with
  $\gr(x) \geq L$. There are associated spectra
  $\Sp(\Cat_{\geq L}, \iota, \phi)$. Projections between CW complexes
  (collapsing the lower dimensional cells up to some degree) induce
  maps
  \[
    \Sp(\Cat_{\geq L'}, \iota, \phi) \leftarrow \Sp(\Cat_{\geq L},
    \iota, \phi)
  \]
  for all $L' \leq L$. This gives an inverse system
  of spectra, i.e., a pro-spectrum as in \cite{CJS}. In this case, we
  may define $\Sp(\Cat, \iota, \phi)$ to be this pro-spectrum. This
  would be useful if one were interested in a minus version of the
  knot Floer spectrum (which will not be discussed in the current
  paper).
\end{remark}

\subsection{Filtrations}
\label{sec:filteredsp}
We make the following definition, which is somewhat non-standard, but sufficient for our purposes.

\begin{definition}
  Let $H$ be a partially ordered set, and $X$ a spectrum. Let $H^+$ be
  $H$ with an additional terminal element $\infty$. An {\em
    $H$-filtration on $X$} is a functor from $H^+$ to the category of
  spectra, such that $\infty$ is mapped to $X$.

  More concretely, an $H$-filtration is a collection of spectra $X(h)$
  for $h \in H$, together with maps
  \[
    \psi_{h, h'} \from X(h) \to X(h') \ \text{for all } h \leq h'
  \]
  and
  \[
    \psi_h\from X(h) \to X \ \text{for all } h,
  \]
  satisfying: $\psi_{h, h} = \text{id}$,
  $\psi_{h', h''}\circ \psi_{h, h'}=\psi_{h, h''}$ and
  $\psi_h = \psi_{h'} \circ \psi_{h, h'}$, whenever
  $h \leq h' \leq h''$.
\end{definition}

\begin{definition}
  Let $X$ and $Y$ be $H$-filtered spectra with maps
  $\psi_h,\psi_{h, h'}$ and $\phi_h,\phi_{h, h'}$, respectively.
  \begin{enumerate}[label=(\alph*),leftmargin=*]
  \item An {\em $H$-filtered map} from $X$ to $Y$ consists of maps of
    spectra $f\from X \to Y$ and $f_h\from X(h) \to Y(h)$ for all
    $h \in H$, such that
    $\phi_{h, h'} \circ f_h = f_{h'} \circ \psi_{h, h'}$ and
    $\phi_h \circ f_h = f \circ \psi_h$ for all $h \leq h'$. (By a
    slight abuse of notation, we use $f$ to also denote to the whole
    data of a $H$-filtered map.)
  \item An $H$-filtered map $f\from X \to Y$ is called an {\em
      $H$-filtered equivalence} if $f$ and $f_h$ (for all
    $h \in H$) are equivalences of spectra.
  \end{enumerate}
\end{definition}

Filtrations on spectra arise in our setting as follows. 
\begin{definition}\label{def:H-filtered-flow-cat}
  Let $H$ be a partially ordered set. An {\em $H$-filtered framed flow
    category} consists of a framed flow category $(\Cat, \iota, \phi)$
  and a collection of framed flow subcategories
  $(\Cat(h), \iota, \phi)$ for $h \in H$, with
  $\Cat(h) \subseteq \Cat(h')$ whenever $h \leq h'$. (By a slight
  abuse of notation, we still use $\iota$ and $\phi$ to denote their
  restrictions.)
\end{definition}

Suppose we have a compact $H$-filtered framed flow category, with $\gr$
bounded below. Then, we obtain inclusions of the form
\[
  |\Cat(h)_{\leq K}|_{\iota, \phi, A,B} \hookrightarrow |\Cat_{\leq
    K}|_{\iota, \phi, A,B}.
\]
Taking homotopy colimits, we arrive at an $H$-filtration on
$\Sp(\Cat, \iota, \phi)$ made of the spectra
$\Sp(\Cat(h), \iota, \phi)$ for $h \in H$.

Furthermore, given an $H$-filtered framed flow category and $h \in H$,
we can define the associated graded framed flow category $g\Cat(h)$
using objects that are in $\Cat(h)$ but not in any $\Cat(h')$ for
$h' < h$. If compact with grading bounded below, the corresponding
spectrum is denoted $g\Sp(\Cat(h), \iota, \phi)$. The associated
graded spectrum for the whole filtration is set to be
\[
  g\Sp(\Cat, \iota, \phi) = \bigvee_{h \in H} g\Sp(\Cat(h), \iota, \phi).
\]

\subsection{The knot Floer spectrum}
\label{sec:kfs}
We now define the multi-filtered spectrum $\GS(\Grid)$ from a grid
diagram $\Grid$, as advertised in the introduction.

In Section~\ref{sec:gluing} we glued together the moduli spaces
$\bM_0(D)$ for $D$ in the same equivalence class, with the result
being smooth $\langle k \rangle$-manifolds $\bModuli([D])$. Since
$\vN =\vec{0}$, the thick dimension of these moduli spaces equals
their actual dimension $k$, so the internal framings are empty. We
have neat embeddings of $\bM([D])$ into $\E^d_l$, as well as normal
(external) framings for these embeddings.

Recall from Section~\ref{sec:grid-complex-background} that the grid
complex $\Cp(\Grid)$ has generators
\[
  [x, j_2, \dots, j_n] = U_2^{-j_2} \dots U_n^{-j_n} x
\]
for $x \in \S$ and $j_2, \dots, j_n \in \NN$. The generators also have
Alexander gradings $(A_1, \dots, A_{\ell}) \in (\hZ)^\ell$, one for
each component of the link $L$ represented by $\Grid$.

We define a framed flow category $\Cat(\Grid)$ as follows. The objects are the generators of $\Cp$, and we let $\gr$ be the Maslov grading. Given two objects $[x, j_2, \dots, j_n]$ and $[y, i_2, \dots, i_n]$, consider the equivalence class $[D]$  made of domains $D \in \pdomains(x, y)$ with $\Os(D) = (j_2 - i_2, \dots, j_n - i_n)$. We let 
\begin{equation}\label{eq:gridmodulispaces-final}
  \bM([x, j_2, \dots, j_n], [y, i_2, \dots, i_n])=\bM([D]).
\end{equation}
The
enumeration of strata in Section~\ref{sec:enum} ensures that the
conditions in the definition of a flow category are
satisfied. Furthermore, the neat embeddings and the external framings
turn $\Cat(\Grid)$ into a framed flow category. The compatibility
conditions in Definitions~\ref{def:necat} and \ref{def:ffc} are
satisfied because the moduli spaces (along with their embeddings and
framings) in Section~\ref{sec:construction} were constructed
coherently with respect to the lower dimensional strata.

There are finitely many generators in each Maslov grading, so
$\Cat(\Grid)$ is compact, and the generators are bounded below in
Maslov grading; therefore, we obtain a {\em knot Floer spectrum}
\[
  \GS(\Grid)\defeq \Sp(\Cat(\Grid)).
\]

For each $h=(h_1, \dots, h_\ell) \in (\hZ)^\ell$, we define a framed flow subcategory $\Cat(\Grid, h) \subset \Cat(\Grid)$, using only the generators of $\Cp$ with Alexander multi-grading $\leq h$. From here we obtain spectra $\GS(\Grid, h)$, producing a $(\hZ)^\ell$-filtration on $\GS(\Grid).$

There is an additional structure given by the $U_i$ maps.  Consider the framed flow quotient category obtained from  $\Cat(\Grid)$ by removing the objects $[x, j_2, \dots, j_n]$ where $j_i =0$. By mapping 
\[ [x, j_2, \dots, j_i, \dots, j_n] \mapsto [x, j_2, \dots, j_i -1,
  \dots, j_n] \] we get an isomorphism between this quotient category
and $\Sigma^2\Cat(\Grid)$, a category which is the same as
$\Cat(\Grid)$ except the Maslov grading is shifted by $2$. At the
level of the associated CW complexes and then spectra, we obtain a map
\[ U_i \from \GS(\Grid) \to \Sigma^2\GS(\Grid)\]
given by collapsing the cells corresponding to generators $[x, j_1, \dots, j_n]$ where $j_i =0$. 

The $U_i$ maps interact with the multi-filtration as follows. Suppose the marking $U_i$ lies on the $k^{\text{\tiny th}}$ component of the link. Then, we get commutative diagrams
\[ 
  \xymatrix{
    \GS(\Grid, h) \ar[r]^-{U_i} \ar[d] & \Sigma^2\GS(\Grid, h - \ve_k) \ar[d]\\
    \GS(\Grid) \ar[r]^-{U_i} &  \Sigma^2\GS(\Grid),
  }
\]
where $\ve_k$ is the unit vector in the $k\th$ coordinate.

\subsection{Other versions}
Two other variants of grid complexes, $\Chat$ and $\Ct$, were
mentioned at the end of Section~\ref{sec:grid-complex-background}. We
construct knot Floer spectra $\widehat{\GS}(\Grid)$ and
$\widetilde{\GS}(\Grid)$ in the same way as we did for $\Cp$, but
using fewer generators to define the framed Floer categories. In the
case of $\Chat$, we only use those $[x, j_1, \dots, j_n]$ where
$j_i=0$ for one index $i$ chosen from the $O$-markings on each link
component. In the case of $\Ct$, we only use the generators where
$j_i=0$ for all $i$. We then take the framed flow subcategories of
$\Cat(\Grid)$ with those generators as objects.

We also have associated graded framed flow categories, with the
resulting spectra denoted $\gGS(\Grid)$, $\ghGS(\Grid)$ and
$\gtGS(\Grid).$ In particular, we have a decomposition
\[
  \ghGS(\Grid) = \bigvee_{h \in (\hZ)^{\ell}} \ghGS(\Grid, h).
\]
By construction, the (reduced) homology of $\ghGS(\Grid, h)$ is the link Floer homology in Alexander multi-grading $h$:
\[
  H_i(\ghGS(\Grid, h); \Z) = \widehat{\HFL}_{i}(L, h).
\]

\begin{remark}
  \label{rem:limitedU}
  If $O_i$ and $O_j$ are two markings on the same link component, one
  can show that the $U_i$ and $U_j$ actions on a grid complex are
  filtered chain homotopic. (See Lemma 2.11 in
  \cite{MOSzT-hf-combinatorial} for a proof for the minus version.)
  Although we will not pursue it here, one should be able to adapt the
  proof to the Floer spectra. This would imply that it suffices to
  consider one $U_i$ map on $\GS(\Grid)$ for each link component not
  containing $O_1$. For the link component containing $O_1$, the
  corresponding map is trivial. In particular, in the case of knots,
  all the $U_i$ maps should be trivial. Another consequence is that
  the tilde version $\tGS(\Grid)$ should be equivalent to the wedge
  sum of several (shifted) copies of the hat version $\hGS(\Grid)$.
\end{remark}

\section{Invariance}
\label{sec:invariance}

In this section we prove Theorem~\ref{thm:main}. That is, we fix the
grid $\Grid$ and the marking $O_1$, and prove that the filtered
stable homotopy type of $\GS(\Grid)$ is independent of the further
auxiliary choices. Indeed, we will prove that framed flow categories
associated to two sets of choices are related by a map, which induces
an equivalence between the filtered spectra. Towards this end, we will
discuss the notion of maps of framed flow categories in
Section~\ref{sec:invariance-bimodules}, mention the relevant
modifications in the definition for grid diagrams when we study more
general stratified spaces in the presence of bubbles in
Section~\ref{sec:invariance-bimodules-bubbles}, and finally list all
the auxiliary choices and prove invariance in
Section~\ref{sec:invariance-choices}.

\subsection{Bimodules}\label{sec:invariance-bimodules}
Morphisms of framed flow categories are called bimodules, and they
induce maps of spectra associated to the the two framed flow
categories. Different versions of this notion have appeared in the
literature; see \cite[Def.~3.6]{Large},
\cite[Section~4]{ABFoundations}, \cite[Def.~3.1]{CoteKartal} and
\cite[Def.~6.1]{Bonciocat}. (It also appears implicitly in
\cite[Lemma~3.32]{LipshitzSarkar}.) Many of these definitions work
with flow categories where the moduli spaces are stably tangentially
framed; this is different from the notion that we have used in this
paper---our moduli spaces were embedded in specific Euclidean spaces
$\E^d_{i,j}$ and normally framed. Therefore, we will give a different
definition of bimodules of framed flow categories, modeled on normally
framed moduli spaces; this definition will also be more amenable to
the modifications in presence of bubbles (in
Section~\ref{sec:invariance-bimodules-bubbles}).

Towards this end, we need a mild generalization of our notion of
$\langle n\rangle$-manifolds from Section~\ref{sec:nmflds}. For any
poset $P$, a $\langle P\rangle$-manifold is a multifaceted
manifold $X$, together with multifacets $\bdy_i X$ for all $i\in P$
such that
\begin{itemize}
\item $\bigcup_{i\in P}\bdy_i X=\bdy X$;
\item $\bdy_i X\cap \bdy_j X$ is a multifacet of both $\bdy_i X$ and $\bdy_j X$ for all $i\neq j$;
\item $\bdy_iX\cap\bdy_j X=\varnothing$ if $i,j$ are not comparable in the poset $P$.
\end{itemize}
Reusing notation from Equation~\eqref{eq:deliX}, we will let
$\bdy_I X\defeq \bigcap_{i\in I}\bdy_i X$ for any subset $I\subset P$;
unless $I$ is a chain (i.e., a totally ordered subset) of $P$, the multifacet 
$\bdy_I X$ is empty.

The special case when
$P=n\defeq\{1<2<\dots<n\}$ produces $\langle n\rangle$-manifolds; we
will be interested in the case when $P$ is the poset $2\ltimes n$
whose elements are
$\{0\}\times\{1,2,\dots,n\}\cup\{1\}\times\{0,1,\dots,n-1\}$, and the
partial order is given by $ (i_0,i_1)\leq (j_0,j_1)$ if and only if
$i_0\leq j_0,i_1\leq j_1$.

We are now ready to present the definition of flow bimodules
$\Bat\from\Cat\to\Cat'$, which is essentially
\cite[Def.~3.1]{CoteKartal}. Let $\bM(x, y)$ and $\bM'(x, y)$ denote the compactified moduli spaces in the flow categories $\Cat$ and $\Cat'$, respectively.

The central idea is that to each
$x\in\Cat,y\in\Cat'$, a flow bimodule will associate a multifaceted
manifold $\bN(x,y)$, as well as maps
$\bM(x,y)\times\bN(y,z)\to\bdy\bN(x,z)$ (for any
$x,y\in\Cat,z\in\Cat'$) and $\bN(x,y)\times\bM'(y,z)\to\bdy\bN(x,z)$
(for any $x\in\Cat,y,z\in\Cat'$). If $\gr(x)-\gr(z)=k$, then
$\bN(x,z)$ will be $k$-dimensional and its boundary will consist of
$2k$ multifacets, which are naturally indexed by the poset
$2\ltimes k$: for $m=1,\dots,k$, the multifacet $\bdy_{0,m}\bN(x,z)$
is the image of the maps $\bM(x,y)\times\bN(y,z)$ with
$\gr(y)=\gr(x)-m$, and for $m=0,\dots,k-1$, the multifacet
$\bdy_{1,m}\bN(x,z)$ is the image of the maps
$\bN(x,y)\times\bM'(y,z)$ with $\gr(y)=\gr(x)-m$. See the left half of Figure~\ref{fig:bimodules-explanatory}.

\begin{definition}\label{def:bimodule-basic}
  Let $\Cat$ and $\Cat'$ be two flow categories. A {\em flow
    bimodule} $\Bat\from\Cat\to\Cat'$ consists of the following data:
  \begin{itemize}
  \item For all $x\in\Cat,y\in\Cat'$ with $\gr(x)-\gr(y)=k$, a compact
    $k$-dimensional
    $\langle 2\ltimes k\rangle$-manifold $\bN(x,y)$;
  \item For all $x,y\in\Cat,z\in\Cat'$ with $\gr(x)-\gr(y)=m$, a map
    \[
      \circ \from \bM(x,y)\times\bN(y,z)\to \bN(x,z),
    \]
    which is an embedding into $\bdy_{0,m}\bN(x,z)$, satisfying
    \[
      \circ^{-1}(\bdy_{i_0,i_1}\bN(x,z))=\begin{cases}
        \bdy_{i_1}\bM(x,y)\times\bN(y,z)&\text{for $i_0=0,i_1<m$}\\
        \bM(x,y)\times\bdy_{0,i_1-m}\bN(y,z)&\text{for $i_0=0,i_1>m$}\\
        \bM(x,y)\times\bdy_{1,i_1-m}\bN(y,z)&\text{for $i_0=1,i_1\geq m$};
      \end{cases}
    \]
    moreover, these maps induce a diffeomorphism
    \[
      \bdy_{0,m}\bN(x,z)\cong\coprod_{\{y\mid\gr(y)=\gr(x)-m\}}\bM(x,y)\times\bN(y,z);
    \]
  \item For all $x\in\Cat,y,z\in\Cat'$ with $\gr(x)-\gr(y)=m$, a map
    \[
      \circ \from \bN(x,y)\times\bM'(y,z)\to \bN(x,z),
    \]
    which is an embedding into $\bdy_{1,m}\bN(x,z)$, satisfying
    \[
      \circ^{-1}(\bdy_{i_0,i_1}\bN(x,z))=\begin{cases}
        \bdy_{0,i_1}\bN(x,y)\times\bM'(y,z)&\text{for $i_0=0,i_1\leq m$}\\
        \bdy_{1,i_1}\bN(x,y)\times\bM'(y,z)&\text{for $i_0=1,i_1<m$}\\
        \bN(x,y)\times\bdy_{i_1-m}\bM'(y,z)&\text{for $i_0=1,i_1> m$};
      \end{cases}
    \]
    moreover, these maps induce a diffeomorphism
    \[
      \bdy_{1,m}\bN(x,z)\cong\coprod_{\{y\mid\gr(y)=\gr(x)-m\}}\bN(x,y)\times\bM'(y,z).
    \]
  \end{itemize}
\end{definition}

Note that we may view the $\langle 2\ltimes k\rangle$-manifold
$\bN(x,y)$ as a $\langle k\rangle$-manifold by declaring
$\bdy_i\bN(x,y)=\bdy_{0,i}\bN(x,y)\amalg\bdy_{1,i-1}\bN(x,y)$. This
forgets a little bit of its structure, but after that, a flow bimodule
is simply a flow category $\Bat$ with $\Cat'$ as a subcategory and
$\Sigma\Cat$ as the corresponding quotient category (where
$\Sigma\Cat$ is the flow category obtained from $\Cat$ by increasing
the grading of each object by $1$).

Recall from Definitions~\ref{def:necat} and~\ref{def:ffc} that we defined
framed flow categories by coherently embedding and framing the moduli
spaces $\bM(i,j)$ in the spaces $\E^d_{i,j}$. These spaces
$\E^d_{i,j}$ had the salient feautures that they were also
$\langle i-j-1\rangle$-manifolds, and had maps---indeed
diffeomorphisms---of the form 
\[\
  \E^d_{i,j}\times\E^d_{j,k}\to\bdy_{i-j}\E^d_{i,k}\subset
  \E^d_{i,k}.
\]
While not necessary, it was also useful that the spaces
$\E^d_{i,j}$ were polyhedra, that is, intersections of finitely many
half-spaces in some $\RR^n$.

In order to define framed flow
bimodules connecting framed flow categories, we need to embed and
frame these spaces $\bN(x,y)$ in certain new spaces, which we now
describe. As before, let
\[
  \bN(i,j)=\coprod_{x\in\Cat_i,y\in\Cat'_j}\bN(x,y),
\]
and call a flow bimodule \emph{compact} if $\bN(i,j)$ is compact for
all $i,j$.  Fix integers $d,e$, which we may assume are even in order
to make sign calculations easier. We will now construct polyhedra
$\FF^{d,e}_{i,j}$ (with similar properties as the spaces $\E^d_{i,j}$)
where the spaces $\bN(i,j)$ will be embedded. Indeed, we will define
spaces $\FF^{d,e}_n$, and set $\FF^{d,e}_{i,j}=\FF^{d,e}_{i-j}$.

Recall the $n$-dimensional permutohedron
$\Pi_{n+1}\subset\Affine^{n}\subset\RR^{n+1}$ from
Section~\ref{sec:permutohedron-basic}. Consider $(n+1)$ of its vertices
\begin{align*}
  w_0&=v_{[(n+1)12\cdots n]}=(2,3,\dots,n+1,1),\\
  w_1&=v_{[1(n+1)2\cdots n]}=(1,3,\dots,n+1,2),\\
  &\qquad\qquad\cdots\\
  w_{n}&=v_{[12\cdots n(n+1)]}=(1,2,\dots,n,n+1).
\end{align*}
(The left half of Figure~\ref{fig:bimodules-explanatory} sheds light
on our motive for considering these vertices.)  Consider the facets
$F_S$ that contain at least one of these vertices. From
Item~\ref{item:permuto-facet-vertices}, there are exactly $2n$ such
facets $F_{\{1\}}$, $F_{\{n+1\}}$, $F_{\{1,2\}}$, $F_{\{1,n+1\}}$,
$F_{\{1,2,3\}}$, $F_{\{1,2,n+1\}},\ldots$; that is, $S$ is either of the form $\{1,2,\dots, k\}$ or $\{1, 2, \dots, k-1, n+1\}$. Consider the corresponding
half-spaces $\HalfSpace_S\subset\Affine^n$, and define $\wt{\Delta}^n$
to be the intersection of those half-spaces. The polyhedron
$\wt{\Delta}^n$ has vertices $w_0,\dots,w_{n}$, and has $2n$ facets,
which we denote $G_S=\wt{\Delta}^n\cap\bdy\HalfSpace_S$. See
Figures~\ref{fig:tilde-delta-first-few}
and~\ref{fig:tilde-delta-explicit}.

\begin{figure}
  \centering
  \begin{tikzpicture}
    \foreach\n/\xsh in {1/0,2/0.8,3/2}{
      \begin{scope}[xshift=5.5*\xsh cm]

        \ifnum\n=0
        \node at (0,0) {$\bullet$};
        \node[anchor=south] at (0,0) {\small $w_0$};
        \fi

        \ifnum \n=1
        \draw[thick] (0,0)--(0,1) node[pos=0] {$\bullet$} node[pos=0,anchor=north,align=center] {\small $w_0=G_{\{2\}}$\\\small$=\bdy_{1,0}$} node[pos=1] {$\bullet$} node[pos=1,anchor=south,align=center] {\small $w_1=G_{\{1\}}$\\\small $=\bdy_{0,1}$};
        \fi

        \ifnum \n=2
        \coordinate (w1) at (-1,1);
        \coordinate (w2) at (0,0);
        \coordinate (w3) at (1,1);
        \coordinate (d) at (0,2);

        \fill[black!30] (w1)--(w2)--(w3);
        \draw[thick] (w1)--(w2)--(w3);
        \draw[thick,dashed] (w1)--(w3);
        \draw[thick,-latex] (w1)--($(w1)+(d)$);
        \draw[thick,-latex](w3)--($(w3)+(d)$) node[pos=1,anchor=south,align=center] {\small $e_1-e_2$\\\small$=\vv_1$};
        \foreach \i in {1,2,3}{\node at (w\i) {$\bullet$};}

        \node[anchor=east] at (w1) {\small $w_0$};
        \node[anchor=north] at (w2) {\small $w_1$};
        \node[anchor=west] at (w3) {\small $w_2$};

        \node[anchor=east,align=right] at ($(w1)+0.5*(d)$) {\small $G_{\{3\}}$\\\small $=\bdy_{1,0}$};
        \node[anchor=north east,align=right] at ($(w1)!0.5!(w2)$) {\small $G_{\{1,3\}}$\\\small$=\bdy_{1,1}$};
        \node[anchor=north west,align=left] at ($(w2)!0.5!(w3)$) {\small $G_{\{1\}}$\\\small$=\bdy_{0,1}$};
        \node[anchor=west,align=left] at ($(w3)+0.5*(d)$) {\small $G_{\{1,2\}}$\\\small $=\bdy_{0,2}$};
        
        \fi

        \ifnum \n=3

        \begin{scope}[scale=3,x={(-1cm,0.2cm)},y={(0.8cm,0.4cm)},z={(0,1cm)}]
          \coordinate (w1) at (0,0,0);
          \coordinate (w2) at (0,-0.3,0.3333);
          \coordinate (w3) at (-0.3,0,0.6667);
          \coordinate (w4) at (0,0,1);

          \coordinate (e1) at (0.7,0,0);
          \coordinate (e2) at (0,0.7,0);

          \fill[black!30] (w1)--(w2)--(w4)--(w3);
          \draw[thick] (w1)--(w2)--(w3)--(w4);
          \draw[thick,dashed] (w3)--(w1)--(w4)--(w2);

          \foreach\i in {1,2}{\draw[thick,-latex] (w\i)--($(w\i)+(e1)$);}
          \foreach\i in {1,3}{\draw[thick,-latex] (w\i)--($(w\i)+(e2)$);}

          \draw[thick,-latex] (w4)--($(w4)+(e1)$) node[inner sep=0,outer sep=0,pos=1,anchor=south east,align=right] {\small $e_2-e_3$\\\small$=\vv_2$};
          \draw[thick,-latex] (w4)--($(w4)+(e2)$) node[inner sep=0,outer sep=0,pos=1,anchor=south west,align=left] {\small $e_1-e_2$\\\small$=\vv_1$};

          \foreach\i in {1,2,3,4}{\node at (w\i) {$\bullet$};}

        \node[anchor=north] at (w1) {\small $w_0$};
        \node[anchor=south east] at (w2) {\small $w_1$};
        \node[anchor=north west] at (w3) {\small $w_2$};
        \node[anchor=south] at (w4) {\small $w_3$};

        \node[anchor=east,align=left] at ($(w1)!0.5!(w2)+(e1)$) {\small $G_{\{1,4\}}$\\\small $=\bdy_{1,1}$};
        \node[align=right] at ($(w1)!0.5!(w3)+(e2)$) {\small $G_{\{1,2,4\}}$\\\small $=\bdy_{1,2}$};
        \node[align=left] at ($(w2)!0.5!(w4)+(e1)$) {\small $G_{\{1\}}$\\\small$=\bdy_{0,1}$};
        \node[align=right] at ($(w3)!0.5!(w4)+(e2)$) {\small $G_{\{1,2\}}$\\\small $=\bdy_{0,2}$};
        \node at ($(w4)+0.8*(e1)+0.8*(e2)$) {\small $G_{\{1,2,3\}}=\bdy_{0,3}$};
          
        \end{scope}
        \fi

        \node[anchor=north] at (0,-1) {$n=\n$};
        
      \end{scope}
    }
  \end{tikzpicture}
  \caption{The polyhedra $\wt{\Delta}^n$ for small values of $n$. (We
    have skipped the $n=0$ case of a point.) We have also shown the
    vertices $w_i$ and the facets $G_S=\bdy_{i_0,i_1}$ (except the
    facet $G_{\{4\}}=\bdy_{1,0}$ for $n=3$, which is at the bottom and
    hence not labeled). The vectors $\vv_1,\dots,\vv_{n-1}$ are
    labeled and the simplex $\Delta^n$ is shaded and its edges are
    dashed (the dashed lines are not edges of
    $\wt{\Delta}^n$).}\label{fig:tilde-delta-first-few}
\end{figure}

\begin{figure}
  \centering
  \begin{tikzpicture}[scale=0.8,%
    x={(-0.4cm,-0.6cm)},
    y={(1cm,0)},
    z={(0cm,1cm)}]

    \draw[thick,gray,->] (0,0,0)--(4,0,0) node[pos=1,anchor=north east] {\small $x_1$};
    \draw[thick,gray,->] (0,0,0)--(0,4,0) node[pos=1,anchor=west] {\small $x_2$};
    \draw[thick,gray,->] (0,0,0)--(0,0,4) node[pos=1,anchor=south] {\small $x_3$};

    \foreach \sigma/\x/\y/\z/\pos [count=\c from 0] in {
      312/2/3/1/north west,
      132/1/3/2/west,
      123/1/2/3/south}{

      \coordinate (v\sigma) at (\x,\y,\z);
      \node at (v\sigma) {$\bullet$};
      \node[anchor=\pos] at (v\sigma) {\small $v_{[\sigma]}=w_\c=(\x,\y,\z)$};
    }

    \fill[black!30] (v123)--(v132)--(v312);
    
    \draw[thick] (v312)--(v132)--(v123);
    \draw[thick,dashed] (v312)--(v123);

    \draw[thick,-latex] (v123)--++(1,-1,0); 
    \draw[thick,-latex] (v312)--++(1,-1,0) node[inner sep=0,outer sep=0,pos=1,anchor=east,align=right] {\small $e_1-e_2$\\\small$=\vv_1$};
    
  \end{tikzpicture}
  \caption{Continuing from Figure~\ref{fig:permuto-basics}, we have
    shown the (non-compact) polyhedron $\wt{\Delta}^2$ sitting in the
    ambient $\RR^3$, and have labeled the vertices. The other
    conventions are same as in
    Figure~\ref{fig:tilde-delta-first-few}.}\label{fig:tilde-delta-explicit}
\end{figure}

It is also useful to note that this polyhedron $\wt{\Delta}^n$ is the
Minkowski sum of $\Delta^n$ and $\RR_+^{n-1}$. Specifically, the
$n$-simplex $\Delta^n\subset\Affine^n$ is the convex hull $[w_0,\dots,w_n]$ of the
vertices $w_0,\dots,w_{n}$, and $\RR_+^{n-1}$ is the positive cone $\RR_+\langle \vv_1,\dots,\vv_{n-1}\rangle$ of 
the vectors $\vv_1=e_1-e_2,\vv_2=e_2-e_3,\dots,\vv_{n-1}=e_{n-1}-e_{n}$ (where $e_i$ are the
standard unit vectors of the ambient $\RR^{n+1}$). Then $\wt{\Delta}^n$
is the Minkowski sum
\begin{equation}\label{eq:minkowski-sum}
\wt{\Delta}^n=\Delta^n+\RR_+^{n-1}=\{x+y\mid x\in [w_0,\dots,w_n],y\in\RR_+\langle \vv_1,\dots,\vv_{n-1}\rangle\}.
\end{equation}
More specifically, consider the $2^{n-1}$ subsets
$I\subset\{1,2,\dots,n-1\}$, and let
$\Delta_I^n=[(w_i)_{i\in I\cup\{0,n\}}]$ be the face of
$\Delta^n$ which is the convex hull of vertices $w_i$ for
$i\in I\cup\{0,n\}$. (These are precisely the $2^{n-1}$ faces of
$\Delta^n$ that contain $w_0,w_n$.) Also let
$\bdy_I\RR_+^{n-1}=\RR_+\langle (\vv_i)_{i\notin I}\rangle$ be the positive cone
of the vectors $\vv_i$ for $i\notin I$.
Then we have a decomposition into Minkowski sums
\[
  \wt{\Delta}^n=\bigcup_{I\subset \{1,\dots,n-1\}}(\Delta_I^n+\bdy_I\RR_+^{n-1}).
\]
Moreover, $\Delta_I^n$ is perpendicular to $\bdy_I\RR_+^{n-1}$, so the Minkowski sum $\Delta_I^n+\bdy_I\RR_+^{n-1}$ is diffeomorphic to the product $\Delta_I^n\times\bdy_I\RR_+^{n-1}$, and we have a diffeomorphism
\begin{equation}\label{eq:minkow-is-product}
  \wt{\Delta}^n\cong\bigcup_{I\subset \{1,\dots,n-1\}}(\Delta_I^n\times\bdy_I\RR_+^{n-1}).
\end{equation}
See again Figures~\ref{fig:tilde-delta-first-few}
and~\ref{fig:tilde-delta-explicit}. Figure~\ref{fig:tilde-delta-decompose}
shows the decomposition of $\wt{\Delta}^3$ into these four pieces.

\begin{figure}
  \centering
  \begin{tikzpicture}
    
    \foreach \xsh [count=\case from 0] in {0,1,-1,0}{
      \begin{scope}[xshift=3.8*\case cm]

        \ifnum\case=0
        \node[anchor=north,align=center] at (0,-0.8) {\small $\Delta_{\{1,2\}}^3\times\bdy_{\{1,2\}}\RR_+^2$\\\small $[w_0,w_1,w_2,w_3]\times\RR_+\langle\rangle$};
        \fi
        \ifnum\case=1
        \node[anchor=north,align=center] at (0,-0.8) {\small $\Delta_{\{1\}}^3\times\bdy_{\{1\}}\RR_+^2$\\\small $[w_0,w_1,w_3]\times\RR_+\langle\vv_2\rangle$};
        \fi
        \ifnum\case=2
        \node[anchor=north,align=center] at (0,-0.8) {\small $\Delta_{\{2\}}^3\times\bdy_{\{2\}}\RR_+^2$\\\small $[w_0,w_2,w_3]\times\RR_+\langle\vv_1\rangle$};
        \fi
        \ifnum\case=3
        \node[anchor=north,align=center] at (0,-0.8) {\small $\Delta_{\varnothing}^3\times\bdy_{\varnothing}\RR_+^2$\\\small $[w_0,w_3]\times\RR_+\langle\vv_1,\vv_2\rangle$};
        \fi
        
        \begin{scope}[scale=3,xshift=0.3*\xsh cm,x={(-1cm,0.2cm)},y={(0.8cm,0.4cm)},z={(0,1cm)}]
          \coordinate (w0) at (0,0,0);
          \coordinate (w1) at (0,-0.3,0.3333);
          \coordinate (w2) at (-0.3,0,0.6667);
          \coordinate (w3) at (0,0,1);

          \coordinate (e1) at (0.7,0,0);
          \coordinate (e2) at (0,0.7,0);

          \ifnum\case=0
          
          \draw[thick] (w2)--(w0)--(w1)--(w2)--(w3)--(w1);
          \draw[thick,dashed] (w0)--(w3);

          \foreach\i in {0,1,2,3}{\node at (w\i) {$\bullet$};}
          \node[anchor=north] at (w0) {\small $w_0$};
          \node[anchor=south east] at (w1) {\small $w_1$};
          \node[anchor=north west] at (w2) {\small $w_2$};
          \node[anchor=south] at (w3) {\small $w_3$};

          \fi

          \ifnum\case=1
          
          \draw[thick] (w0)--(w1)--(w3)--(w0);

          \foreach\i in {0,1}{\draw[thick,-latex] (w\i)--($(w\i)+(e1)$);}
          \draw[thick,-latex] (w3)--($(w3)+(e1)$) node[inner sep=0,outer sep=0,pos=1,anchor=south east] {\small$\vv_2$};

          \foreach\i in {0,1,3}{\node at (w\i) {$\bullet$};}
          \node[anchor=north] at (w0) {\small $w_0$};
          \node[anchor=south east] at (w1) {\small $w_1$};
          \node[anchor=south] at (w3) {\small $w_3$};

          \fi

          \ifnum\case=2
          
          \draw[thick] (w0)--(w2)--(w3)--(w0);

          \foreach\i in {0,2}{\draw[thick,-latex] (w\i)--($(w\i)+(e2)$);}
          \draw[thick,-latex] (w3)--($(w3)+(e2)$) node[inner sep=0,outer sep=0,pos=1,anchor=south west] {\small$\vv_1$};

          \foreach\i in {0,2,3}{\node at (w\i) {$\bullet$};}
          \node[anchor=north] at (w0) {\small $w_0$};
          \node[anchor=north west] at (w2) {\small $w_2$};
          \node[anchor=south] at (w3) {\small $w_3$};

          \fi

          \ifnum\case=3
          
          \draw[thick] (w0)--(w3);

          \foreach\i in {0}{\draw[thick,-latex] (w\i)--($(w\i)+(e1)$);}
          \foreach\i in {0}{\draw[thick,-latex] (w\i)--($(w\i)+(e2)$);}

          \draw[thick,-latex] (w3)--($(w3)+(e1)$) node[inner sep=0,outer sep=0,pos=1,anchor=south east] {\small$\vv_2$};
          \draw[thick,-latex] (w3)--($(w3)+(e2)$) node[inner sep=0,outer sep=0,pos=1,anchor=south west] {\small$\vv_1$};

          \foreach\i in {0,3}{\node at (w\i) {$\bullet$};}
          \node[anchor=north] at (w0) {\small $w_0$};
          \node[anchor=south] at (w3) {\small $w_3$};

          \fi

        \end{scope}

      \end{scope}}
  \end{tikzpicture}
  \caption{Decomposition of $\wt{\Delta}^3$ into four pieces.}\label{fig:tilde-delta-decompose}
\end{figure}

These $\wt{\Delta}^n$ are simple polyhedra, that is, each vertex is
contained in the minimal number of $n$ facets. Therefore, they are
multifaceted manifolds. They can be endowed with the structure of a
$\langle 2\ltimes n\rangle$-manifold by declaring
\begin{equation}
  \bdy_{i_0,i_1}\wt{\Delta}^n=\begin{cases}
    G_{\{1,2,\dots,i_1\}}&\text{if $i_0=0$,}\\
    G_{\{1,2,\dots,i_1-1,n+1\}}&\text{if $i_0=1$.}
    \end{cases}
\end{equation}
Define
\begin{equation}\label{eq:F-d-e-n}
  \FF^{d,e}_n=\wt{\Delta}^n\times\RR^e\times\RR^{dn}
\end{equation}
with the induced $\langle 2\ltimes n\rangle$-manifold structure, and set
\begin{equation}\label{eq:F-d-e-i-j}
\FF^{d,e}_{i,j}=\FF^{d,e}_{i-j}
\end{equation}
for all $i\geq j$.

Consider the facet $\bdy_{0,k}\wt{\Delta}^n=G_{\{1,\dots,k\}}$. The map
\[
  \begin{tikzpicture}[scale=1]
    \node (a) at (0,0) {$\RR^k\times\RR^{n+1-k}$};
    \node (b) at (6,0) {$\RR^{n+1}$};
    \draw[->] (a) -- (b) node[midway,anchor=south] {\footnotesize $+(1,2,\dots,k,k,k,\dots,k)$};
  \end{tikzpicture}
\]
identifies the product
$\RR_+\langle \vv_1,\dots,\vv_{k-1}\rangle \times \wt{\Delta}^{n-k}$
with the facet $G_{\{1,\dots,k\}}$. (This map is similar to the one
from Item~\ref{item:permuto-facet-identify}, except we add
$(1,2,\dots,k)$ instead of $(0,0,\dots,0)$ to the first $\RR^k$ factor
in order to translate the cone
$\RR_+\langle \vv_1,\dots,\vv_{k-1}\rangle$ and make it start from
the point $(1,2,\dots,k)\in\RR^k$.)  Similarly, for the facet
$\bdy_{1,k-1}\wt{\Delta}^n=G_{\{1,\dots,k-1,n+1\}}$, we use the map
\[
  \begin{tikzpicture}[scale=1]
    \node (a) at (0,0) {$\RR^k\times\RR^{n+1-k}$};
    \node (b) at (6,0) {$\RR^{n+1}$};
    \node (c) at (13,0) {$\RR^{n+1}$};
    \draw[->] (a) -- (b) node[midway,anchor=south] {\footnotesize $+(0,\dots,0,k+1,k+2,\dots,n+1)$};
    \draw[->] (b) -- (c) node[midway,anchor=south] {\footnotesize $(x_1,\dots,x_n)\mapsto (x_1,\dots,x_{k-1},x_{k+1},\dots,x_n,x_k)$};
  \end{tikzpicture}
\]
where we have postcomposed with a shuffle that sends the $k\th$
$\RR$-factor to the end. (This is again similar to what we did in
Item~\ref{item:permuto-facet-identify}.)  This map identifies
$\wt{\Delta}^{k-1}\times\RR_+\langle \vv_{k+1},\dots,\vv_{n}\rangle$
with the facet $G_{\{1,\dots,k-1,n+1\}}$. That is, we get
identifications
\begin{equation}
  \bdy_{0,k}\wt{\Delta}^n\cong\RR_+^{k-1}\times\wt{\Delta}^{n-k},\qquad\qquad\bdy_{1,k-1}\wt{\Delta}^n\cong\wt{\Delta}^{k-1}\times\RR_+^{n-k}.
\end{equation}
Then for any $i\geq j\geq k$, we get induced identifications
\begin{equation}\label{eq:Fij-product-structure}
  \bdy_{0,i-j}\FF^{d,e}_{i,k}\cong\E^d_{i,j}\times\FF^{d,e}_{j,k},\qquad\qquad\bdy_{1,i-j}\FF^{d,e}_{i,k}\cong\FF^{d,e}_{i,j}\times\E^d_{j,k}.
\end{equation}
by just shuffling the factors.

\begin{figure}
  \centering
  \begin{tikzpicture}

    \begin{scope}[xscale=4]
      \node[anchor=east,draw] at (0,2) {$x$};
      \node[anchor=west,draw] at (1,0) {$y$};

      \node[anchor=east] at (0,1) {$\bdy_{0,1}$};
      \node[anchor=east] at (0,0) {$\bdy_{0,2}$};
      \node[anchor=west] at (1,2) {$\bdy_{1,0}$};
      \node[anchor=west] at (1,1) {$\bdy_{1,1}$};

      \foreach \i/\j in {0/1,0/0,1/2,1/1}{\node at (\i,\j) {$\bullet$};}
      
      \node[anchor=north east] (c) at (0,-0.2) {$\Cat$};
      \node[anchor=north west] (cp) at (1,-0.2) {$\Cat'$};
      
      \node[anchor=east] at (-0.1,1.5) {\textbf 1};
      \node[anchor=east] at (-0.1,0.5) {\textbf 2};
      \node at ($(c)!0.5!(cp)$) {\textbf 3};

      \draw[thick] (0,2) to[out=0,in=180] node[inner sep=0,outer sep=0,midway,anchor=south] {\small $w_0=v_{[312]}$} (1,2) to[out=-105,in=105] (1,1) to[out=-105,in=105] (1,0);
      \draw[thick,dashed] (0,2) to[out=-60,in=60] (0,1) to[out=0,in=180] node[inner sep=0,outer sep=0,midway,anchor=south] {\small $w_1=v_{[132]}$} (1,1) to[out=-120,in=120] (1,0);
      \draw[thick,densely dotted] (0,2) to[out=-75,in=75] (0,1) to[out=-75,in=75] (0,0) to[out=0,in=180] node[inner sep=0,outer sep=0,midway,anchor=south] {\small $w_2=v_{[123]}$} (1,0);
    \end{scope}

    \begin{scope}[xshift=7cm,xscale=4]
      \node[anchor=east,draw] at (0,2) {$x$};
      \node[anchor=west,draw] at (1,0) {$y$};

      \node[anchor=west] at (1.1,1) {$\RR_+$};
      \node[anchor=west] at (1.1,1.5) {$\RR^d$};
      \node[anchor=west] at (1.1,0.5) {$\RR^d$};

      \foreach \i/\j in {0/1,0/0,1/2,1/1}{\node at (\i,\j) {$\bullet$};}
      
      \node[anchor=north east] (c) at (0,-0.2) {$\Cat$};
      \node[anchor=north west] (cp) at (1,-0.2) {$\Cat'$};
      \node at ($(c)!0.5!(cp)$) {$\RR^e$};

      \draw[thick,dashed] (0,2)--
      node[pos=0.8,anchor=south west, inner sep=0,outer sep=0] {\begin{tikzpicture}[scale=0.2]
          \coordinate (w1) at (-1,1);
          \coordinate (w2) at (0,0);
          \coordinate (w3) at (1,1);
          \coordinate (d) at (0,2);
          \fill[black!20] (w1)++(d)--(w1)--(w2)--(w3)--++(d);
          \draw[thick,solid] (w1)--(w2)--(w3);
          \draw[thick,solid,-latex] (w1)--++(d);
          \draw[thick,solid,-latex] (w3)--++(d);
          \end{tikzpicture}}
      (1,0);

      \draw[thick,densely dotted] (0,2)--
      node[pos=0.3,anchor=south west, inner sep=1pt,outer sep=0] {$\wt{\Delta}^1=\begin{tikzpicture}[scale=0.2]
          \coordinate (w1) at (-1,1);
          \coordinate (w2) at (0,0);
          \coordinate (w3) at (1,1);
          \coordinate (d) at (0,2);
          \draw[thick,solid] (w1)--(w2);
          \end{tikzpicture}$}
        (1,1) (0,1)--
      node[pos=0.5,anchor=north east, inner sep=-2pt,outer sep=0] {$\wt{\Delta}^1=\begin{tikzpicture}[scale=0.2]
          \coordinate (w1) at (-1,1);
          \coordinate (w2) at (0,0);
          \coordinate (w3) at (1,1);
          \coordinate (d) at (0,2);
          \draw[thick,solid] (w2)--(w3);
          \end{tikzpicture}$}        
        (1,0);

    \end{scope}

  \end{tikzpicture}
  \caption{Left: We consider $x\in\Cat$ (left column) and $y\in\Cat'$
    (right column) with $\gr(x)-\gr(y)=2$. The boundary of $\bN(x,y)$
    corresponds to broken trajectories, so there are four types
    $\bdy_{0,1},\bdy_{0,2},\bdy_{1,0},\bdy_{1,1}$, indexed by
    $2\ltimes 2$, depending on where they break (shown by black
    dots). The completely broken trajectories are $\bdy_I$ where $I$
    is a maximal chain of $2\ltimes 2$, so there are exactly three:
    $\bdy_{1,0}\cap\bdy_{1,1}$ (solid lines),
    $\bdy_{0,1}\cap\bdy_{1,1}$ (dashed lines), and
    $\bdy_{0,1}\cap\bdy_{0,2}$ (dotted lines), and they correspond to
    the vertices $w_0,w_1,w_2$ of $\wt{\Delta}^2$, respectively. (If
    we number the `gaps' between gradings as \textbf{1}, \textbf{2},
    and the `gap' between $\Cat$ and $\Cat'$ as \textbf{3}, then the
    permutation $\sigma$ in $w_i=v_{[\sigma]}$ can easily be
    determined from the broken trajectory.) Right: The
    $\langle 2\ltimes 2\rangle$-manifold $\bN(x,y)$ is embedded in
    $\FF^{d,e}_2=\wt{\Delta}^2\times\RR^e\times\RR^{2d}$, and its
    multifacets $\bdy_{0,1},\bdy_{0,2},\bdy_{1,0},\bdy_{1,1}$ are
    embedded in $\RR^d\times\wt{\Delta}^1\times\RR^e\times\RR^d$,
    $\RR^d\times\RR_+\times\RR^d\times\RR^e$,
    $\RR^e\times\RR^d\times\RR_+\times\RR^d$, and
    $\wt{\Delta}^1\times\RR^e\times\RR^d$, respectively; here the two
    copies of $\RR^d$ are associated to the two `gaps' between
    gradings, $\RR^e$ is associated to the `gap' between $\Cat$ and
    $\Cat'$, the half-line $\RR_+$ is associated to the grading between $\gr(x)$
    and $\gr(y)$, the two copies of the interval $\wt{\Delta}^1$ are
    associated to the two dotted lines, and $\wt{\Delta}^2$ is
    associated to the dashed line.}\label{fig:bimodules-explanatory}
\end{figure}

We will require $\bN(i,j)$ to be neatly embedded in $\FF^{d,e}_{i,j}$.
The notion of neat embedding is similar to
Definition~\ref{def:neatangle}. Being a polyhedron, $\FF^{d,e}_n$ is
Thom-Mather stratified (as in Definition~\ref{def:stratified}) with
$Y_I\defeq\mathring{\bdy}_I\FF^{d,e}_n$ as the open strata for
$I\subset 2\ltimes n$; we may further assume that for any open stratum $Y_I$, the projection map from its tubular neighborhood,
\[
  \pi_{Y_I}\from T_{Y_I}\to Y_I,
\]
is the orthogonal projection. A smooth embedding of a 
$\langle 2\ltimes n\rangle$-manifold $X$ into $\FF^{d,e}_n$ is called
{\em neat} if
\begin{itemize}
\item It respects the strata, that is, for every $i\in 2\ltimes n$, $\bdy_iX=X\cap\bdy_i\FF^{d,e}_n$.
\item For every $I \subset 2\ltimes n$,
  \[
    T_{Y_I} \cap X = \pi_{Y_I}^{-1}(\mathring{\del}_IX).
  \]
\end{itemize}
Now we are all set to define framed flow bimodules; compare Definitions~\ref{def:necat}
and~\ref{def:ffc}.

\begin{definition}\label{def:bimodule}
  Let $\Cat=(\Cat, \iota, \phi)$ and $\Cat'=(\Cat', \iota', \phi')$ be
  two framed flow categories, both relative some $d\in\NN$. A {\em
    framed flow bimodule} $\Bat=(\Bat,\jmath,\psi)$ from $\Cat$ to
  $\Cat'$ consists of the following data:
  \begin{itemize}
  \item A flow bimodule $\Bat\from\Cat\to\Cat'$;
  \item Neat embeddings $\jmath_{i,j}\from\bN(i,j)\to\FF^{d,e}_{i,j}$ (for some fixed integer $e$) for all $i\geq j$;
  \item Framings $\psi_{i,j}$ of their normal bundles $\mu_{i,j}$;
  \end{itemize}
  satisfying, for all $i\geq j\geq k$,
  \begin{itemize}
  \item The following diagrams commute:
    \[
    \begin{tikzpicture}
      \foreach\i in {0,1}{
        \begin{scope}[xshift=8*\i cm,xscale=2.5,yscale=1.2]
          \ifnum\i=0
          \node (a) at (0,0) {$\bM(i,j)\times\bN(j,k)$};
          \node (d) at (0,-1) {$\E^d_{i,j}\times\FF^{d,e}_{j,k}$};
          \else
          \node (a) at (0,0) {$\bN(i,j)\times\bM'(j,k)$};
          \node (d) at (0,-1) {$\FF^{d,e}_{i,j}\times\E^{d}_{j,k}$};
          \fi
          \node (b) at (1.2,0) {$\bdy_{\i,i-j}\bN(i,k)$};
          \node (e) at (1.2,-1) {$\bdy_{\i,i-j}\FF^{d,e}_{i,k}$};
          \node (c) at (2,0) {$\bN(i,k)$};
          \node (f) at (2,-1) {$\FF^{d,e}_{i,k}$};

          \draw[->] (d)--(e) node[midway,anchor=south] {\small $\cong$};
          
          \draw[right hook->] (b)--(c);
          \draw[right hook->] (e)--(f);
          \draw[->] (a)--(b);

          \ifnum\i=0
          \draw[right hook->] (a)--(d) node[midway,anchor=east] {\small $\iota_{i,j}\times\jmath_{j,k}$};
          \else
          \draw[right hook->] (a)--(d) node[midway,anchor=east] {\small $\jmath_{i,j}\times\iota'_{j,k}$};
          \fi
          \draw[right hook->] (b)--(e);
          \draw[right hook->] (c)--(f) node[midway,anchor=west] {\small $\jmath_{i,k}$};
          
      \end{scope}}
  \end{tikzpicture}
  \]
\item In the above two cases, the product framing
  $\phi_{i,j}\times\psi_{j,k}$ on $\nu_{i,j}\times\mu_{j,k}$ and the
  product framing $\psi_{i,j}\times\phi'_{j,k}$ on
  $\mu_{i,j}\times\nu'_{j,k}$ agree with the pullback framing of
  $\psi_{i,k}$ on $\nu_{i,k}$.
  \end{itemize}
\end{definition}

See the right half of Figure~\ref{fig:bimodules-explanatory}. In line
with earlier notation, for $x\in\Cat_i,y\in\Cat'_j$, we will denote
the restrictions of $\jmath_{i,j},\mu_{i,j},\psi_{i,j}$ to $\bN(x,y)$
as $\jmath_{x,y},\mu_{x,y},\psi_{x,y}$.

\begin{example}\label{ex:identity-bimodule}
  The only explicit example of a framed flow bimodule that we will
  provide is that of the identity bimodule connecting a framed flow
  category $\Cat=(\Cat,\iota,\phi)$ to itself. Assume $\Cat$ is
  embedded relative $d$, and fix $e=0$.

  For the $0$-dimensional spaces, if $\gr(x)=\gr(y)$, define
  $\bN(x,y)$ to be empty if $x\neq y$, and to be a point $p_x$
  embedded in $\FF^{d,0}_0=\RR^0$, otherwise.

  For the $1$-dimensional spaces, if $\gr(x)=\gr(y)+1$, define
  $\bN(x,y)=\Delta^1\times\bM(x,y)$, embedded in
  $\FF^{d,0}_1=\wt{\Delta}^1\times\RR^d=\Delta^1\times\RR^d$ by the
  product embedding $(\Id,\iota_{x,y})$; the normal bundle $\mu_{x,y}$
  is the pullback of $\nu_{x,y}$ under the projection map
  $\Delta^1\times\bM(x,y)\to\bM(x,y)$, and framing $\psi_{x,y}$ is
  given by $\phi_{x,y}$. Its boundary consists of two pieces:
  $\bdy_{0,1}=\bM(x,y)\times p_y$ and $\bdy_{1,0}=p_x\times\bM(x,y)$,
  and they are embedded in $\bdy_{0,1}\FF^{d,0}_1$ and
  $\bdy_{1,0}\FF^{d,0}_1$ by the product embedding as well.

  In general, if $\gr(x)=\gr(y)+k$, the space $\bM(x,y)$ is a
  $\langle k-1\rangle$-manifold, so it has closed strata $\bdy_I\bM(x,y)$ for
  $I\subset\{1,\dots,k-1\}$, which are embedded and framed in
  $\bdy_I\RR_+^{k-1}\times\RR^{kd}$. Equation~\eqref{eq:minkow-is-product} produces a decomposition
  \[
    \FF^{d,0}_k\cong \bigcup_{I\subset\{1,\dots,k-1\}}(\Delta^k_I\times\bdy_I\RR^{k-1}_+\times\RR^{kd}),
  \]
  where $\Delta^k_I=[(w_i)_{i\in I\cup\{0,k\}}]$ is the convex hull of
  vertices $w_i$ fo $i\in I\cup\{0,k\}$ as before.  Define $\bN(x,y)$
  to be corresponding union
  \[
    \bN(x,y)\defeq \bigcup_{I\subset\{1,\dots,k-1\}}(\Delta^k_I\times\bdy_I\bM(x,y)),
  \]
  with the subspace $\Delta^k_I\times\bdy_I\bM(x,y)$ embedded in the
  subspace
  $\Delta^k_I\times\bdy_I\RR^{k-1}_+\times\RR^{kd}\subset\FF^{d,0}_k$
  by the product embedding $(\Id,\iota_{x,y})$ and framed by
  $\phi_{x,y}$. The smooth structure and the
  $\langle 2\ltimes k\rangle$-manifold structure on $\bN(x,y)$ are
  induced as a subspace of $\FF^{d,0}_k$. See
  Figure~\ref{fig:identity-bimodule}.
\end{example}

\begin{figure}
  \centering
  \begin{tikzpicture}[x={(1cm,0cm)},y={(0cm,1cm)},z={(-0.6cm,-0.3cm)}]
    \coordinate (w0) at (-1,1);
    \coordinate (w1) at (0,0);
    \coordinate (w2) at (1,1);

    \coordinate (p) at (0,0,-1);
    \coordinate (q) at (0,0,1);

    \coordinate (h) at (0,1.5);

    \fill[black!20] ($(p)+(w0)$)--($(p)+(w1)$)--($(p)+(w2)$) ..controls ($(p)+(w2)+(0,1)$) and ($(h)+(w2)+(0,-0.1,-0.3)$) .. ($(h)+(w2)$)--($(h)+(w0)$);

    \fill[black!50] ($(q)+(w0)$)--($(q)+(w1)$)--($(q)+(w2)$) ..controls ($(q)+(w2)+(0,1)$) and ($(h)+(w2)+(0,0.2,0.6)$) .. ($(h)+(w2)$)--($(h)+(w0)$) ..controls($(h)+(w0)+(0,0.2,0.6)$) and ($(q)+(w0)+(0,1)$)..($(q)+(w0)$);
    
    \foreach \i in {0,1,2}{
      \draw[ultra thin,->] (w\i)++(0,0,-2.5)--++(0,0,5.5);
      \ifnum\i=1
      \else
      \draw[ultra thin,->] (w\i)++(0,2,-2.5)--++(0,0,5.5);
      \fi
      \node at ($(p)+(w\i)$) {$\bullet$};
      \node at ($(q)+(w\i)$) {$\bullet$};
    }

    \foreach\j in {-2,2}{
      \draw[ultra thin] ($(w0)+(0,0,\j)$)--($(w1)+(0,0,\j)$)--($(w2)+(0,0,\j)$);
      \draw[ultra thin,->] ($(w0)+(0,0,\j)$)--++(0,2.5);
      \draw[ultra thin,->] ($(w2)+(0,0,\j)$)--++(0,2.5);
    }

    \foreach \x in {p,q}{
      \draw[thick] ($(\x)+(w0)$)--($(\x)+(w1)$)--($(\x)+(w2)$);
      \draw[thick,dashed] ($(\x)+(w0)$)--($(\x)+(w2)$);
    }

    \draw ($(h)+(w0)$)--($(h)+(w2)$) coordinate[midway] (hmid);
    \foreach \i in {0,2}{
      \draw[thick] ($(p)+(w\i)$) ..controls ($(p)+(w\i)+(0,1)$) and ($(h)+(w\i)+(0,-0.1,-0.3)$) .. ($(h)+(w\i)$) .. controls ($(h)+(w\i)+(0,0.2,0.6)$) and ($(q)+(w\i)+(0,1)$) .. ($(q)+(w\i)$);
    }

    \node[anchor=south] at ($(w2)+(0,2.5,-2)$) {\small $\RR_+$};
    \node[anchor=north east] at ($(w1)+(0,0,3)$) {\small $\RR^{2d}$};

    \node[anchor=south west] at ($(h)+(w2)$) {\small $[w_0,w_2]\times\bM(x,y)$};
    \node[anchor=north west] at ($(q)+(w1)$) {\small $[w_0,w_1,w_2]\times\bdy\bM(x,y)$};

  \end{tikzpicture}
  \caption{If $\gr(x)-\gr(y)=2$, the space $\bN(x,y)$ is the union of
    two pieces $[w_0,w_2]\times\bM(x,y)$ and
    $[w_0,w_1,w_2]\times\bdy\bM(x,y)$ (separated by dashed lines), and
    is embedded in
    $\FF^{d,0}_2=\big([w_0,w_2]\times\RR_+\times\RR^{2d}\big)\cup
    \big([w_0,w_1,w_2]\times\RR^{2d}\big)$. We drew here the case
    where $\bM(x, y)$ is an interval and hence $\bN(x,y)$ is a
    hexagon.}\label{fig:identity-bimodule}
\end{figure}

If $\Bat=(\Bat,\jmath,\psi)$ is a compact framed flow bimodule
connecting compact framed flow categories $\Cat=(\Cat,\iota,\phi)$ and
$\Cat'=(\Cat',\iota',\phi')$, then there is an associated chain map
$f_{\Bat}\from C_*(\Cat)\to C_*(\Cat')$, defined on generators
$x\in\Cat$ by
\begin{equation}\label{eq:flowbimodule-to-chainmap}
  f_{\Bat}(x)=\sum_{\substack{y\in\Cat'\\\gr(y)=\gr(x)}}\#\bN(x,y) y.
\end{equation}
Compare Equation~\eqref{eq:flowcat-to-chaincx}.

\begin{proposition}\label{prop:bimodule}
  Let $\Cat=(\Cat,\iota,\phi)$ and $\Cat'=(\Cat',\iota',\phi')$ be two
  compact framed flow categories with gradings bounded below.  A
  compact framed flow bimodule $\Bat=(\Bat,\jmath,\psi)$ connecting
  them induces a map of spectra $F_{\Bat}\from S(\Cat) \to S(\Cat')$
  so that the induced maps on homology
  $H(F_{\Bat})\from H(S(\Cat))\to H(S(\Cat'))$ and
  $H(f_{\Bat})\from H(\Cat)\to H(\Cat')$ agree.
\end{proposition}

\begin{proof}
  Mimicking the construction from Section~\ref{sec:tospectra}, we will
  associate a CW complex to $\Bat$. Again, first assume both
  $\Cat,\Cat'$ have gradings supported in some interval $[B+1,A]$,
  with $A,B\in\ZZ$ and $B<0$, and set $D_{d,e}(A,B)=e+(A-B)d-B-1$.  Our CW
  complex $|\Bat|_{\jmath,\psi,A,B}$ again has a single $0$-cell, and
  for each $x\in\Cat\amalg\Cat'$, a cell $\Dell(x)$ of dimension
  $\gr(x)+D_d(A,B)$ if $x\in\Cat'$, or $1+\gr(x)+D_d(A,B)$ if
  $x\in\Cat$.

  As before, choose small $\epsilon>0$ and large $R$ such that
  $\iota_{i,j}$ (respectively $\iota'_{i,j}$) extends via the normal
  framings to embeddings of
  $\bM(i,j)\times[-\epsilon,\epsilon]^{(i-j)d}$ (respectively,
  $\bM'(i,j)\times[-\epsilon,\epsilon]^{(i-j)d}$) inside
  $[0,R]^{i-j-1}\times[-R,R]^{(i-j)d}\subset\RR_+^{i-j-1}\times\RR^{(i-j)d}\cong\E^d_{i,j}$. Just
  as $[0,R]^{i-j-1}\times[-R,R]^{(i-j)d}$ plays the role of a large
  compact subset of $\E^d_{i,j}$, we will pick a suitably large
  compact subset of $\FF^{d,e}_{i,j}$ as follows: Consider
  $\RR_+^{n-1}=\RR_+\langle \vv_1,\dots,\vv_{n-1}\rangle$ from
  Equation~\eqref{eq:minkowski-sum}, and let
  \[
    [0,R]^{n-1}=[0,R]\langle \vv_1,\dots,\vv_{n-1}\rangle=\big\{\sum_i t_i\vv_i\mid 0\leq t_i\leq R,\forall i\big\}\subset\RR_+^{n-1}.
  \]
  Let
  \[
    \wt{\Delta}^n_R=\Delta^n+[0,R]^{n-1}\subset \Delta^n+\RR_+^{n-1}=\wt{\Delta}^n
  \]
  be the corresponding compact subset; the corresponding compact
  subset of $\FF^{d,e}_{i,j}$ (from Equations~\eqref{eq:F-d-e-n}
  and~\eqref{eq:F-d-e-i-j}) is
  \[
    \wt{\Delta}^{i-j}_R\times[-R,R]^e\times[-R,R]^{(i-j)d}\subset
    \wt{\Delta}^{i-j}\times\RR^e\times\RR^{(i-j)d}=\FF^{d,e}_{i,j}.
  \]
  We will assume that the embeddings $\jmath_{i,j}$ extend via the
  normal framing to embeddings of
  $\bN(i,j)\times[-\epsilon,\epsilon]^e\times[-\epsilon,\epsilon]^{(i-j)d}$
  inside $\wt{\Delta}^{i-j}_R\times[-R,R]^e\times[-R,R]^{(i-j)d}$.

  Let $T_e$ be the set of $e$ $\RR$-factors appearing in $\RR^e$, and let
  $S^e=\bigwedge_{T_e}S^1$ denote the corresponding
  $e$-dimensional sphere. Then $T_e\amalg T_d(A,B)$ is a set of
  cardinality $D_{d,e}(A,B)$, and $S^e\smas S_d(A,B)$
  is the corresponding $D_{d,e}(A,B)$-dimensional sphere.

  The CW complex $|\Bat|_{\jmath,\psi,A,B}$ has
  $\Sigma^e|\Cat'|_{\iota',\phi',A,B}=S^e\smas|\Cat'|_{\iota',\phi',A,B}$
  as a subcomplex. In more detail, for any $x\in\Cat'$ with $\gr(x)=m$, define the cell
  \[
    \Dell(x)=[-\epsilon,\epsilon]^{(A-m)d}\times[0,R]^{m-B-1}\times[-\epsilon,\epsilon]^e\times[-R,R]^{(m-B)d};
  \]
  letting $\bdy_i\Dell(x)$ denote the subset where the coordinate in the $i\th$ $[0,R]$-factor is $0$, we get an embedding
  \[
    \bM'(x,y)\times\Dell(y)\into\bdy_{m-l}\Dell(x)
  \]
  for any $y\in\Cat'$ with $\gr(y)=l<m$. The attaching map on
  $\bdy\Dell(x)$ maps the image $\Dell_y(x)$ of this embedding to
  $\Dell(y)$ and $\bdy\Dell(x)\setminus\bigcup_y\Dell_y(x)$ to the
  basepoint. (This construction is exactly as described in
  Section~\ref{sec:tospectra}, but with an extra factor of
  $[-\epsilon,\epsilon]^e$.)

  Now for $x\in\Cat$ with $\gr(x)=m$, define the cell
  \begin{equation}\label{eq:dell-x}
    \Dell(x)=\underbrace{[-\epsilon,\epsilon]^{(A-m)d}}_{\subset\E^d_{A,m}}\times\underbrace{\wt{\Delta}^{m-B}_R\times[-R,R]^e\times[-R,R]^{(m-B)d}}_{\subset\FF^{d,e}_{m,B}}\subset\FF^{d,e}_{A,B};
  \end{equation}
  see Figure~\ref{fig:dell-x}.
  There is a map $\Dell(x)\to\wt{\Delta}^{m-B}$ by first projecting to
  the $\wt{\Delta}^{m-B}_R$ factor, followed by the inclusion
  $\wt{\Delta}^{m-B}_R\into \wt{\Delta}^{m-B}$. By an abuse of
  notation, for any $(i_0,i_1)\in 2\ltimes (m-B)$, let
  $\bdy_{i_0,i_1}\Dell(x)$ denote the preimage of
  $\bdy_{i_0,i_1}\wt{\Delta}^{m-B}$ under this map.
  
  For $y\in\Cat$ with $\gr(y)=l<m$, we have an embedding
  \[
    \bM(x,y)\times[-\epsilon,\epsilon]^{(m-l)d}\into [0,R]^{m-l-1}\times[-R,R]^{(m-l)d}\subset\E^d_{m,l};
  \]
  extending by identity, we get an embedding
  \begin{gather*}
    \bM(x,y)\times\Dell(y)\cong[-\epsilon,\epsilon]^{(A-m)d}\times \big(\bM(x,y)\times[-\epsilon,\epsilon]^{(m-l)d}\big) \times \big(\wt{\Delta}^{l-B}_R\times[-R,R]^e\times[-R,R]^{(l-B)d}\big)\\
    \into\underbrace{[-\epsilon,\epsilon]^{(A-m)d}}_{\subset\E^d_{A,m}}\times \underbrace{[0,R]^{m-l-1}\times[-R,R]^{(m-l)d}}_{\subset\E^d_{m,l}} \times\underbrace{\wt{\Delta}^{l-B}_R\times[-R,R]^e\times[-R,R]^{(l-B)d}}_{\subset\FF^{d,e}_{l,B}}=\bdy_{0,m-l}\Dell(x),
  \end{gather*}
  and let $\Dell_y(x)$ be its image. Similarly, for $y\in\Cat'$ with $\gr(y)=l\leq m$, we have an embedding
  \[
    \bN(x,y)\times[-\epsilon,\epsilon]^e\times[-\epsilon,\epsilon]^{(m-l)d}\into \wt{\Delta}^{m-l}_R\times[-R,R]^e\times[-R,R]^{(m-l)d}\subset\FF^{d,e}_{m,l};
  \]
  once again, extending by identity, we get an embedding
  \begin{gather*}
    \bN(x,y)\times\Dell(y)\cong[-\epsilon,\epsilon]^{(A-m)d}\times\big(\bN(x,y)\times[-\epsilon,\epsilon]^e\times[-\epsilon,\epsilon]^{(m-l)d}\big)\times\big([0,R]^{l-B-1}\times[-R,R]^{(l-B)d}\big)\\
    \into \underbrace{[-\epsilon,\epsilon]^{(A-m)d}}_{\subset\E^d_{A,m}}\times \underbrace{\wt{\Delta}^{m-l}_R\times[-R,R]^e\times[-R,R]^{(m-l)d}}_{\subset\FF^{d,e}_{m,l}} \times\underbrace{[0,R]^{l-B-1}\times[-R,R]^{(l-B)d}}_{\subset\E^{d}_{l,B}} =\bdy_{1,m-l}\Dell(x),
  \end{gather*}
  and let $\Dell_y(x)$ again denote its image. As before, the
  attaching map on $\bdy\Dell(x)$ maps $\Dell_y(x)$ to $\Dell(y)$ (for
  all $y\in\Cat\amalg\Cat'$), and maps
  $\bdy\Dell(x)\setminus\bigcup_y\Dell_y(x)$ to the basepoint.

  \begin{figure}
    \centering
    \begin{tikzpicture}
      \begin{scope}[scale=3,x={(-1cm,0.2cm)},y={(0.8cm,0.4cm)},z={(0,1cm)}]
          \coordinate (w1) at (0,0,0);
          \coordinate (w2) at (0,-0.3,0.3333);
          \coordinate (w3) at (-0.3,0,0.6667);
          \coordinate (w4) at (0,0,1);

          \coordinate (e1) at (0.7,0,0);
          \coordinate (e2) at (0,0.7,0);

          \fill[black!30] ($(w3)!0.5!(w4)$)--($(w2)!0.75!(w4)+(e1)$)--($(w4)+(e1)$)
          --($(w4)+(e1)+(e2)$)--($(w4)+(e2)$)--($(w3)!0.5!(w4)+(e2)$)--cycle;

          
          \draw[thick] (w1)--(w2)--(w3)--(w4);

          \foreach\i in {1,2,4}{\draw[thick] (w\i)--($(w\i)+(e1)$);}
          \foreach\i in {1,3,4}{\draw[thick] (w\i)--($(w\i)+(e2)$);}

          \draw[thick] ($(w1)+(e1)$)--($(w2)+(e1)$)--($(w4)+(e1)$)--($(w4)+(e1)+(e2)$)--($(w4)+(e2)$)--($(w3)+(e2)$)--($(w1)+(e2)$);
          
          \node at ($(w4)+0.5*(e1)+0.5*(e2)$) {\small $\bdy_{0,3}$};
          
        \end{scope}
      \end{tikzpicture}
      \caption{The cell $\Dell(x)$ from Equation~\eqref{eq:dell-x};
        here $n=m-B=3$, and we have only shown the $\wt{\Delta}^{3}_R$
        factor. (Compare the $n=3$ case in Figure~\ref{fig:tilde-delta-first-few}.) The top facet is $\bdy_{0,3}\Dell(x)$, and
        its $\delta$-neighborhood $\wt{\Dell}(x)=N_\delta\bdy_{0,3}\Dell(x)$ is
        shown in gray.}\label{fig:dell-x}
    \end{figure}
  
    By construction, $\Sigma^e|\Cat'|_{\iota',\phi',A,B}$ is a
    subcomplex of $|\Bat|_{\jmath,\psi,A,B}$, and let
    $\wt{|\Cat|}_{\iota,\phi,A,B}$ be the corresponding quotient
    complex. Therefore, we have a homotopy equivalence
    \begin{equation}\label{eq:puppe-homotopy-1}
      \wt{|\Cat|}_{\iota,\phi,A,B}\simeq
      \Cone\big(\Sigma^e|\Cat'|_{\iota',\phi',A,B}\to
      |\Bat|_{\jmath,\psi,A,B}\big).
    \end{equation}
    The quotient complex also has a cell $\Dell(x)$ for each
    $x\in\Cat$, with attaching map sending $\Dell_y(x)$ to $\Dell(y)$
    for all $y\in\Cat$, and the rest
    $\bdy\Dell(x)\setminus\bigcup_y\Dell_y(x)$ to the
    basepoint. Consider the subspace
    $\bdy_{0,m-B}\Dell(x)\cong[0,R]^{m-B-1}$, and let
    $\wt{\Dell}(x)=N_\delta\bdy_{0,m-B}\Dell(x)\cong[0,\delta]\times\bdy_{0,m-B}\Dell(x)$
    be a small (closed) $\delta$-neighborhood, see
    Figure~\ref{fig:dell-x}.  Comparing with
    Equation~\eqref{eq:dell-x}, we get
    \[
      \wt{\Dell}(x)\cong
      [-\epsilon,\epsilon]^{(A-m)d}\times
      [0,\delta]\times[0,R]^{m-B-1}\times[-R,R]^e\times[-R,R]^{(m-B)d}.
    \]
    These $\delta$-neighborhoods are constructed inductively. We
    construct $\delta$-neighborhoods
    $N_\delta\bdy_{0,m-B}\wt{\Delta}^{m-B}_R\cong[0,\delta]\times
    \bdy_{0,m-B}\wt{\Delta}^{m-B}_R$ of
    $\bdy_{0,m-B}\wt{\Delta}_R^{m-B}$ inside $\wt{\Delta}_R^{m-B}$,
    such that, restricted to any facet
    $\bdy_{0,m-l}\wt{\Delta}^{m-B}_R\cong
    [0,R]^{m-l-1}\times\wt{\Delta}^{l-B}_R$, it is $[0,R]^{m-l-1}$
    times the $\delta$-neighorbood
    $N_\delta\bdy_{0,l-B}\wt{\Delta}^{l-B}_R\cong[0,\delta]\times
    \bdy_{0,l-B}\wt{\Delta}^{l-B}_R$ of
    $\bdy_{0,l-B}\wt{\Delta}^{l-B}_R$ inside $\wt{\Delta}^{l-B}_R$;
    and then define the $\delta$-neighborhood
    $N_\delta\bdy_{0,m-B}\Dell(x)$ to be this $\delta$-neighborhood
    $N_\delta\bdy_{0,m-B}\wt{\Delta}^{m-B}_R$ times
    $[-\epsilon,\epsilon]^{(A-m)d}\times[-R,R]^e\times[-R,R]^{(m-B)d}$. The
    embedding $\bM(x,y)\times\Dell(y)\into\bdy_{0,m-l}\Dell(x)$
    restricts to give an embedding
    \[
      \bM(x,y)\times\wt{\Dell}(y)\into \bdy_{0,m-l}\wt{\Dell}(x),
    \]
    with image say $\wt{\Dell}_y(x)$. Therefore, we can define a new
    CW complex with cells $\wt{\Dell}(x)$ for $x\in\Cat$, and
    attaching maps sending $\wt{\Dell}_y(x)$ to $\wt{\Dell}(y)$ and
    the rest $\bdy\wt{\Dell}(x)\setminus\bigcup_y\wt{\Dell}_y(x)$ to
    the basepoint. This gives precisely the construction from
    Section~\ref{sec:tospectra}, but with an extra factor of
    $[0,\delta]\times[-\epsilon,\epsilon]^e$, and therefore, this new CW
    complex is $\Sigma\Sigma^e|\Cat|_{\iota,\phi,A,B}$, where the
    first suspension corresponds to the $[0,\delta]$-factor, and the
    second suspension is smash product with $S^e$ from earlier. On the
    other hand, we have a quotient map
    \begin{equation}\label{eq:puppe-homotopy-2}   
      \wt{|\Cat|}_{\iota,\phi,A,B}\xrightarrow{\sim} \Sigma\Sigma^e|\Cat|_{\iota,\phi,A,B},
    \end{equation}
    that maps each $\Dell(x)$ to $\wt{\Dell}(x)$ by quotienting
    $\Dell(x)\setminus\wt{\Dell}(x)$ to the basepoint. This map is
    identity on the cellular chain complex, and hence a homotopy
    equivalence.

    Combining Equations~\eqref{eq:puppe-homotopy-1}
    and~\eqref{eq:puppe-homotopy-2}, we get a homotopy equivalence
    \[
      \Sigma\Sigma^e|\Cat|_{\iota,\phi,A,B}\simeq \Cone\big(\Sigma^e|\Cat'|_{\iota',\phi',A,B}\to
      |\Bat|_{\jmath,\psi,A,B}\big).
    \]
    Quotienting $|\Bat|_{\jmath,\psi,A,B}$ to the basepoint, we get a
    map
    $\Cone\big(\Sigma^e|\Cat'|_{\iota',\phi',A,B}\to
    |\Bat|_{\jmath,\psi,A,B}\big)\to \Sigma
    \Sigma^e|\Cat'|_{\iota',\phi',A,B}$; composing, we get our
    required map
    \[
      F_\Bat\from\Sigma\Sigma^e|\Cat|_{\iota,\phi,A,B}\to\Sigma\Sigma^e|\Cat'|_{\iota',\phi',A,B}.
    \]
    (This is the standard Puppe construction.)  As in
    Equation~\eqref{eq:flowcat-to-chaincx}, the shifted reduced
    cellular chain complex of $|\Bat|_{\jmath,\psi,A,B}$ is isomorphic
    to the mapping cone of the map $f_\Bat$ from
    Equation~\eqref{eq:flowbimodule-to-chainmap}
    \[
      \Cone(f_\Bat)\cong\Sigma^{-D_d(A,B)}\wt{C}^{\mathrm{CW}}_*(|\Bat|_{\jmath,\psi,A,B});
    \]
    Therefore, $F_\Bat$ induces the chain map
    $f_\Bat\from C_*(\Cat)\to C_*(\Cat')$ (after shifting gradings by
    $1+D_d(A,B)$).

    We pass to spectra as before. (The cone appearing in
    Equation~\eqref{eq:puppe-homotopy-1} is a special case of homotopy
    colimit, so is taken levelwise when constructing spectra.) By
    taking the suspension spectra, and desuspending
    $1+D_d(A,B)=1+e+C_d(A,B)$ times, we get induced maps (still
    denoted $F_\Bat$),
    \[
      F_\Bat\from \Omega\Omega^e\Sigma\Sigma^e\Omega^{C_d(A,B)}|\Cat|_{\iota,\phi,A,B}\smas\SphereS\to \Omega\Omega^e\Sigma\Sigma^e\Omega^{C_d(A,B)}|\Cat'|_{\iota',\phi',A,B}\smas\SphereS.
    \]
    Via the equivalence between $X$ and $\Omega\Sigma X$ for any spectrum $X$, we get the map
    \[
      F_\Bat\from \Omega^{C_d(A,B)}|\Cat|_{\iota,\phi,A,B}\smas\SphereS\to \Omega^{C_d(A,B)}|\Cat'|_{\iota',\phi',A,B}\smas\SphereS.
    \]
    Taking homotopy colimits as $A\to\infty$ and $B\to-\infty$, we get our map on spectra
    \begin{gather*}
      F_\Bat=\Sp(\Cat,\iota,\phi)=\hocolim_{A,B} \Omega^{C_d(A,B)}|\Cat|_{\iota, \phi, A,B}\smas\SphereS\\
      \qquad\qquad\qquad\to \Sp(\Cat',\iota',\phi')=\hocolim_{A,B} \Omega^{C_d(A,B)}|\Cat'|_{\iota', \phi', A,B}\smas\SphereS.
    \end{gather*}

    Finally, if the gradings of $\Cat$ and $\Cat'$ were only bounded
    below, we modify the construction as in
    Equation~\eqref{eq:triple-hocolim}. We truncate $\Cat,\Cat',\Bat$
    to only objects with $\gr\leq K$, and only consider $A\geq K$, and
    then take homotopy colimits:
    \begin{gather*}
      F_\Bat=\Sp(\Cat,\iota,\phi)=\hocolim_{A,B,K} \Omega^{C_d(A,B)}|\Cat_{\leq K}|_{\iota, \phi, A,B}\smas\SphereS\\
      \qquad\qquad\qquad\to \Sp(\Cat',\iota',\phi')=\hocolim_{A,B,K} \Omega^{C_d(A,B)}|\Cat'_{\leq K}|_{\iota', \phi', A,B}\smas\SphereS.\qedhere
    \end{gather*}
\end{proof}

Mimicking Definition~\ref{def:flowcat-sub-quotient}, a {\em framed flow
sub-bimodule} of a framed flow bimodule $\Bat\from\Cat\to\Cat'$ is a
pair of framed flow subcategories $\Dat\subset\Cat$,
$\Dat'\subset\Cat'$, such that $\bN(x,y)=\varnothing$ for all
$x\in\Dat$ and $y\in\Cat'\setminus\Dat'$. (The embeddings and framings
on the framed flow sub-bimodule are induced from those on $\Bat$.)

We can now extend the notion of framed flow bimodules to the
$H$-filtered setting. Specifically, if $\Cat$ and $\Cat'$ are
$H$-filtered framed flow categories (as in
Definition~\ref{def:H-filtered-flow-cat}), an $H$-filtered framed flow
bimodule $\Bat$ between them is a framed flow bimodule
$\Bat\from\Cat\to\Cat'$ which restricts on the objects of $\Cat(h)$
and $\Cat'(h)$ to give framed flow sub-bimodules
$\Bat(h)\from\Cat(h)\to\Cat'(h)$, for all $h\in H$.  In this case,
Proposition~\ref{prop:bimodule} produces an $H$-filtered map of spectra.

\subsection{Bimodules from grid diagrams}\label{sec:invariance-bimodules-bubbles}

Assume $(\Cat,\iota,\phi)$ and $(\Cat',\iota',\phi')$ are two framed
flow categories associated to some fixed grid diagram $\Grid$, both
constructed following the steps in Section~\ref{sec:construction}; let
$\bM_{\vN,\vlambda}(D)$ and $\bM'_{\vN,\vlambda}(D)$ be the moduli
spaces used in the two constructions. In this section, we outline the
steps that will be needed in order to construct a framed flow bimodule
$(\Bat,\jmath,\psi)$ connecting them. The construction is similar, so
we will only describe the major waypoints, and will skip the details.

For each triple $(D, \vN, \vlambda) \in \DNlambda$, there is a new
``interpolating'' moduli space $\Noduli_{\vN, \vlambda}(D)$, with
compactification $\bNoduli_{\vN, \vlambda}(D)$. This plays the role
of the moduli space of continuation maps in Floer theory with domain
$D$, and $|\vN|$ marked points on the boundary, with the $N_j$ points
corresponding to $H_j$ or $V_j$ bubbles, which have been grouped
according to the partition $\lambda_j$. There is no $\R$ translation
for continuation maps, so compared to Equation~\eqref{eq:dimension},
the dimension is given by
\[
  \dim\Noduli_{\vN, \vlambda}(D) = \gr(D,\vN,\vlambda)=\mu(D)+|\vlambda|.
\]
Recall from
Section~\ref{sec:moduli} that in the case of ordinary moduli spaces,
we did not consider the trivial case
$(D, \vN, \vlambda)=(c_x, \vec{0}, \vec{0})$. This has
$\gr(c_x, \vec{0}, \vec{0})=0$ and therefore the formal dimension of
$\Moduli_{\vec{0}, \vec{0}}(c_x)$ should be $-1$. The interpolating
moduli space $\Noduli_{\vec{0}, \vec{0}}(c_x)$ has dimension zero and
we no longer ignore it. Rather, we define it to be a point $p_x$.

These new spaces are stratified, and similar to Equation~\eqref{eq:moduli-bdy}, the strata of $\bNoduli_{\vN, \vlambda}(D)$ are given by
\begin{equation}\label{eq:noduli-strata}
 \prod_{j=1}^{i-1}\Moduli_{\vN^j+\coefficients{E^j}+\coefficients{F^j}, \vlambda^j}(D^j) \times  \Noduli_{\vN^i+\coefficients{E^i}+\coefficients{F^i}, \vlambda^i}(D^i)\times  \prod_{j=i+1}^{r}\Moduli'_{\vN^j+\coefficients{E^j}+\coefficients{F^j}, \vlambda^j}(D^j)
 \end{equation}
for $1 \leq i \leq r$, subject to similar conditions.  Similarly to
Section~\ref{sec:enum}, each $\bNoduli_{\vN, \vlambda}(D)$ has a
single codimension zero stratum, namely $\Noduli_{\vN, \vlambda}(D)$;
and similarly to Section~\ref{sec:codim1}, the codimension-one strata
correspond to trajectory breaking into two pieces
\[
  \Moduli_{\vN^1,\vlambda^1}(D^1)\times\Noduli_{\vN^2,\vlambda^2}(D^2)\qquad\text{or}\qquad \Noduli_{\vN^1,\vlambda^1}(D^1)\times\bM'_{\vN^2,\vlambda^2}(D^2)
\]
with $D^1*D^2=D$, $\vN^1+\vN^2=\vN$, and
$\lambda^1_j*\lambda^2_j=\lambda_j$ for all $j=2,\dots,n$; or a single boundary degeneration (Type II)
\begin{equation}\label{eq:noduli-II}
  \Noduli_{\vN+\ve_j,\vlambda'}(D^1),\text{ with }D^1+H_j=D\text{ or }D^1+V_j=D
\end{equation}
and $\lambda'_s=\lambda_s$ for all $s\neq j$ and $\lambda'_j\in\UE(\lambda_j)$; or an elementary coarsening of partitions (Type III)
\begin{equation}\label{eq:noduli-III}
  \Noduli_{\vN,\vlambda'}(D)
\end{equation}
with $\lambda'_s=\lambda_s$ for all $s\neq j$ and $\lambda'_j\in\EC(\lambda_j)$. Specifically, the codimension-one strata where the $\bM$ or $\bM'$ factor is zero-dimensional is either a Type I product
\begin{equation}\label{eq:noduli-I}
  \Moduli_{0}(R)\times\Noduli_{\vN,\vlambda}(E)\qquad\text{or}\qquad \Noduli_{\vN,\vlambda}(E)\times\bM'_{0}(R)
\end{equation}
where $R$ is a rectangle and $R*E=D$, resp.~$E*R =D$; or a Type IV product  
\begin{equation}\label{eq:noduli-IV}
  \Moduli_{N\ve_j,(N)_j}(c_x)\times\Noduli_{\vM,\vmu}(D)\qquad\text{or}\qquad \Noduli_{\vM,\vmu}(D)\times\bM'_{N\ve_j,(N)_j}(c_y)
\end{equation}
where $M_s=N_s$ and $\mu_j=\lambda_j$ for $j\neq s$,
$N+M_j=N_j$, and $(N)*\mu_j=\lambda_j$,
resp.~$\mu_j*(N)=\lambda_j$. Note that the terms appearing in
$\dII(D,\vN,\vlambda)$, $\dIII(D,\vN,\vlambda)$,
$\dI(D,\vN,\vlambda)$, and $\dIV(D,\vN,\vlambda)$ (from
Formula~\eqref{eq:delta123}) correspond to the strata from
Formulas~\eqref{eq:noduli-II}--\eqref{eq:noduli-IV}.

Similarly to Section~\ref{sec:modelsmoduli}, each
$\bN_{\vN,\vlambda}(D)$ is a closed stratum of $\bN_0(\wt{D})$ with
$\wt{D}=D+\wt{E}+\wt{F}$; we define
$\bN([\wt{D}])=\bigcup\bN_0(\wt{D})$, which are
$2\ltimes (\gr(x)-\gr(y))$-manifolds. The local models will be based
on the stratified space $\tZ_N$ with strata $\tZ(p^-,p^0,p^+;\lambda)$
from Equation~\eqref{eq:tZ-strata}; the more general local model is
$\tZ_{\vN}=\prod_j \Sym^{N_j}(\CC)$ (compare Equation~\eqref{eq:ZvN}))
with strata $\tZ(\vp^-,\vp^0,\vp^+;\vlambda)$; and the most general
local model is of the form
$\R^a \times \R_+^{r-1} \times \prod_{j=1}^{i-1}Z_{\vN^i} \times
\tZ_{\vN^i} \times \prod_{j=i+1}^r Z_{\vN^r}$ (compare
Equation~\eqref{eq:mostgeneral}). The standard frame of any stratum
inside this space is same as the standard frame from
Definition~\ref{def:standardframe3}. We require the local model of the
stratum from Equation~\eqref{eq:noduli-strata} inside $\bN([\wt{D}])$
to be same as the local model of
\begin{gather*}
  \RR^a\times\{0\}\times \prod_{j=1}^{i-1}Z(0,\vN^j+\Os(E^j)+\Os(F^j),0;\vlambda^j) \times
  \tZ(0,\vN^i+\Os(E^i)+\Os(F^i),0;\vlambda^i)\\
  \qquad\qquad\qquad\times \prod_{j=i+1}^r Z(0,\vN^j+\Os(E^j)+\Os(F^j),0;\vlambda^j)
\end{gather*}
inside
\[
  \RR^a\times\RR_+^{r-1}\times \prod_{j=1}^{i-1}Z_{\vN^j+\Os(E^j)+\Os(F^j)} \times
  \tZ_{\vN^i+\Os(E^i)+\Os(F^i)} \times \prod_{j=i+1}^r Z_{\vN^j+\Os(E^j)+\Os(F^j)}.
\]

Similarly to Section~\ref{sec:neat}, $\bN_{\vN,\vlambda}(D)$ is neatly
embedded and framed inside $\FF^{d,e}_l$ where $l=\mu(D)+2|\vN|$ is
the thick dimension. The data of the neat embedding consists of an 
$l$-dimensional thickening such that the local model of each stratum
inside the thickening is identified with the local model of that
stratum inside $\bN([\wt{D}])$; this is recorded via the internal
frame, which is the image of the standard frame under this
identification. The external frame is a framing of the thickening
inside the ambient space $\FF^{d,e}_l$. We require the neat embeddings
and framings to be coherent with respect to the lower dimensional
strata (cf.~Definitions~\ref{def:compatible-neat}
and~\ref{def:compatible-ext-framing}). (For the stratum from
Equation~\eqref{eq:noduli-strata}, the neat embeddings and framings on
the $\Moduli$ and $\Moduli'$ factors are induced from $(\iota,\phi)$
and $(\iota',\phi')$, respectively.)

Finally, for grid generators $[x,j_2,\dots,j_n]$ and $[y,i_2,\dots,i_n]$, we set
\[
  \bN([x,j_2,\dots,j_n],[y,i_2,\dots,i_n])=\bN([D])
\]
with induced embedding $\jmath$ and framing $\psi$, similarly to
Equation~\eqref{eq:gridmodulispaces-final}.

\subsection{Invariance under auxiliary choices}\label{sec:invariance-choices}

Finally we are ready to prove Theorem~\ref{thm:main}. After fixing the
grid $\Grid$ (with numbered markings $O_1,\dots,O_n,X_1,\dots,X_n$)
and and designating the marking $O_1$ as distinguished, the further
auxiliary choices in the construction of the filtered Floer spectrum
$\GS(\Grid)$ are listed below, reusing the notaion from the steps in
Section~\ref{sec:construction}.
\begin{enumerate}[leftmargin=*, label=(I-\arabic*)]

\item \label{item:choiceDim} The (even) dimension $d \gg 0$ chosen at the the start of Section~\ref{sec:neat};

\item \label{item:choicePsi} The diffeomorphism $\Psi_n$ from \eqref{eq:Psi}, which is used to embed the permutohedron in Section~\ref{sec:permutohedron}. In turn, to construct $\Psi_n$, we needed:
\begin{enumerate}[leftmargin=*, label=(\alph*)]
\item \label{item:choiceEmb} The proper smooth embedding in Proposition~\ref{prop:proper-h-cobordism}. Any two such differ by an element in the relative diffeomorphism group $\operatorname{Diff}_{\del} (\R^m \times \R_+)$;

\item \label{item:choiceCollar1} The embedding $i'$ in Proposition~\ref{prop:collaring-unique};

\item \label{item:choiceCollar2} The collar neighborhood in part~(\ref{item:permu-extendable-collar}) of the construction of $\Psi_n$, where we extended $\Psi_|$.
\end{enumerate}

\item \label{item:choiceDelta} The construction of the neighborhoods $U(\bdy\bModuli_{\vN,\vlambda}(D))$ and $V(\bdy\bModuli_{\vN,\vlambda}(D))$ in Step~\ref{item:exponentiate}, involving choice of small $\epsilon_Y>0$;

\item \label{item:choiceSmoothing} The smoothing in Step~\ref{item:smooth}, involving choice of different small $\epsilon_Y>0$ (from Lemma~\ref{lemma:Nbhd}) and $\delta>0$ (from Definition~\ref{def:smoothn});

\item \label{item:choiceb} The element $\mf{b}\in \Hom(\CDP'_{k},\tOmega^{k-1}_\fr)$ in Step~\ref{item:change};

\item \label{item:choiceManifold} The framed manifold $M(F,\vP,\vnu)$
  representing $-\mf{b}(F,\vP,\vnu)$, also in Step~\ref{item:change}, along with its thickening;

\item \label{item:choiceFilling} The filling $\bModuli'_{\vN, \vlambda}(D)$ (along with its embedding and framing) in Step~\ref{item:fill};

\item \label{item:choiceThick}The thickening in Step~\ref{item:newinternal};

\item \label{item:choiceEpsR} The values of $\epsilon$, $R$, $A$ and $B$ in the construction of a spectrum from a flow category in Section~\ref{sec:tospectra}.

\end{enumerate}

We seek to prove that making different choices in each of these cases produces equivalent (filtered) spectra.

\begin{lemma}
\label{lem:ah}
The filtered spectrum $\GS(\Grid)$ is independent of the choices in \ref{item:choiceDim} and \ref{item:choiceEpsR}, up to filtered equivalence. 
\end{lemma}

\begin{proof}
For \ref{item:choiceDim}, we can turn a construction for  $d$ to one for a higher $d'$ by a stabilization move in the sense of \cite[Section 3.1]{LOS2}. Given a framed flow category $\Cat$ with a neat embedding relative $d$, we get one relative $d'$ by post-composing each embedding   
\[
  \iota_{x, y} \from \bM(x, y) \hookrightarrow \E^d_{\gr(x),\gr(y)}
\]
with the standard inclusion of $\E^d_{\gr(x),\gr(y)}$ into
$\E^{d'}_{\gr(x),\gr(y)}$. We also let the new framings be obtained
from the old by letting them be trivial in the new directions. It is
proved in \cite[Lemma 3.26]{LipshitzSarkar} that the resulting CW
complex is obtained from the original one by suspending
${C_{d'}(A,B) - C_d(A,B)} = {(d'-d)(A-B)}$ times, (in the notation of
Equation~\eqref{eq:Cdba}), or more specifically, by smash product with
$\bigwedge_{T_{d'}(A,B)\setminus T_d(A,B)}S^1$ (in notation of
Equation~\eqref{eq:SdAB-sphere}).  Since in the definition of the
spectrum in Equation~\eqref{eq:SDef} we de-suspended $C_d(A,B)$ times,
or more specifically, we looped with respect to
$S_d(A,B)=\bigwedge_{T_d(A,B)}S^1$, it follows that the two
spectra are equivalent. The same argument applies to their filtered
pieces.

Independence of the quantities in \ref{item:choiceEpsR} was proved (for general framed flow categories) in Lemmas 3.25 and 3.26 in  \cite{LipshitzSarkar}.
\end{proof}

\begin{lemma}
\label{lem:contra}
The filtered spectrum $\GS(\Grid)$ is independent of the choices in \ref{item:choicePsi}, \ref{item:choiceDelta},  and \ref{item:choiceSmoothing}, up to filtered equivalence. 
\end{lemma}

\begin{proof}
We claim that all of these choices are taken from a contractible set. This is clearly the case for \ref{item:choiceDelta} and \ref{item:choiceSmoothing}. It is also the case for \ref{item:choiceCollar1} and \ref{item:choiceCollar2} from \ref{item:choicePsi}, because the space of collars is contractible by \cite{Cerf}. As for part \ref{item:choiceEmb} from \ref{item:choicePsi}, we will prove in Lemma~\ref{lem:DiffCon} below that $\operatorname{Diff}_{\del} (\R^m \times \R_+)$ is contractible.

In all these cases, if we make two choices we can relate the respective framed flow categories $(\Cat_0, \iota_0, \phi_0)$ and $(\Cat_1, \iota_1, \phi_1)$ by a smooth $1$-parameter family $(\Cat_t, \iota_t, \phi_t)_{t \in [0,1]}$. The flow categories have the same objects, the morphism spaces are diffeomorphic (in a way compatible with gluing), and the neat embeddings and framings vary smoothly. This yields isomorphic CW complexes $|\Cat_t|_{\iota(t), \phi(t), B, A}$, and hence equivalent spectra; compare \cite[Lemma 3.25]{LipshitzSarkar}.
\end{proof}

To complete the proof of Lemma~\ref{lem:contra}, it remains to study $\operatorname{Diff}_{\del} (\R^m \times \R_+)$. We will be using the weak topology on this group (i.e., that of uniform $C^{\infty}$ convergence on compact subsets). The weak topology is sufficient for our purposes because in the end the diffeomorphism of $\R^m \times \R_+$ is used to construct an embedding of the permutohedron into a Euclidean space (with corners). The permutohedron is compact, and hence when we relate two elements of $\operatorname{Diff}_{\del} (\R^m \times \R_+)$ by a path, we only have to show that their restriction to a compact subset is varying smoothly.

\begin{lemma}
\label{lem:DiffCon}
The space $\operatorname{Diff}_{\del} (\R^m \times \R_+)$ (in the weak topology) is contractible.
\end{lemma}

\begin{proof}
We will relate an arbitrary element $f \in \operatorname{Diff}_{\del} (\R^m \times \R_+)$ to the identity by a  continuous path, which is canonical up to a contractible set of choices. 

We first apply a theorem of Kirby and Siebenmann  \cite[Theorem A.2]{KS-book} about the uniqueness of collars to map $f(\R^m \times [0, \epsilon))$ into $\R^m \times [0, \epsilon)$. This gives an isotopy $f_t$ of $\R^m \times \R_+$ so that $f_0 = f$ and $f_1$ is the identity on $\R^m \times [0, \epsilon)$. (Also, $f_t$ is constantly the identity on the boundary.)

Next, we relate $f_1$ to the identity by making the collar $\R^m \times [0, \epsilon)$ bigger and bigger. Precisely, consider the re-scaling map \[\psi_t\from \R^m \times \R_+ \to \R^m \times \R_+, \ \psi_t(x, s) = (x, ts)\]  
and let $h\from [1,2) \to [1, \infty)$ be an increasing homeomorphism. Then, for $t \in [1,2]$, the diffeomorphism
\[ f_t \from  \R^m \times \R_+ \to \R^m \times \R_+, \ \ f_t =\psi_{h(t)} \circ f_1 \circ \psi_{h(t)}^{-1}\]
is the identity on the collar $\R^m \times [0, h(t) \epsilon)$. Furthermore, in the limit as $t\to 2$, we have that $f_t$ converges to the identity (which we denote by $f_2$) in the weak topology. 

Altogether, the family $(f_t)_{t \in [0,2]}$ gives a path of diffeomorphisms from $f_0$ to the identity.
\end{proof}

\begin{lemma}
\label{lem:cobo}
The filtered spectrum $\GS(\Grid)$ is independent of the choices in \ref{item:choiceb}, \ref{item:choiceManifold}, \ref{item:choiceFilling}, and \ref{item:choiceThick}, up to filtered equivalence. 
\end{lemma}

\begin{proof}
Suppose $(\Cat, \iota, \phi)$ and
  $(\Cat', \iota', \phi')$ are two framed flow categories for the grid
  that were constructed by making different choices in \ref{item:choiceb},
  \ref{item:choiceManifold}, \ref{item:choiceFilling}, or
  \ref{item:choiceThick}. We will construct a compact
  framed flow bimodule $(\Bat,\jmath,\psi)$ (as in
  Definition~\ref{def:bimodule}) connecting them, and show that the induced map
  (from Proposition~\ref{prop:bimodule}) is a filtered  equivalence. 
  
Since we already have invariance under
  \ref{item:choicePsi}, we may assume that the two framed flow
  categories agree on $\DNlambdad$. The $0$-dimensional moduli spaces
  were constructed in Section~\ref{sec:base} and the only choice was
  the choice of embedding of points in contractible spaces (the frames
  were chosen explicitly); therefore, we may assume that the two
  framed flow categories agree on $\DNlambda'_0$ as well. 
  
  As discussed in Section~\ref{sec:invariance-bimodules-bubbles}, in
  order to construct the framed flow bimodule $(\Bat,\jmath, \psi)$,
  we need to construct the stratified spaces $\bN_{\vN,\vlambda}(D)$,
  along with coherent neat embeddings and framings. The construction
  follows the same steps as in Section~\ref{sec:construction}. For
  Step~\ref{item:emb}, we will construct $\bN_{\vN,\vlambda}(D)$ for
  all triples $(D,\vN,\vlambda)\in\DNlambdad\cup\DNlambda'_{0}$. Since $(\Cat, \iota, \phi)$ and $(\Cat', \iota', \phi')$
  agree on $\DNlambdad\cup\DNlambda'_0$, we set $e=0$ for the moment, and define
  $\bN_{\vN,\vlambda}(D)$ as in Example~\ref{ex:identity-bimodule}. In
  particular, the $0$-dimensional spaces $\bN_0(c_x)$ for
  $(c_x,0,0)\in\DNlambda_{-1}$ are points $p_x$, embedded in
  $\RR^0$. This ensures that the chain map
  $f_\Bat\from C_*(\Cat)\to C_*(\Cat')$ (from
  Equation~\eqref{eq:flowbimodule-to-chainmap}) is the identity.
  
  When we go to the next steps, we may need to use a larger value of $e$. Starting from the above construction with  $e=0$, we  post-compose by the standard inclusion
  $\FF^{d,0}_{i,j}\into\FF^{d,e}_{i,j}$, and define the new framing to be 
  the old one, plus the standard frame in the new $\RR^e$ direction.

  The rest of the steps are almost identical to what was done in Section~\ref{sec:construction}; we summarize
  them below. Assume $\bN_{\vN,\vlambda}(D)$ have been constructed (with
  coherent neat embeddings and coherent external framings) for all
  $(D,\vN,\vlambda)\in\DNlambdad\cup\DNlambda'_{\leq k-1}$. For any
  $(D,\vN,\vlambda)\in\DNlambda'_k$, its $k$-dimensional boundary
  $\bdy\bN_{\vN,\vlambda}(D)$ is already constructed. By smoothing, we
  get an element
  $[\bdy'\bN_{\vN,\vlambda}(D)]\in\tOmega^k_\fr$. Together, these
  produce a cochain (indeed a cocycle) 
  $\mf{o}_k\in\Hom(\CDP'_{k+1},\tOmega^k_\fr)$. Therefore, $\mf{o}_k$
  is the coboundary of some $\mf{b}\in\Hom(\CDP'_{k},\tOmega^k_\fr)$;
  change each $k$-dimensional spaces $\bN_{\vM,\vmu}(E)$ for
  $(E,\vM,\vmu)\in\DNlambda'_{k-1}$ by $-\mf{b}(E,\vM,\vmu)$, and then
  proceed inductively to construct the spaces $\bN_{\vN,\vlambda}(D)$
  for all $(D,\vN,\vlambda)\in\DNlambda'_k$.
  
  In the end, we obtain the framed flow bimodule $(\Bat,\jmath,\psi)$
  connecting $(\Cat,\iota,\phi)$ to $(\Cat',\iota',\phi')$. As
  discussed above, the induced chain map $f_\Bat$ is identity, so the
  induced map on spectra is an equivalence. The same goes for the maps
  on the filtration levels of the spectra.
\end{proof}

\begin{remark}
The method in the proof of Lemma~\ref{lem:cobo} can also provide alternate proofs of
  invariance under \ref{item:choiceDelta} and
  \ref{item:choiceSmoothing}.
\end{remark}

\begin{proof}[Proof of Theorem~\ref{thm:main}]
This is a combination of Lemmas~\ref{lem:ah}, \ref{lem:contra} and \ref{lem:cobo}.
\end{proof}

\section{Examples}
\label{sec:examples}

In this section we will compute the knot Floer spectra in several examples. In all the cases we considered, the spectra are determined by the information already existing in the knot Floer complexes. At the end we will discuss how one may look for an example where this is not the case. 

\subsection{Generalities} 
We will make repeated use of the following well-known result.
\begin{proposition}
\label{prop:hurewicz}
Suppose that the (reduced) homology of a spectrum $\GS$ is finitely
generated free Abelian and supported in at most two consecutive
gradings; i.e. it is isomorphic to $\Z^k \oplus \Z^l$, with $\Z^k$ in
homological grading $d$ and $\Z^l$ in homological grading $d+1$. Then
$\GS$ is homotopy equivalent to the wedge sum of $k$ copies of
$\SphereS^d$ and $l$ copies of $\SphereS^{d+1}$.
\end{proposition}

\begin{proof}
  By the Hurewicz theorem we have $\pi_d(\GS)\cong \Z^k$, so there is
  a map $f\from \bigvee^k \SphereS^d \to \GS$ inducing an isomorphism on $H_d$. The
  cone of this map has homology $\Z^l$ supported in degree $d+1$;
  applying the Hurewicz theorem again, together with Whitehead's
  theorem, tells us that this cone is equivalent to
  $\bigvee^{l} \SphereS^{d+1}$. From the coexact sequence of $f$ it follows that
  $ \GS$ is equivalent to the cone of a map
  $\bigvee^{l} \SphereS^d \to \bigvee^k \SphereS^d$. This map is zero on homology, and
  therefore zero by the Hopf theorem. The conclusion follows.
\end{proof}

Consider now a grid diagram $\Grid$, and suppose it represents a knot
$K$ (rather than a link with more components). In that case the
spectra $\GS(\Grid)$ and $\hGS(\Grid)$ coincide. Furthermore, in view
of Remark~\ref{rem:limitedU}, we expect the $U_i$ actions on the
spectrum $\GS(\Grid)$ to be trivial. Thus, we will just focus on
$\GS(\Grid)$ as a filtered spectrum. If we ignore the filtration, the
homology of $\GS(\Grid)$ is just $\widehat{\HF}(S^3) =\Z$, so by
Proposition~\ref{prop:hurewicz} we have that $\GS(\Grid)$ is
equivalent to $\SphereS^0$. Thus, the interesting information is in the
filtered pieces $\GS(\Grid, h)$ for $h \in \Z$, and the maps
$\GS(\Grid, h-1) \to \GS(\Grid, h)$. Note also that the cone of the
map $\GS(\Grid, h-1) \to \GS(\Grid, h)$ is equivalent to the
associated graded spectrum $\gGS(\Grid, h)=\ghGS(\Grid, h)$.

The filtered chain complex underlying $\GS(\Grid)$ is $\Cp = \Chat$,
which is filtered chain homotopy equivalent to $\CFKhat(K)$, the hat
version of the knot Floer complex in the notation of \cite[Section
11.3]{LOT} and \cite{Msurvey}.\footnote{Warning: In the original paper
  \cite{OS-knots} where they defined knot Floer homology, Ozsv\'ath
  and Szab\'o used the notation $\CFKhat(K)$ for the associated graded
  complex that is denoted $g\CFKhat(K)$ here, and whose homology is
  $\HFKhat(K)$. Compare \cite[Remark 3.17]{Msurvey}.} The Alexander
filtration on $\CFKhat(K)$ is denoted
\[ \cdots \subseteq \Filt_{h-1} \CFKhat(K) \subseteq \Filt_{h}
  \CFKhat(K) \subseteq \cdots \] We will use the notation $\HFK(K, h)$
for the homology $H_*(\Filt_h \CFKhat(K)) = H_*(\GS(\Grid, h))$.  The
associated graded of the filtered complex $\CFKhat(K)$ is denoted
$g\CFKhat(K)$, with homology $\HFKhat_*(K, h) = H_*(\gGS(\Grid, h))$.

The filtered complex $\CFKhat(K)$ can be defined from any Heegaard
diagram for the knot $K$, not necessarily a grid diagram; for example,
we can use a doubly-pointed diagram and ask for disks to not pass
through one basepoint. As such, we can appeal to the computations of
knot Floer complexes from the literature, or to the computer program
\cite{HFKCalc}. It turns out that in many cases, the information in
the filtered complex $\widehat{\CFK}(K)$ determines the filtered
equivalence class of the spectrum $\GS(\Grid)$. To formulate a general
condition that makes this hold, consider the inclusion
$\Filt_h \CFKhat(K) \to \CFKhat(K)$ which induces on homology a map
$\HFK(K, h) \to \HFhat(S^3) \cong \Z$. Let
\[
  \HFKnew (K, h) = \ker (\HFK(K, h) \to \Z).
\]
Following \cite{OS-tau}, we let $\tau(K)$ be the minimum value of $h$ such that the map $\HFK(K, h) \to \HFhat(S^3) \cong \Z$ is nontrivial. It follows that
\begin{equation}
\label{eq:tau} \HFK(K, h) \cong \begin{cases}
\HFKnew(K, h) & \text{if } h < \tau(K),\\
\HFKnew (K, h) \oplus \Z & \text{if } h \geq \tau(K).
\end{cases} 
\end{equation}
(In the second case, the isomorphism
$\HFK(K, h)\cong\HFKnew (K, h) \oplus \Z$ is not canonical; we only
have a canonical short exact sequence
$0\to\HFKnew (K, h)\to\HFK(K, h)\to\Z\to0$.)

\begin{definition}
\label{def:constrained}
A knot $K \subset S^3$ is called {\em Floer constrained} if the following conditions are satisfied:
\begin{enumerate}[label=(\arabic*),ref=(\arabic*)]
\item\label{item:floer-constrained-1} For every $h \in \Z$, the homology $\HFK(K, h)$ is torsion-free;\footnote{This condition holds in practice. To date, no torsion has been observed in the knot Floer homologies of knots in $S^3$: neither in $\HFKhat(K)$ nor in $\HFK(K, h)$.} 
\item\label{item:floer-constrained-2} For any fixed $h \in \Z$, the homology groups $\HFKnew(K, h)$ are supported in at most two homological degrees and, if there are two, then these degrees are consecutive;
\item\label{item:floer-constrained-3} (Monotonicity) There are no integers $h < h'$ and $k > k'$ such that $\HFKhat_k(K, h)$ and $\HFKhat_{k'}(K, h')$ are both non-zero.  
\end{enumerate}
\end{definition}

\begin{proposition}
\label{prop:constrained}
For any Floer constrained knot,  the filtered homotopy type of $\GS(\Grid)$ is determined by the filtered chain homotopy type of $\CFKhat(K)$. \end{proposition}

\begin{proof}
  The filtration on $\CFKhat(K)$ induces a spectral sequence from
  $\bigoplus_{i \leq h} \HFKhat(K, i)$ to $\HFK(K, h)$, so
  $\HFK(K, h)$ is supported (in homological grading) on a subset of
  the union of the supports of $ \HFKhat(K, i)$ for $i \leq h$. Using
  this fact, we deduce from Condition~\ref{item:floer-constrained-3} in
  Definition~\ref{def:constrained} that monotonicity also holds for
  $\CFK(K, h)$: there are no $h < h'$ and $k > k'$ such that
  $\HFK_k(K, h)$ and $\HFK_{k'}(K, h')$ are both non-zero.

At the spectrum level, the map $\HFK(K, h) \to  \Z$ is realized by the inclusion $\GS(\Grid, h) \to \GS(\Grid) \simeq \SphereS^0.$ Let us define a spectrum $\GSnew(\Grid, h)$ by
\[
  \GSnew(\Grid, h)  = \begin{cases}
    \GS(\Grid, h) & \text{if } h < \tau(K),\\
    \Omega\Cone( \GS(\Grid, h) \to \SphereS^0) & \text{if } h \geq \tau(K).
  \end{cases}
\]
The homology of $\GSnew(\Grid, h)$ is $\HFKnew(K,
h)$. Condition~\ref{item:floer-constrained-1} implies that
$\HFKnew(K, h)$ is torsion-free. Using Proposition~\ref{prop:hurewicz}
and Condition~\ref{item:floer-constrained-2}, we see that
$\GSnew(\Grid, h)$ must be a wedge of spheres, of at most two
(consecutive) dimensions.

We claim that each $\GS(\Grid, h)$ is also a wedge of spheres. For
$h < \tau(K)$, this is clearly true, because
$\GS(\Grid, h)\simeq \GSnew(\Grid, h)$. For $h \geq \tau(K)$, notice
first that, since $\HFK(K, \tau(K)-1)$ maps trivially to
$\HFhat(S^3) \cong \Z$ but $\HFK(K, \tau(K))$ does not, the associated
graded $\HFKhat(K, \tau(K))$ must have some generator in degree
$0$. By monotonicity, we have that $\HFKhat(K, h)$ is supported in
degrees $\geq 0$ for $h > \tau(K)$. Since every element of
$\HFKnew(K, h)$ is eventually mapped to zero in
$\HFKhat(K, h') \cong \Z$ (for $h' \gg 0$), it must be cancelled by an
element of $\HFKhat$ that lies in homological degree one higher (and
also in a higher Alexander grading). Therefore, for $h \geq \tau(K)$
we have $\HFKnew(K, h)$ must be supported in degrees $\geq -1$. Since
$\Sigma\GSnew(\Grid, h)=\Sigma\Omega\Cone(\GS(\Grid, h) \to
\SphereS^0)$ is equivalent to $\Cone(\GS(\Grid, h) \to \SphereS^0)$,
we get a cofibration sequence
\[
  \GS(\Grid, h) \to \SphereS^0 \to \Sigma \GSnew(\Grid, h) \to\Sigma\GS(\Grid, h)\to \cdots
\]
which shows that $\Sigma\GS(\Grid, h)$ is equivalent to the cone of a map
$\SphereS^0 \to \Sigma \GSnew(\Grid, h)$. This map sends $\SphereS^0$ to a wedge of
spheres of non-negative dimensions, and therefore it is determined by
its effect on homology. Since we know that
$\HFK(K, h) = H_*( \GS(\Grid, h))$ is torsion-free by
Condition~\ref{item:floer-constrained-1}, we deduce that
$\GS(\Grid, h)$ is a wedge of spheres, and in fact
\[
  \GS(\Grid, h) \simeq \GSnew(\Grid, h) \vee \SphereS^0
\]
for $h \geq \tau(K)$. Thus, $\GS(\Grid, h)$ is determined by its homology.

To fully determine the filtered homotopy type of $\GS(\Grid)$, we also
need the maps $\GS(\Grid, h) \to \GS(\Grid, h+1)$. By the monotonicity
of $\HFK(K, h)$, these are maps between wedges of spheres of
non-decreasing dimensions, and therefore they are determined by what
they do on homology.
\end{proof}

We proceed to give several examples of Floer constrained knots.

\subsection{The unknot}
The simplest example is when $\Grid$ represents the unknot $U$. The knot Floer homology $\widehat{\HFK}(U)$ is isomorphic to $\Z$, supported in Maslov and Alexander degrees $0$. The spectrum $\GS(\Grid, h)$ has homology $\HFK(K, h) \cong \Z$ (in degree $0$) for $h \geq 0$, and trivial for $h < 0$. We deduce that
\[ \GS(\Grid, h) \cong \begin{cases}
\SphereS^0 & \text{if } h \geq 0, \\
* &\text{otherwise}.\end{cases} \]
Moreover, for $h \geq 0$, the maps $\GS(\Grid, h) \to \GS(\Grid, h+1)$ are isomorphisms.

\subsection{Thin and almost thin knots}
For a knot $K\subset S^3$, we will represent its knot Floer homology 
\[\HFKhat(K) = \bigoplus_{k, h \in \Z}  \HFKhat_k(K, h)\]
by a collection of Abelian groups in the plane, with the two coordinate axes being the Alexander grading $A=h$ and the homological (Maslov) grading $\gr = k$.

A knot $K$ is called {\em Floer homologically thin} (or {\em thin}) if its knot Floer homology is supported in a single diagonal line given by $A - \gr = \delta$, for some $\delta \in \Z$. Thin knots were first studied in \cite{RasmussenThesis}, \cite{RasmussenSurvey}.  All alternating and quasi-alternating knots are thin \cite{OS-alt}, \cite{MO-quasi}. 

More generally, for any $K$, we define its {\em diagonal support} to be 
\[ \deltasupp(K) = \{\delta \mid \HFKhat_k(K, h) \neq 0 \text{ for some } k, h \in \Z \text{ with } h-k = \delta \} \]
and its {\em Floer thickness} as
\[
  \thick(K) = 1+\max \deltasupp(K) - \min \deltasupp(K).
\]
Thus, $K$ is thin iff $\thick(K) = 1$. We can use thickness to define almost thin knots, see also Figure~\ref{fig:almostthin}.
\begin{definition}
A knot $K$ is called {\em almost thin} if $\thick(K) = 2.$
\end{definition}

\begin{figure}
  \centering
  \begin{tikzpicture}[scale=0.7]
    \draw[->] (-2,0)--++(5,0) node[pos=1,anchor=west] {\small $A$};
    \draw[->] (0,-2)--++(0,5)  node[pos=1,anchor=south] {\small $\gr$};

    \draw[dashed] (1,-2)-- node[pos=0,anchor=east] {\small $h$}++(0,5) (1.5,-2)-- node[pos=0,anchor=west] {\small $h+1$}++(0,5);

    \draw[black!50] (-2,-1.5)-- node[black,pos=0,anchor=east] {\small $A-\gr=\delta+1$} ++(4.5,4.5) (-2,-1)-- node[black,pos=0,anchor=east] {\small $A-\gr=\delta$}++(4,4);

    \foreach \i/\j in {1/1.5,1/2,1.5/2,1.5/2.5}{
      \node at (\i,\j) {$\bullet$};}
    
  \end{tikzpicture}
\caption{The knot Floer homology of an almost thin knot. The two diagonals represent the support of $\HFKhat$. The homology $\CFK(K, h)$ is that of the complex in the half-plane $\{A\leq h\}$.}\label{fig:almostthin}
\end{figure}

\begin{proposition}
Let $K$ be a knot that is either thin or almost thin. Suppose that the groups $\HFKhat(K, h)$ and $\HFK(K, h)$ are torsion-free for all $h$. Then $K$ is Floer constrained. 
\end{proposition}

\begin{proof}
  Condition~\ref{item:floer-constrained-1} in
  Definition~\ref{def:constrained} is satisfied by
  hypothesis. Moreover, since $\thick(K) \leq 2$, we have that in any
  Alexander grading $A=h$, the homology $\HFKhat(K, h)$ is supported
  in at most two consecutive gradings $h-\delta-1$ and
  $h-\delta$. This implies that
  Condition~\ref{item:floer-constrained-3} is satisfied.

  It remains to check Condition~\ref{item:floer-constrained-2}. We
  prove by induction on $h$ that $\HFKnew(K, h)$ is supported in
  degrees $ h-\delta-1$ and $h-\delta$. For $h \ll 0$, we have that
  $\HFKnew(K, h) = \HFK(K, h) = 0$. For the inductive step, suppose
  the claim is true for $h$ and let us prove it for $h+1$. From
  Equation~\eqref{eq:tau} we see that $\HFK(K, h)$ is supported in degrees
  $ h-\delta-1$, $h-\delta$ and possibly $0$ (where the last comes
  from an element mapping nontrivially to $\HFhat(S^3)$). Then, the
  long sequence
  \[
    \cdots\to \HFK_k(K,h) \to \HFK_k(K,h+1) \to \HFKhat_k(K, h+1) \to \cdots,
  \]
  together with the almost thin condition, tells us that that
  $\HFK(K, h+1)$ can only be supported in degrees
  $h-\delta -1, h-\delta, h+1-\delta$ and $0$. Thus, $\HFKnew(K, h+1)$
  is supported in some of the degrees
  $ h-\delta -1, h-\delta, h+1-\delta$. To reach the conclusion, it
  remains to exclude degree $h-\delta-1$. If
  $\HFKnew_{h-\delta-1}(K,h+1) \neq 0$, then by looking at the same
  long exact sequence for higher $h$, and taking into account the
  almost thin condition on $\HFKhat$, we deduce that the elements of
  $\HFKnew_{h-\delta-1}(K,h+1) \subset \HFK_{h-\delta-1}(K,h+1)$
  survive under the maps to any higher $\HFK_{h-\delta-1}(K,h')$. For
  $h' \gg 0$, we get that they survive to $\HFhat(S^3)$, which
  contradicts the definition of $\HFKnew$. The claim is proven.
\end{proof}

\subsection{L-space knots}
Another important class of knots are L-space knots, which were introduced in  \cite{OS-lens}. These are the knots such that large surgeries on them are L-spaces (manifolds such that $\HFhat$ has rank one in each $\Spinc$ structure). For example, all positive torus knots are L-space knots.

\begin{proposition}[Ozsv\'ath-Szab\'o \cite{OS-lens}]
\label{prop:LS}
Given an L-space knot $K$ there exist two increasing sequences of integers
\[ n_{-l} < \dots < n_l, \ \ \delta_{-l} < \dots < \delta_l=0\]
such that for all $j$ between $-l$ and $l$ we have
\[
  \widehat{\HFK}_{\delta_j}(K, n_j) \cong \Z,
\]
and all other groups $\widehat{\HFK}_k(K, h)$ are zero.
\end{proposition}

\begin{example}
See Figure~\ref{fig:T34} for the knot Floer homology of the torus knot $T(3,4)$ and its mirror.
\end{example}

\begin{figure}
  \centering
  \begin{tikzpicture}[scale=0.5]
    \foreach \case in {0,1}{
      \begin{scope}[xshift=10*\case cm,yshift=-6*\case cm]
        \draw[->] (-3.5,0)--(3.5,0) node[pos=1,anchor=west] {\small $A$};

        \ifnum\case=0
        \draw[->] (0,-6.5)--(0,0.5)  node[pos=1,anchor=south] {\small $\gr$};
        \else
        \draw[->] (0,-0.5)--(0,6.5)  node[pos=1,anchor=south] {\small $\gr$};
        \fi
        
        \foreach \i in {-3,-2,-1,1,2,3}{
          \ifnum\case=0
          \node[anchor=south] at (\i,0) {\tiny $\i$};
          \else
          \node[anchor=north] at (\i,0) {\tiny $\i$};
          \fi
          \node at (\i,0) {\tiny $|$};
        }

        \ifnum\case=0
        \foreach \i in {-6,...,-1}{
          \node[anchor=east] at (0,\i) {\tiny $\i$};
          \node at (0,\i) {$-$};
        }
        \else
        \foreach \i in {6,...,1}{
          \node[anchor=west] at (0,\i) {\tiny $\i$};
          \node at (0,\i) {$-$};
        }
        \fi
        
        \foreach \i/\j [count=\c from 0] in {3/0,2/-1,0/-2,-2/-5,-3/-6}{
          \ifnum\case=0
          \node[inner sep=1pt] (v\c) at (\i,\j) {$\bullet$};
          \else
          \node[inner sep=1pt] (v\c) at (-\i,-\j) {$\bullet$};
          \fi
        }
        
        \foreach \i/\j in {1/2,3/4}{
          \ifnum\case=0
          \draw[->] (v\i)--(v\j);
          \else
          \draw[->] (v\j)--(v\i);
          \fi
        }
    \end{scope}}
  \end{tikzpicture}
\caption{The knot Floer homology of $T(3,4)$ (left) and its mirror (right). Each dot is a generator of $\HFKhat$. The arrows represent the extra differentials in $\CFKhat$ going over the $z$ basepoint (or $X$-markings on the grid). }\label{fig:T34}
\end{figure}

\begin{proposition}
L-space knots and their mirrors are Floer constrained.
\end{proposition}

\begin{proof}
  If $K$ is an L-space knot, we use Proposition~\ref{prop:LS}. Since
  $\delta_l=0$, we must have $\tau(K) = n_l$; that is, in $\CFKhat(K)$
  all the elements in Alexander gradings less than $n_l$ cancel in
  pairs, with the surviving element being in Alexander grading
  $n_l$. Furthermore, because of the monotonicity of the sequences
  $(n_j)$ and $(\delta_j)$, the cancellation can only happen between
  elements in Alexander gradings $n_i$ and $n_{i+1}$, for various
  $i$. It follows that $\HFK(K, h)$ is either trivial or a copy of
  $\Z$ in degree $\delta_j$. This easily implies
  Conditions~\ref{item:floer-constrained-1}
  and~\ref{item:floer-constrained-2} of
  Definition~\ref{def:constrained}. Condition~\ref{item:floer-constrained-3}
  is a direct consequence of Proposition~\ref{prop:LS}.

  For the mirrors, recall from \cite[Proposition 3.7]{OS-knots} that
  \[
    \HFKhat_k(K, h) \cong \HFKhat^{-k}(m(K), -h),
  \]
  where on the right hand side we have Floer cohomology. When $K$ is
  an L-space knot, Proposition~\ref{prop:LS} implies (almost) the same
  conclusion for the mirror $m(K)$, with $n_i$ and $\delta_i$ replaced
  by $-n_{-i}$ and $-\delta_{-i}$. (We use here the universal
  coefficients formula for cohomology.) The one difference is that we
  now have $\delta_{-l} = 0$ instead of $\delta_l =0$. Therefore, in
  this case the surviving element in $\HFhat(S^3)$ comes at the
  beginning: $\tau(m(K)) = n_{-l}$. It follows that $\HFK(K, h)$ is
  either trivial or consists of copies of $\Z$ in at most two
  gradings: $\delta_j$ and $0$. (The latter appears for example for
  $h=-2$ on the right hand side of Figure~\ref{fig:T34}.)
  Nevertheless, we still have that $\HFKnew(K, h)$ is either trivial
  or is $\Z$ in degree $\delta_j$. As before, we deduce that the
  conditions in Definition~\ref{def:constrained} are satisfied.
\end{proof}

\subsection{Another example}
Most small knots are thin, almost thin, or L-space; or, even if they
do not fit into any of these categories, they may still be Floer
constrained. (An example of this kind is the knot $13n5016$.) One has
to go quite far in the list of knots to find one that is not Floer
constrained. An example is the $14$-crossing knot $14n26580$.  Its
knot Floer homology, computed using the program \cite{HFKCalc}, is
depicted in Figure~\ref{fig:14n26580}. Observe that the conditions in
Definition~\ref{def:constrained} are satisfied, with one exception:
the monotonicity of $\HFKhat$ in Alexander gradings $1$ through $3$.

\begin{figure}
  \centering
  \begin{tikzpicture}[scale=0.5,xscale=2]
    \draw[->] (-4.5,0)--(4.5,0) node[pos=1,anchor=west] {\small $A$};
    
    \draw[->] (0,-7.5)--(0,1.5)  node[pos=1,anchor=south] {\small $\gr$};

    \foreach \i in {-4,-3,-2,-1,1,2,3,4}{
      \node[anchor=south] at (\i,0) {\tiny $\i$};
      \node at (\i,0) {\tiny $|$};
    }

    \foreach \i in {-7,...,-1,1}{
      \node[anchor=east] at (0,\i) {\tiny $\i$};
      \node at (0,\i) {$-$};
    }
    
    \foreach \i/\j [count=\c from 0] in {4/1,3.1/0,2.9/0,2/1,2/0,2/-1,
      1.1/0,0.9/0,1/-1,0.1/-1,-0.1/-1,0/-2,
      -0.9/-2,-1.1/-2,-1/-3,-2/-3,-2/-4,-2/-5,
      -2.9/-6,-3.1/-6,-4/-7
    }{
      \node[inner sep=1pt] (v\c) at (\i,\j) {$\bullet$};
    }
        
    \foreach \i/\j in {0/1,2/5,3/6,7/9,8/11,10/12,
    13/15,14/16,17/18,19/20}{
      \draw[->] (v\i)--(v\j);
    }

    \draw[dashed] (2,0) circle (0.5 and 1.5);
    \draw[dashed] (-2,-4) circle (0.5 and 1.5);
  \end{tikzpicture}
\caption{The knot Floer homology of $14n26580$. The two ovals are places where it looks like there are two options for the spectrum $\gGS(\Grid, h)$. However, only one option remains after taking into account the differentials.}\label{fig:14n26580}
\end{figure}

Interestingly, if we look only at $\HFKhat$, there seems to be some
flexibility in the spectrum in two places: Alexander gradings $-2$ and
$2$, where the support is in three homological gradings so
Proposition~\ref{prop:hurewicz} cannot be applied. When $A=-2$, from
knowing only the homology $\HFKhat(K,-2)$, there are two possible
choices for the associated graded spectrum $\gGS(\Grid, -2)$: it could
be $\SphereS^{-5} \vee \SphereS^{-4} \vee \SphereS^{-3}$ or
$\Sigma^{-7}\mathbb{CP}^2 \vee \SphereS^{-4}$. (We abuse notation and
write $\mathbb{CP}^2$ to denote its suspension spectrum
$\mathbb{CP}^2\smas\SphereS$.) However, the additional information in
$\HFK(K, h)$ eliminates the second choice, because the conditions in
Definition~\ref{def:constrained} hold in that grading range, and we
can use the arguments in the proof of
Proposition~\ref{prop:constrained}. More concretely, we can argue that
that the second Steenrod square $\Sq_2$ (viewed as a stable homology
operation) cannot map the generator in grading $-3$ to the generator
in grading $-5$, since it would not commute with the extra
differentials (shown by arrows in Figure~\ref{fig:14n26580}).

In Alexander grading $2$, we have two possibilities for the spectrum
$\gGS(\Grid, 2)$: it could be $\SphereS^{-1} \vee \SphereS^0 \vee \SphereS^1$ or
$\Sigma^{-3}\mathbb{CP}^2 \vee \SphereS^{0}$. This time, because of the
failure of monotonicity for $\HFKhat$, both possibilities are
compatible with the extra arrows. They each give rise to a unique
possibility for the filtered spectrum $\GS(\Grid)$: the only ambiguity
is in the map
\[
  \GS(\Grid, 1) \simeq \SphereS^0  \to \GS(\Grid, 2) \simeq \SphereS^{-1} \vee \SphereS^0
\]
or, more precisely, in its composition with projection to the $\SphereS^{-1}$
summand. The resulting map $\SphereS^0 \to \SphereS^{-1}$ can be either the zero map
(in which case $\gGS(\Grid, 2)$ is a wedge of spheres) or the Hopf map
(in which case $\gGS(\Grid, 2)$ has the $\Sigma^{-3}\mathbb{CP}^2$
summand). Our expectation is that the first case occurs. While the
Hopf map is not excluded by the structure of the filtration on
$\GS(\Grid)$, it would be excluded if we could enhance the spectrum to
also include domains that go over $O_1$. Indeed, apart from the 
arrows in Figure~\ref{fig:14n26580} (which come from going over $z$
basepoints, i.e., the $X$-markings), we would also have arrows from
going over $w$ basepoints, i.e., the $O$-markings. The two kinds of
arrows are related by the symmetry $(A, \gr) \to (-A, \gr-2A)$ in the
plane, under which the two ovals in Figure~\ref{fig:14n26580}
are interchanged. We conjecture that this symmetry lifts to the level
of spectra. Assuming that, the fact that $\gGS(\Grid, -2)$ was a wedge
of spheres would imply the same about $\gGS(\Grid, 2)$.

\subsection{An algebraic example}
Given the (proved and conjectural) constraints on the filtered
spectrum $\GS(\Grid)$ from the chain level, the reader may wonder if
these always force the knot Floer spectra to be wedges of spheres. We
give here a potential counterexample: a model for a chain complex
$\CFKhat(K)$ that satisfies all the structural properties of a knot
Floer complex, and may produce an interesting Floer spectrum. Of
course, the catch is that we have not found any knot $K$ with this
knot Floer complex.

Consider the complex shown in Figure~\ref{fig:potential}. Suppose it
appears as $\CFKhat(K)$ for a knot $K$. Then, the homology
$\HFKhat(K, -1) =\HFK(K, -1)$ is supported in gradings $-2$ and
$0$. There are two possibilities for the spectrum $\GS(\Grid, -1)$: it
could be either $\SphereS^{-2} \vee \SphereS^0$ or $\Sigma^{-4} \CP^2$, according to
whether the attaching map between the two cells is zero or the Hopf
map. Because of compatibility with the differentials, each possibility
uniquely determines the rest of the filtered spectrum $\GS(\Grid)$. In
the first case, all the spectra $\GS(\Grid, h)$ are wedges of spheres,
whereas in the more interesting second case, we have
\[
  \GS(\Grid, -1) \simeq \Sigma^{-4} \CP^2, \ \ \GS(\Grid, 0) \simeq \SphereS^0 \vee \Sigma^{-3} \CP^2, \ \ \GS(\Grid, 1) \simeq \SphereS^0.
\]
Further, the map $\GS(\Grid, -1) \to \GS(\Grid, 0)$ is trivial,
whereas the map $\GS(\Grid, 0) \to \GS(\Grid, 1)$ maps the summand
$\SphereS^0$ isomorphically to $\SphereS^0$, and vanishes on the other summand.
\begin{figure}
  \centering
  \begin{tikzpicture}[scale=0.5,xscale=2]
    \draw[->] (-1.5,0)--(1.5,0) node[pos=1,anchor=west] {\small $A$};
    
    \draw[->] (0,-2.5)--(0,2.5)  node[pos=1,anchor=south] {\small $\gr$};

    \foreach \i in {-1,1}{
      \node[anchor=south] at (\i,0) {\tiny $\i$};
      \node at (\i,0) {\tiny $|$};
    }

    \foreach \i in {-2,-1,1,2}{
      \node[anchor=east] at (0,\i) {\tiny $\i$};
      \node at (0,\i) {$-$};
    }
    
    \foreach \i/\j [count=\c from 0] in {1/2,1/0,0.1/1,-0.1/1,0/0,0.1/-1,-0.1/-1,-1/0,-1/-2}{
      \node[inner sep=1pt] (v\c) at (\i,\j) {$\bullet$};
    }
        
    \foreach \i/\j in {0/2,1/5,3/7,6/8}{
      \draw[->] (v\i)--(v\j);
    }

  \end{tikzpicture}
\caption{A potential knot Floer complex that does not uniquely determine its Floer spectrum.}  \label{fig:potential}
\end{figure}

\begin{remark}
The discussion in this section makes clear the difficulties in finding a knot whose Floer spectrum contains more information than the knot Floer complex. Once the functoriality properties of knot Floer spectra (under knot cobordisms) are established, an easier thing to look for would be a non-trivial map between spectra that is trivial on homology. Many such maps exist even between spheres (e.g., the Hopf map).
\end{remark}
\bibliographystyle{plain}
\bibliography{gridbibfile}
\vskip10pt
\end{document}